\documentclass[12pt,reqno]{amsart}
\usepackage{amsmath,amsthm,amssymb,mathrsfs,stmaryrd,color,mathtools}

\usepackage[all]{xy}
\usepackage{url}
\usepackage{tikz-cd}
\usepackage[margin=1in,footskip=.5in]{geometry}

\usepackage[utf8]{inputenc}
\usepackage[T1]{fontenc}

\usepackage{relsize}
\usepackage[bbgreekl]{mathbbol}
\usepackage{amsfonts}

\DeclareSymbolFontAlphabet{\mathbb}{AMSb}
\DeclareSymbolFontAlphabet{\mathbbl}{bbold}

\usepackage{enumitem}
\usepackage[colorlinks=true,hyperindex, linkcolor=magenta, pagebackref=false, citecolor=cyan,pdfpagelabels]{hyperref}
\usepackage[capitalize]{cleveref}

\theoremstyle{plain}
\newtheorem{thm}{Theorem}[section]
\newtheorem{thm2}{Theorem}

\newtheorem{conv}[thm]{Convention}

\newtheorem{prop}[thm]{Proposition}
\newtheorem{prop2}[thm2]{Proposition}

\newtheorem{cor}[thm]{Corollary}

\newtheorem{lem}[thm]{Lemma}

\theoremstyle{definition}
\newtheorem{defi}[thm]{Definition}

\newtheorem{rmk}[thm]{Remark}

\newtheorem{exa}[thm]{Example}

\newtheorem{const}[thm]{Construction}
\newtheorem{ass}[thm]{Assumption}

\newtheorem{sett}[thm]{Setting}

\numberwithin{equation}{section}

\newcommand{\bb}[1]{\mathbb{#1}}
\newcommand{\cl}[1]{{\mathcal{#1}}}

\newcommand{\msf}[1]{{\mathsf{#1}}}
\newcommand{\mfr}[1]{{\mathfrak{#1}}}
\newcommand{\mrm}[1]{{\mathrm{#1}}}

\newcommand{\mbf}[1]{\mathbf{#1}}

\newcommand{\ov}[1]{{\overline{#1}}}
\newcommand{\und}[1]{{\underline{#1}}}
\newcommand{\wtd}[1]{{\widetilde{#1}}}

\newcommand{\Qla}{{\overline{\mathbb{Q}}_\ell}}
\newcommand{\Qlax}{{\overline{\mathbb{Q}}_\ell^\times}}

\newcommand{\LG}{{{}^LG}}

\newcommand{\LM}{{{}^LM}}

\newcommand{\HT}{{\operatorname{HT}}}

\newcommand{\Flag}{{\mathscr{F}\!\ell}}

\newcommand{\BL}{{\operatorname{BL}}}

\newcommand{\Lie}{{\operatorname{Lie}}}
\newcommand{\Img}{{\operatorname{Im}}}

\newcommand{\ints}{{\operatorname{int}}}

\newcommand{\Hck}{{\operatorname{Hck}}}

\newcommand{\BdR}{B_{\operatorname{dR}}}
\newcommand{\BdRp}{B^+_{\operatorname{dR}}}

\newcommand{\Eq}{{\operatorname{Eq}}}

\newcommand{\Gr}{{\operatorname{Gr}}}
\newcommand{\GM}{{\operatorname{GM}}}
\newcommand{\Spd}{{\operatorname{Spd}}}
\newcommand{\Spec}{{\operatorname{Spec}}}
\newcommand{\Spa}{{\operatorname{Spa}}}
\newcommand{\Spf}{{\operatorname{Spf}}}
\newcommand{\Perf}{{\operatorname{Perf}}}
\newcommand{\Perfd}{{\operatorname{Perfd}}}

\newcommand{\Bun}{{\operatorname{Bun}}}

\newcommand{\Sh}{{\operatorname{Sh}}}
\newcommand{\univ}{{\operatorname{univ}}}
\newcommand{\lis}{{\operatorname{lis}}}

\newcommand{\un}{{\operatorname{un}}}

\newcommand{\et}{{\acute{e}t}}

\newcommand{\perf}{{\operatorname{perf}}}

\newcommand{\bas}{{\operatorname{bas}}}

\newcommand{\Ker}{{\operatorname{Ker}}}

\newcommand{\can}{{\operatorname{can}}}

\newcommand{\id}{{\operatorname{id}}}

\newcommand{\colim}{{\operatorname{colim}}}
\newcommand{\pr}{{\operatorname{pr}}}
\newcommand{\pre}{{\operatorname{pre}}}
\newcommand{\ad}{{\operatorname{ad}}}
\newcommand{\red}{{\operatorname{red}}}

\newcommand{\Div}{{\operatorname{Div}}}

\newcommand{\Hom}{{\operatorname{Hom}}}

\newcommand{\cInd}{{\operatorname{cInd}}}
\newcommand{\tr}{{\operatorname{tr}}}
\newcommand{\Res}{{\operatorname{Res}}}
\newcommand{\res}{{\operatorname{res}}}

\newcommand{\End}{{\operatorname{End}}}

\newcommand{\der}{{\operatorname{der}}}

\newcommand{\IC}{{\operatorname{IC}}}
\newcommand{\Nm}{{\operatorname{Nm}}}

\newcommand{\asym}{{\operatorname{asym}}}
\newcommand{\sym}{{\operatorname{sym}}}

\newcommand{\diag}{{\operatorname{diag}}}

\newcommand{\FS}{{\operatorname{FS}}}

\newcommand{\CM}{{\operatorname{CM}}}
\newcommand{\Rep}{{\operatorname{Rep}}}
\newcommand{\Irr}{{\operatorname{Irr}}}
\newcommand{\Gal}{{\operatorname{Gal}}}

\newcommand{\GL}{{\operatorname{GL}}}
\newcommand{\GSp}{{\operatorname{GSp}}}

\newcommand{\Sht}{{\operatorname{Sht}}}

\newcommand{\sm}{{\operatorname{sm}}}

\newcommand{\Ad}{{\operatorname{Ad}}}
\newcommand{\ab}{{\operatorname{ab}}}

\newcommand{\ord}{{\operatorname{ord}}}

\usepackage[citestyle=alphabetic,bibstyle=alphabetic,backend=bibtex, maxalphanames=10, maxnames=10, url=true, doi=false]{biblatex}
\title{Geometry around unramified CM points}
\author{Yuta Takaya}
\address{Graduate School of Mathematical Sciences, The University of Tokyo, 3-8-1 Komaba,
Meguro-ku, Tokyo 153-8914, Japan}
\email{takaya@ms.u-tokyo.ac.jp}
\begin{document}
\begin{abstract}
    We compute an explicit contribution of supercuspidal representations arising from Yu's construction to the middle \'{e}tale cohomology of local Shimura varieties. For the proof, we construct special affinoids in local Shimura varieties whose reductions are isomorphic to Heisenberg Deligne-Lusztig varieties and then apply positive-depth Deligne-Lusztig theory. Along the way, we propose a stacky analogue of special affinoids using the stack of local shtukas and observe that the `general canonical subgroup problem' can be solved purely on the stack of local shtukas, which yields the corresponding claim on Shimura varieties via the Hodge-Tate period map. 
\end{abstract}

\maketitle
\setcounter{tocdepth}{2}
\tableofcontents

\section*{Introduction}

\subsection{Main results}

Let $F$ be a finite extension of $\bb{Q}_p$ and let $G$ be a connected reductive group over $F$. The local Langlands correspondence associates an $L$-parameter $\varphi_\pi$ to each smooth irreducible representation $\pi$ of $G(F)$. Although the local Langlands correspondence is established for $\GL_n(F)$ by \cite{HT01}, it remains largely open for general $G$. Currently, there are two approaches toward that generality: a geometric approach developed by Fargues-Scholze \cite{FS24} in the style of geometric Langlands, and a type-theoretic approach of Kaletha (\cite{Kal19} and \cite{Kal21}) based on Yu's construction \cite{Yu01}. On the one hand, Fargues-Scholze constructed a semisimple $L$-parameter $\varphi_\pi^\FS$, called the \textit{Fargues-Scholze parameter}, for every $\pi$ that should be the semisimplification of $\varphi_\pi$. On the other hand, Kaletha formed supercuspidal $L$-packets consisting of regular (or nonsingular) supercuspidal representations. 

Recently, the compatibility of these two constructions has been an active research area. Hansen \cite{Han26} deduced the constancy of Fargues-Scholze parameters in Kaletha's regular supercuspidal $L$-packets from their atomic stability\footnote{It is verified at least for \textit{sufficiently large} $p$ by \cite[Corollary 3.4.12]{Var24}.}, and a more direct comparison of $L$-parameters has been carried out by Cotner-Feng (see \cite{Cot26}, \cite{CF26_I} and \cite{CF26_II}) and independently announced by Fu in the depth-zero case. In this paper, we will study the contribution of supercuspidal representations arising from Yu's construction to the \'{e}tale cohomology of local Shimura varieties. 

In the literature, geometric realization of Yu's construction has been studied in the special fiber using positive-depth Deligne-Lusztig varieties (see e.g.\ \cite{CO25}, \cite{IN26} and \cite{Tak26_Yu} for recent developments and further references). Our method is to relate local Shimura varieties to positive-depth Deligne-Lusztig varieties via the nearby cycles of \textit{special affinoids}. 

To simplify the introduction, suppose that $G$ is split. Let $x \in \cl{B}(G, F)$ be a hyperspecial point of the reduced Bruhat-Tits building and let $\{ \mu \}$ be a geometric conjugacy class of minuscule cocharacters of $G$. Let $[b] \in B(G, \mu)_\bas$ be the unique basic $\sigma$-conjugacy class bounded by $\{\mu \}$ and let $G_b$ be the associated extended pure inner form of $G$. To each $x$ and $\{\mu\}$, we associate a \textit{common} unramified elliptic twisted Levi subgroup 
\[
    G \hookleftarrow M \hookrightarrow G_b
\]
satisfying $x \in \cl{B}(M, F)$ (see \Cref{ssec:unramified_length_0_triple}). Let $x_b \in \cl{B}(G_b, F)$ be the image of $x$. When $G = \GL_n$ and $\mu(t) = \diag(t^{-1}, 1, \ldots, 1)$, the above sequence reduces to
\[
    G = \GL_n \hookleftarrow F_n^\times \hookrightarrow D^\times = G_b. 
\]
Here, $F_n / F$ is the unramified extension of degree $n$ and $D / F$ is the division algebra with invariant $1 / n$. This corresponds to the Lubin-Tate case treated in \cite{Yos10} and \cite{BW16}. 

In Yu's construction, we consider transfers of irreducible supercuspidal representations along $M \subset G$ and $M \subset G_b$. Let $\rho$ be a smooth irreducible representation of $M(F)_x$ such that $\tau = \cInd_{M(F)_x}^{M(F)} \rho$ is irreducible (hence supercuspidal) and $\rho$ is `suitably' generic of depth $n \geq 1$ (see \Cref{defi:supergeneric_representation} for the precise condition). 
Then, we get irreducible supercuspidal representations
\[
    \pi = \cInd_{G(F)_{x, n/2} M(F)_x}^{G(F)} (\kappa_{x, n, \phi} \otimes \rho) ,\quad 
    \pi_b = \cInd_{G_b(F)_{x_b, n/2} M(F)_x}^{G_b(F)} (\kappa_{x_b, n, \phi} \otimes \rho). 
\]
Here, $\kappa_{x, n, \phi}$ and $\kappa_{x_b, n, \phi}$ are twisted Heisenberg-Weil representations introduced in \cite{FKS23} (and in \cite{Tak26_Yu} when $p = 2$). They are associated to a generic character $\phi \colon M(F) \to \Qlax$ of depth $n$ that extends the central character $\rho\vert_{M(F)_{x, n}}$. Our main result is a geometric realization of these transfers under the relevant Hecke operator 
\[
    i_b^* T_{-\mu} i_{1!} \colon \cl{D}(G(F), \Qla) \to \cl{D}(G_b(F), \Qla)^{BW_F}
\]
introduced in \cite{FS24}. Now, let $E$ be the splitting field of $M$.  

\begin{thm2}\textup{(\Cref{cor:Hecke_operator_unramified_CM})} \label{thm2:middle_Hecke_operator} 
    In the above setting, we have an inclusion
    \[
        \pi_b \boxtimes \xi_E \subset H^0(i_b^* T_{-\mu} i_{1!}(\pi))
    \] 
    for some smooth character $\xi_E \colon W_E \to \Qlax$ after restricting the Weil action to $W_E$. In particular, we get an equality $\varphi^\FS_\pi = \varphi^\FS_{\pi_b}$ of Fargues-Scholze parameters. 
\end{thm2}

More generally, we expect that $\varphi_\pi^\FS = \varphi_\tau^\FS = \varphi_{\pi_b}^\FS$ should hold for a suitable $L$-embedding $\LM \subset \LG$. Notably, those parameters could be non-supercuspidal since $\tau$ can be a character twist of a unipotent supercuspidal representation. The main feature of \Cref{thm2:middle_Hecke_operator} is the appearance of $\pi_b$ in the \textit{middle} degree of $i_b^* T_{-\mu} i_{1!}(\pi)$. 
This phenomenon is beyond the scope of the usual vanishing conjecture (e.g.\ \cite[Conjecture 1.1 (2)]{HJ25}) since $\varphi_\pi^\FS$ could be non-supercuspidal. The difficulty of these problems lies in the fact that the degree information is lost once one passes to character computations such as \cite{HKW22}. 

The strategy of our proof is to find \textit{special affinoids} in the relevant local Shimura variety $\cl{M}_{G, b, \mu, K}$. This approach had previously been applied only to Lubin-Tate spaces for a long time (e.g.\ \cite{Yos10} and \cite{BW16}) and the author \cite{Tak25z} recently found a way to apply this method to general local Shimura varieties for depth-zero regular supercuspidal representations. Here, we develop a positive-depth generalization of \cite{Tak25z}. 

Let $K_\phi \subset G(F)_{x, n, n/2}$ be a compact open subgroup containing $G(F)_{x, n+, n/2 +}$ such that 
\[
    K_\phi / G(F)_{x, n+, n/2+} \cong \Ker(X) \subset \msf{m}_{x, n} = M(F)_{x, n} / M(F)_{x, n+}
\]
for a suitable $k$-linear map $X\colon \msf{m}_{x, n} \to k$ associated to $\phi$ (see \Cref{ssec:special_affinoids_at_positive_depth}). Let $m = \lceil \tfrac{n}{2} \rceil$. We construct a positive-depth special affinoid $\cl{W}(n)_\phi$ satisfying the following properties. 

\begin{thm2} \textup{(\Cref{thm:property_of_special_affinoids})} \label{thm2:HDLV_as_reductions}
    There is an explicit affinoid $\cl{W}(n)_{\phi, \bb{C}_p} \subset \cl{M}_{G, b, \mu, K_\phi, \bb{C}_p}$ with good reduction of the local Shimura variety at level $K_\phi$ such that its reduction $\msf{W}(n)_\phi$ is isomorphic to a Heisenberg Deligne-Lusztig variety for $n \geq 2$. 
\end{thm2}

Here, Heisenberg Deligne-Lusztig varieties refer to those varieties  studied in \cite[Theorem 6]{Tak26_Yu} and their \'{e}tale cohomology geometrically realizes twisted Heisenberg-Weil representations. The reduction $\msf{W}(n)_\phi$ is a generalization of the variety $X$ studied in \cite{BW16} for Lubin-Tate spaces and the exceptional behavior at depth $n = 1$ is a common feature as in \cite{BW16}. Its \'{e}tale cohomology  lifts to local Shimura varieties via the nearby cycles functor. 

Our proof of \Cref{thm2:HDLV_as_reductions} follows the method of \cite{Tak25z} and relies on the explicit $\cl{G}$-BKF module over the universal deformation introduced in \cite{Ito25a}. One main advantage of our method over \cite{BW16} is that we directly work with local Shimura varieties \textit{at finite levels} and our special affinoids have good reduction. 

\subsection{Canonical level structures}

In modern $p$-adic geometry, stacky approaches have been successfully applied to many situations. We adopt this perspective and formulate a stacky analogue of special affinoids, which we call special open neighborhoods (see \Cref{defi:special_open_neighborhoods}), in terms of the stack $\Sht_{\cl{G}}$ of local shtukas. They are formulated so that \Cref{thm2:middle_Hecke_operator} follows formally from the existence of special open neighborhoods. In our context, they are directly constructed from explicit geometry of local Shimura varieties (see \Cref{prop:restatement_LSV}). 

In the course of developing a stacky interpretation, we find that a particular part of our construction on canonical trivializations (see \Cref{sec:canonical_trivialization}) can be generalized to \textit{local special points}, which we introduce in \Cref{sec:canonical_level_structures} as a local analogue of special points on Shimura varieties. It turns out that the ordinary case solves the `general canonical subgroup problem' on Shimura varieties via the pullback along the Hodge-Tate period map
\[
    \pi_\HT \colon \Sh_{\msf{K}}(\mbf{G}, \mbf{X})_E^\diamondsuit \to \Sht_{\cl{G}, -\mu}. 
\]
Here, $\cl{G}$ is a special maximal parahoric group scheme of $G = \mbf{G}_{\bb{Q}_p}$ associated to $x \in \cl{B}(G, \bb{Q}_p)$ and $\mu$ is a minuscule cocharacter of $G$ defined over $\bb{Q}_p$. The latter condition on $\mu$ is necessary for the existence of the ordinary locus (e.g.\ \cite[{}1.6.3]{Wed99}). In this setting, we introduce the ordinary locus $\Sht_{\cl{G}, -\mu}^\ord \subset \Sht_{\cl{G}, -\mu}$ (see \Cref{ssec:ordinary_locus}), which is isomorphic to $[\Spd(\bb{Q}_p) / \und{\cl{P}_\mu(\bb{Z}_p)}]$ with $\cl{P}_\mu = \cl{G}^{\mu \leq 0}$. 

\begin{prop2}\textup{(\Cref{cor:canonical_filtration})} \label{cor2:canonical_filtration}
    The ordinary locus $\lvert \Sh_{\msf{K}}(\mbf{G}, \mbf{X})_E\rvert^{\ord} = \pi_\HT^{-1}(\Sht_{\cl{G}, -\mu}^\ord)$ of the Shimura variety $\Sh_{\msf{K}}(\mbf{G}, \mbf{X})_E$ admits a decreasing family of open neighborhoods
    \[
        \Sh_{\msf{K}}(\mbf{G}, \mbf{X})_E^{\ord}(r) \subset \Sh_{\msf{K}}(\mbf{G}, \mbf{X})_E^\ad
    \]
    indexed by $r > 0$ with the following properties. 
    \begin{enumerate}
        \item Let $\msf{K}_{p, r} = G(\bb{Q}_p)_{x, r} \cl{P}_\mu(\bb{Z}_p)$.
        There is a section
        \[
            \can_r \colon \Sh_{\msf{K}}(\mbf{G}, \mbf{X})_E^{\ord}(r) \to \Sh_{\msf{K}_{p, r} \msf{K}^p}(\mbf{G}, \mbf{X})_E^\ad
        \]
        extending the truncation of the canonical filtration on $\Sh_{\msf{K}}(\mbf{G}, \mbf{X})_E^{\diamondsuit, \ord}$. 
        \item We have $\bigcap_{r > 0} \lvert \Sh_{\msf{K}}(\mbf{G}, \mbf{X})_E^{\ord}(r) \rvert = \lvert \Sh_{\msf{K}}(\mbf{G}, \mbf{X})_E\rvert^{\ord}$ and the composition map
        \[
            U_p \colon \Sh_{\msf{K}}(\mbf{G}, \mbf{X})_E^{\ord}(1) \xrightarrow{\can_1} \Sh_{\msf{K}_{p, 1} \msf{K}^p}(\mbf{G}, \mbf{X})_E^\ad \xrightarrow{\mu(p^{-1})} \Sh_{\msf{K}}(\mbf{G}, \mbf{X})_E^\ad
        \]
        restricts to a finite \'{e}tale surjection $\Sh_{\msf{K}}(\mbf{G}, \mbf{X})_E^{\ord}(r+1) \to \Sh_{\msf{K}}(\mbf{G}, \mbf{X})_E^{\ord}(r)$. 
    \end{enumerate}
\end{prop2}

For some Shimura varieties of PEL type, $\can_r$ was constructed by \cite[Theorem A]{GK12} and \cite[Théorème 8]{Far11} using integral models and finding such a map was called the `general canonical subgroup problem' in \cite[Introduction]{GK12}. We observe in \Cref{thm:canonical_level_structures} that this phenomenon can be proved purely locally on the stack of local shtukas at every local special point. In fact, the ordinary case is close to the analysis of \cite{BP21} for flag varieties. 

We think that local special points provide a natural class that contains both the ordinary locus and CM points as two extreme cases. For our application, unramified CM points are relevant in the construction of $\cl{W}(n)_\phi$. We specify special open balls $\cl{U}(n) \subset \cl{M}_{\cl{G}, b, \mu}$ at hyperspecial level and then construct canonical trivializations
\[
    \can_{n, m} \colon \cl{U}(n)_{\bb{C}_p} \to \cl{M}_{G, b, \mu, G(F)_{x, n, m}}. 
\]
Our special affinoid $\cl{W}(n)_\phi$ is constructed as the inverse image of $\Img(\can_{n, m})$ to $\cl{M}_{G, b, \mu, K_\phi}$. In \Cref{sec:cohomology_computation}, we show that $\cl{W}(n)_\phi \to \cl{U}(n)$ has good reduction and the reduction $\msf{W}(n)_\phi \to \msf{U}(n)$ is a finite \'{e}tale torsor of an affine space studied in \cite{Tak26_Yu}. 

\addtocontents{toc}{\protect\setcounter{tocdepth}{1}}

\subsection*{The structure of the paper}

In \Cref{sec:canonical_level_structures}, we introduce local special points and show explicit overconvergence of truncations of canonical level structures at local special points. In \Cref{sec:special_open_neighborhoods}, we formulate special open neighborhoods and introduce relevant unramified CM points. In \Cref{sec:depth_zero_case}, we study the period map $\pi_\GM$ and review the depth-zero case $n = 0$ by incorporating the results of \cite{Tak25z}. In \Cref{sec:special_open_balls}, we introduce special open balls $\cl{U}(n)$ for $n > 0$ and study their stability under group actions. In \Cref{sec:canonical_trivialization}, we construct canonical trivializations $\can_{n, m}$. In \Cref{sec:cohomology_computation}, we introduce special affinoids $\cl{W}(n)_\phi$ and compute their \'{e}tale cohomology by relating the reductions $\msf{W}(n)_\phi$ to Heisenberg Deligne-Lusztig varieties. In \Cref{sec:summary}, we summarize the results and deduce \Cref{thm2:middle_Hecke_operator}. 

\subsection*{Acknowledgements}
I would like to thank my advisor Yoichi Mieda for his constant support and encouragement. I would also like to thank Amoru Fujii, Naoki Imai, Haruto Onishi and Kazuki Tokimoto for helpful discussions related to this work. This work was supported by the WINGS-FMSP program at the Graduate School of Mathematical Sciences, the University of Tokyo and JSPS KAKENHI Grant number JP24KJ0865.

\subsection*{Notation}

We fix two primes $p \neq \ell$ and the coefficient field of representations will always be $\Qla$. All rings are assumed to be commutative. The reduced underlying scheme of a formal scheme $\mfr{X}$ is denoted by $\mfr{X}_\red$. The perfection of a scheme $X$ over $\bb{F}_p$ is denoted by $X^\mrm{perf}$. For a scheme $X$ over a ring $A$, its base change to an $A$-algebra $B$ is denoted by $X_B$ or $X\otimes_A B$. For a field $k$, an algebraic closure of $k$ is denoted by $\ov{k}$. 

For a locally profinite group $K$ that is compact modulo its center $Z$ and a central character $\chi \colon Z \to \Qlax$, $\cl{D}(K, \Qla)_\chi$ denotes the derived category of smooth representations of $K$  over $\Qla$ with central character $\chi$ and let $\Rep(K)_\chi = \cl{D}(K, \Qla)_\chi^\heartsuit$. Let $\Irr^\sm(K)$ be the set of irreducible smooth representations of $K$. For every $K$ in our context, Schur's lemma holds, i.e.\ $\End_K(\rho) \cong \Qla$ for $\rho \in \Irr^\sm(K)$ (see \cite[Section 2.6]{BH06}). For $\rho \in \Irr^\sm(K)$ with central character $\chi$ and $V \in \Rep(K)_\chi$, the $\rho$-isotypic part of $V$ is denoted by 
\[
    V[\rho] = \Hom_K(\rho, V) \otimes \rho \subset V. 
\]
As $K$ is compact modulo $Z$, it is an exact functor on $\Rep(K)_\chi$ and extends to $\cl{D}(K, \Qla)_\chi$. 

The center of an algebraic group $G$ is denoted by $Z_G$. The identity component of $G$ is denoted by $G^\circ$. For a cocharacter $\lambda \colon \bb{G}_m \to G$ over a ring $A$, $\lambda(a)$ is also denoted by $a^\lambda$ for $a\in A^\times$. 
Let $G^{\lambda=0} \subset G$ denote the centralizer of $\lambda$ and let $G^{\lambda \geq 0} \subset G$ denote the parabolic subgroup
\[
	G^{\lambda \geq 0} = \{ g\in G \mid \lim_{t \to 0} \lambda(t) g \lambda(t)^{-1} \hspace{3pt} \mrm{exists}\}. 
\]
Let $G^{\lambda > 0}$ denote the unipotent radical of $G^{\lambda \geq 0}$. Note that $G^{\lambda \leq 0}$ is also denoted by $P_\lambda$. For a torus $T$, $X^*(T)$ (resp.\ $X_*(T)$) denotes the character (resp.\ cocharacter) group of $T$ and the natural pairing between $X^*(T)$ and $X_*(T)$ is denoted by $\langle - , - \rangle$. 

Let $F$ be a non-archimedean local field with ring of integers $O_F$ and residue field $k$. Let $\pi$ be a uniformizer of $F$ and let $q$ be the cardinality of $k$. For a perfect $k$-algebra $R$, let 
\[  
    W_{O_F}(R)= \left\{ \begin{alignedat}{4}
        & W(R) \otimes_{W(k)} O_F & \quad & (\mrm{char}(F) = 0) \\
        & R\llbracket \pi \rrbracket & \quad & (\mrm{char}(F) = p). 
    \end{alignedat} \right.
\]
The Frobenius on $W_{O_F}(R)$ relative to $k$ is denoted by $\sigma$. Let $\breve{F}$ be the completed maximal unramified extension of $F$ and let $\ov{k}$ be the residue field of $\breve{F}$. For an $O_F$-scheme $X$ and a perfect $k$-algebra $R$, the pullback along $\sigma$ induces a map 
\[
    X(W_{O_F}(R)) \to X(W_{O_F(R)}) ,\quad x \mapsto \sigma^*(x). 
\]
To simplify our notation, $\sigma^*(x)$ is denoted by $\sigma(x)$. 

For a connected reductive group $G$ over $F$, the reduced (resp.\ enlarged) Bruhat-Tits building is denoted by $\cl{B}(G, F)$ (resp.\ $\cl{\wtd{B}}(G, F)$). For each $x \in \cl{B}(G, F)$ and $r \geq 0$, the Moy-Prasad filtrations of $G(F)$ (resp.\ $\Lie(G)$) are denoted by $G(F)_{x, r}$ (resp.\ $\Lie(G)_{x, r}$). For $r > 0$, the graded piece of this filtration is denoted in the sans-serif type by $\msf{g}_{x, r}$ and we will freely use the Moy-Prasad isomorphism 
\[
    \msf{g}_{x, r} = G(F)_{x, r} / G(F)_{x, r+} \cong \Lie(G)_{x, r} / \Lie(G)_{x, r+}. 
\]
The parahoric group scheme of $G$ associated to $x$ is denoted in the calligraphic type by $\cl{G}$ and its maximal reductive quotient $\msf{G} = G(F)_{x, 0} / G(F)_{x, 0+}$ is denoted in the sans-serif type. We will omit the subscript $x$ when there is no confusion. 

Let $R$ be a perfectoid ring over $O_F$. Its tilt is denoted by $R^\flat$ and the Frobenius on $W_{O_F}(R^\flat)$ is denoted by $\varphi_R$. Let
\[
    \theta_R\colon W_{O_F}(R^\flat)\to R
\]
be the canonical surjection and let $\xi_R$ denote a generator of $\Ker(\theta_R)$. When we just say that $R$ is a perfectoid ring, we apply the above notation for $O_F = \bb{Z}_p$. When $F$ is a finite extension of $\bb{Q}_p$, we fix its completed algebraic closure $\bb{C}_p$ and a compatible system $\{\pi^{1 / n}\}_{n \geq 1}$ of roots of $\pi$ in $\bb{C}_p$. There is an associated system $\{ \pi^{\flat, 1 / n} \}_{n \geq 1}$ of pseudo-uniformizers in $\bb{C}_p^\flat$. 

When a $v$-sheaf $X$ is partially proper (i.e.\ $X(R, R^+)$ is independent of the choice of $R^+$), $X(R, R^+)$ is simply denoted by $X(R)$. For a $v$-sheaf $X$ over $\Spd(\bb{Z}_p)$, let $X_s = X \times_{\Spd(\bb{Z}_p)} \Spd(\bb{F}_p)$ and $X_\eta = X \times_{\Spd(\bb{Z}_p)} \Spd(\bb{Q}_p, \bb{Z}_p)$.

\addtocontents{toc}{\protect\setcounter{tocdepth}{2}}

\section{Explicit overconvergence of canonical level structures} \label{sec:canonical_level_structures}

In this section, we introduce local special points on the stack of local shtukas and prove explicit overconvergence of truncations of their canonical level structures (see \Cref{thm:canonical_level_structures}). By specializing to the ordinary locus, we obtain a solution to the `general canonical subgroup problem' on Shimura varieties (see \Cref{cor:canonical_filtration}). 

Here, let $\Perf$ be the category of perfectoid spaces over $k$ and let $\Perfd$ be the category of perfectoid spaces over $F$. All $v$-stacks refer to $v$-stacks on $\Perf$ and $v$-stacks on $\Perfd$ are identified with $v$-stacks over $\Spd(F)$. This notation will be modified after \Cref{sec:special_open_neighborhoods}. 

\subsection{Stacks of local shtukas} \label{ssec:local_shtuka}

In this section, we recall the definition of local shtukas introduced in \cite{SW20}. For each $S \in \Perfd$, there is an analytic adic space $\cl{Y}_{O_F, S}$. When $S = \Spa(R,R^+)$ with a pseudo-uniformizer $\varpi \in R^{\flat+}$, it is given by
\[
    \cl{Y}_{O_F, S} = \Spa(W_{O_F}(R^{\flat+})) - V([\varpi]). 
\]
When there is no confusion, $\cl{Y}_{O_F, S}$ is denoted by $\cl{Y}_S$. For each rational number $r>0$, let 
\[
    \cl{Y}_{S,[0, r]} = \{ \lvert \pi \rvert^r \leq \lvert [\varpi] \rvert \neq 0 \} \subset \cl{Y}_S. 
\]
By the proof of \cite[Proposition II.1.1]{SW20}, we have
\begin{equation}
    \cl{Y}_{S, [0, r]} = \Spa(B_{S,[0, r]},B^+_{S,[0, r]}), \quad B_{S, [0, n]}^\wedge \cong W_{O_F}(R^\flat). \label{eq:YSn}
\end{equation}
Here, the wedge denotes the $\pi$-adic completion. Then, the Fargues-Fontaine curve $\cl{X}_S$ is defined as the quotient 
\[
    \cl{X}_S = \cl{Y}_{S, (0, \infty)} / \varphi^\bb{Z}
\]
where $\cl{Y}_{S, (0, \infty)} = \cl{Y}_{S} - V(\pi)$. Let $G$ be a connected reductive group over $F$ and let $\cl{G}$ be a smooth affine model of $G$ with connected fibers over $O_F$. 

\begin{defi}(\cite[Section 23.1]{SW20})
    For each $S \in \Perfd$, a local $\cl{G}$-shtuka over $S$ is a $\cl{G}$-torsor $\cl{P}$ over $\cl{Y}_S$ equipped with a Frobenius map
    \[
        \varphi_{\cl{P}}\colon \varphi^*\cl{P}\vert_{\cl{Y}_S\backslash S^\sharp} \cong \cl{P}\vert_{\cl{Y}_S\backslash S^\sharp}
    \]
    meromorphic along $S^\sharp$. The $v$-stack of local $\cl{G}$-shtukas on $\Perfd$ is denoted by $\Sht_{\cl{G}}$. 
\end{defi}

Let $\{\mu \colon \bb{G}_m \to G_{\ov{F}}\}$ be a geometric conjugacy class of cocharacters and let $E$ be the reflex field of $(G, \{ \mu \})$. 
A local $\cl{G}$-shtuka $\cl{P}$ over a perfectoid space $S$ over $E$ is bounded by $\mu$ (resp.\ of type $\mu$) if the relative position of $\varphi_{\cl{P}}$ at each geometric point of $S$ is bounded by (resp.\ equal to) $\mu$. The substack of local $\cl{G}$-shtukas bounded by (resp.\ of type) $\mu$ is denoted by
\[
    \Sht_{\cl{G}, \leq \mu} \subset \Sht_{\cl{G}, E} \quad (\textup{resp.}\ \Sht_{\cl{G}, \mu} \subset \Sht_{\cl{G}, E}). 
\]
Here, the first inclusion is closed and $\Sht_{\cl{G}, \mu} \subset \Sht_{\cl{G}, \leq \mu}$ is open. When $\mu$ is minuscule, we have $\Sht_{\cl{G}, \mu} = \Sht_{\cl{G}, \leq \mu}$. 

An efficient way to describe $\Sht_{\cl{G}, \mu}$ is via the $\BdRp$-affine Grassmannian. For each affinoid perfectoid space $S = \Spa(R,R^+)$, the formal completion along $S \subset \cl{Y}_S$ is denoted by $\BdRp(S)$. A generator of the Cartier divisor $S \subset \cl{Y}_S$ is denoted by $\xi_S$ and let $\BdR(S) = \BdRp(S)[\tfrac{1}{\xi_S}]$. Then, $\BdRp$ and $\BdR$ are partially proper $v$-sheaves on $\Perfd$. 

When $S$ is defined over $E$, there is a natural embedding $S \subset \cl{Y}_{O_E, S}$ and $\BdRp(S)$ is identified with the formal completion along $S \subset \cl{Y}_{O_E, S}$. In particular, $\BdRp(S)$ and $\BdR(S)$ can be regarded as $E$-algebras when $S$ is defined over $E$. 

\begin{defi}\textup{(\cite[Definition 19.1.1]{SW20})}
   The $\BdRp$-affine Grassmannian $\Gr_G$ is the \'{e}tale sheafification of the functor sending $S \in \Perfd$ to $G(\BdR(S)) / G(\BdRp(S))$. The substack $\Gr_{G, \mu} \subset \Gr_{G, E}$ consists of elements $g \in \Gr_G(S)$ that lie in the Schubert cell
   \[
        G(\BdRp) \cdot \mu(\xi) \cdot G(\BdRp)
   \]
   at each geometric point of $S$. 
\end{defi}

An element $g \in G(\BdR(S)) / G(\BdRp(S))$ defines a modification $g^{-1} \colon \cl{E} \dashrightarrow \cl{P}$ at $S$ of the trivial $\cl{G}$-torsor $\cl{E}$ over $\cl{Y}_S$. To get a local $\cl{G}$-shtuka $\cl{P}$ associated to $g$, one further takes a modification $\cl{E} \dashrightarrow \cl{P}$ at $\cup_{n \geq 1} \varphi^n(S)$, so that $\cl{E} \dashrightarrow \cl{P}$ is given by $\varphi^n(g)^{-1}$ at each $\varphi^n(S)$. Then, $\cl{P}$ can be enhanced to a local $\cl{G}$-shtuka and we get a map
\[
    \Gr_{G} \to \Sht_{\cl{G}}. 
\]
After base change to $E$, it restricts to a map $\Gr_{G, \mu} \to \Sht_{\cl{G}, -\mu}$. 

\begin{prop}\textup{(cf.\ \cite[Corollary 3.3.9]{DHKZ26})} \label{prop:ShtvsGr}
    The above map provides the identification
    \[
        \Sht_{\cl{G}, -\mu} \cong \und{\cl{G}(O_F)} \backslash \Gr_{G, \mu}. 
    \] 
\end{prop}
\begin{proof}
    The claim is proved loc. cit. for quasi-parahoric group schemes and the same argument works here. It essentially follows from \cite[Proposition 22.6.1]{SW20}. 
\end{proof}

For this reason, we will use the following notation. 

\begin{defi} \label{defi:local_shtuka_general_level}
    For each open subgroup $K \subset G(F)$, we set $\Sht_{K, -\mu} = \underline{K} \backslash \Gr_{G, \mu}$. 
\end{defi}

\subsection{The ordinary locus} \label{ssec:ordinary_locus}

In this section, we introduce the ordinary locus $\Sht_{\cl{G}, -\mu}^\ord \subset \Sht_{\cl{G}, -\mu}$. It serves as a motivating example of local special points on $\Sht_{\cl{G}, -\mu}$ (see \Cref{ssec:extremal_points}).  

Here, we impose $E=F$ and $\{ \mu \}$ admits an $F$-rational representative $\mu \colon \bb{G}_m \to G$. It is a necessary condition to consider the ordinary locus as observed in \cite[{}1.6.3]{Wed99} for Shimura varieties of PEL type. 

Let $\Bun_G$ be the $v$-stack of $G$-bundles over the Fargues-Fontaine curve. The underlying space of $\Bun_G$ is homeomorphic to the set of $\sigma$-conjugacy classes $B(G)$ (see \cite[Theorem 1.1]{Vie24}). Note that this claim holds over $k$ since Newton points and Kottwitz values are invariant under $\Gal(\ov{k} / k)$. 

For each local $\cl{G}$-shtuka $\cl{P}$ over $S$, its restriction to $\cl{Y}_{S, [r, \infty)}$ for sufficiently large $r$ provides a $G$-bundle over the Fargues-Fontaine curve $\cl{X}_S$. This induces a natural map 
\begin{equation} \label{eq:Beauville_Laszlo}
    \BL_{\cl{G}} \colon \Sht_{\cl{G}, -\mu} \to \Bun_G. 
\end{equation}
The image is equal to the $\mu$-bounded locus $B(G, \mu) \subset B(G)$ (cf.\ \cite[Proposition A.9]{Sch18}). It has a unique maximal element represented by $\mu(\pi)$. 

\begin{defi} \label{defi:ord_locus_stack}
    The closed substack $\Sht_{\cl{G}, -\mu}^\ord = \BL_{\cl{G}}^{-1}([\mu(\pi)])$ is the ordinary locus of $\Sht_{\cl{G}, -\mu}$. 
\end{defi}

Now, we provide a simple description of $\Sht_{\cl{G}, -\mu}^\ord$. For this, we suppose that $\cl{G}$ is a special maximal parahoric group scheme of $G$. Since $\cl{G}$ contains the standard integral model of a maximal $F$-split torus, we may assume that $\mu$ extends to a cocharacter $\mu \colon \bb{G}_m \to \cl{G}$ by replacing $\mu$ with a suitable $G(F)$-conjugate. 

Let $P_{\mu} = G^{\mu \leq 0}$ and $\Flag_{G, \mu} = G / P_{\mu}$. We also set $\cl{P}_{\mu} = \cl{G}^{\mu \leq 0}$. 

\begin{prop} \label{prop:ordinary_locus}
    There is a natural inclusion $\und{\Flag_{G, \mu}(F)} \subset \Gr_{G, \mu}$ inducing an identification 
    \[
        \Sht_{\cl{G}, -\mu}^\ord = \und{\cl{G}(O_F)} \backslash \und{\Flag_{G, \mu}(F)} \cong [\Spd(F) / \und{\cl{P}_{\mu}(O_F)}]
    \]
    when $\cl{G}$ is special maximal parahoric and $\mu$ is defined over $O_F$. 
\end{prop}
\begin{proof}
    The proof of \cite[Corollary IX.7.3]{FS24} uses the fact that the ordinary locus of $\Gr_{G, \mu}$ equals $\und{\Flag_{G, \mu}(F)}$, which is justified by \cite[Proposition 4.7]{GI23} via the Tannakian interpretation. Then, \Cref{prop:ShtvsGr} implies $\Sht_{\cl{G}, -\mu}^\ord = \und{\cl{G}(O_F)} \backslash \und{\Flag_{G, \mu}(F)}$. The second isomorphism follows from the Iwasawa decomposition $G(F) = \cl{G}(O_F) P_\mu(F)$ (see \cite[Theorem 5.3.4]{KP23}). 
\end{proof}

In particular, there is a \textit{canonical filtration}
\[
    \Sht_{\cl{G}, -\mu}^\ord \to \und{\cl{P}_\mu(O_F)} \backslash \Gr_{G, \mu}
\]
on the ordinary locus. 

\subsection{Local special points on the stack of local shtukas} \label{ssec:extremal_points}

In this section, we introduce the notion of local special points on $\Sht_{\cl{G}, -\mu}$ as a generalization of the ordinary locus. It is a local analogue of special points on Shimura varieties (see \Cref{ssec:special_points}) and it also includes CM points. The following datum is the input to specify a special point. 

\begin{defi} \label{defi:local_special_pair}
    A pair $(M, \mu)$ is a \textit{local special pair} for $(G, \{ \mu \})$ if $M \subset G$ is a twisted Levi subgroup and $\mu$ is a representative of $\{ \mu \}$ defined over a finite extension $E / F$ such that 
    \[
        M_{\ov{F}} = \bigcap_{\tau \in \Gal(\ov{F} / F)} G_{\ov{F}}^{\tau\mu = 0}. 
    \]
    We say that $(M, \mu)$ is tamely ramified (resp.\ unramified) if $M$ is tamely ramified (resp.\ unramified). For each $(M, \mu)$, let $H_\mu \subset G$ be the closed subgroup containing $M$ such that 
    \[
        H_{\mu, \ov{F}} = \bigcap_{\tau \in \Gal(\ov{F} / F)} G_{\ov{F}}^{\tau\mu \leq 0}.
    \]
\end{defi}

We will heavily rely on the Bruhat-Tits theory and the Moy-Prasad filtration. For this reason, we will always assume that $M  \subset G$ is tamely ramified. In particular, we follow the following convention throughout this paper. 

\begin{conv} \label{conv:tamely_ramified_minimal}
    In this paper, local special pairs are always assumed to be tamely ramified and $E$ denotes the defining field of $\mu$ (rather than the reflex field of $\{ \mu \}$) unless otherwise stated. 
\end{conv}

\begin{rmk}
    We replace $E$ by a sufficiently ramified Galois extension in some proofs (e.g.\ \Cref{prop:characterize_Kr}, \Cref{thm:canonical_level_structures}). Nevertheless, we follow \Cref{conv:tamely_ramified_minimal} in every statement. 
\end{rmk}

Since $\mu$ is central in $M$, $\Gr_{M, \mu} = \Spd(E)$. Let $x_\mu \in \Gr_{G, \mu}(E)$ denote the point associated to $\Gr_{M, \mu} \to \Gr_{G, \mu}$. For simplicity, we assume that $\cl{G}$ is parahoric. Let $\cl{\wtd{B}}(G, F)$ be the enlarged Bruhat-Tits building of $G$ and let $x \in \cl{\wtd{B}}(G, F)$ be a point such that $\cl{G}(O_F) = G(F)_{x, 0}$. 


\begin{lem} \label{lem:level_structure_special_points}
    The $\cl{G}(O_F)$-orbit of $x_\mu\colon \Gr_{M, \mu} \to \Gr_{G, \mu}$ induces a closed immersion
    \[
        [\Spd(E) / \und{H_{\mu, x}}] \hookrightarrow  \Sht_{\cl{G}, -\mu, E}
    \]
    by setting $H_{\mu, x} = H_\mu(F) \cap \cl{G}(O_F)$. This closed substack is denoted by $\Sht_{\cl{G}, -\mu}^{[M, \mu]}$ and called the \textup{local special point} of $\Sht_{\cl{G}, -\mu}$ associated to $(M, \mu)$. 
\end{lem}

\begin{proof}
    First, we show that the stabilizer of $x_\mu$ under the $G(F)$-action is $H_\mu(F)$. Let $\bb{C}_p$ be the completed algebraic closure of $E$ and let $\xi = \xi_{\bb{C}_p} \in \BdRp(\bb{C}_p)$. For each $g \in G(F)$, $g \cdot x_\mu = x_\mu$ is equivalent to 
    \begin{equation} \label{eq:condition_g}
        g \mu(\xi) \in \mu(\xi) \cdot G(\BdRp(\bb{C}_p)). 
    \end{equation}
    If $g \in H_\mu(F) \subset P_\mu$, we have $\mu(\xi)^{-1} g \mu(\xi) \in P_\mu(\BdRp(\bb{C}_p)) \subset G(\BdRp(\bb{C}_p))$. On the other hand, if \eqref{eq:condition_g} holds, we have $g \in G(\BdRp(\bb{C}_p)) \cap \mu(\xi) \cdot G(\BdRp(\bb{C}_p)) \cdot \mu(\xi)^{-1}$. The right-hand side maps to $P_\mu(\bb{C}_p)$ under the reduction modulo $\xi$. Thus, $g \in P_\mu(\bb{C}_p)$, so we get $g \in H_\mu(F)$ from $G(F) \cap P_\mu(\bb{C}_p) = H_\mu(F)$. Since $H_{\mu, x} = H_\mu(F) \cap \cl{G}(O_F) $, the $\cl{G}(O_F)$-orbit of $x_\mu$ provides an inclusion 
    \[
        [\Spd(E) / \und{H_{\mu, x}}] \hookrightarrow  \Sht_{\cl{G}, -\mu, E}. 
    \]
    It remains to see that it is a closed immersion. For this, it is enough to prove that $\und{\cl{G}(O_F) / H_{\mu, x}} \times \Spd(E) \to \Gr_{G, \mu, E}$ is a closed immersion. Since $\cl{G}(O_F) / H_{\mu, x}$ is a profinite space, the left-hand side is proper over $\Spd(E)$. Since $\Gr_{G, \mu, E} \to \Spd(E)$ is separated (see \cite[Proposition 20.2.3]{SW20}), $\und{\cl{G}(O_F) / H_{\mu, x}} \times \Spd(E) \to \Gr_{G, \mu, E}$ is a proper monomorphism. Thus, it is a closed immersion (see \cite[Lemma 2.1]{AGLR22}).  
\end{proof}

As in the case of the ordinary locus, there is a \textit{canonical level structure}
\[
    \Sht_{\cl{G}, -\mu}^{[M, \mu]} \to \und{H_{\mu, x}} \backslash \Gr_{G, \mu}. 
\]
Our goal is to show that it can be extended to an \textit{explicit} open neighborhood of $\Sht_{\cl{G}, -\mu}^{[M, \mu]}$ after \textit{truncations}, i.e.\ after passing to the composition  
\[
    \Sht_{\cl{G}, -\mu}^{[M, \mu]} \to \und{H_{\mu, x}} \backslash \Gr_{G, \mu} \to \Sht_{K, -\mu}
\]
for a compact open subgroup $K \subset \cl{G}(O_F)$ containing $H_{\mu, x}$. 

\begin{rmk} \label{rmk:overconvergence_formal}
    The intuition to expect the overconvergence of such $K$-level structures is the overconvergence of finite \'{e}tale sites \cite[Lemma 12.17]{Sch17}: it implies that for any local $\cl{G}$-shtuka $\cl{P} \colon X \to \Sht_{\cl{G}, -\mu}$ on a qcqs $v$-sheaf $X$ whose \textit{$(M, \mu)$-special} locus 
    \[
        X^{[M, \mu]} = \cl{P}^{-1}(\Sht_{\cl{G}, -\mu}^{[M, \mu]}) \subset X
    \]
    can be written as the intersection $\bigcap_{i \in I} U_i$ of qcqs open neighborhoods, the composition
    \[
        X^{[M, \mu]} \to \und{H_{\mu, x}} \backslash \Gr_{G, \mu} \to \Sht_{K, -\mu}
    \]
    extends to some $U_i$. In applications, it is still important to have an explicit family of overconvergent loci (cf.\ \Cref{prop:compact_Up_operator}). 
\end{rmk}

To simplify the situation, we impose the following compatibility between $(M, \mu)$ and $x$. 

\begin{defi} \label{defi:local_special_triple}
    A triple $(M, \mu, x)$ is a local special triple if $(M, \mu)$ is a local special pair and $x$ is a rational point in the enlarged Bruhat-Tits building $\wtd{\cl{B}}(M, F)$ such that $\cl{G}(O_F) = G(F)_{x, 0}$ for some admissible embedding $\wtd{\cl{B}}(M, F) \subset \wtd{\cl{B}}(G, F)$ (see \cite[Section 14.2]{KP23}). 
\end{defi}

From now on, we fix a local special triple $(M, \mu, x)$ and a compatible admissible embedding $\wtd{\cl{B}}(M, F) \subset \wtd{\cl{B}}(G, F)$. 

\subsection{Examples of local special points} \label{ssec:example_special_points}

In this section, we examine some basic examples of local special points. The motivating examples for us are the following. 

\begin{defi} \label{defi:special_case_of_extremal_points}
    A local special pair $(M, \mu)$ is \textit{ordinary} (resp.\ \textit{CM}) if $\mu$ is defined over $F$ (resp.\ $M$ is elliptic, i.e.\ $Z_M^\circ / Z_G^\circ$ is anisotropic). In each case, we say that the associated local special point $\Sht_{\cl{G}, -\mu}^{[M, \mu]}$ is ordinary (resp.\ CM). 
\end{defi}

\begin{lem}
    Let $(M, \mu)$ be a local special pair. We have the following. 
    \begin{enumerate}
        \item When $(M, \mu)$ is ordinary, $M = G^{\mu = 0}$ and $H_\mu = P_\mu$. In particular, $\Sht_{\cl{G}, -\mu}^{[M, \mu]} = \Sht_{\cl{G}, -\mu}^\ord$ when the image of $x$ in $\cl{B}(G, F)$ is a special vertex. 
        \item When $(M, \mu)$ is CM, $H_\mu = M$ and $H_{\mu, x} = M(F)_{x, 0}$. 
    \end{enumerate}
\end{lem}
\begin{proof}
    For (1), the first claim is immediate and the second claim is a reformulation of \Cref{prop:ordinary_locus}. For (2), the claim follows from 
    $
        \sum_{\tau \in \Gal(\wtd{E} / F)} \tau \mu \in X_*(Z_G^\circ)
    $
    for the Galois closure $\wtd{E}$ of $E / F$ since $Z_M^\circ / Z_G^\circ$ is anisotropic. 
\end{proof}

These two examples of local special points are opposite extremes. To investigate some examples in the middle, we present a construction of unramified local special pairs. 

\begin{const} \label{const:special_pair_from_Weyl_element}
    Suppose that $G$ is unramified and $x \in \cl{B}(G, F)$ is a hyperspecial point. Take a maximal torus $T \subset G$ such that $x \in \cl{A}(G, T)$ and let $\cl{T} \subset \cl{G}$ be its integral model. For each Weyl element $w \in N_{\cl{G}}(\cl{T})(O_{\breve{F}})$, we take an element $p_w \in \cl{G}(O_{\breve{F}})$ such that 
    \[  
        p_w^{-1} \sigma(p_w) = w. 
    \]  
    Fix a representative $\mu \in X_*(T)$ of $\{ \mu \}$. Then, $\mu_w = \Ad(p_w)(\mu)$ provides an unramified local special pair $(M_w, \mu_w)$ such that 
    \[
        M_w = \Ad(p_w)\left( \bigcap_{n \geq 0} G^{(w\sigma)^n \mu = 0} \right). 
    \]
\end{const}

Now, we will explore this construction when $G = \GL_n$ and $\mu(t) = \diag(t^{-1}, 1, \ldots, 1)$. For each $1 \leq k \leq n$, let $I_k$ denote the identity matrix of size $k \times k$ and consider a Weyl element
\[
    w_k = {\footnotesize \begin{pmatrix}
        0 & 1 & 0 \\
        I_{k-1} & 0 & 0 \\
        0 & 0 & I_{n-k}
    \end{pmatrix}}. 
\]
Let $P_k \subset \GL_n$ be the parabolic subgroup
\[
    P_k = \left\{ {\small \begin{pmatrix}
        A & B \\
        0 & D
    \end{pmatrix}} \in \GL_n \; \bigg\vert \; A \in \GL_k, D \in \GL_{n-k}
    \right\}. 
\]
Let $\pi_k \colon P_k \to \GL_k \times \GL_{n-k}$ denote the Levi quotient. 
\begin{lem}
    Let $(M_k, \mu_k)$ be the local special pair associated to $w_k$. Let $F_k$ be the unramified extension of $F$ of degree $k$. Then, 
    \[
        M_k \cong \Res_{F_k / F} \bb{G}_m \times \GL_{n-k} \subset \GL_k \times \GL_{n-k} ,\quad H_{\mu_k, x} = \pi_k^{-1}(M_k) \cap \GL_n(O_F). 
    \]
\end{lem}
\begin{proof}
    Since $(w_k\sigma)^i \mu = \diag(1, \ldots, 1, t^{-1}, 1, \ldots, 1)$ where $t$ appears (under $0$-based index) in the $(i \bmod k)$-th coordinate, it is easy to verify the claim by following the definition. 
\end{proof}

In particular, $(M_k, \mu_k)$ is ordinary when $k = 1$, and is CM when $k = n$. Recall that when $F = \bb{Q}_p$, $\Sht_{\cl{G}, -\mu}$ can be thought of as the stack of $p$-divisible groups of height $n$ and dimension $1$ via \cite{SW13} (e.g.\ \cite[Theorem 24.2.5]{SW20}). Under this identification, for each $1 \leq k \leq n$, the canonical level structure at 
$
    \Sht_{\cl{G}, -\mu}^{[M_k, \mu_k]}
$
corresponds to a filtration of $p$-divisible groups
\[
    0 \to X_0 \to X \to X_1 \to 0
\]
such that $X_0$ is of height $k$ and equipped with a CM structure under $\bb{Q}_{p^k}$. Here, $\bb{Q}_{p^k}$ is the unramified extension of $\bb{Q}_p$ of degree $k$. 

\subsection{Explicit neighborhoods of local special points}

In this section, we construct an explicit family of open neighborhoods $\Sht_{\cl{G}, -\mu}^{[M, \mu]}(r)$ of the local special point $\Sht_{\cl{G}, -\mu}^{[M, \mu]}$ so that the canonical level structure extends to $\Sht_{\cl{G}, -\mu}^{[M, \mu]}(r)$ after the ``depth-$r$'' truncation. 

Let $\Flag_{G, \mu} = G_E / P_\mu$. Though $\Flag_{G, \mu}$ can be defined over the reflex field of $(G, \{ \mu \})$, we only need the $E$-rational structure here. We use the following map to access $\Gr_{G, \mu, E}$. 

\begin{prop}\textup{(\cite[Proposition 19.4.2]{SW20})} \label{prop:BB_map}
    There is a natural Bialynicki-Birula map
    \[
        \pi_\mu \colon \Gr_{G, \mu, E} \to \Flag_{G, \mu}^\diamondsuit.
    \]
    When $\mu$ is minuscule, $\pi_\mu$ is an isomorphism. In general, $\pi_\mu$ is $\ell$-cohomologically smooth (hence universally open) and has connected fibers. 
\end{prop}
\begin{proof}
    The proof is given loc. cit. except for the last claim. For the full proof, we begin with a purely group-theoretic construction of $\pi_\mu$. 
    
    Let $S = \Spa(R, R^+) \in \Perfd_E$. Here, we write $\BdRp = \BdRp(S)$, $\BdR = \BdR(S)$ and $\xi = \xi_S$ for simplicity. Now, $\Gr_{G, \mu, E}$ is the sheafification of the functor
    \[
        S \mapsto G(\BdRp) \mu(\xi) G(\BdRp) / G(\BdRp) \cong G(\BdRp) / Z
    \]
    by setting $Z = G(\BdRp) \cap \Ad(\mu(\xi))(G(\BdRp))$. Let $K_n \subset G(\BdRp)$ be the kernel of the reduction map $G(\BdRp) \to G(\BdRp / \xi^n)$ and let $Z_n = Z K_n$. For sufficiently large $N$, we have 
    \[
        G(\BdRp) / Z = G(\BdRp) / Z_N \to G(\BdRp) / Z_{N-1} \to \cdots \to G(\BdRp) / Z_1 = G(R) / P_\mu(R). 
    \]
    The composition provides $\Gr_{G, \mu, E} \to \Flag_{G, \mu}^\diamondsuit$. When $\mu$ is minuscule, we can take $N = 1$, so $\pi_\mu$ is an isomorphism. It remains to prove the last claim. 

    Let $\Lie(G_E)^{\mu > n} \subset \Lie(G_E)$ be the subspace where the $\mu$-weight is larger than $n$. Then, each transition map $G(\BdRp) / Z_{n + 1} \to G(\BdRp) / Z_n$ is a fiber bundle under $(\Lie(G_E)^{\mu > n})(R)$ due to the smoothness of $G$. In particular, by passing to the $v$-sheafification, $\pi_\mu$ is written as an iteration of fiber bundles under vectorial groups $(\Lie(G_E)^{\mu > n})^\diamondsuit$. 
    
    Since $(\Lie(G_E)^{\mu > n})^\diamondsuit$ is $\ell$-cohomologically smooth over $\Spd(E)$ by \cite[Proposition 24.4]{Sch17}, each of the fiber bundle under $(\Lie(G_E)^{\mu > n})^\diamondsuit$ is $\ell$-cohomologically smooth by \cite[Proposition 23.15]{Sch17} (hence universally open by \cite[Proposition 23.11]{Sch17}) and has connected fibers. It is easy to see that the class of universally open morphisms with connected fibers is closed under composition. Thus, we get the last claim.
\end{proof}

It is easier to work with $\Flag_{G, \mu}^\diamondsuit$ since it comes from a rigid analytic variety $\Flag_{G, \mu}^\ad$. 
Let $U_{-\mu} = G_E^{\mu > 0}$. Then, $U_{-\mu} \subset \Flag_{G, \mu}$ is Zariski open. We will introduce rigid-analytic open neighborhoods $[1] \in \bb{B}_{x, r} \subset U_{-\mu}^\ad$ indexed by depth $r > 0$.

Let $S$ be a maximal $E$-split torus of $M$ such that $x$ lies in the enlarged apartment $\wtd{\cl{A}}(M, S)$. As in \cite[Section 2]{Yu01}, we can take a maximal torus $T \subset M$ that is the centralizer of a maximal $\breve{E}$-split torus containing $S$. Then, $T$ splits over a tamely ramified finite Galois extension $\wtd{E} / F$ containing $E$. 

There is a natural embedding $\cl{B}(G, E) \subset \cl{B}(G, \wtd{E})$ of the reduced Bruhat-Tits buildings (see \cite[Section 12.9]{KP23}). Let $\Phi$ be the set of roots of $G_{\wtd{E}}$ with respect to $T_{\wtd{E}}$ and let $\Phi_{\mu > 0} = \{ \alpha \in \Phi \vert \langle \alpha, \mu \rangle > 0 \}$. Then, we have a decomposition 
\begin{equation} \label{eq:product_description}
    U_{-\mu, \wtd{E}} \cong \prod_{\alpha \in \Phi_{\mu > 0}} U_{\alpha}
\end{equation}
into root groups. Here, we fix an order of $\Phi_{\mu > 0}$ to take the product. Now, we fix a special point $o \in \cl{A}(G_{\wtd{E}}, T_{\wtd{E}})$. We regard each $\alpha \in \Phi$ as a linear function on the reduced apartment $\cl{A}(G_{\wtd{E}}, T_{\wtd{E}})$ so that its vectorial part is $\alpha$ and $\alpha(o) = 0$. Let $\cl{U}_{\alpha, o}$ be the integral model of $U_\alpha$ associated to $o$. We have $\cl{U}_{\alpha, o} \cong \bb{A}^1$ and fix a coordinate $u_\alpha$ on $\cl{U}_{\alpha, o}$. 

Since we assume that $x$ is rational (cf.\ \cite[Definition 13.4.1]{KP23}), $\alpha(x)$ is a rational number for every $\alpha \in \Phi$. Note that the notion of $x$ being rational depends only on the relative position of $x$ in the unique facet that contains $x$ in its interior. 


\begin{prop} \label{prop:definition_B_xr}
    For every rational number $r > 0$,
    \[
        \bb{B}_{x, r} = \{ \lvert u_\alpha \rvert \leq \lvert \pi \rvert^{- \alpha(x) + r} \neq 0 \mid\alpha \in \Phi_{\mu > 0} \} \subset U_{-\mu, \wtd{E}}^\ad
    \]
    is closed under multiplication and independent of the choice of $T$, $o$ and the product order in \eqref{eq:product_description}. Moreover, $\bb{B}_{x, r}$ can be defined over $E$. 
\end{prop}
\begin{proof}
    Since the function $\alpha \mapsto - \alpha(x) + r$ on $\Phi_{\mu > 0}$ is concave, $\bb{B}_{x, r}$ is closed under multiplication. Then, it implies the independence of the product order in \eqref{eq:product_description}. 

    We prove the independence of $o$. For another choice of a special point $o' \in \cl{A}(G_{\wtd{E}}, T_{\wtd{E}})$, we have $o' = o + \lambda$ for some $\lambda \in X_*(T_{\wtd{E}})_{\bb{Q}}$. Then, $\pi^{ \langle \alpha, \lambda \rangle} u_\alpha$ is a coordinate on $\cl{U}_{\alpha, o'} \cong \bb{A}^1$. The choice of $o'$ decreases $\alpha(x)$ by $\langle \alpha, \lambda \rangle$, so the independence of $o$ follows. 

    Next, we prove the independence of $T$. For another choice of a maximal torus $T' \subset M$ with $x \in \wtd{\cl{A}}(G_{\wtd{E}}, T'_{\wtd{E}})$, there is an element $h \in M(\wtd{E})$ such that $h x = x$ and $T'_{\wtd{E}} = \Ad(h)(T_{\wtd{E}})$ (see \cite[4.1.12 (3)]{KP23}). Since $\Ad(h)(\mu) = \mu$, we have $\Ad(h)(U_{-\mu, \wtd{E}}) = U_{-\mu, \wtd{E}}$. By using the choice of a special point $ho \in \cl{A}(G_{\wtd{E}}, T'_{\wtd{E}})$, it is easy to see $\Ad(h)(\bb{B}_{x, r}) = \bb{B}_{x, r}$. 

    For the last claim, it is enough to check $\tau(\bb{B}_{x, r}) = \bb{B}_{x, r}$ for every $\tau \in \Gal(\wtd{E} / E)$. Since $x \in \cl{A}(M, S)$, we have $\tau(x) = x$ and the $\tau$-action only changes the choice of a special point $o$ to $\tau(o)$. Then, the claim follows from the independence of $\bb{B}_{x, r}$ on the choice of $o$. 
\end{proof}

From now on, we will study the property of $\bb{B}_{x, r}$. First, we introduce truncations of the canonical level $H_{\mu, x}$. 

\begin{defi}
    We say that $K_{\mu, r} = G(F)_{x, r} H_{\mu, x}$ is the truncation of $H_{\mu, x}$ of depth $r > 0$. 
\end{defi}


\begin{prop} \label{prop:characterize_Kr}
    Let $\wtd{\cl{G}}$ be the parahoric group scheme of $G_E$ associated to $x$ and let $\cl{P}_\mu = \wtd{\cl{G}}^{\mu \leq 0}$. For every $r  > 0$, we have 
    \[
        K_{\mu, r} = \cl{G}(O_F) \cap G(E)_{x, r} \cl{P}_{\mu}(O_{E}). 
    \]
    Moreover, the same claim holds true if one replaces $E$ by its finite tame extension. 
\end{prop}
\begin{proof}
    The inclusion $K_{\mu, r} \subset \cl{G}(O_F) \cap G(E)_{x, r} \cl{P}_{\mu}(O_{E})$ is immediate. For the converse, we may enlarge $E$ so that $E / F$ is Galois and $G_E$ is split. Let $\pi_E$ be the uniformizer of $E$ and let $k_E$ be the residue field of $E$. 

    First, we study the following $\Gal(E / F)$-stable subgroup scheme
    \[
        \wtd{\cl{H}}_\mu = \bigcap_{\tau \in \Gal(E / F)} \wtd{\cl{G}}^{\tau\mu \leq 0} \subset \wtd{\cl{G}}. 
    \]  
    It is easy to see that $\wtd{\cl{H}}_\mu$ is the product of the parahoric group scheme $\wtd{\cl{M}}$ of $M_E$ associated to $x$ and a suitable collection of root groups $\cl{U}_{\alpha, x}$. Now, we need the following lemma. 

\begin{lem} \label{lem:modify_h}
    Let $s > 0$ and let $h \in \wtd{\cl{H}}_\mu(O_E)$ be an element such that $h^{-1} \tau(h) \in G(E)_{x, s}$ for every $\tau \in \Gal(E / F)$. Then, there is an element $h' \in H_{\mu, x}$ such that $h' \in h \cdot G(E)_{x, s}$. 
\end{lem}
\begin{proof}
    Let $h_s = h$. For every $i \geq s$, we will inductively construct $h_i \in \wtd{\cl{H}}_\mu(O_E)$ so that $h_{i+} \in  h_{i} \cdot G(E)_{x, i}$ and $h_i^{-1} \tau(h_i) \in G(E)_{x, i}$ for every $\tau \in \Gal(E / F)$. For indices $i$, we only consider jumps of the Moy-Prasad filtration $G(E)_{x, (-)}$ and $i+$ denotes the successor of $i$. In particular, the index set is discrete.  
    
    Suppose that we have $h_i$ for some $i$. Let $F_\un$ be the maximal unramified subextension of $E / F$ and let $\Lie(\wtd{H}_\mu)_{[i]}$ be the image of $\wtd{\cl{H}}_\mu(O_E) \cap G(E)_{x, i}$ in the Moy-Prasad quotient 
    \[
        G(E)_{x, i} / G(E)_{x, i+} \cong \Lie(G_E)_{x, i} / \Lie(G_E)_{x, i+}. 
    \]
    Let $\Lie(G_E)_{[i]}$ denote the right-hand side. Since $\wtd{\cl{H}}_\mu$ is a product of $\wtd{\cl{M}}$ and some root groups, it is easy to see $\Lie(\wtd{H}_\mu)_{[i]} = \bigcap_{\tau \in \Gal(E / F)} \Lie(G_E)_{[i]}^{\tau\mu \leq 0}$. Since $\Lie(\wtd{H}_\mu)_{[i]}$ is a $k_E$-vector space and $E / F_\un$ is totally tamely ramified, the $1$-cocycle
    \[
        \Gal(E / F_\un) \ni \tau \mapsto h_i^{-1} \tau(h_i) \bmod{G(E)_{x,i+}} \in \Lie(\wtd{H}_\mu)_{[i]}
    \]
    can be resolved. Thus, there is $h_{i + 1}'\in \wtd{\cl{H}}_\mu(O_E)$ such that $h_i \equiv h'_{i+1} \pmod{G(E)_{x, i}}$ and $\tau(h'_{i+1}) \equiv h'_{i+1} \pmod{G(E)_{x, i+}}$ for every $\tau \in \Gal(E / F_\un)$. 
    
    Let $V = (\Lie(\wtd{H}_\mu)_{[i]})^{\Gal(E / F_\un)}$. It admits a natural semilinear $\Gal(k_E/k)$-action and the following $1$-cocycle is well-defined:
    \[
        \Gal(k_E / k) = \Gal(F_\un / F) \ni \sigma \mapsto h'^{-1}_{i+1} \sigma(h'_{i+1}) \bmod{G(E)_{x,i+}} \in V. 
    \]
    It follows from a general fact that $H^1(k_E / k, V) = 0$ since we have $V \cong V^{\Gal(k_E / k)} \otimes_k k_E$. Thus, the above $1$-cocycle can be resolved and we get $h_{i + 1}$. 

    Now, $\varprojlim_i h_i$ converges to an element $h' \in \wtd{\cl{H}}_\mu(O_E)$. By construction, $h' \in h \cdot G(E)_{x, s}$ and  $h' \in G(F) \cap \wtd{\cl{H}}_\mu(O_E) = H_{\mu, x}$. Thus, we get the claim. 
\end{proof}

    Now, we go back to the proof of the original claim. Let $g \in \cl{G}(O_F) \cap G(E)_{x, r} \cl{P}_{\mu}(O_E)$. First, consider the image 
    \[
        g \bmod G(E)_{x, 0+} \in G(E)_{x, 0} / G(E)_{x, 0+}. 
    \]
    The right-hand side equals the set of $k_E$-valued points of the maximal reductive quotient $\msf{G}_x$ of $\cl{\wtd{G}}_{k_E}$. Then, $g \bmod G(E)_{x, 0+}$ lies in $\msf{H}_\mu = \bigcap_{\tau \in \Gal(E / F)} \msf{G}_x^{\tau\mu \leq 0}$ and it lifts to $h_0 \in \wtd{\cl{H}}_\mu(O_E)$. By applying \Cref{lem:modify_h} to $h_0$ for $s = 0+$, we may assume $h_0 \in H_{\mu, x}$. By replacing $g$ with $h_0^{-1}g$, we reduce to the case $g \in G(E)_{x, 0+}$. 

    Now, suppose $g \in G(E)_{x, i}$ for some index $0 < i < r$. Then, consider the image 
    \[
        g \bmod G(E)_{x, i+} \in G(E)_{x, i} / G(E)_{x, i+} = \Lie(G_E)_{[i]}. 
    \]
    Since $g \in \cl{G}(O_F) \cap G(E)_{x, r} \cl{P}_{\mu}(O_E)$, it lies in $\Lie(\wtd{H}_\mu)_{[i]}$ and admits a lift $h_i \in \wtd{\cl{H}}_\mu(O_E) \cap G(E)_{x, i}$. By applying \Cref{lem:modify_h} to $h_i$ for $s = i+$, we may assume $h_i \in H_{\mu, x} \cap G(E)_{x, i}$. By replacing $g$ with $h_i^{-1}g$, we reduce to the case $g \in G(E)_{x, i+}$. 

    By induction, we reduce to the case $g \in G(E)_{x, r}$. In this case, $G(F) \cap G(E)_{x, r} = G(F)_{x, r}$ by \cite[Proposition 12.9.4]{KP23}, so we get the claim. 
\end{proof}

\begin{thm} \label{thm:canonical_level_structures}
    For every $r > 0$, $\bb{B}_{x, r}$ satisfies the following properties. 
    \begin{enumerate}
        \item We have $\bb{B}_{x, r} = g \cdot \bb{B}_{x, r}$ for $g \in K_{\mu, r}$ and $\bb{B}_{x, r} \cap g \cdot \bb{B}_{x, r} = \emptyset$ for $g \in \cl{G}(O_F) - K_{\mu, r}$. 
        \item There is a natural open immersion 
        \[
            \Sht_{\cl{G}, -\mu}^{[M, \mu]}(r) = \und{K_{\mu, r}} \backslash \pi_\mu^{-1}(\bb{B}_{x, r})^\diamond \subset \Sht_{\cl{G}, -\mu, E}
        \]
        and the image contains $\Sht_{\cl{G}, -\mu}^{[M, \mu]}$. Moreover, there is a unique section
        \[
            \can_r \colon \Sht_{\cl{G}, -\mu}^{[M, \mu]}(r) \to \Sht_{K_{\mu, r}, -\mu}
        \]
        extending the truncation of the canonical filtration on $\Sht_{\cl{G}, -\mu}^{[M, \mu]}$ along $H_{\mu, x} \subset K_{\mu ,r}$. 
        \item For $s> r $, we have $\Sht_{\cl{G}, -\mu}^{[M, \mu]}(s) \subset \Sht_{\cl{G}, -\mu}^{[M, \mu]}(r)$. Moreover, $\bigcap_{r>0} \Sht_{\cl{G}, -\mu}^{[M, \mu]}(r) = \Sht_{\cl{G}, -\mu}^{[M, \mu]}$ when $\mu$ is minuscule. 
    \end{enumerate}
\end{thm}
\begin{proof}
    For (1), we first consider the case $g \in K_{\mu, r}$. By the Iwahori decomposition of $G(E)_{x, r}$ (see \cite[Proposition 13.2.5 (3)]{KP23}) with respect to $P_\mu$, we have $g \in U_{-\mu}(E)_{x, r} \cl{P}_\mu(O_E)$. Now, $U_{-\mu}(E)_{x, r} \subset \bb{B}_{x, r}(E)$, so the left multiplication under $U_{-\mu}(E)_{x, r}$ stabilizes $\bb{B}_{x, r}$. On the other hand, the stability under the left multiplication by $\wtd{\cl{G}}^{\mu < 0}(O_E)$ follows from the commutator relation between $U_{-\mu}$ and $U_{\mu}$. Since $\wtd{\cl{G}}^{\mu=0}(O_F)$ is also generated by root groups and the center (see \cite[Definition 7.3.3]{KP23}), the stability under $\wtd{\cl{G}}^{\mu=0}(O_F)$ also follows. 
    
    Next, we show that $\bb{B}_{x, r} \cap g \cdot \bb{B}_{x, r} \neq \emptyset$ implies $g \in K_{\mu, r}$. Let $b \in \bb{B}_{x, r}(C, C^+)$ be a geometric point such that $gb \in \bb{B}_{x, r}$. Let $u_b, u_g \in U_{-\mu}(C)$ be elements associated to $b$ and $gb$. Then, $g \cdot [u_b] = [u_g]$ implies 
    $
        u_g^{-1} g u_b \in P_\mu(C). 
    $

    We may enlarge $E$ so that the ramification index of $E / F$ is sufficiently large and $u_b$ and $u_g$ can be lifted to $\cl{U}_{-\mu, x, r'}(C^+)$ for some $0 < r' < r$ such that $K_{\mu, r} = K_{\mu, r'}$. For this choice of $E$, we have $u_g^{-1} g u_b \in \cl{P}_\mu(C^+)$. Now, the function 
    \[
        \alpha \mapsto \left\{
            \begin{alignedat}{4}
                & r' & \; & (\alpha \in \Phi_{\mu > 0}) \\
                & 0 & \; & (\alpha \notin \Phi_{\mu > 0})
            \end{alignedat}
        \right.
    \]
    is concave and the associated Moy-Prasad group scheme of $G_E$ is $\cl{U}_{-\mu, x, r'} \times \cl{P}_\mu$. Thus
    \[
        g \in (\cl{U}_{-\mu, x, r'} \times \cl{P}_\mu)(C^+) \cap G(E) = G(E)_{x, r'}\cl{P}_\mu(O_E). 
    \]
    Now, \Cref{prop:characterize_Kr} implies $g \in K_{\mu, r'} = K_{\mu, r}$, so we get the claim. 

    For (2), the result of (1) implies
    \[
        \coprod_{g \in \cl{G}(O_F) / K_{\mu, r}} g \cdot \pi_\mu^{-1}(\bb{B}_{x, r}) = \bigcup_{g \in \cl{G}(O_F)} g \cdot \pi_\mu^{-1}(\bb{B}_{x, r}) \subset \Gr_{G, \mu, E}. 
    \]
    Thus, $\Sht_{\cl{G}, -\mu}^{[M, \mu]}(r) \cong \und{\cl{G}(O_F)} \backslash \bigcup_{g \in \cl{G}(O_F)} g \cdot \pi_\mu^{-1}(\bb{B}_{x, r})$ and it is an open substack of $\Sht_{\cl{G}, -\mu, E}$. Since $\pi_\mu(x_\mu) = 1 \in \bb{B}_{x, r}$, $\Sht_{\cl{G}, -\mu}^{[M, \mu]}(r)$ contains $\Sht_{\cl{G}, -\mu}^{[M, \mu]}$. Moreover, $\pi_\mu^{-1}(\bb{B}_{x, r}) \subset \Gr_{G, \mu, E}$ induces a section
    \[
        \can_r \colon \Sht_{\cl{G}, -\mu}^{[M, \mu]}(r) \to \Sht_{K_{\mu, r}, -\mu}. 
    \]
    It is easy to see that the restriction to $\Sht_{\cl{G}, -\mu}^{[M, \mu]}$ is given by the composition
    \[
        \Sht_{\cl{G}, -\mu}^{[M, \mu]} \to \und{H_{\mu, x}} \backslash \Gr_{G, \mu} \to \Sht_{K_{\mu, r}, -\mu}. 
    \]
    It remains to show the uniqueness of $\can_r$. Since $\Sht_{K_{\mu, r}, -\mu} \to \Sht_{\cl{G}, -\mu}$ is finite \'{e}tale, the equalizer of any two sections $\Sht_{\cl{G}, -\mu}^{[M, \mu]}(r) \to \Sht_{K_{\mu, r}, -\mu}$ is closed and open in $\Sht_{\cl{G}, -\mu}^{[M, \mu]}(r)$. Then, such a section is uniquely determined by the image of $\Sht_{\cl{G}, -\mu}^{[M, \mu]}$ because $\bb{B}_{x, r}$ is connected and $\Sht_{\cl{G}, -\mu}^{[M, \mu]}(r)$ is connected by \Cref{prop:BB_map}. 

    For (3), the inclusion $\Sht_{\cl{G}, -\mu}^{[M, \mu]}(s) \subset \Sht_{\cl{G}, -\mu}^{[M, \mu]}(r)$ is obvious from $\bb{B}_{x, s} \subset \bb{B}_{x, r}$. Now, suppose that $\mu$ is minuscule, so that $\pi_\mu$ is an isomorphism. It is enough to show
    \[  
        \bigcap_{r > 0} (\cl{G}(O_F) / K_{\mu, r}) \times \bb{B}_{x, r} = \und{\cl{G}(O_F) / H_{\mu, x}} \subset \Flag_{G, \mu}^\ad. 
    \]
    Let $b \in \Flag_{G, \mu}(C, C^+)$ be a geometric point on the left-hand side. For each $r > 0$, take $g_r \in \cl{G}(O_F)$ so that $g_r^{-1} b \in \bb{B}_{x, r}$. By (1), $g_r \in g_s K_{\mu, r}$ for $s > r$, so we may modify the family $\{g_r\}$ so that $g_r \in g_s G(F)_{x, r}$ for $s > r$. Then, we can take the limit $g = \lim_{r \to \infty} g_r \in \cl{G}(O_F)$. Now, we have $b = [g] \in \und{\cl{G}(O_F) / H_{\mu, x}}$ and we get the claim. 
\end{proof}

\begin{exa} \label{exa:interpretation_p_divisible_groups}
    Here, we interpret \Cref{thm:canonical_level_structures} in terms of $p$-divisible groups when $\cl{G} = \GL_n / \bb{Z}_p$ and $\mu = \diag(t^{-1}, 1, \ldots, 1)$. Recall the notation in \Cref{ssec:example_special_points} and consider
    \[
        \can_m \colon \Sht_{\cl{G}, -\mu}^{[M_k, \mu_k]}(m) \to \Sht_{K_{\mu_k, m}, -\mu}
    \]
    for each $1 \leq k \leq n$ and $m \geq 1$. In this case, 
    \[
        K_{\mu_k, m} = \left\{ g \in \GL_n(\bb{Z}_p)  \; \bigg\vert \; g \bmod{\pi^m} = {\small \begin{pmatrix}
            A & B \\
            0 & D
        \end{pmatrix}},\; A \in (\bb{Z}_{p^k} / p^m)^\times,\; D \in \GL_{n-k}
        \right\}
    \]
    where $\bb{Z}_{p^k}$ is the unramified extension of $\bb{Z}_p$ of degree $k$. Then, $\can_m$ corresponds to an $m$-truncated Barsotti-Tate subgroup of the $p^m$-torsion of the universal $p$-divisible group $X$
    \[
        H \subset X[p^m]
    \]
    that is of height $k$. Moreover, $H$ is equipped with a CM structure under $\bb{Z}_{p^k} / p^m$. 
\end{exa}

\subsection{Canonical subgroup around the ordinary locus}

In this section, we specialize \Cref{thm:canonical_level_structures} in the case of the ordinary locus (see \Cref{defi:special_case_of_extremal_points}) and provide some consequences on the ordinary locus of Shimura varieties. 

Here, we assume that $\mu$ is defined over $F$ and $\cl{G}(O_F) = G(F)_{x, 0}$ is a special maximal parahoric subgroup. Then, $M = G^{\mu = 0}$ and $H_\mu = P_\mu$. In this case, 
\[
    \Sht_{\cl{G}, -\mu}^{[M, \mu]} = \Sht_{\cl{G}, -\mu}^\ord
\]
by \Cref{prop:ordinary_locus}. In particular, \Cref{thm:canonical_level_structures} provides the following family of open neighborhoods of $\Sht_{\cl{G}, -\mu}^\ord$. 

\begin{thm} \label{thm:canonical_filtration}
    There is a decreasing family of open neighborhoods $\Sht_{\cl{G}, -\mu}^{\ord}(r)$ of the ordinary locus $\Sht_{\cl{G}, -\mu}^\ord$ indexed by $r > 0$ with the following properties. 
    \begin{enumerate}
        \item Let $K_{\mu, r} = G(F)_{x, r} \cl{P}_\mu(O_F)$. There is a unique section
        \[
            \can_r \colon \Sht^{\ord}_{\cl{G}, -\mu}(r) \to \Sht_{K_{\mu, r}, -\mu}
        \]
        extending the truncation of the canonical filtration on $\Sht^\ord_{\cl{G}, -\mu}$. 
        \item When $\mu$ is minuscule, $\bigcap_{r > 0} \Sht_{\cl{G}, -\mu}^\ord(r) = \Sht^{\ord}_{\cl{G}, -\mu}$ and the composition map
        \[
            U_p \colon \Sht^{\ord}_{\cl{G}, -\mu}(1) \xrightarrow{\can_1} \Sht_{K_{\mu, 1}, -\mu} \xrightarrow{\mu(\pi^{-1})} \Sht_{\cl{G}, -\mu}
        \]
        restricts to a finite \'{e}tale surjection $\Sht_{\cl{G}, -\mu}^\ord(r+1) \to \Sht_{\cl{G}, -\mu}^\ord(r)$ for every $r > 0$. Here, the second map is induced by the inclusion $\Ad(\mu(\pi^{-1}))(K_{\mu, 1}) \subset \cl{G}(O_F)$. 
    \end{enumerate}
\end{thm}
\begin{proof}
    First, (1) and the first claim of (2) are consequences of \Cref{thm:canonical_level_structures}. For the latter claim of (2), we first verify $\Ad(\mu(\pi^{-1}))(K_{\mu, 1}) \subset \cl{G}(O_F)$. By the Iwahori decomposition of $G(F)_{x, 1}$, we have $K_{\mu, 1} = U_{-\mu}(F)_{x, 1} \cl{P}_\mu(O_F)$. Then, it is easy to see that $\Ad(\mu(\pi^{-1}))$ sends both $U_{-\mu}(F)_{x, 1}$ and $\cl{P}_\mu(O_F)$ into $\cl{G}(O_F)$. 
    
    Now, the $U_p$-operator sends a geometric point of $\Sht_{\cl{G}, - \mu}^{\ord}(1)$ represented by $b \in \bb{B}_{x,1}(C, C^+)$ to $[\mu(\pi^{-1})b] \in \und{\cl{G}(O_F)} \backslash \Flag_{G, \mu}^\diamond$. Since $\Ad(\mu(\pi^{-1}))$ induces $\bb{B}_{x, r+1} \cong  \bb{B}_{x, r}$, the $U_p$-operator restricts to a finite \'{e}tale surjection
    \[
        \und{K_{\mu, r+1}} \backslash \bb{B}_{x, r+1}^\diamond \xrightarrow{\Ad(\mu(\pi^{-1}))} \und{\Ad(\mu(\pi^{-1}))(K_{\mu, r+1})} \backslash \bb{B}_{x, r}^\diamond \to \und{K_{\mu, r}} \backslash \bb{B}_{x, r}^\diamond
    \]
    The inclusion $\Ad(\mu(\pi^{-1}))(K_{\mu, r+1}) \subset K_{\mu, r}$ follows as previously from the Iwahori decomposition $K_{\mu, r+1} = U_{-\mu}(F)_{x, r+1} \cl{P}_\mu(O_F)$. Thus, we get the claim. 
\end{proof}

\begin{rmk}
    The case of the ordinary locus is quite close to the analysis in \cite[{}3.3.6]{BP21}. Our main contribution is to reinterpret their analysis in terms of the ordinary locus of the stack of local shtukas and generalize to other local special points such as CM points. 
\end{rmk}

For some Shimura varieties of PEL type, a counterpart of $\can_r$ was constructed by \cite[Theorem A]{GK12} and \cite[Théorème 8]{Far11}. Finding such a section $\can_r$ for Shimura varieties is called the `general canonical subgroup problem' in \cite[Introduction]{GK12}. These results can be recovered from \Cref{thm:canonical_filtration} by pulling back along the Hodge-Tate period map
\[
    \pi_\HT \colon \Sh_{\msf{K}_p \msf{K}^p}(\mbf{G}, \mbf{X})_E^\diamondsuit \to \Sht_{\msf{K}_p, -\mu}. 
\]
Note that $\pi_\HT$ was first constructed by \cite{Sch15} for Siegel moduli spaces, and then extended to general Shimura varieties by \cite{CS17} and \cite{PR24}. Here, our notation is as follows. 

First, $(\mbf{G}, \mbf{X})$ is a Shimura datum satisfying the axiom (SV5). We set $G = \mbf{G}_{\bb{Q}_p}$ and let $\msf{K}_p = \cl{G}(\bb{Z}_p)$ be a special maximal parahoric subgroup of $G(\bb{Q}_p)$ corresponding to a special point $x \in \cl{B}(G, \bb{Q}_p)$. Let $\msf{K}^p\subset \mbf{G}(\bb{A}^p_f)$ be a sufficiently small (or neat) compact open subgroup and let $\msf{K} = \msf{K}_p\msf{K}^p \subset \mbf{G}(\bb{A}_f)$. Let $\mbf{E}$ be the reflex field of $(\mbf{G},\mbf{X})$. We fix a place $v$ of $\mbf{E}$ over $p$ and set $E=\mbf{E}_v$. 

Now, the Shimura variety $\Sh_\msf{K}(\mbf{G},\mbf{X})$ is determined and there is a universal $\cl{G}$-shtuka over $\Sh_\msf{K}(\mbf{G},\mbf{X})_E^\diamondsuit$ (see \cite[Proposition 4.1.2]{PR24}) corresponding to the Hodge-Tate period map $\pi_\HT$. Here, the geometric conjugacy class $\{\mu\}$ of cocharacters of $G$ is taken to be the opposite of the usual one associated to $\mbf{X}$. The pullback along $\pi_\HT$ provides the following.

\begin{cor} \label{cor:canonical_filtration}
    Suppose that $\{ \mu \}$ admits a representative $\mu \colon \bb{G}_m \to \cl{G}$ defined over $\bb{Z}_p$ and let $\cl{P}_\mu = \cl{G}^{\mu \leq 0}$. Then, the closed subspace $\lvert \Sh_{\msf{K}}(\mbf{G}, \mbf{X})_E\rvert^{\ord} = \pi_\HT^{-1}(\Sht_{\cl{G}, -\mu}^\ord)$ of the Shimura variety $\Sh_{\msf{K}}(\mbf{G}, \mbf{X})_E$ admits a decreasing family of open neighborhoods
    \[
        \Sh_{\msf{K}}(\mbf{G}, \mbf{X})_E^{\ord}(r) \subset \Sh_{\msf{K}}(\mbf{G}, \mbf{X})_E^\ad
    \]
    indexed by $r > 0$ with the following properties. 
    \begin{enumerate}
        \item Let $\msf{K}_{p, r} = G(\bb{Q}_p)_{x, r} \cl{P}_\mu(\bb{Z}_p)$.
        There is a section
        \[
            \can_r \colon \Sh_{\msf{K}}(\mbf{G}, \mbf{X})_E^{\ord}(r) \to \Sh_{\msf{K}_{p, r} \msf{K}^p}(\mbf{G}, \mbf{X})_E^\ad
        \]
        extending the truncation of the canonical filtration on $\Sh_{\msf{K}}(\mbf{G}, \mbf{X})_E^{\diamondsuit, \ord}$. 
        \item We have $\bigcap_{r > 0} \lvert \Sh_{\msf{K}}(\mbf{G}, \mbf{X})_E^{\ord}(r) \rvert = \lvert \Sh_{\msf{K}}(\mbf{G}, \mbf{X})_E\rvert^{\ord}$ and the composition map
        \[
            U_p \colon \Sh_{\msf{K}}(\mbf{G}, \mbf{X})_E^{\ord}(1) \xrightarrow{\can_1} \Sh_{\msf{K}_{p, 1} \msf{K}^p}(\mbf{G}, \mbf{X})_E^\ad \xrightarrow{\mu(p^{-1})} \Sh_{\msf{K}}(\mbf{G}, \mbf{X})_E^\ad
        \]
        restricts to a finite \'{e}tale surjection $\Sh_{\msf{K}}(\mbf{G}, \mbf{X})_E^{\ord}(r+1) \to \Sh_{\msf{K}}(\mbf{G}, \mbf{X})_E^{\ord}(r)$. 
    \end{enumerate}
\end{cor}
\begin{proof}
    All the claims follow from \Cref{thm:canonical_filtration} by pullback along $\pi_\HT$. Note that $\can_r$ is a priori defined after $v$-sheafification, but since the domain and codomain of $\can_r$ are \'{e}tale over $\Sh_{\msf{K}}(\mbf{G}, \mbf{X})_E^\ad$, we may apply \cite[Lemma 15.6]{Sch17} to realize it as a map of adic spaces. 
\end{proof}

Our construction exhibits the philosophy that the Hodge-Tate period map is a rigid-analytic analogue of the Hasse invariant. It would be interesting to directly compare the classical Hasse bound with the Hodge-Tate bound by $\bb{B}_{x, r}$. For example, each pair $(G, \mu)$ corresponds to the following classical geometric object.

\vspace{5pt}
\begin{center}
    \begin{tabular}{c|c} 
        $(G, \mu)$ & classical geometric objects \\ \hline
        $(\GL_h, (- 1^{(d)}, 0^{(h-d)}))$  & $p$-divisible groups of height $h$ and dimension $d$ \\ \hline
        $(\GSp_{2g}, (- 1^{(g)}, 0^{(g)}))$ & polarized $p$-divisible groups of height $2g$ 
    \end{tabular}
\end{center}
\vspace{5pt}

As explained in \Cref{rmk:overconvergence_formal}, the mere existence of $\can_r$ follows formally and the property (2) is rather important In applications. As a simple consequence of (2), we construct a compact $U_p$-operator in the following setting. 

We fix a family of automorphic vector bundles $\omega_{\msf{K}}$ on each $\Sh_{\msf{K}}(\mbf{G}, \mbf{X})$ (see \cite[Section 4.1.1]{BP21}). It comes from an algebraic representation of $G^{\mu = 0}$ and is stable under any Hecke correspondence. 

\begin{lem}
    When $\Sh_{\msf{K}}(\mbf{G}, \mbf{X})$ is proper, $\Sh_{\msf{K}}(\mbf{G}, \mbf{X})_E^{\ord}(r)$ is quasicompact for $r > 0$. In particular, the set of overconvergent sections 
    \[
        H^0(\Sh_{\msf{K}}(\mbf{G}, \mbf{X})_E^{\ord}(r), \omega_{\msf{K}})
    \] 
    is naturally equipped with the structure of a Banach $E$-vector space. 
\end{lem}
\begin{proof}
    If $\Sh_{\msf{K}}(\mbf{G}, \mbf{X})$ is proper, the Hodge-Tate period map $\pi_\HT$ is quasicompact, so the first claim follows from the quasicompactness of $\bb{B}_{x, r}$. Then, the second claim follows formally (cf.\ \cite[Lemma 2.5.20]{BP21}). 
\end{proof}

\begin{prop} \label{prop:compact_Up_operator}
    Let $U_{p, r} \colon \Sh_{\msf{K}}(\mbf{G}, \mbf{X})_E^{\ord}(r+1) \to \Sh_{\msf{K}}(\mbf{G}, \mbf{X})_E^{\ord}(r)$ be the finite \'{e}tale map coming from the $U_p$-operator. Then, we have $\omega_{\msf{K}} \cong U_{p, r}^* \omega_{\msf{K}}$ and the normalized trace map 
    \[
        \tr_{U_p} \colon H^0(\Sh_{\msf{K}}(\mbf{G}, \mbf{X})_E^{\ord}(r+1), \omega_{\msf{K}}) \to H^0(\Sh_{\msf{K}}(\mbf{G}, \mbf{X})_E^{\ord}(r), \omega_{\msf{K}})
    \]
    is well-defined. Then, the following composition
    \[
        H^0(\Sh_{\msf{K}}(\mbf{G}, \mbf{X})_E^{\ord}(r), \omega) \xrightarrow{\res} H^0(\Sh_{\msf{K}}(\mbf{G}, \mbf{X})_E^{\ord}(r+1), \omega) \xrightarrow{\tr_{U_p}} H^0(\Sh_{\msf{K}}(\mbf{G}, \mbf{X})_E^{\ord}(r), \omega)
    \]
    provides a compact operator on $H^0(\Sh_{\msf{K}}(\mbf{G}, \mbf{X})_E^{\ord}(r), \omega_{\msf{K}})$ when $\Sh_{\msf{K}}(\mbf{G}, \mbf{X})$ is proper. 
\end{prop}
\begin{proof}
    Recall the definition of $U_p$ in \Cref{cor:canonical_filtration} and let
    \[
        [1], [\mu(p^{-1})] \colon \Sh_{\msf{K}_{p, 1} \msf{K}^p}(\mbf{G}, \mbf{X})_E \to \Sh_{\msf{K}}(\mbf{G}, \mbf{X})_E
    \]
    denote the Hecke correspondences attached to the elements $1, \mu(p^{-1}) \in G(\bb{Q}_p)$ respectively. The compatibility of automorphic vector bundles $\{ \omega_{\msf{K}} \}_{\msf{K}}$ with Hecke correspondences provides
    \[
        [1]^* \omega_{\msf{K}} \cong \omega_{\msf{K}_{p, 1} \msf{K}^p} \cong [\mu(p^{-1})]^* \omega_{\msf{K}}.
    \]
    By restricting along $\can_1$, we get $\omega_{\msf{K}} \cong U_{p, r}^* \omega_{\msf{K}}$. Since $U_{p, r}$ is finite \'{e}tale, it formally follows that $\tr_{U_p}$ can be naturally defined as a continuous map. 

    When $\Sh_{\msf{K}}(\mbf{G}, \mbf{X})$ is proper, the restriction map $\res$ is compact by \cite[Lemma 2.5.23]{BP21} since the closure of $\bb{B}_{x, r + 1}$ in $\Flag_{G, \mu}$ is contained in $\bb{B}_{x, r}$. Thus, the composition $\tr_{U_p} \circ \res$ is a compact operator on the Banach $E$-vector space $H^0(\Sh_{\msf{K}}(\mbf{G}, \mbf{X})_E^{\ord}(r), \omega_{\msf{K}})$. 
\end{proof}
\begin{rmk}
    When $\Sh_{\msf{K}}(\mbf{G}, \mbf{X})$ is not proper, one usually needs to work with toroidal compactifications. In \cite{BP21}, Boxer-Pilloni dealt with this difficulty and developed higher Coleman theory for Shimura varieties of abelian type such that $G$ is quasi-split. 
\end{rmk}

\subsection{Relation to special points of Shimura varieties} \label{ssec:special_points}

In this section, we explain a relation between special points of Shimura varieties and local special points on the stack of local shtukas. First, we review some terminology for special points on Shimura varieties.

Let $(\mbf{T}, \mbf{x}) \hookrightarrow (\mbf{G}, \mbf{X})$ be an embedding of Shimura data (satisfying (SV4)) such that $\mbf{T}$ is a torus. It is called a special pair in \cite[Definition 12.5]{Mil05}. Let $\mbf{E}_{\mbf{x}} / \mbf{E}$ be the reflex field of $(\mbf{T}, \mbf{x})$ and let $T = \mbf{T}_{\bb{Q}_p}$. We suppose that for a suitable compact open subgroup $\msf{K}_{\mbf{T}} = \msf{K}_{\mbf{T}, p} \msf{K}_{\mbf{T}}^p \subset \mbf{T}(\bb{A}_f)$, we have a closed immersion
\begin{equation} \label{eq:Shimura_embedding}
    \Sh_{\msf{K}_{\mbf{T}}}(\mbf{T}, \mbf{x}) \hookrightarrow \Sh_{\msf{K}}(\mbf{G}, \mbf{X})_{\mbf{E}_{\mbf{x}}}
\end{equation}
(see e.g.\ \cite[Section 4.3]{PR24} for such functoriality). At the place $p$, we fix an admissible embedding $\cl{\wtd{B}}(T, \bb{Q}_p) \subset \cl{\wtd{B}}(G, \bb{Q}_p)$ of enlarged Bruhat-Tits buildings and take $\msf{K}_{\mbf{T}, p} = T(\bb{Q}_p)_{x, 0}$ and $\msf{K}_{p} = G(\bb{Q}_p)_{x, 0}$ for some point $x \in \cl{B}(T, \bb{Q}_p)$. 

Let $E_{\mbf{x}}$ be the completion of $\mbf{E}_{\mbf{x}}$ at a place over $v$. By the functoriality\footnote{This follows from the functoriality in \cite[Proposition 1.1.1]{PR24}.} of universal shtukas on Shimura varieties, we have the following commutative diagram. 
\begin{equation} \label{eq:diag_HT}
    \begin{tikzcd}
        \Sh_{\msf{K}_{\mbf{T}}}(\mbf{T}, \mbf{x})_{E_{\mbf{x}}}^\diamondsuit \ar[r, hook] \ar[d, "\pi_{\HT}^{\mbf{T}}"] & \Sh_{\msf{K}}(\mbf{G}, \mbf{X})_{E_{\mbf{x}}}^\diamondsuit \ar[d, "\pi_\HT"] \\
        \Sht_{\msf{K}_{\mbf{T}, p}, -\mu} \ar[r] & \Sht_{\msf{K}_{p}, -\mu}. 
    \end{tikzcd}
\end{equation}
Here, the vertical arrows denote Hodge-Tate period maps. Note that $\mbf{x} \in \mbf{X}$ determines a representative $\mu \in X_*(T_{\ov{\bb{Q}}_p})$ of the geometric conjugacy class $\{ \mu \}$. To apply the study of special points on $\Sht_{\msf{K}_{p}, -\mu}$, we assume that $T$ is tamely ramified. In this case, we say that the special pair $(\mbf{T}, \mbf{x})$ is tame at $p$. 

\begin{prop} \label{prop:image_extremal}
    Let $M = \bigcap_{\tau \in \Gal(\ov{\bb{Q}}_p / \bb{Q}_p)} G^{\tau \mu = 0}$ be a tamely ramified twisted Levi subgroup of $G$ containing $T$. Then, $(M, \mu)$ is a local special pair of $G$ and the image of 
    \[
        \Sht_{\msf{K}_{\mbf{T}, p}, -\mu} \to \Sht_{\msf{K}_{p}, -\mu}
    \]
    is equal to the associated special point $\Sht_{\msf{K}_{p}, -\mu}^{[M, \mu]}$. 
\end{prop}
\begin{proof}
    Since $T$ is tamely ramified, $(M, \mu)$ is a local special pair by construction. Moreover, we have the following commutative diagram. 
    \begin{center}
        \begin{tikzcd}
           \Gr_{T, \mu} \ar[r, hook] \ar[d] & \Gr_{G, \mu} \ar[d] \\
            \Sht_{\msf{K}_{\mbf{T}, p}, -\mu} \ar[r] & \Sht_{\msf{K}_{p}, -\mu}. 
        \end{tikzcd}
    \end{center}
    The above horizontal map factors as $\Gr_{T, \mu} \xrightarrow{\sim} \Gr_{M, \mu} \to \Gr_{G, \mu}$ since $\mu$ is central in $M$. In particular, the image of $\Sht_{\msf{K}_{\mbf{T}, p}, -\mu} \to \Sht_{\msf{K}_{p}, -\mu}$ is the special point $\Sht_{\msf{K}_{p}, -\mu}^{[M, \mu]}$. 
\end{proof}

\begin{cor} \label{cor:canonical_level_structure}
    For every $s \in \Sh_{\msf{K}}(\mbf{G}, \mbf{X})(\bb{C}_p)$ that lies in $\Sh_{\msf{K}_{\mbf{T}}}(\mbf{T}, \mbf{x})(\bb{C}_p)$ for some special pair $(\mbf{T}, \mbf{x})$ tame at $p$, $\pi_\HT(s) \in \Sht_{\msf{K}_p, -\mu}(\bb{C}_p)$ is a local special point. In particular, for a local special pair $(M, \mu)$ associated to $\pi_\HT(s)$, there is a decreasing family of open neighborhoods 
    \[
        \Sh_{\msf{K}}(\mbf{G}, \mbf{X})(s; r)
    \]
    of $s$ indexed by $r > 0$ on which there is a canonical level structure
    \[
        \can_r \colon \Sh_{\msf{K}}(\mbf{G}, \mbf{X})(s; r) \to \Sh_{\msf{K}_{p, r} \msf{K}^p}(\mbf{G}, \mbf{X}) 
    \]
    at the level $\msf{K}_{p, r} = G(\bb{Q}_p)_{x, r} H_{\mu, x}$ associated to $(M, \mu)$. 
\end{cor}
\begin{proof}
    The first claim follows from \Cref{prop:image_extremal} and the commutativity of \eqref{eq:diag_HT}. Then, the second claim follows formally from \Cref{thm:canonical_level_structures} by taking the pullback along $\pi_{\HT}$.
\end{proof}

\begin{rmk}
    By \Cref{rmk:overconvergence_formal}, the overconvergence follows formally for the canonical level structure of $s$ at the level $G(\bb{Q}_p)_{x, r} \msf{K}_{\mbf{T}, p}$. Here, \Cref{cor:canonical_level_structure} at least claims that such level structure is independent of the choice of $(\mbf{T}, \mbf{x})$ by enlarging the level to $\msf{K}_{p, r}$. 
\end{rmk}

Note that special points on $\Sh_{\msf{K}}(\mbf{G}, \mbf{X})$ are those points whose Hecke orbits intersect with $\Sh_{\msf{K}_{\mbf{T}}}(\mbf{T}, \mbf{x})$ for some special pair $(\mbf{T}, \mbf{x})$ and \Cref{cor:canonical_level_structure} only treats those points lying in $\Sh_{\msf{K}_{\mbf{T}}}(\mbf{T}, \mbf{x})$ for some tame $(\mbf{T}, \mbf{x})$. This is because local special points depend on the choice of $x \in \cl{B}(G, \bb{Q}_p)$, which is not preserved under Hecke correspondences at $p$. 

\section{Special open neighborhoods of CM points} \label{sec:special_open_neighborhoods}


In this section, we first review representation-theoretic transfers associated to CM points in \Cref{ssec:representation_theory_preparation}. Then, we introduce special open neighborhoods of CM points in the stack of local shtukas (see \Cref{defi:special_open_neighborhoods}) and specify unramified CM points relevant to our main theorem (see \Cref{thm:explicit_geometry}). 

From now on, all $v$-stacks are considered over $\ov{k}$ and $\ast = \Spd(\ov{k})$. The completed maximal unramified extension of $F$ (resp.\ $E$) with residue field $\ov{k}$ is denoted by $\breve{F}$ (resp.\ $\breve{E}$). 

\subsection{Representation-theoretic setup} \label{ssec:representation_theory_preparation}


In this section, we set up our notation for CM points and review representation-theoretic transfers associated to them. Let $(M, \mu)$ be a tamely ramified CM pair. Let $b \in Z^\circ_M(\breve{F})$ be an element such that $[b] \in B(Z_M^\circ)$ corresponds to $\mu$ under the Kottwitz's isomorphism 
\[
    X_*(Z_M^\circ)_{\Gal(\ov{F} / F)} \cong B(Z_M^\circ)
\]
(see \cite[{}2.4]{Kot84}). We take a representative $b = \Nm_{\breve{E} / \breve{F}}(\mu(\pi_E))$. Since $Z_M^\circ / Z_G^\circ$ is anisotropic, the Newton point of $[b] \in B(G)$ is central and $[b]$ is basic in $B(G)$. Fix a representative $b \in Z_M^\circ(\breve{F})$ of $[b]$ and take an extended pure inner form $G_b$ of $G$ so that 
\[
    G_b(F) = \{ g \in G(\breve{F}) \mid b \sigma(g) = g b \}. 
\]
There is a natural embedding $M \subset G_b$ and we have canonical admissible embeddings 
\[
    \cl{B}(M, F) \subset \cl{B}(G, F), \quad \cl{B}(M, F) \subset \cl{B}(G_b, F). 
\]
of the reduced Bruhat-Tits buildings. Fix a rational point $x \in \cl{B}(M, F)$. Its image in $\cl{B}(G, F)$ (resp.\ $\cl{B}(G_b, F)$) is denoted by $x$ (resp.\ $x_b$). There is a canonical identification 
\[
    \cl{B}(G, F) \subset \cl{B}(G, \breve{F}) \cong \cl{B}(G_b, \breve{F}) \supset \cl{B}(G_b, F)
\]
and $x$ and $x_b$ are distinct in general under this identification. Nevertheless, $x_b \in \cl{B}(G, \breve{F})$ is independent of the choice of a representative $b \in [b] \subset Z_M^\circ(\breve{F})$. 

For $0 < r \leq 2s$, the Moy-Prasad subgroups associated to twisted Levi embeddings $M \subset G$ and $M \subset G_b$ are denoted by
\[ 
    G(F)_{x, r, s} = G(F) \cap G(E)_{x, r, s} ,\quad 
    G_b(F)_{x_b, r, s} = G_b(F) \cap G_b(E)_{x_b, r, s}
\]
following \cite[Section 2.5]{Fin21} (see also \cite[Section 2.2]{Tak26_Yu}). For $r = 0$, we set $G(F)_{x, 0, s} = G(F)_{x, 0+, s} M(F)_{x, 0}$ and $G_b(F)_{x_b, 0, s} = G_b(F)_{x_b, 0+, s} M(F)_{x, 0}$. 

\begin{defi} \label{defi:level_subgroups}
    For each $r \geq 0$, let 
    \[
        K_r = G(F)_{x, r/2} M(F)_{x} \subset G(F) ,\quad 
        K_{b, r} = G_b(F)_{x_b, r / 2} M(F)_{x} \subset G_b(F)
    \]
    be open subgroups that are compact modulo center. Here, $M(F)_x$ denotes the full stabilizer of $x \in \cl{B}(M, F)$ under $M(F)$. Moreover, we set 
    \[
        K_r^0 = K_r \cap \cl{G}(O_F) ,\quad 
        K_{b, r}^0 = K_{b, r} \cap \cl{G}_b(O_F). 
    \]
\end{defi}

Here, $K_0 \subset G(F)_x$ is not necessarily an equality. These open subgroups play an important role in Yu's construction \cite{Yu01}. In fact, there are transfers from smooth representations of $M(F)_{x}$ that are ``generic'' of depth $r$ to $K_r$ and $K_{b, r}$. As the situation is symmetric between $G$ and $G_b$, we review the construction of this transfer in terms of $M \subset G$. 

Recall that an unramified torus $S \subset G$ is \textit{maximally unramified} if $S$ is of maximal dimension among unramified subtori of $G$ (see \cite[Definition 3.4.2]{Kal19}). For transfers at depth $r = 0$, we will assume the following setting. 

\begin{ass} \label{ass:when_r=0}
    For $r = 0$, we assume that $M$ contains a maximally unramified torus $S \subset G$ such that $x \in \cl{A}(S, F)$ and fix a parabolic subgroup $P \subset G_{\breve{F}}$ with Levi factor $M_{\breve{F}}$. 
\end{ass}

\begin{sett} \label{sett:parabolic_r=0}
    Since $S \subset M$ is maximally unramified in $G$, $M_{\breve{F}} \subset G_{\breve{F}}$ is a Levi subgroup and the existence of $P$ is not an additional condition. The special fiber $\msf{S}$ of the standard integral model of $S$ is a maximal torus of $\msf{G}$ and there is a parabolic subgroup
    \[
        \msf{M}_{\ov{k}} \subset \msf{P} = \msf{G}_{\ov{k}}^{\lambda \geq 0} \subset \msf{G}_{\ov{k}}
    \]
    associated to $\msf{P}$ where $\lambda \in X_*(\msf{S}) \cong X_*(S)$ is a cocharacter satisfying $P = G_{\breve{F}}^{\lambda \geq 0}$. Let $\msf{N}$ be the unipotent radical of $\msf{P}$ and let 
    \[
        \msf{Y}_{\msf{P}} = \{ g \in \msf{G}_{\ov{k}} / \msf{N} \mid g ^{-1} \sigma(g) \in \msf{N} \cdot \sigma(\msf{N}) \}
    \]
    be the parabolic Deligne-Lusztig variety associated to $\msf{P}$. The adjoint action of $M(F)_x$ on $\msf{G} = G(F)_{x, 0} / G(F)_{x, 0+}$ induces an action $\Ad\colon M(F)_x \circlearrowright \msf{Y}_{\msf{P}}$. It is compatible with the natural $\msf{G} \times \msf{M}$-action in the sense that $\Ad(m) = (m, m) \circlearrowright \msf{Y}_{\msf{P}}$ for every $m \in M(F)_{x, 0}$. 
\end{sett}

\begin{defi}
    An irreducible smooth representation $\rho \in \Irr^\sm(M(F)_{x})$ is \textit{generic of depth $r$} with respect to $G$ if it is trivial on $M(F)_{x, r+}$ and satisfies the following condition. 
    \begin{enumerate}
        \item When $r = 0$, \Cref{ass:when_r=0} holds and $\rho \vert_{\msf{M}}$ belongs to blocks associated to $(\msf{G}, \msf{M})$-regular rational series in the sense of \cite[{}11.4]{BR03}. 
        \item When $r > 0$, $\rho \vert_{M(F)_{x, r} / M(F)_{x, r+}}$ is given by a scalar $\psi \circ X$ for some nontrivial character $\psi \colon k \to \Qlax$ and a $k$-linear $r$-generic map 
        \[
            X \colon \msf{m}_{x, r} \cong M(F)_{x, r} / M(F)_{x, r+} \to k
        \]
        with respect to $G$ in the sense of \cite[Definition 3.1]{Tak26_Yu}. 
    \end{enumerate}
\end{defi}

\begin{defi} \label{defi:R_induction}
    Let $\rho \in \Irr^\sm(M(F)_{x})$ be an irreducible smooth representation generic of depth $r$. We define a smooth representation $R^{K_r}_M(\rho)$ of $K_r$ as follows. 
    \begin{enumerate}
        \item When $r = 0$, let $d = \dim \msf{Y}_{\msf{P}}$ and we set
        \[
            R^{K_r}_M(\rho) = H_c^d(\msf{Y}_{\msf{P}}, \Qla) \mathop{\otimes}\limits_{\msf{M}} \rho\vert_{\msf{M}}
        \]
        where $G(F)_{x, 0}$ acts via $\msf{G} \circlearrowright H_c^d(\msf{Y}_{\msf{P}}, \Qla)$ and $M(F)_x$ acts diagonally. In particular, $R^{K_r}_M(\rho)\vert_{\msf{G}}$ is the usual Deligne-Lusztig induction along $\msf{M} \subset \msf{G}$. 
        \item When $r > 0$, consider the Heisenberg-Weil representation 
        \[
            \kappa_{x, r, \phi} \in \Irr(G(F)_{x, r, r/2} / G(F)_{x, r+, r/2+} \rtimes M(F)_{x})
        \]
        introduced in \cite[Lemma 3.2]{Tak26_Yu} depending only on $\psi \circ X$. We set
        \[
            R^{K_r}_M(\rho) = \kappa_{x, r, \phi} \otimes \rho
        \]
        where $G(F)_{x, r, r/2}$ acts trivially on $\rho$ and $M(F)_x$ acts diagonally on each factor. 
    \end{enumerate}
    We write $R^H_M(\rho) = \cInd_{K_r}^H R^{K_r}_M(\rho)$ for subgroups $H \subset G(F)$ containing $K_r$. 
\end{defi}

\begin{prop}
    In \Cref{defi:R_induction} (1), $R^{K_r}_M(\rho)$ is independent of the choice of $P \subset G_{\breve{F}}$ made in \Cref{ass:when_r=0}. 
\end{prop}
\begin{proof}
    This is proved in \cite[Section 2.6, (2.17)]{Kal21} when $P$ is Borel and the same argument works by using the parabolic case of \cite[Theorem 1.3]{BDR17}. For the reader's convenience, we review an argument here. 
    
    First, $R\Gamma_c(\msf{Y}_{\msf{P}}, \Qla) \otimes_{\msf{M}} \rho\vert_{\msf{M}}$ is concentrated in degree $d$ by \cite[Corollaire 10.9]{BR03}. Then, by the Grothendieck-Lefschetz trace formula, the character of $R^{K_r}_M(\rho)$ is $\ell$-independent, so we may assume that $\ell$ is sufficiently large so that \cite{BDR17} can be applied. In particular, $R^{K_r}_M(\rho) \vert_{\msf{G}}$ is independent of $P$ by \cite[Theorem 1.3]{BDR17}. To treat the full action of $\wtd{\msf{G}} = K_0 / G(F)_{x, 0+}$, one needs to work with a disconnected variant of the Deligne-Lusztig variety
    \[
        \msf{\wtd{Y}}_{\msf{P}} = \msf{Y}_{\msf{P}} \times^{\msf{G}}  \wtd{\msf{G}}. 
    \] 
    In fact, $\msf{\wtd{G}}$ can be equipped with the structure of a group scheme so that $\msf{G} = \msf{\wtd{G}}^\circ$ and
    \[
        \wtd{\msf{Y}}_{\msf{P}} \cong \{ g \in \msf{\wtd{G}}_{\ov{k}} / \msf{N} \mid g ^{-1} \sigma(g) \in \msf{N} \cdot \sigma(\msf{N}) \}
    \]
    (see \cite[Lemma 2.6.1]{Kal21}). Let $\msf{\wtd{M}} = M(F)_x / M(F)_{x, 0+}$. Then, $\wtd{\msf{Y}}_{\msf{P}}$ is equipped with the $\msf{\wtd{G}} \times \msf{\wtd{M}}$-action and $R^{K_r}_M(\rho) = H_c^d(\msf{\wtd{Y}}_{\msf{P}}, \Qla) \otimes_{\msf{\wtd{M}}} \rho$. Now, for any other choice of a parabolic subgroup $\msf{Q} \subset \msf{G}_{\ov{k}}$ with Levi factor $\msf{M}_{\ov{k}}$, \cite[Theorem 6.2]{BDR17} extends to $\wtd{\msf{Y}}_{\msf{P}}$ as in \cite[Lemma 2.6.7]{Kal21} and we get an isomorphism at middle degrees
    \[
         H_c^d(\msf{\wtd{Y}}_{\msf{P}}, \Qla) \otimes_{\msf{\wtd{M}}} \rho \cong H_c^{d_{\msf{P}, \msf{Q}}}(\msf{\wtd{Y}}_{\msf{P}, \msf{Q}}, \Qla) \otimes_{\msf{\wtd{M}}} \rho \cong H_c^{d_{\msf{Q}}}(\msf{\wtd{Y}}_{\msf{Q}}, \Qla) \otimes_{\msf{\wtd{M}}} \rho. 
    \]
\end{proof}

Now, $R_M^{K_r}$ transfers irreducible supercuspidal representations from $M(F)$ to $G(F)$. 

\begin{prop} \label{prop:transfer_regular_supercuspidal}
    Let $\rho \in \Irr^\sm(M(F)_x)$ be a smooth representation generic of depth $r$ such that $\cInd^{M(F)}_{M(F)_x} \rho$ is irreducible and supercuspidal and the following condition holds. 
    \begin{enumerate}
        \item When $r = 0$, there is an elliptic maximal torus $T \subset M$ containing some $S$ such that there is a regular depth-zero character $\theta \colon T(F) \to \Qlax$ with respect to $T \subset G$ and 
        \[
            \rho = R_T^{M(F)_x}(\theta). 
        \]
        \item When $r > 0$, there is a $(G, M)$-supergeneric character $\phi \colon M(F) \to \Qlax$ of depth $r$ in the sense of \cite[Definition 3.1]{Tak26_Yu} such that $\rho \otimes \phi^{-1}$ is trivial on $M(F)_{x, r}$. 
    \end{enumerate}
    Then, $R^{G(F)}_M(\rho)$ is irreducible and supercuspidal. In particular, $R_M^{H}(\rho)$ is irreducible for every subgroup $H \subset G(F)$ containing $K_r$. 
\end{prop}
\begin{proof}
    For $r = 0$, it is easy to see from the associativity of (disconnected) Deligne-Lusztig induction that
    \[
        R^{K_0}_M(\rho) = \cInd^{K_0}_{G(F)_{x, 0} T(F)} R_T^G(\theta). 
    \]
    Then, by construction (see \cite[Lemma 3.4.20]{Kal19}), $\cInd^{G(F)}_{K_0} R^{K_0}_M(\rho)$ is a regular depth-zero supercuspidal representation associated to $(T, \theta)$. For $r > 0$, the claim is a reformulation of \cite[Theorem 3.5]{Tak26_Yu}. 
\end{proof}

\begin{rmk}
    When $\cInd^{M(F)}_{M(F)_x} \rho$ is a regular supercuspidal representation associated to an elliptic pair $(T, \theta)$ with $x \in \cl{B}(T, F)$ that is also regular in $G$ in the sense of \cite{Kal19}, then $R_M^{G(F)}(\rho)$ is also a regular supercuspidal representation associated to $(T, \theta)$.
\end{rmk}


Later, we will compare two transfers along $M \subset G$ and $M \subset G_b$. As a condition imposing genericity for both inclusions, we introduce the following terminology. 

\begin{defi} \label{defi:supergeneric_representation}
    When $\rho \in \Irr^\sm(M(F)_x)$ satisfies the conditions in \Cref{prop:transfer_regular_supercuspidal}, we say that $\rho$ is \textit{$(G, M)$-supergeneric} of depth $r$. When $\rho$ is also $(G_b, M)$-supergeneric, we say that $\rho$ is \textit{$(G, G_b, M)$-supergeneric} of depth $r$. 
\end{defi}

\subsection{Formulation of special open neighborhoods} \label{ssec:formulation_special_open}

In this section, we introduce the axioms of \textit{special open neighborhoods} of CM points in the stack of local shtukas, which provide a stacky interpretation of the theory of special affinoids in local Shimura varieties along the lines of \cite{Yos10} and \cite{BW16}. 

Since we work over $\Perf_{\ov{k}}$, $\Sht_{K, -\mu}$ is a $v$-stack over $\Spd(\breve{E})$ for each open subgroup $K \subset G(F)$. To keep track of the Frobenius action, we first endow it with the structure over $\Div_E^1$. 

\begin{defi}\textup{(\cite[Definition II.1.19]{FS24})}
    The $v$-sheaf $\Div^1_E = \Spd(\breve{E}) / \varphi^{\bb{Z}}$ is a $v$-sheaf sending $S \in \Perf_{\ov{k}}$ to the set of Cartier divisors on $\cl{X}_S$ that $v$-locally arise from untilts. 
\end{defi}

For each untilt $S^\sharp \in \Div^1_E(S)$, $\BdRp(S^\sharp)$ and $\BdR(S^\sharp)$ can be defined as previously. Thus, the affine $\BdRp$-Grassmannian $\Gr_{G, \mu}$ is defined over $\Div^1_E$ and $\Sht_{K, -\mu}$ can be defined over $\Div^1_E$ by adopting \Cref{defi:local_shtuka_general_level}. 

\begin{lem}
    There is a natural isomorphism 
    \[  
        \Sht_{M(F)_x, -\mu} \cong [\ast / \und{M(F)_x}] \times \Div^1_E \cong \Sht_{M(F)_x, \mu}. 
    \]
\end{lem}
\begin{proof}
    Since $\mu$ is central in $M$, $\Gr_{M, \mu} \cong \Div^1_E \cong \Gr_{M, -\mu}$. Moreover, the action of $M(F)_x$ is trivial on each side, so we get the claim. 
\end{proof}

Since $\Bun_G \cong \Bun_{G_b}$ over $\Perf_{\ov{k}}$, we may consider the following open substacks. 

\begin{defi}
    We define the admissible open loci of stacks of local shtukas
    \[
        \Sht_{K, -\mu}^b = \BL_K^{-1}(\Bun_G^b) \subset \Sht_{K, -\mu}, \quad 
        \Sht_{K_b, \mu}^1 = \BL_{K_b}^{-1}(\Bun_G^1) \subset \Sht_{K_b, \mu}. 
    \]
    for open subgroups $K \subset G(F)$ and $K_b \subset G_b(F)$. Here, $\BL_K$ (resp.\ $\BL_{K_b})$ denotes the Beauville-Laszlo map (see \eqref{eq:Beauville_Laszlo}) at level $K$ (resp.\ $K_b$). 
\end{defi}

For our purpose, we work with CM points (cf.\ \Cref{defi:special_case_of_extremal_points}) at levels such as $K_r$. 

\begin{lem} \label{lem:CM_points_general_depth}
    Let $K \subset G(F)$ and $K_b \subset G_b(F)$ be open subgroups compact modulo $Z_G(F)$ such that $K \cap M(F) = K_b \cap M(F) = M(F)_x$. Then, there are natural closed immersions
    \[
        \Sht_{M(F)_x, -\mu} \subset \Sht_{K, -\mu}^b ,\quad 
        \Sht_{M(F)_x, \mu} \subset \Sht_{K_{b}, \mu}^1. 
    \]
\end{lem}
\begin{proof}
    It suffices to prove the claim for $\Sht_{K, -\mu}^b$. Since $x_\mu \in \Gr_{G, \mu}$ is stable under $M(F)_x$, the stabilizer of $x_\mu$ under $K$ is $M(F)_x$ by \Cref{lem:level_structure_special_points}. Then, we get an inclusion $\Sht_{M(F)_x, -\mu} \subset \Sht_{K, -\mu}$ and it is closed since $K / M(F)_x$ is compact. Since the diagram
    \begin{center}
        \begin{tikzcd}
            \Sht_{M(F)_x, -\mu} \ar[r] \ar[d] & \Bun_M^b \ar[d] \\
            \Sht_{K, -\mu} \ar[r] & \Bun_G
        \end{tikzcd}
    \end{center}
    commutes, we have $\Sht_{M(F)_x, -\mu} \subset \Sht_{K, -\mu}^b$ and the claim follows.
\end{proof}

\begin{lem} \label{lem:sht_comparison}
    In the setting of \Cref{lem:CM_points_general_depth}, there is a natural isomorphism
    \[
        \Sht_{K, -\mu}^b \times_{[\ast / \und{G_b(F)}]} [\ast / \und{K_b}] \cong \Sht_{K_b, \mu}^1 \times_{[\ast / \und{G(F)}]} [\ast / \und{K}]. 
    \]
    Moreover, this stack admits a natural map from $[\ast / \und{M(F)_x}] \times \Div^1_E$ that is a common lift of the CM points on $\Sht_{K, -\mu}$ and $\Sht_{K_b, \mu}$. 
\end{lem}
\begin{proof}
    Let $\Hck_{G, -\mu}^{1, b}$ be the global Hecke stack over the Fargues-Fontaine curve (see \cite[Chapter IX]{FS24}) that classifies modifications with one leg of type $-\mu$ from admissible $G$-bundles to admissible $G_b$-bundles. It is defined over $\Div^1_E$ to consider modifications of type $-\mu$ and fits in the following commutative diagram. 
    \begin{center}
        \begin{tikzcd}
            & \Sht_{M(F)_x, -\mu} \ar[rr, "\sim"] \ar[ld] & & \Sht_{M(F)_x, \mu} \ar[rd] & \\
            \Sht_{K, -\mu}^b \ar[rr] \ar[d] & & \Hck^{1, b}_{G, -\mu} \ar[ld] \ar[d] \ar[rd] & & \Sht_{K_b, \mu}^1 \ar[ll] \ar[d] \\
            \lbrack \ast / \und{K} \rbrack \ar[r] & \lbrack \ast / \und{G(F)} \rbrack & \Div^1_E & \lbrack \ast / \und{G_b(F)} \rbrack & \lbrack \ast / \und{K_b} \rbrack \ar[l]. 
        \end{tikzcd}
    \end{center}
    The bottom two squares are both Cartesian by construction, so the claim follows. 
\end{proof}

\Cref{lem:sht_comparison} provides a common \'{e}tale localization of $\Sht_{K, -\mu}^b$ and $\Sht_{K_{b}, \mu}^1$. Thus, the local geometry around the CM point should be the same between $\Sht_{K, -\mu}^b$ and $\Sht_{K_{b}, \mu}^1$. 

Now, we formulate the axioms of special open neighborhoods using lisse-\'{e}tale and $\ell$-adic sheaves (see \cite[{}VII.6]{FS24} and \cite[Section 26]{Sch17}). We avoid (the base change of) the motivic six-functor formalism \cite{Sch26} as the involved space is not partially proper. Let $d = \dim \Gr_{G, \mu}$ and let $\IC_\mu = \Qla(\tfrac{d}{2})[d]$ denote the intersection cohomology complex on $\Sht_{K, -\mu}$ and $\Sht_{K_{b}, \mu}$. Here, we choose $\sqrt{q} \in \Qla$ to consider half Tate twists on $\Div^1_E$. 

\begin{defi} \label{defi:special_open_neighborhoods}
    Let $r \geq 0$ be a rational number and let 
    \[  
        K_r \subset J_r \subset G(F), \quad K_{b, r} \subset J_{b, r} \subset G_b(F)
    \]
    be open subgroups compact modulo $Z_G(F)$. Let 
    \[
        \Sht_{M(F)_x, -\mu} \subset \Sht_{J_r, -\mu}^b(r) \subset \Sht_{J_r, -\mu}^b ,\quad 
        \Sht_{M(F)_x, \mu} \subset \Sht_{J_{b, r}, \mu}^1(r) \subset \Sht_{J_{b, r}, \mu}^1
    \]
    be open neighborhoods of the CM points equipped with an isomorphism 
    \[
        \iota \colon \Sht_{J_r, -\mu}^b(r) \cong \Sht_{J_{b, r}, \mu}^1(r)
    \]
    over $\Hck_{G, -\mu}^{1, b}$ sending $\Sht_{M(F)_x, -\mu}$ to $\Sht_{M(F)_x, \mu}$. We call $(\Sht_{J_r, -\mu}^b(r), \Sht_{J_{b, r}, \mu}^1(r), \iota)$ \textit{special open neighborhoods} at levels $(J_r, J_{b, r})$ of depth $r$ if the following condition holds. 
    \begin{quote}
        Consider the correspondence
        \begin{center}
            \begin{tikzcd}
                \lbrack \ast / \und{J_r} \rbrack & \Sht_{J_r, -\mu}^b(r) \ar[r, "\sim", "\iota"'] \ar[l, "h_1"'] & \Sht_{J_{b, r}, \mu}^1(r) \ar[r, "h_2"] & \lbrack \ast / \und{J_{b, r}} \rbrack \times \Div^1_E. 
            \end{tikzcd}
        \end{center}
        Then, $h_2$ is quasicompact and the induced functor 
        \[
            \cl{D}(J_r, \Qla) \to \cl{D}(J_{b, r}, \Qla)^{BW_E}, \quad \rho \mapsto (h_2 \circ \iota)_! (h_1^*\rho \otimes \IC_\mu)
        \]
        satisfies that for every $(G, G_b, M)$-supergeneric $\rho \in \Irr^\sm(M(F)_{x})$ of depth $r$, the natural map 
        \begin{equation} \label{eq:condition_2_question}
            (h_2 \circ \iota)_! (h_1^*R^{J_r}_M(\rho) \otimes \IC_\mu)[R^{J_{b, r}}_M(\rho)] \to (h_2 \circ \iota)_* (h_1^*R^{J_r}_M(\rho) \otimes \IC_\mu)[R^{J_{b, r}}_M(\rho)]
        \end{equation}
        is an isomorphism and both sides are isomorphic to $R^{J_{b, r}}_M(\rho)$ (after forgetting the $W_E$-action). 
    \end{quote}
\end{defi}
\begin{rmk} \label{rmk:restriction_to_jumps}
    Here, $r$ should be a jump for the Moy-Prasad filtration of $M$ at $x$. In particular, we may assume $r \in \bb{Z}_{\geq 0}$ when $x \in \cl{B}(M, F)$ is a hyperspecial point. 
\end{rmk}

Now, we explain rigorous definitions of functors in \Cref{defi:special_open_neighborhoods}. For the first functor 
\[
    \rho \mapsto (h_2 \circ \iota)_! (h_1^*\rho \otimes \IC_\mu), 
\]
we use lisse-\'{e}tale sheaves developed in \cite[{}VII.6]{FS24}. As in the case of $\Bun_G$ (see \cite[{}VII.7.1, IX.2.3]{FS24}), we have natural identifications
\[
    \cl{D}_\lis([\ast / \und{J_r}], \Qla) \cong \cl{D}(J_r, \Qla), \quad 
    \cl{D}_\lis([\ast / \und{J_{b, r}}] \times \Div^1_E, \Qla) \cong \cl{D}(J_{b, r}, \Qla)^{BW_E}. 
\]
Since $h_2$ is $\ell$-cohomologically smooth (cf.\ \Cref{prop:BB_map}), $(h_2 \circ \iota)_!$ is well-defined for lisse-\'{e}tale sheaves (see \cite[p.161]{HI25}). Now, the problem is the definition of the second functor 
\[
    \rho \mapsto (h_2 \circ \iota)_* (h_1^*\rho \otimes \IC_\mu)
\]
and the construction of a natural map 
\[
    (h_2 \circ \iota)_! (h_1^*\rho \otimes \IC_\mu) \to (h_2 \circ \iota)_* (h_1^*\rho \otimes \IC_\mu)
\]
for $\rho \in \Irr^\sm(J_r)$. As the behavior of $(h_2 \circ \iota)_*$ for lisse-\'{e}tale sheaves is rather difficult to describe explicitly, we define it by hand using $\ell$-adic sheaves in the case of interest. First, we interpret $(h_2 \circ \iota)_!$ in the $\ell$-adic sheaf formalism. 

\begin{lem} \label{lem:qc_smooth_preserve_ell_adic}
    Let $f \colon Y \to X$ be a quasicompact and $\ell$-cohomologically smooth map of small $v$-stacks. Let $\{ \cl{F}_n \in \cl{D}_\et(Y, \bb{Z} / \ell^n) \}_{n \geq 1}$ be a collection with $\cl{F}_{n} \cong \cl{F}_{n + 1} \otimes_{\bb{Z}/\ell^{n+1}} \bb{Z}/\ell^n$. Then, we have 
    \[
        f_\natural( \varprojlim\nolimits_{n \geq 1} \cl{F}_n) \cong \varprojlim\nolimits_{n \geq 1} f_\natural(\cl{F}_n). 
    \]
\end{lem}
\begin{proof}
    Take a canonical compactification $Y \xrightarrow{j} \ov{Y} \xrightarrow{g} X$ of $f$. Since $f$ is quasicompact and $g$ is separated, $j$ is also quasicompact. By \cite[Proposition 8.10]{Man26_l_adic}, $f_\natural \cong f_!((-) \otimes f^! \bb{Z}_\ell)$ for $\ell$-adically complete sheaves, which include $\cl{F} = \varprojlim_{n \geq 1} \cl{F}_n$. Here, we work with solid sheaves introduced in \cite[{}VII.1]{FS24}. Since $f^! \bb{Z}_\ell$ is invertible, it is enough to show 
    \[
        f_!( \varprojlim\nolimits_{n \geq 1} \cl{F}_n) \cong \varprojlim\nolimits_{n \geq 1} f_!(\cl{F}_n). 
    \]
    By construction, $f_! = g_* j_!$ and it is formal that $g_*$ preserves $\ell$-adically complete sheaves. Since $j$ is quasicompact, $j_!$ also preserves $\ell$-adically complete sheaves by \cite[Lemma 5.2 (iv)]{Man26_l_adic}. By taking the composition, we get the claim. 
\end{proof}

As previously, we set $\IC_\mu = \ov{\bb{Z}}_\ell(\tfrac{d}{2})[d]$ when we work with $\ell$-adic torsion coefficients. 

\begin{lem}
    Suppose that $\rho \in \Irr^\sm(J_r)$ admits a $J_r$-stable $\ov{\bb{Z}}_\ell$-lattice $\rho_0$. Then
    \[
        (h_2 \circ \iota)_! (h_1^*(\rho_0 / \ell^n) \otimes \IC_\mu) \in \Perf_{\ov{\bb{Z}}_\ell / \ell^n}(J_{b, r})^{BW_E}
    \] 
    \[
        (h_2 \circ \iota)_* (h_1^*(\rho_0 / \ell^n) \otimes \IC_\mu) \in \Perf_{\ov{\bb{Z}}_\ell / \ell^n}(J_{b, r})^{BW_E}
    \]
    in the torsion \'{e}tale sheaf theory \cite{Sch17} for every $n \geq 0$. Moreover, we have 
    \[
        (h_2 \circ \iota)_! (h_1^*\rho \otimes \IC_\mu) = \bigl(\varprojlim_{n\geq 1} (h_2 \circ \iota)_! (h_1^*(\rho_0 / \ell^n) \otimes \IC_\mu) \bigr)\bigl[\tfrac{1}{\ell}\bigr] \in \Perf_{\Qla}(J_{b, r})^{BW_E}
    \]
\end{lem}
\begin{proof}
    Since $h_2$ is quasicompact and $\ell$-cohomologically smooth, the first claim follows from \cite[Proposition 23.12 (ii)]{Sch17} and the Verdier duality \cite[Proposition 23.3 (i)]{Sch17}. Moreover, we may apply \Cref{lem:qc_smooth_preserve_ell_adic} to see that $(h_2 \circ \iota)_!$ preserves $\ell$-adically complete sheaves. Since $h_1^*$ also preserves $\ell$-adically complete sheaves by \cite[{}VII.3.1 (i)]{FS24}, we get 
    \[
        (h_2 \circ \iota)_! (h_1^*\rho \otimes \IC_\mu) = \bigl(\varprojlim\nolimits_{n\geq 1} (h_2 \circ \iota)_! (h_1^*(\rho_0 / \ell^n) \otimes \IC_\mu) \bigr)\bigl[\tfrac{1}{\ell}\bigr] \in \Perf_{\Qla}(J_{b, r})^{BW_E}. 
    \]
\end{proof}

Let $D$ be the maximal $F$-split toral quotient of $G$. Since $G / G^\der \cong G_b / G_b^\der$, $D$ is also the maximal $F$-split toral quotient of $G_b$. 

\begin{lem} \label{lem:twist_lattice}
    For every $\rho \in \Irr^\sm(J_r)$, there is a character $\chi \colon D(F) \to \Qlax$ such that $\rho \otimes \chi$ admits a $J_r$-stable $\ov{\bb{Z}}_\ell$-lattice. Any two such characters differ by a character $D(F) \to \ov{\bb{Z}}_\ell^\times$. 
\end{lem}
\begin{proof}
    Let $Z_s \subset Z_G$ be the maximal $F$-split subtorus and let $\chi_\rho$ be the central character of $\rho$. Since $K_r$ is compact modulo $Z_G(F)$, $\rho$ admits a $J_r$-stable $\ov{\bb{Z}}_\ell$-lattice if $\chi_\rho \vert_{Z_s(F)}$ takes values in $\ov{\bb{Z}}_\ell^\times$. Since $X^*(D)_{\bb{Q}} \to X^*(Z_s)_{\bb{Q}}$ is an isomorphism, the norm character $\nu(\chi_\rho)$ lifts uniquely to $D(F)$ and any $\chi \colon D(F) \to \Qla$ such that $\nu(\chi) = - \nu(\chi_\rho)$ satisfies the given condition. 
\end{proof}

\begin{lem} \label{lem:twisting_D_character}
    For every character $\chi \colon D(F) \to \Qlax$ and $\rho \in \Irr^\sm(J_r)$, we have 
    \[
        (h_2 \circ \iota)_! (h_1^*(\rho \otimes \chi) \otimes \IC_\mu) \cong (h_2 \circ \iota)_! (h_1^*\rho \otimes \IC_\mu) \otimes \chi. 
    \]
    If $\chi$ takes values in $\ov{\bb{Z}}_\ell^\times$ and $\rho$ admits a $J_r$-stable $\ov{\bb{Z}}_\ell$-lattice $\rho_0$, we have 
    \[
        (h_2 \circ \iota)_* (h_1^*(\rho_0 / \ell^n \otimes \chi) \otimes \IC_\mu) \cong (h_2 \circ \iota)_* (h_1^*(\rho_0 / \ell^n) \otimes \IC_\mu) \otimes \chi.  
    \]
    In both cases, the $W_E$-action on the right-hand side is twisted by a smooth character $\xi_\chi$ corresponding to $\chi \circ \Nm_{D(E) / D(F)} \circ \mu^{-1} \colon E^\times \to \Qlax$ via local class field theory. 
\end{lem}
\begin{proof}
    Since $\chi$ (resp.\ $\xi_\chi$) can be regarded as a $\Qla$-local system of rank $1$ on each of $[\ast / \und{J_r}]$ and $[\ast / \und{J_{b, r}}]$ (resp.\ $\Div^1_E$), it is enough for both claims to show $h_1^*\chi \cong (h_2 \circ \iota)^* (\chi \boxtimes \xi_\chi)$ on $\Sht^b_{\cl{G}, -\mu}(r)$. Since $\iota$ is a map over $\Hck_{G, -\mu}^{1, b}$ and there is a natural map $\Hck_{G, -\mu}^{1, b} \to \Hck_{D, -\mu}$, it is enough to show $\overleftarrow{h}^*\chi \cong \overrightarrow{h}^*(\chi \boxtimes \xi_\chi)$ for the diagram
    \[
        [\ast / \und{D(F)}] \xleftarrow{\overleftarrow{h}} \Hck_{D, -\mu} \xrightarrow{\overrightarrow{h}} [\ast / \und{D(F)}] \times \Div^1_E. 
    \]
    Since $D$ is a torus, $\overrightarrow{h}$ is an isomorphism, so it is enough to show $T_{-\mu}(\chi) = \chi \boxtimes \xi_\chi$. This follows from the categorical local Langlands correspondence for tori (see \cite[Theorem 6.4.1]{Zou24}) since $\xi_\chi = \mu^{-1} \circ \varphi_\chi \vert_{W_E}$ where $\varphi_\chi \colon W_F \to {}^LD$ is the $L$-parameter of $\chi$.
\end{proof}

\begin{defi} \label{defi:map_from_!_to_*}
    For every $\rho \in \Irr^\sm(J_r)$, take a character $\chi \colon D(F) \to \Qlax$ such that $\rho \otimes \chi$ admits a $J_r$-stable $\ov{\bb{Z}}_\ell$-lattice $\rho_0$ and we define
    \[
        (h_2 \circ \iota)_* (h_1^*\rho \otimes \IC_\mu) = \bigl(\varprojlim_{n\geq 1} (h_2 \circ \iota)_* (h_1^*(\rho_0 / \ell^n) \otimes \IC_\mu) \bigr)\bigl[\tfrac{1}{\ell}\bigr] \otimes (\chi \boxtimes \xi_\chi)^{-1} \in \Perf_{\Qla}(J_{b, r})^{BW_E}. 
    \]
    It is independent of the choice of $\chi$ by \Cref{lem:twisting_D_character}. Moreover, we have a natural map 
    \begin{equation} \label{eq:compact_supp_and_ordinary}
        (h_2 \circ \iota)_! (h_1^*\rho \otimes \IC_\mu) \to (h_2 \circ \iota)_* (h_1^*\rho \otimes \IC_\mu)
    \end{equation}
    as the limit of $(h_2 \circ \iota)_! (h_1^*(\rho_0 / \ell^n) \otimes \IC_\mu) \to (h_2 \circ \iota)_* (h_1^*(\rho_0 / \ell^n) \otimes \IC_\mu)$. 
\end{defi}

Since the diagonal action of $Z_G(F)$ on the infinite-level local shtuka space is trivial (see \Cref{lem:diagonal_central_action}), both sides of \eqref{eq:compact_supp_and_ordinary} have the same central character $\chi$ as that of $\rho$. In \eqref{eq:condition_2_question}, we take the isotypic part in the derived category $\cl{D}(J_{b, r}, \Qla)_\chi$ of $J_{b,r}$-representations with fixed central character $\chi$. 


\subsection{Relation to local Shimura varieties and Hecke operators}

In this section, we first present some consequences of the existence of special open neighborhoods to the Hecke operator $T_{-\mu}$. Then, we explain how to construct them via the geometry of local Shimura varieties. First, we recall the definition of local shtuka spaces. 

\begin{defi}\textup{(cf.\ \cite[Section 23.3]{SW20})}
    For every open subgroup $K \subset G(F)$, the local shtuka space $\cl{M}_{G, b, \mu, K}$ is a locally spatial diamond over $\Spd(\breve{E})$ given by the fiber product 
    \[
        \cl{M}_{G, b, \mu, K} = \Sht^b_{K, -\mu} \times_{\Bun_G^b} \ast. 
    \]
    It is equipped with a Weil descent datum relative to $E$
    \[
        \sigma_E \colon \cl{M}_{G, b, \mu, K} \to \sigma_E^*\cl{M}_{G, b, \mu, K}
    \]
    coming from the structure of $\Sht^b_{\cl{G}, -\mu}$ over $\Div^1_E$. There is a natural $G_b(F)$-equivariant \'{e}tale map called the (Grothendieck-Messing) period map
    \[
        \pi_{\GM} \colon \cl{M}_{G, b, \mu, K} \to \Gr_{G_b, -\mu, \breve{E}}
    \]
    that is \'{e}tale locally isomorphic to $G(F) / K$ over the admissible locus $\Gr_{G_b, -\mu, \breve{E}}^1$. When $\mu$ is minuscule, $\cl{M}_{G, b, \mu, K}$ is uniquely represented by a smooth rigid analytic variety over $\breve{E}$. In that case, $\cl{M}_{G, b, \mu, K}$ denotes such rigid analytic variety and is called a local Shimura variety. 

    When $K = \cl{G}(O_F)$ for some smooth affine model $\cl{G}$ of $G$ over $O_F$, $\cl{M}_{G, b, \mu, K}$ is also denoted by $\cl{M}_{\cl{G}, b, \mu}$. The infinite-level diamond $\varprojlim_K \cl{M}_{G, b, \mu, K}$ is denoted by $\cl{M}_{G, b, \mu, \infty}$. 
\end{defi}

\begin{rmk}
    Our notation is non-standard when $\mu$ is non-minuscule. In that case, the local shtuka space usually refers to the larger fiber product 
    $
        \cl{M}_{G, b, \leq \mu, K} = \Sht^b_{K, \leq -\mu} \times_{\Bun_G^b} \ast. 
    $
\end{rmk}

Note that our sign convention on $\mu$ is opposite to \cite{SW20}. Moreover, we follow the following convention for the group action unless otherwise stated. 

\begin{conv} \label{conv:group_action}
    The $G_b(F)$-action on $\cl{M}_{G, b, \mu, K}$ is always taken to be a left action. For any subgroup $H \subset G(F)$ normalizing an open subgroup $K \subset G(F)$, the $H$-action on $\cl{M}_{G, b, \mu, K}$ is also taken to be a left action unless otherwise stated. 
\end{conv}

Though the symmetry between $G$ and $G_b$ breaks down by passing to local shtuka spaces, they recover classical contexts such as Lubin-Tate spaces and Rapoport-Zink spaces. 
The following is a standard fact on infinite-level local shtuka spaces. 

\begin{lem} \label{lem:diagonal_central_action}
    The diagonal action of $Z_G(F)$ on $\cl{M}_{G, b, \mu, \infty}$ is trivial. 
\end{lem}
\begin{proof}
    It follows from the explicit moduli interpretation of $\cl{M}_{G, b, \mu, \infty}$ in \cite[p.219]{SW20}. 
\end{proof}

Next, we review the Hecke operators relevant to our context. Let 
\[
    T_{-\mu} \colon \cl{D}_\lis(\Bun_G, \Qla) \to \cl{D}_\lis(\Bun_G, \Qla)^{BW_E}
\]
be the Hecke operator associated to the highest weight representation $V_{-\mu}$ of $\widehat{G} \rtimes W_E$ with extremal weight $-\mu$ (see \cite[Proposition IX.2.1]{FS24}). For $\rho \in \Irr^\sm(J_r)$ admitting a $J_r$-stable $\bb{\ov{Z}}_\ell$-lattice $\rho_0$, the pullback along $\cl{M}_{G, b, \leq \mu, J_r} \to [\ast / \und{J_r}]$ induces an \'{e}tale $\Qla$-local system $\cl{L}_\rho$ equipped with a $\bb{\ov{Z}}_\ell$-lattice $\cl{L}_{\rho_0}$ on $\cl{M}_{G, b, \leq \mu, J_r}$. By construction and proper base change, the induced functor 
\[
    i_b^* T_{-\mu} i_{1!} \colon \cl{D}(G(F), \Qla) \to \cl{D}(G_b(F), \Qla)
\]
satisfies that for every such $\rho$, we have 
\begin{equation} \label{eq:Hecke_vs_LSV}
    i_b^* T_{-\mu} i_{1!}(\cInd_{J_r}^{G(F)} \rho) \cong R\Gamma_c(\cl{M}_{G, b, \leq \mu, J_r, \bb{C}_p}, \cl{L}_\rho \otimes \IC_{\leq \mu}). 
\end{equation}
Here, $\IC_{\leq \mu}$ denotes the intersection cohomology complex\footnote{Precisely, it is the pullback of the $\ell$-adic Satake sheaf $\cl{S}_{-\mu}$ associated to $V_{-\mu}$ under the geometric Satake equivalence \cite[Chapter VI]{FS24}.} on $\cl{M}_{G, b, \leq \mu, J_r}$ and its restriction to $\cl{M}_{G, b, \mu, J_r}$ equals $\IC_\mu = \Qla(\tfrac{d}{2})[d]$. As in \cite[{}IX.3]{FS24}, the right-hand side of \eqref{eq:Hecke_vs_LSV} denotes the colimit of $\ell$-adic cohomology complexes over quasicompact open subsets. 

\begin{prop} \label{prop:consequence_Hecke_operator}
    Let $\rho \in \Irr^\sm(M(F)_x)$ be a $(G, G_b, M)$-supergeneric representation of depth $r$ and let $\pi = R^{G(F)}_M(\rho)$ and $\pi_b =  R^{G_b(F)}_M(\rho)$. If there exist special open neighborhoods of depth $r$ for $(M, \mu, x)$, then for some smooth character $\xi_E \colon W_E \to \Qlax$, we have 
    \[
        \pi_b \boxtimes \xi_E \subset H^0(i_b^* T_{-\mu} i_{1!}(\pi)). 
    \]
    In particular, their Fargues-Scholze parameters are equal, i.e.\ $\varphi^\FS_\pi = \varphi^\FS_{\pi_b}$. 
\end{prop}
\begin{proof}
    First, suppose that $\rho$ admits an $M(F)_x$-stable $\bb{\ov{Z}}_\ell$-lattice $\rho_0$. Let 
    \[
        \cl{U}(r) = \Sht_{J_r, -\mu}^b(r) \times_{[\ast / \und{J_{b, r}}]} \ast
    \]
    where $\Sht_{J_r, -\mu}^b(r) \to [\ast / \und{J_{b, r}}]$ is given by $h_2 \circ \iota$. Via the natural open embedding 
    \[
        [\ast / \und{J_{b, r}}] \hookrightarrow [\ast / \und{J_{b, r}}] \times_{[\ast / \und{G_b(F)}]} [\ast / \und{J_{b, r}}], 
    \]
    we get an open embedding $\cl{U}(r) \subset \cl{M}_{G, b, \mu, J_r}$. By construction, $\cl{U}(r)$ is stable under $J_{b, r}$ and 
    \[
        j \cdot \cl{U}(r) \cap \cl{U}(r) = \emptyset \quad (j \in G_b(F) - J_{b, r}). 
    \]
    Since the central character of $\rho$ takes values in $\ov{\bb{Z}}_\ell^\times$, $R^{J_r}_M(\rho)$ also admits a central character with values in $\ov{\bb{Z}}_\ell^\times$ and it admits a $J_r$-stable $\bb{\ov{Z}}_\ell$-lattice. Then, $\cl{U}(r) \subset \cl{M}_{G, b, \mu, J_r}$ induces a commutative diagram of $\ell$-adic cohomology groups
    \begin{center}
        \begin{tikzcd}
            H_c^0(\cl{U}(r)_{\bb{C}_p}, \cl{L}_{R^{J_r}_M(\rho)}(\tfrac{d}{2})[d]) \ar[r] \ar[d] & H_c^0(\cl{M}_{G, b, \leq \mu, J_r, \bb{C}_p}, \cl{L}_{R^{J_r}_M(\rho)} \otimes \IC_{\leq \mu}) \ar[d] \\
            H^0(\cl{U}(r)_{\bb{C}_p}, \cl{L}_{R^{J_r}_M(\rho)}(\tfrac{d}{2})[d]) & H^0(\cl{M}_{G, b, \leq \mu, J_r, \bb{C}_p}, \cl{L}_{R^{J_r}_M(\rho)} \otimes \IC_{\leq \mu}). \ar[l]
        \end{tikzcd}
    \end{center}
    Here, $H^0(\cl{M}_{\cl{G}, b, \leq \mu, \bb{C}_p}, \cl{L}_{R^{J_r}_M(\rho)} \otimes \IC_{\leq \mu})$ denotes the limit of $\ell$-adic cohomology groups over quasicompact open subsets. By qcqs base change \cite[Proposition 17.6]{Sch17} along $\Spd(\bb{C}_p) \to [\ast / \und{J_{b, r}}] \times \Div^1_E$, the left vertical arrow is identified with \eqref{eq:compact_supp_and_ordinary}, so its $R_M^{J_{b, r}}(\rho)$-isotypic part is an isomorphism by the required condition in \Cref{defi:special_open_neighborhoods}. If we forget the $W_E$-action, it is isomorphic to $R^{J_{b, r}}_M(\rho)$, so the diagram implies 
    \[
        R^{J_{b, r}}_M(\rho) \subset H_c^0(\cl{M}_{G, b, \leq \mu, J_r, \bb{C}_p}, \cl{L}_{R^{J_r}_M(\rho)} \otimes \IC_{\leq \mu})
    \]
    as a $J_{b, r}$-subrepresentation. Now, we apply the same argument to finite disjoint unions 
    \[
        \bigsqcup_{j \in \Lambda} j \cdot \cl{U}(r)_{\bb{C}_p} \subset \cl{M}_{G, b, \mu, J_r, \bb{C}_p}
    \]
    running over finite subsets $\Lambda \subset G_b(F) / J_{b, r}$. Then, we get a $G_b(F)$-equivariant inclusion
    \[
        \pi_b = \cInd_{J_{b, r}}^{G_b(F)} R^{J_{b, r}}_M(\rho) \subset H_c^0(\cl{M}_{G, b, \leq \mu, J_r, \bb{C}_p}, \cl{L}_{R^{J_r}_M(\rho)} \otimes \IC_{\leq \mu}). 
    \]
    Then, the claim on the Fargues-Scholze parameter is a formal consequence (see e.g.\ \cite[Corollary 7.23]{Tak25z}). Since $\Sht_{J_r, -\mu}^b(r)$ is defined over $\Div^1_E$, $\bigsqcup_{j \in G_b(F) / J_{b, r}} j \cdot \cl{U}(r)$ is stable under the Weil descent datum. It implies that $\pi_b \subset H^0(i_b^* T_{-\mu} i_{1!}(\pi))$ is stable under the $W_E$-action. Since $\pi_b$ is irreducible and the $W_E$-action on $H^0(i_b^* T_{-\mu} i_{1!}(\pi))$ commutes with $G_b(F)$, $W_E$ acts on $\pi_b$ by a continuous character $\xi_E$. Since $W_E^\ab = E^\times$, $\xi_E$ is automatically smooth. 

    Finally, we treat the general case. As in \Cref{lem:twist_lattice}, we take a character $\chi \colon D(F) \to \Qlax$ so that $\rho \otimes \chi$ admits an $M(F)_x$-stable lattice. Then, we have
    \[
        (\pi_b \otimes \chi) \boxtimes \xi_E \subset H^0(i_b^* T_{-\mu} i_{1!}(\pi \otimes \chi))
    \]
    for some smooth character $\xi_E \colon W_E \to \Qlax$. As in the proof of \Cref{lem:twisting_D_character}, we have $\overleftarrow{h}^*(\chi) \cong \overrightarrow{h}^*(\chi \boxtimes \xi_\chi)$ on $\Hck_{G, -\mu}^{1, b}$, so we get 
    $
        \pi_b \boxtimes (\xi_E \otimes \xi_\chi^{-1}) \subset H^0(i_b^* T_{-\mu} i_{1!}(\pi)). 
    $
\end{proof}
\begin{rmk}
    In fact, \Cref{defi:special_open_neighborhoods} is designed to imply \Cref{prop:consequence_Hecke_operator}. The main motivation for this is that if we set $\tau = \cInd_{M(F)_x}^{M(F)} \rho$, we should have $\varphi_\pi^\FS = \varphi_\tau^\FS = \varphi_{\pi_b}^\FS$ for a suitable $L$-embedding ${}^LM \to {}^LG$ (cf.\ \cite[Lemma 5.2.6]{Kal19}): this follows from the conjectural compatibility between Fargues-Scholze $L$-parameters and Kaletha's $L$-packets when $\tau$, $\pi$ and $\pi_b$ are regular supercuspidal representations. In that case, $\pi_b$ would appear in degree $0$ by the vanishing conjecture (see \cite[Conjecture 1.1 (2)]{HJ25}). Moreover, under the categorical local Langlands conjecture, $\pi_b$ would correspond on the spectral side to the action of the extremal weight space $V_{-\mu}[-\mu] \subset V_{-\mu}$ to $\pi$, so that $W_E$ would act by scalar
    \[
        W_E \xrightarrow{\varphi_\tau^\FS} \widehat{M} \rtimes W_E \to \widehat{Z_M^\circ} \rtimes W_E \xrightarrow{-\mu} \bb{G}_m. 
    \]
\end{rmk}

Recall that the study of explicit geometry around CM points has its origin in the study of stable models of modular curves (\cite{DR73}, \cite{Wei16}, etc.) and special affinoids in Lubin-Tate spaces (\cite{Yos10}, \cite{BW16}, etc.) We will explain how to construct special open neighborhoods directly from the geometry of local shtuka spaces. To simplify the geometry, we will assume the following mild condition. Later, we will restrict to the unramified case. 

\begin{lem} \label{lem:zero_dimensional_LSV}
    Suppose that the following coinvariant for the inertial action 
    \[
        X_*(Z_M^\circ)_{I_F}
    \]
    is torsion-free. Then, we have an isomorphism 
    \[
        \cl{M}_{\cl{M}, b, \mu} \cong \und{M(F) / \cl{M}(O_F)} \times \Spd(\breve{E}). 
    \]
    Moreover, under $\lvert \cl{M}_{\cl{M}, b, \mu} \rvert\cong M(F) / \cl{M}(O_F)$, the Weil descent datum is given by 
    \[
        \sigma_E(m) = z_\mu m
    \]
    for $z_\mu = \Nm_{E/ F}(\mu(\pi_E)) \in Z_M(F)$. Moreover, $z_\mu \in Z_G(F)$ if $(M, \mu)$ is unramified. 
\end{lem}
\begin{proof}
    Since $\mu$ is central in $M$, the Grothendieck-Messing period map
    \[
        \cl{M}_{\cl{M}, b, \mu} \to \Gr_{M, \mu} = \Spd(\breve{E})
    \]
    is \'{e}tale locally isomorphic to $\und{M(F) / \cl{M}(O_F)}$ (see the proof of \cite[Proposition 23.3.3]{SW20}). Thus, for the first claim, it is enough to show that $\cl{M}_{\cl{M}, b, \mu}$ admits an $\breve{E}$-valued point. We set $Z = Z_M^\circ$ and $\cl{Z} = Z_{\cl{M}}^\circ$. As previously, we have 
    \[
        \cl{M}_{\cl{Z}, b, \mu, \bb{C}_p} \cong \und{Z(F) / \cl{Z}(O_F)} \times \Spd(\bb{C}_p). 
    \]
    Fix a point $x \in \cl{M}_{\cl{Z}, b, \mu}(\bb{C}_p)$ and consider the inertial orbit $I_E \cdot x$. Since $ \cl{M}_{\cl{Z}, b, \mu}$ is \'{e}tale over $\Spd(\breve{E})$, $x$ is defined over a finite extension of $\breve{E}$, so $I_E \cdot x$ is finite. On the other hand, let $z_\tau \in Z(F)$ be an element satisfying $\tau \cdot x = z_\tau \cdot x$ for each $\tau \in I_E$. Since the inertial action and the $Z(F)$-action commute, $\tau \cdot y = z_\tau \cdot y$ for every $y \in \cl{M}_{\cl{Z}, b, \mu}(\bb{C}_p)$. In particular, 
    \[
        I_E \to Z(F) / \cl{Z}(O_F) ,\quad \tau \mapsto [z_\tau]
    \]
    is a homomorphism with finite image. However, the given assumption on $X_*(Z_M^\circ)_{I_F}$ implies that $Z(F) / \cl{Z}(O_F)$ is torsion-free (see \cite[Corollary 11.7.2]{KP23}). Thus, the above homomorphism is trivial and $x$ is defined over $\breve{E}$. Thus, we get the first claim. 

    Next, we compute the Weil descent datum. It commutes with the inner action, so we may again reduce to $\cl{M}_{\cl{Z}, b, \mu}$. Here, we choose $b = \Nm_{\breve{E} / \breve{F}}(\mu(\pi_E))$ so that it is defined over the maximal unramified subextension $F_\un / F$ of $E$. Then, $\sigma_E(b) = b$, so the second claim follows from the explicit form of the Weil descent datum in \cite[(3.1.6)]{PR24} (see also \cite[Construction 5.10]{Tak25z}), which implies 
    \[
        \sigma_E(m) = \prod_{0 \leq i < [F_\un \colon F]} \sigma^i(b) \cdot m = z_\mu m
    \]
    for each $m \in M(F) / \cl{M}(O_F)$. When $M$ is unramified, $F_\un = E$ and we may take $\pi_E = \pi \in F$. Then $z_\mu = \mu_z(\pi)$ with $\mu_z = \sum_{0 \leq i < [F_\un \colon F]} \sigma^i \mu$. Since $\mu_z$ is $\Gal(F_\un / F)$-invariant and $Z_M^\circ / Z_G^\circ$ is anisotropic, we have $\mu_z \in X_*(Z_G^\circ)$. 
\end{proof}
\begin{rmk}
    The condition on $X_*(Z_M^\circ)_{I_F}$ automatically holds when $(M, \mu)$ is unramified or $Z_M^\circ$ is the product of induced tori. In particular, it always holds for $G = \GL_n$ since $Z_M^\circ$ is a product of the multiplicative groups $E^\times$ of some finite extensions $E / F$. 
\end{rmk}

In the classical setting, local Shimura varieties are considered at (sometimes deeper than) parahoric levels, so the geometry of local Shimura varieties is first related to $\Sht_{J_r^0, -\mu}^b$ where 
\[
    J_r^0 = J_r \cap \cl{G}(O_F) , \quad 
    J_{b, r}^0 = J_{b, r} \cap \cl{G}_b(O_F). 
\]
One needs to find an $M(F)_x$-stable open substack $\Sht_{J_r^0, -\mu}^b(r) \subset \Sht_{J_r^0, -\mu}^b$ to obtain special open neighborhoods at levels $(J_r, J_{b, r})$. For this construction, we find some interaction with \Cref{thm:canonical_level_structures}: later, we will construct $\Sht_{J_r^0, -\mu}^b(r)$ as the image of the \textit{canonical trivialization}
\[
    \Sht_{\cl{G}, -\mu}^b(r) \to \Sht_{J_r^0, -\mu}
\]
on some open neighborhood $\Sht_{\cl{G}, -\mu}^b(r) \subset \Sht_{\cl{G}, -\mu}^b$ of the CM point. We will rewrite this construction in terms of local shtuka spaces. 

First, $M(F)_x$ acts on $\cl{M}_{\cl{G}, b, \mu}$ since $M(F)_x$ normalizes $\cl{G}(O_F) = G(F)_{x, 0}$. For each $m \in M(F)_x$, let $\Delta(m) = (m , m) \in G(F) \times G_b(F)$ act on $\cl{M}_{\cl{G}, b, \mu}$ diagonally through the $M(F)_x$-action and the $G_b(F)$-action. Now, recall $d = \dim \Gr_{G, \mu} = \dim \cl{M}_{\cl{G}, b, \mu}$. 

\begin{prop} \label{prop:restatement_LSV}
    Suppose that the coinvariant $X_*(Z_M^\circ)_{I_F}$ is torsion-free and we have
    \[
        J_r = J_r^0 M(F)_x ,\quad J_{b, r} = J_{b, r}^0 M(F)_x
    \]
    for some $r \geq 0$. If there is a connected quasicompact open subset 
    $
        \cl{U}(r) \subset \cl{M}_{\cl{G}, b, \mu}
    $
    with the following conditions, then there are special open neighborhoods at levels $(J_r, J_{b, r})$ of depth $r$. 
    \begin{enumerate}
        \item $\cl{U}(r)$ is stable under $J^0_{b, r}$ and $j \cdot \cl{U}(r) \cap \cl{U}(r) = \emptyset$ for $j \in G_b(F) - J_{b, r}^0$. Moreover, 
        \[
            \Delta(m)(\cl{U}(r)) = \cl{U}(r) \quad (m \in M(F)_x) ,\quad
            \sigma_E(\cl{U}(r)) = z_\mu \cdot \cl{U}(r)
        \]
        as open subsets of $\cl{M}_{\cl{G}, b, \mu}$. Here, $z_\mu$ acts via the $G_b(F)$-action. 
        \item The restriction of $\pi_\GM$ to $\bigsqcup_{j \in \cl{G}_b(O_F) / J^0_{b, r}} j \cdot \cl{U}(r)$ is an open immersion. 
        \item There is a unique $\breve{E}$-valued point $x_\CM \in \cl{U}(r) \cap \cl{M}_{\cl{M}, b, \mu}$ and a unique section
        \[
            \can_r \colon \cl{U}(r) \to \cl{M}_{G, b, \mu, J^0_r}
        \]
        sending $x_\CM$ into its natural lift via $\cl{M}_{\cl{M}, b, \mu} \subset \cl{M}_{G, b, \mu, J^0_r}$. Moreover, the open image $\Img(\can_r) \subset \cl{M}_{G, b, \mu, J^0_r}$ is stable under $J_{b, r}^0$. 
        \item For every $(G, G_b, M)$-supergeneric $\rho \in \Irr^\sm(M(F)_x)$ of depth $r$, the natural map
        \begin{equation} \label{eq:map_from_compact_to_usual}
            R\Gamma_c(\cl{U}(r)_{\bb{C}_p}, \can_r^*\cl{L}_{R^{J_r}_M(\rho)})[R^{J_{b, r}}_M(\rho)] \to R\Gamma(\cl{U}(r)_{\bb{C}_p}, \can_r^*\cl{L}_{R^{J_r}_M(\rho)})[R^{J_{b, r}}_M(\rho)]
        \end{equation}
        is an isomorphism and isomorphic to $R^{J_{b, r}}_M(\rho)[-d]$ in $\cl{D}(J_{b, r}, \Qla)$. 
    \end{enumerate}
\end{prop}

Here, $\cl{L}_{R_M^{J_r}(\rho)}$ is a $G_b(F)$-equivariant \'{e}tale $\Qla$-local system on $\cl{M}_{G, b, \mu, J_r}$ and its pullback to $\cl{M}_{G, b, \mu, J^0_r}$ admits a $\bb{\ov{Z}}_\ell$-lattice since $J_r^0$ is compact. As $\Img(\can_r)$ is stable under $J_{b, r}^0$ and the diagonal action of $M(F)_x$ by the uniqueness of $\can_r$, the complexes in (4) naturally admit actions of $J_{b, r}^0$ and $M(F)_x$. As they are compatible and give rise to the action of $J_{b, r} = J_{b, r}^0 M(F)_x$, we may require condition (4). 

\begin{proof}
    First, we construct $\Sht_{J_r, -\mu}^b(r)$. Let $\cl{U}^\infty(r) \subset \cl{M}_{G, b, \mu, \infty}$ be the inverse image of $\Img(\can_r) \subset \cl{M}_{G, b, \mu, J^0_r}$. Then, conditions (1) and (3) imply 
    \[
        (g, j) \cdot \cl{U}^\infty(r) \cap \cl{U}^\infty(r) = \emptyset \quad ((g, j) \in \cl{G}(O_F) \times G_b(F) - J^0_r \times J^0_{b, r}).
    \]
    Moreover, $\Img(\can_r)$ is stable under $\Delta(m)$ for $m \in M(F)_x$ since $\cl{U}(r)$ is connected and $\can_r$ is uniquely determined by the value of $x_\CM$. Thus, $\Delta(m) \cdot \cl{U}^\infty(r) = \cl{U}^\infty(r)$ for every $m \in M(F)_x$. Now, since $J_r = J_r^0 M(F)_x$, 
    \[
        \bigsqcup_{j \in G_b(F) / J_{b, r}^0} j \cdot \cl{U}^\infty(r) \subset \cl{M}_{G, b, \mu, \infty}
    \]
    is stable under $J_r \times G_b(F)$. Thus, we may take an open substack 
    \[
        \Sht_{J_r, -\mu}^b(r) = \bigsqcup_{j \in G_b(F) / J_{b, r}^0} j \cdot \cl{U}^\infty(r) / \und{J_r \times G_b(F)} \subset \Sht_{J_r, -\mu}^b 
    \]
    over $\Spd(\breve{E})$. By the uniqueness of $\can_r$ and condition (1), we also have $\sigma_E(\Img(\can_r)) = z_\mu \cdot \Img(\can_r)$. Thus, $\Sht_{J_r, -\mu}^b(r)$ can be defined over $\Div_E^1$. 

    Next, we construct $\Sht_{J_{b, r}, \mu}^1(r)$. We will show 
    \begin{equation} \label{eq:stabilizer_U_infty}
        (g, j) \cdot \cl{U}^\infty(r) \cap \cl{U}^\infty(r) = \emptyset \quad ((g, j) \in G(F) \times \cl{G}_b(O_F) - J^0_{r} \times J^0_{b, r}).
    \end{equation} 
    Suppose that $(g, j) \cdot \cl{U}^\infty(r) \cap \cl{U}^\infty(r)$ is nonempty. Let $x \in \cl{U}^\infty(r)$ be a geometric point such that $y = (g, j) \cdot x \in \cl{U}^\infty(r)$. Since the Grothendieck-Messing period map
    \[
        \pi_{\GM} \colon \cl{M}_{G, b, \mu, \infty} \to \Gr_{G_b, -\mu, \breve{E}}
    \]
    is a $\und{G(F)}$-torsor on the admissible locus, $\pi_\GM(y) = \pi_\GM(j \cdot x)$. Let the subscript $(-)_0$ denote the projection $\cl{U}^\infty(r) \to \Img(\can_r)$. Then, $\pi_\GM(y_0) = \pi_\GM(j \cdot x_0)$, so condition (2) implies $y_0 = j \cdot x_0$. In particular, $y \in J^0_r \cdot (j \cdot x)$ and we also have $j \in J^0_{b, r}$ by condition (1). Since the $G(F)$-action on $\cl{M}_{G, b, \mu, \infty}$ is free, we get $g \in J^0_r$. Thus, \eqref{eq:stabilizer_U_infty} follows. 
    
    By \cite[Corollary 23.3.2]{SW20}, there is an isomorphism
    \[
        \cl{M}_{G, b, \mu, \infty} \cong \cl{M}_{G_b, b^{-1}, -\mu, \infty}. 
    \]
    By \eqref{eq:stabilizer_U_infty}, since $J_{b, r} = J_{b, r}^0 M(F)_x$, we have an open subspace stable under $G(F) \times J_{b, r}$
    \[
        \bigsqcup_{g \in G(F) / J_{r}^0} g \cdot \cl{U}^\infty(r) \subset \cl{M}_{G_b, b^{-1}, -\mu, \infty}. 
    \]
    Thus, we may take an open substack 
    \[
        \Sht_{J_{b, r}, \mu}^1(r) = \bigsqcup_{g \in G(F) / J_{r}^0} g \cdot \cl{U}^\infty(r) / \und{G(F) \times J_{b, r}} \subset \Sht_{J_{b, r}, \mu}^1.  
    \]
    As $\sigma_E(\cl{U}^\infty(r)) = z_\mu \cdot \cl{U}^\infty(r)$, $\Sht_{J_{b, r}, \mu}^1(r)$ can be defined over $\Div^1_E$. Now, we have
    \[
        \Sht_{J_r, -\mu}^b(r) \cong \bigsqcup_{j \in J_{b, r} / J^0_{b, r}} j \cdot \cl{U}^\infty(r) / \und{J_r \times J_{b, r}} =  \bigsqcup_{g \in J_{r} / J^0_{r}} g \cdot \cl{U}^\infty(r) / \und{J_r \times J_{b, r}} \cong \Sht_{J_{b, r}, \mu}^1(r)
    \]
    since $J_r / J_r^0 \cong M(F)_x / M(F)_{x, 0} \cong J_{b, r} / J_{b, r}^0$. Since $\bigsqcup_{j \in J_{b, r} / J^0_{b, r}} j \cdot \cl{U}^\infty(r)$ is stable under $\sigma_E$, this provides $\iota \colon \Sht_{J_r, -\mu}^b(r) \cong \Sht_{J_{b, r}, \mu}^1(r)$ over $\Div^1_E$. 

    It remains to see that condition (4) implies the condition of \Cref{defi:special_open_neighborhoods}. In terms of the notation loc. cit., the pullback of $h_2 \circ \iota$ along $\Spd(\bb{C}_p) \to [\ast / \und{J_{b, r}}] \times \Div^1$ is equal to
    \[
        \cl{U}(r)_{\bb{C}_p} \to \Spd(\bb{C}_p). 
    \]
    By qcqs base change \cite[Proposition 17.6]{Sch17} along $\Spd(\bb{C}_p) \to [\ast / \und{J_{b, r}}] \times \Div^1$,  
    \[
        (h_2 \circ \iota)_! (h_1^*R^{J_r}_M(\rho) \otimes \IC_\mu) \to (h_2 \circ \iota)_* (h_1^*R^{J_r}_M(\rho) \otimes \IC_\mu)
    \]
    is identified with $R\Gamma_c(\cl{U}(r)_{\bb{C}_p}, \cl{L}_{R^{J_r}_M(\rho)}) \to R\Gamma(\cl{U}(r)_{\bb{C}_p}, \cl{L}_{R^{J_r}_M(\rho)})$ up to the shift by $d$. Thus, the condition of \Cref{defi:special_open_neighborhoods} follows from condition (4). 
\end{proof}

\begin{rmk} \label{rmk:Artin_vanishing}
    In practice, we assume that $\mu$ is minuscule and $\cl{U}(r)$ is constructed to be affinoid. Then, the Artin vanishing (see \cite{Ber96}, \cite{Han20}) implies 
    \[
        R\Gamma_c(\cl{U}(r)_{\bb{C}_p}, \cl{L}_{R^{J_r}_M(\rho)}) \in \cl{D}^{\geq d}, \quad 
        R\Gamma(\cl{U}(r)_{\bb{C}_p}, \cl{L}_{R^{J_r}_M(\rho)}) \in \cl{D}^{\leq d}. 
    \]
    This is why we expect the middle concentration in condition (4). 
\end{rmk}



It seems a bit tricky to treat the $M(F)_x$-action on $\cl{U}(0)$. Nevertheless, we will later assume that $x \in \cl{B}(M, F)$ is hyperspecial, in which case it is essentially enough to treat the action of $M(F)_{x, 0}$ by the following lemma. 

\begin{lem} \label{lem:hyperspecial_M(F)_x}
    When $x \in \cl{B}(M, F)$ is hyperspecial, we have $M(F)_x = Z_G(F) \cl{M}(O_F)$. 
\end{lem}
\begin{proof}
    Since $M$ is unramified and $T$ contains a maximally $F$-split torus $S \subset M$, the Cartan decomposition implies $M(F) = \cl{M}(O_F) \cdot T(F) \cdot \cl{M}(O_F)$ (see \cite[Theorem 5.2.1]{KP23}). Thus
    \[
        M(F)_x = \cl{M}(O_F) \cdot T(F) \cap M(F)_x \cdot \cl{M}(O_F). 
    \]
    Take $t \in T(F) \cap M(F)_x$. Since $T(F) = S(F) \cdot T(F)_0$ by \cite[Lemma 7.1.1]{Kal11}, we may assume $t \in S(F)$ to show $t \in Z_G(F) M(F)_{x, 0}$. Since $S(F) / S(F)_0 \cong X_*(S)$, we may assume $t = \nu(\pi)$ for some $\nu \in X_*(S)$. Since $\cl{A}(M, S)$ is an affine space under $X_*(S / Z_G \cap S) \otimes \bb{R}$, $t \in M(F)_x$ implies $\nu \in X_*(Z_G^\circ)$. Then, $t \in Z_G(F)$ and we get the claim. 
\end{proof}

\subsection{Unramified length-zero triples} \label{ssec:unramified_length_0_triple}

In this section, we introduce the class of CM pairs $(M, \mu)$ for which we construct special open neighborhoods at every depth. It is a generalization of unramified CM points on Lubin-Tate spaces studied in \cite{BW16} and is exactly the class of CM points studied in \cite{Tak25z} at depth zero. 

\begin{defi} \label{defi:unramified_length_0}
    We say that a triple $(M, \mu, x)$ is \textit{unramified and length-zero} if the following conditions hold. 
    \begin{enumerate}
        \item The pair $(M, \mu)$ is an unramified CM pair and $x \in \cl{B}(G, F)$ is a hyperspecial point. 
        \item For the facet $\mfr{f} \subset \cl{B}(G, \breve{F})$ with $x_b \in \mfr{f}^\circ$, $x$ is a vertex of $\mfr{f}$. 
    \end{enumerate}
    In this case, $G$ is also unramified and $E$ is unramified over $F$. 
\end{defi}

\begin{lem} \label{lem:automatic_minuscule}
    For every unramified length-zero triple $(M, \mu, x)$, $\mu$ is minuscule. In fact, $\mu_i = \sum_{0 \leq j < i} \sigma^j \mu$ is minuscule for every $i \geq 0$. 
\end{lem}
\begin{proof}
    Let $T \subset M$ be a maximally $F$-split unramified maximal torus with $x \in \cl{A}(M, T)$. Then, we have $x, x_b \in \mfr{f} \subset \cl{A}(\breve{G}, \breve{T})$. Since $\mfr{f}$ is $b\sigma$-stable and $b = \mu(\pi)$, $\mu_i(\pi) \cdot x$ is also a vertex of $\mfr{f}$. Then, $\mu_i$ is minuscule as $x$ and $\mu_i(\pi) \cdot x$ belong to the same facet $\mfr{f}$. 
\end{proof}

Condition (2) corresponds to the choice of a length-zero representative $b \in [b]$ made in \cite{Tak25z} and automatically implies that $\mu$ is minuscule. It turns out that unramified length-zero triples are essentially bijective to minuscule cocharacters. 

\begin{prop} \label{prop:existence_triple}
    Suppose that $G$ is split and $\{\mu\}$ is minuscule. For each hyperspecial point $x \in \cl{B}(G, F)$, there is an unramified length-zero triple $(M, \mu, x)$. 
\end{prop}
\begin{proof}
    First, we prove the existence of an unramified length-zero triple. It essentially follows from \cite[Remark 2.1, Section 5.1]{Tak25z}. Fix a Borel pair $T \subset B \subset G$ such that $x \in \cl{A}(G, T)$. Let $\mfr{a} \subset \cl{A}(\breve{G}, \breve{T})$ be the base alcove that contains $x$ and lies in the negative chamber with respect to $x$ and $B$. Let $\mu_0 \in X_*(T)$ be the dominant representative of $\{\mu\}$. Then, we may take $w \in N_{\cl{G}}(\cl{T})(O_{\breve{F}})$ so that $w \cdot \mfr{a} = \mu_0(-\pi)^{-1} \cdot \mfr{a}$. By applying \Cref{const:special_pair_from_Weyl_element} to $w$, we get a special pair $(M_w, \mu_w)$. 
    
    First, we show that $(M_w, \mu_w)$ is CM. Let $T_w = \Ad(p_w)(T) \subset M_w$ and $\mfr{a}_w = p_w \cdot \mfr{a}$. Since $p_w \in \cl{G}(O_{\breve{F}})$, we have $x \in \mfr{a}_w \subset \cl{A}(\breve{M}_w, \breve{T}_w)$. By the choice of $w$, $\mfr{a}_w$ is stable under $b\sigma$ by setting $b = \mu_w(-\pi)$. Let $\mfr{f} \subset \mfr{a}_w$ be the minimal facet containing the $b\sigma$-orbit $\{ (b\sigma)^i x\}_{i \geq 0}$. Since $G$ is split, $\mfr{f}$ is a minimal $b\sigma$-facet (cf.\ \cite[Corollary 5.3]{Tak25z}) and the set of $b\sigma$-fixed points $\mfr{f}^{b\sigma}$ is a singleton. By construction, the vectorial part of the facet $\mfr{f}$ is spanned by $\{\sigma^i \mu_w\}_{i \geq 0}$. Then, since $\mfr{f}$ is a facet, its vectorial part is the orthogonal complement of roots $\alpha \in \Phi(\breve{G}, \breve{T}_w)$ such that 
    \[
        \langle \alpha, \sigma^i \mu_w \rangle = 0\quad \forall \; i \geq 0, 
    \]
    which are exactly roots inside $\breve{M}_w$. In particular, $\mfr{f}$ is parallel to $X_*(Z_{\breve{M}_w} / Z_{\breve{G}}) \otimes \bb{R}$. Since $\mfr{f}^{b\sigma}$ is a singleton, it follows that $Z_{M_w} / Z_G$ is anisotropic. 

    To verify condition (2) in \Cref{defi:unramified_length_0}, it is enough to show $x_b = \mfr{f}^{b\sigma}$. By the unramified descent \cite[Axiom 4.1.27]{KP23}, we have  
    \[
        \cl{B}(G_b, F) \supset \cl{A}(\breve{G}, \breve{T}_w)^{b\sigma} \to \cl{A}(\breve{M}_w, \breve{T}_w)^{\sigma}  \subset \cl{B}(M, F). 
    \]  
    The middle map is induced from the quotient $\cl{A}(\breve{G}, \breve{T}_w) \to \cl{A}(\breve{M}_w, \breve{T}_w)$ by $X_*(Z_{\breve{M}_w} / Z_{\breve{G}}) \otimes \bb{R}$. Since $M_w$ is elliptic, $\cl{A}(\breve{G}, \breve{T}_w)^{b\sigma} \to \cl{A}(\breve{M}_w, \breve{T}_w)^{\sigma}$ is an isomorphism and its inverse is the restriction of the natural embedding $\cl{B}(M, F) \subset \cl{B}(G_b, F)$. Thus, $x_b \in \cl{A}(\breve{G}, \breve{T}_w)$ is the unique $b\sigma$-stable point that maps to the image of $x$ in $
    \cl{A}(\breve{M}_w, \breve{T}_w)$. Since $x \in \mfr{f}$ and $\mfr{f}$ is parallel to $X_*(Z_{\breve{M}_w} / Z_{\breve{G}}) \otimes \bb{R}$, it is easy to verify that $\mfr{f}^{b\sigma}$ satisfies this condition. 
\end{proof}

\begin{rmk}
    The splitness condition on $G$ is only used to ensure that $\mfr{f}$ is a minimal $b\sigma$-facet. As in \cite[Corollary 5.3]{Tak25z}, this can be relaxed to the condition that each $F$-simple factor of $G$ is $\breve{F}$-simple. 
\end{rmk}

\begin{prop} \label{prop:uniqueness_triple}
    If we fix $x \in \cl{B}(G, F)$ in the setting of \Cref{prop:existence_triple}, an unramified length-zero triple $(M, \mu, x)$ is unique up to conjugacy under $\cl{G}(O_F)$. 
\end{prop}
\begin{proof}
    Fix a Borel pair $T \subset B \subset G$ such that $x \in \cl{A}(G, T)$ and let $\mfr{a} \subset \cl{A}(\breve{G}, \breve{T})$ be the base alcove as in \Cref{prop:existence_triple}. On the other hand, let $T_x \subset M$ be a maximally $F$-split unramified maximal torus such that $x \in \cl{A}(M, T_x)$. By \cite[4.1.12 (3)]{KP23}, 
    \[
        \breve{T}_x = \Ad(p_w)(\breve{T})
    \]
    for some $p_w \in \cl{G}(O_{\breve{F}})$. Then, $w = p_w^{-1} \sigma(p_w) \in N_{\cl{G}}(\cl{T})(O_{\breve{F}})$. Now, $\mu \in X_*(T_x)$ and let $b = \mu(-\pi)$. By the construction of the embedding $\cl{B}(M, F) \subset \cl{B}(G_b, F)$, $x_b \in \cl{A}(\breve{G}, \breve{T}_x)$ is the unique $b\sigma$-stable point that maps to $x$ under 
    \[
        \cl{A}(\breve{G}, \breve{T}_x) \to \cl{A}(\breve{M}, \breve{T}_x). 
    \]
    In particular, the facet $\mfr{f} \subset \cl{B}(G, \breve{F})$ containing $x_b$ lies in the apartment $\cl{A}(\breve{G}, \breve{T}_x)$. Since $x \in \mfr{f}$, we have $(b\sigma)^i x \in \mfr{f}$ for every $i \geq 0$. 
    
    Fix an alcove $\mfr{a}_x \subset \cl{A}(\breve{G}, \breve{T}_x)$ containing $\mfr{f}$ and let $\mfr{a}_0 = p_w^{-1} \cdot \mfr{a}_x \subset \cl{A}(\breve{G}, \breve{T})$. We may assume $\mfr{a} = \mfr{a}_0$ by replacing $p_w$. Let $\mfr{f}_0 = p_w^{-1} \cdot \mfr{f}$ and  $\mu_0 = \Ad(p_w^{-1})(\mu) \in X_*(T)$. Then, $b\sigma \cdot \mfr{f} = \mfr{f}$ implies $ \mfr{f}_0 \subset \mu_0(-\pi)w \cdot \mfr{a} \cap \mfr{a}$, so we can take a Weyl element $v \in N_{\cl{G}}(\cl{T})(O_{\breve{F}})$ so that $v$ sends $\mu_0(-\pi)w \cdot \mfr{a}$ to $\mfr{a}$ and fixes $\mfr{f}_0$. 

    Since $v$ fixes $\mfr{f}_0$, $v$ commutes with $\mu_0(-\pi)$, so $\mu_0(-\pi)vw \cdot \mfr{a} = \mfr{a}$. In particular, $\mu_0 \in X_*(T)$ is a dominant representative of $\{ \mu \}$ and $(M_{vw}, \mu_{vw}, x)$ is an unramified length-zero triple associated to the Borel pair $T \subset B$ constructed in the proof of \Cref{prop:existence_triple}. It is enough to show that $(M_{vw}, \mu_{vw})$ is conjugate to $(M, \mu)$ under $\cl{G}(O_F)$. 

    By construction, $(M, \mu) = (M_w, \mu_w)$. It is enough to show $\Ad(p_{vw})(\mu) = \Ad(p_w)(\mu)$ for a suitable choice of $p_{vw}$. Let $q_v = p_{vw} p_w^{-1} \in \cl{G}(O_{\breve{F}})$. Then
    \[
        \sigma(q_v) = p_{vw} vw (p_w w)^{-1} = q_v \Ad(p_w)(v). 
    \]
    Since $\Ad(p_w)(v)$ fixes $\mfr{f}$, we have $\Ad(p_w)(v) \in \cl{M}(O_{\breve{F}})$. Then, we can make $q_v$ lie in $\cl{M}(O_{\breve{F}})$. Then, $\Ad(p_{vw})(\mu) = \Ad(q_v)(\Ad(p_w)(\mu)) = \Ad(p_w)(\mu)$, so we get the claim. 
\end{proof}

Now, we can state the main result of this paper. By \Cref{lem:automatic_minuscule}, $\cl{M}_{\cl{G}, b, \mu}$ is a rigid analytic variety over $\breve{F}$ and $\cl{U}(n)$ is constructed to be affinoid as announced in \Cref{rmk:Artin_vanishing}. 

\begin{thm}\textup{(=\Cref{thm:explicit_geometry_strongest})} \label{thm:explicit_geometry}
    Suppose that $F$ is a finite extension of $\bb{Q}_p$. For every unramified length-zero triple $(M, \mu, x)$ and $n \in \bb{Z}_{\geq 0}$, there is an open ball 
    $
        \cl{U}(n) \subset \cl{M}_{\cl{G}, b, \mu}
    $
    satisfying the conditions in \Cref{prop:restatement_LSV} for
    \[
        J_n = K_n, \quad 
        J_{b, n} = \left\{ \begin{alignedat}{4}
            & K_{b, n} & \; & (n = 2m) \\
            & K_{b, n - 1} & \; & (n = 2m - 1). 
        \end{alignedat}
        \right.
    \]
    In particular, there are special open neighborhoods at levels $(J_n, J_{b, n})$ of depth $n$ for every unramified length-zero triple $(M, \mu, x)$. 
\end{thm}

\begin{rmk}
    Since $x \in \cl{B}(M, F)$ is hyperspecial, jumps of the Moy-Prasad filtration at $x$ are integers. This is why we restrict to the case $r = n \in \bb{Z}_{\geq 0}$ (cf.\ \Cref{rmk:restriction_to_jumps}). 
\end{rmk}

From now on, $F$ is a finite extension of $\bb{Q}_p$ and we fix an unramified length-zero triple $(M, \mu, x)$. We will apply the notation introduced so far to the CM pair $(M, \mu)$. We regard $\mu$ as a minuscule cocharacter $\mu \colon \bb{G}_m \to \breve{\cl{G}}$. For later use (i.e.\ \Cref{lem:expBKF}), we set $b = \mu(-\pi) \in Z_M(\breve{F})$. 

\section{Depth-zero case} \label{sec:depth_zero_case}

In this section, we prove \Cref{thm:explicit_geometry} for $n = 0$. The essential part is already developed in \cite{Tak25z} and the main task is to fill in the missing conditions for \Cref{prop:restatement_LSV}. 

\subsection{Depth-zero special affinoids}

In this section, we review the construction of $\cl{U}(0)$ in \cite[Theorem 1]{Tak25z} and its properties in the notation of the present paper. The construction begins with the use of the $v$-sheaf theoretic integral model $\cl{M}^\ints_{\cl{G}, b, \mu}$. 

\begin{defi}\textup{(\cite[Definition 25.1.1]{SW20})} \label{defi:v_sheaf_integral_model}
    The $v$-sheaf theoretic integral model $\cl{M}_{\cl{G},b,\mu}^\ints$ of $\cl{M}_{\cl{G}, b, \mu}$ is a $v$-sheaf sending a perfectoid space $S$ over $O_{\breve{F}}$ to the set of pairs $(\cl{P}, \iota)$ where 
    \begin{enumerate}
        \item $\cl{P}$ is a local $\cl{G}$-shtuka whose Frobenius $\varphi_{\cl{P}}$ is bounded by $-\mu$, and
        \item $\iota \colon \cl{P}\vert_{\cl{Y}_{[r, \infty)}} \cong \cl{E}^b\vert_{\cl{Y}_{[r, \infty)}}$ (for sufficiently large $r$) is a quasi-isogeny from $\cl{P}$ to $\cl{E}^b$. 
    \end{enumerate}
\end{defi}

It is known by \cite[{}2.61, 2.63]{Gle26_local_Shimura} that $\cl{M}^\ints_{\cl{G}, b, \mu}$ is a prekimberlite, a $v$-sheaf theoretic analogue of formal schemes introduced in \cite{Gle24}, and its reduction $(\cl{M}^\ints_{\cl{G}, b, \mu})_\red$ is isomorphic to the affine Deligne-Lusztig variety 
\[
    (\cl{M}_{\cl{G}, b, \mu}^\ints)^\red \cong X_\mu(b) = \{ g \in G(\breve{F}) / \cl{G}(O_{\breve{F}}) \mid g^{-1} b \sigma(g) \in \cl{G}(O_{\breve{F}}) \mu(\pi) \cl{G}(O_{\breve{F}}) \}. 
\]
It contains a closed point $[1] \in X_\mu(b)(\ov{k})$. Then, $\cl{U}(0)$ is constructed inside the tubular neighborhood at $[1]$ using the following explicit description of the formal completion. 

Let $T \subset M$ be a maximally $F$-split unramified maximal torus with $x \in \cl{A}(M, T)$. Let $\Phi$ be the set of roots of $\breve{G}$ with respect to $\breve{T}$ and let $\Phi_{\mu < 0} = \{ \alpha \in \Phi \mid \langle \alpha, \mu \rangle = -1 \}$. The Frobenius action on $X^*(T)$, in particular on $\Phi$, is denoted by $\sigma$. We set 
\[
    R_{\cl{G}, \mu} = O_{\breve{F}} \llbracket u_\alpha \vert \alpha \in \Phi_{\mu < 0} \rrbracket \cong (\cl{\breve{G}}^{\mu < 0})^\wedge_{/1}. 
\]
Here, $u_\alpha$ is a coordinate of the root group $\cl{U}_\alpha \subset \cl{\breve{G}}$ associated to $\alpha \in \Phi$. Precisely, $u_\alpha$ denotes the coordinate for a fixed isomorphism $i_\alpha \colon \bb{A}^1 \cong \cl{U}_\alpha$. For later use, we take $i_\alpha$ as follows. 

\begin{lem} \label{lem:choiceia}
    The family $\{i_\alpha\colon \bb{A}^1 \cong \cl{U}_{\alpha}\}_{\alpha \in \Phi}$ can be taken so that the diagram
    \begin{center}
        \begin{tikzcd}
            \bb{A}^1 \ar[d, "\sigma^*_{\bb{A}^1}"] \ar[r, "i_{\sigma \alpha}"] & \cl{U}_{\sigma \alpha} \ar[d, "\sigma^*_{\cl{G}}"] \ar[r] & \Spec(O_{\breve{F}}) \ar[d, "\sigma^*"] \\
            \bb{A}^1 \ar[r, "i_{\alpha}"] & \cl{U}_{\alpha} \ar[r] & \Spec(O_{\breve{F}})
        \end{tikzcd}
    \end{center}
    commutes for every $\alpha \in \Phi$. Here, $\sigma^*_{\bb{A}^1}$ and $\sigma^*_{\cl{G}}$ denote the natural morphism relative to $\sigma^*$ coming from the $O_F$-structure of $\bb{A}^1$ and $\cl{G}$. 
\end{lem}
\begin{proof}
    Take a $\sigma$-orbit $\alpha, \sigma \alpha, \ldots, \sigma^{t-1}\alpha$ with $\alpha \in \Phi$ and $t \geq 1$. If we choose $i_\alpha$, there is a unique choice of $i_{\sigma^i\alpha}$ for $0\leq i < t$ so that the given diagram for $\sigma^i\alpha$ commutes for $0 \leq i \leq t - 2$. Take $c \in O_{\breve{F}}^\times$ so that $i_{\sigma^{t-1}\alpha}(t) = \sigma^*_{\cl{G}}(i_{\alpha}(ct))$. There is an element $x \in O_{\breve{F}}^\times$ such that $\sigma^t(x) = c^{-1}x$. Then, we may replace the coordinate of $\cl{U}_{\sigma^i\alpha}$ by $\sigma^i(x)u_\alpha$ and this choice makes the given diagram commute for every $\sigma^i\alpha$. 
\end{proof}

\begin{thm}\textup{(\cite[Theorem 5.3.5]{Ito25a})} \label{thm:univdef}
    There is an isomorphism
    \[
        (\cl{M}^\ints_{\cl{G},b,\mu})^\wedge_{/[1]} \cong \Spd(R_{\cl{G}, \mu}). 
    \]
\end{thm}

Ito's theorem also holds at every geometric point of $X_\mu(b)$, but we will only apply it to $[1] \in X_\mu(b)$. As we will see later, $[1]$ is chosen to be the reduction of a CM point in $\cl{M}_{\cl{G}, b, \mu}$, which is necessary for \Cref{prop:restatement_LSV} (3). Now, the essential property in calculation is the following explicit form of the local shtuka over $R_{\cl{G}, \mu}$. Let $\cl{P}^\univ$ be the universal local $\cl{G}$-shtuka over $\Spd(R_{\cl{G}, \mu})$ via the identification in \Cref{thm:univdef}. 

\begin{prop}\textup{(\cite[Corollary 2.9]{Tak25z})} \label{lem:expBKF}
    Let $(R, R^+)$ be a perfectoid Huber pair over $R_{\cl{G}, \mu}$ with a choice of elements $\pi^\flat, u_\alpha^\flat \in R^\flat$ such that $(\pi^\flat)^\sharp = \pi$ and $(u_\alpha^\flat)^\sharp=u_\alpha$. Then, $\cl{P}^\univ(\Spd(R^+))$ is given by the $\cl{G}$-Breuil-Kisin-Fargues \textup{($\cl{G}$-BKF)} module 
    \[
        \bigl(\cl{G}\otimes W_{O_F}(R^{\flat +}), \mu([\pi^\flat]-\pi) \prod_{\alpha \in \Phi_{\mu<0}} i_\alpha([u_\alpha^\flat]) \sigma\bigr). 
    \]
    In particular, $\cl{P}^\univ(R, R^+)$ is given by the restriction of this $\cl{G}$-BKF module to $\cl{Y}_{(R, R^+)}$. 
\end{prop}

Here, we identify local shtukas over $\Spd(R^+)$ with $\cl{G}$-BKF modules over $R^+$ via \cite{Gut23}. 

\begin{proof}
    It looks different from \cite[Corollary 2.9]{Tak25z} by a Weyl element $w$, but the difference stems from the different choice of $\mu$. The one in \cite{Tak25z} is a cocharacter of a maximally $F$-split unramified maximal torus in $G$ and differs from our $\mu$ by $\Ad(p_w)$ (see \Cref{const:special_pair_from_Weyl_element}). After all, the same proof as in \cite[Corollary 2.9]{Tak25z} works here since $(\breve{\cl{G}}, \mu(-\pi)\sigma)$ is the local $\cl{G}$-shtuka associated to $[1] \in X_\mu(b)$. 
\end{proof}

Then, $\cl{U}(0)$ will be defined as a certain explicit open ball inside $\Spf(R_{\cl{G}, \mu})_\eta$. The following formulas again look different from \cite{Tak25z} by a Weyl element, but this is a consequence of our choice of $\mu \in X_*(T)$ (cf.\ the proof of \Cref{lem:expBKF}). 

\begin{defi}\textup{(\cite[Section 3.1]{Tak25z})} \label{defi:parabolic_mu}
    There is a unique rational cocharacter $\lambda \in X_*(T)_{\bb{Q}}$ with $\lambda - q\sigma \lambda = \mu$. The denominator $e$ of $\lambda$ is coprime to $p$. Let
    \[
        \msf{P} = \msf{G}_{\ov{k}}^{e\lambda \geq 0} ,\quad 
        \ov{\msf{P}} = \msf{G}_{\ov{k}}^{e\lambda \leq 0} ,\quad 
        \msf{N} = \msf{G}_{\ov{k}}^{e\lambda > 0} ,\quad
        \ov{\msf{N}} = \msf{G}_{\ov{k}}^{e\lambda < 0}, \quad 
        \msf{U}_{\mu>0} = \msf{G}_{\ov{k}}^{\mu > 0} ,\quad
        \msf{U}_{\mu<0} = \msf{G}_{\ov{k}}^{\mu < 0}.  
    \]
    By \cite[Lemma 3.2, Proposition 3.6]{Tak25z}, we have $\msf{M}_{\ov{k}} = \msf{G}^{e\lambda = 0}$ and $\msf{N} \cap \sigma(\msf{\ov{N}}) = \msf{U}_{\mu > 0}$. 
\end{defi}

\begin{lem}\textup{(\cite[Lemma 3.12]{Tak25z})} \label{lem:vareps}
    Let $r_\alpha = \langle \alpha, q\sigma \lambda \rangle$ for $\alpha \in \Phi$. For each $\alpha \in \Phi_{\mu<0}$, there is a unique positive integer $n_\alpha \geq 1$ such that 
    \[
        \langle \sigma^{-n_\alpha}\alpha, \mu \rangle = 1, \quad 
        \langle \sigma^{- i}\alpha, \mu \rangle = 0 \quad (0 < i < n_\alpha). 
    \]
    Let $\beta_\alpha = - \sigma^{-n_\alpha}\alpha$. The map $\alpha \mapsto \beta_\alpha$ is a permutation of $\Phi_{\mu < 0}$ and we have $r_\alpha = q^{n_\alpha}(1-r_{\beta_\alpha})$ for every $\alpha \in \Phi_{\mu<0}$. Moreover, $1 - q^{-1} <r_\alpha<1$ for every $\alpha \in \Phi_{\mu<0}$. 
\end{lem}

Now, we may summarize the definition and properties of $\cl{U}(0)$ studied in \cite{Tak25z}. Here, we include some basic properties that are not stated loc. cit. but easy to verify. Let $\cl{G}(1)$ be the first congruence subgroup scheme of $\cl{G}$ such that
\[
    \cl{G}(1)(O_F) = \Ker(\cl{G}(O_F) \to \cl{G}(k)). 
\]

\begin{lem} \label{lem:unramified_CM_points}
    The $v$-sheaf $\cl{M}^\ints_{\cl{M}, b, \mu}$ is a disjoint union of copies of $\Spd(O_{\breve{F}})$. Let $X_\mu^{\cl{M}}(b) = (\cl{M}^\ints_{\cl{M}, b, \mu})_\red$ be the affine Deligne-Lusztig variety for $(\cl{M}, b, \mu)$. Then, $X^{\cl{M}}_\mu(b) \to X_\mu(b)$ is injective and $[1] \in X_\mu(b)$ lies in the image. Then, the composition
    \[
        \Spd(O_{\breve{F}}) \cong (\cl{M}^\ints_{\cl{M}, b, \mu})^\wedge_{/[1]} \to (\cl{M}^\ints_{\cl{G}, b, \mu})^\wedge_{/[1]} \cong \Spd(R_{\cl{G}, \mu})
    \]
    is a closed immersion given by the condition $u_\alpha = 0$ for all $\alpha \in \Phi_{\mu < 0}$. 
\end{lem}
\begin{proof}
    By applying \cite[Theorem 5.3.5]{Ito25a} to $\cl{M}^\ints_{\cl{M}, b, \mu}$, we see that the formal completion of $\cl{M}^\ints_{\cl{M}, b, \mu}$ at each closed point is isomorphic to $\Spd(O_{\breve{F}})$. Then, we get the first claim since $X^{\cl{M}}_\mu(b)$ is zero-dimensional by \cite[Theorem 3.1]{Zhu17}. By definition, it is easy to check that $X^{\cl{M}}_\mu(b) \to X_\mu(b)$ is injective and $[1]$ lies in the image. Now, consider the composition
    \[
        \Spd(O_{\breve{F}}) \cong (\cl{M}^\ints_{\cl{M}, b, \mu})^\wedge_{/[1]} \to (\cl{M}^\ints_{\cl{G}, b, \mu})^\wedge_{/[1]} \cong \Spd(R_{\cl{G}, \mu}). 
    \]
    Since this map is defined over $\Spd(O_{\breve{F}})$, it is necessarily a closed immersion. Thus, it is enough to show that it factors through the closed locus $\{ u_\alpha = 0 \mid \alpha \in \Phi_{\mu < 0} \}$.  

    Take a map $f \colon \Spd(O_{\bb{C}_p}) \to \Spd(R_{\cl{G}, \mu})$ corresponding to $O_F \subset O_{\bb{C}_p}$ and the condition $u_\alpha = 0$ for every $\alpha \in \Phi_{\mu < 0}$. By \Cref{lem:expBKF}, the associated $\cl{G}$-BKF module over $O_{\bb{C}_p}$ naturally admits the $\cl{M}$-BKF module structure
    \[
        (\cl{M} \otimes W_{O_F}(O_{\bb{C}_p}^\flat), \mu([\pi^\flat] - \pi) \sigma).
    \] 
    Since this is a deformation of $(\cl{\breve{M}}, \mu(-\pi)\sigma)$ in the sense of \cite{Ito25a}, $f$ lifts to $(\cl{M}^\ints_{\cl{M}, b, \mu})^\wedge_{/[1]}$. The induced map $\Spd(O_{\bb{C}_p}) \to (\cl{M}^\ints_{\cl{M}, b, \mu})^\wedge_{/[1]} \cong \Spd(O_{\breve{F}})$ is surjective, so we get the claim since $f$ maps to the closed locus $\{ u_\alpha = 0 \mid \alpha \in \Phi_{\mu < 0} \}$. 
\end{proof}

\begin{prop}\textup{(cf.\ \cite[Theorem 1]{Tak25z})} \label{prop:U(0)_previous_work}
    The connected rational subdomain
    \[
        \cl{U}(0) = \{ \lvert u_\alpha \rvert \leq \lvert \pi \rvert^{r_\alpha} \neq 0 \mid \alpha \in \Phi_{\mu < 0} \} \subset \Spf(R_{\cl{G}, \mu})_\eta
    \] 
    satisfies the following conditions. 
    \begin{enumerate}
        \item $\cl{U}(0)$ is stable under $\cl{G}_b(O_F)$ and $j \cdot \cl{U}(0) \cap \cl{U}(0) = \emptyset$ for $j \in G_b(F) - \cl{G}_b(O_F)$. Moreover, 
        $
            \sigma_E(\cl{U}(0)) = z_\mu \cdot \cl{U}(0)
        $
        as open subsets of $\cl{M}_{\cl{G}, b, \mu}$. Here, $z_\mu$ acts via the $G_b(F)$-action. 
        \item There is a unique $\breve{F}$-valued point $x_\CM \in \cl{U}(0) \cap \cl{M}_{\cl{M}, b, \mu}$. 
        \item The $\cl{G}(k)$-cover 
        \[
            \cl{W}(0) = \cl{U}(0) \times_{\cl{M}_{\cl{G}, b, \mu}} \cl{M}_{\cl{G}(1), b, \mu}
        \]
        has good reduction over the tame extension $\breve{F}_e / \breve{F}$ of degree $e$. Moreover, its reduction is a parabolic Deligne-Lusztig variety 
        \[
            \msf{W}(0) =  \{g \in \msf{G}/\msf{N}\mid g^{-1} \sigma(g) \in \msf{N} \cdot \sigma(\msf{N}) \}. 
        \]
        We have $\msf{G}_b \cong \msf{M}$ and the action of $\cl{G}(O_F) \times \cl{G}_b(O_F)$ on $\msf{W}(0)$ is the inflation of the natural action $\msf{G} \times \msf{M} \circlearrowleft \msf{W}(0)$. 
    \end{enumerate}
\end{prop}
\begin{proof}
    In \cite{Tak25z}, $\cl{G}_b(O_F)$ is denoted by $G_{b, \mfr{f}}$. First, \cite[Lemma 5.5]{Tak25z} implies $j \cdot \cl{U}(0) \cap \cl{U}(0) = \emptyset$ for $j \in G_b(F) - \cl{G}_b(O_F)$ since $\cl{G}_b(O_F)$ is the stabilizer of $[1] \in X_\mu(b)$. Moreover, $\cl{U}(0)$ is stable under $\cl{G}_b(O_F)$ by \cite[Proposition 5.7]{Tak25z} and (3) follows from \cite[Theorem 3.16, Proposition 5.7]{Tak25z}. Note that $\msf{G}_b \cong \msf{M}$ is proved in \cite[Lemma 5.6]{Tak25z}.
    
    For (2), $(\cl{M}^\ints_{\cl{G}, b, \mu})^\wedge_{/[1]} \cap \cl{M}_{\cl{M}, b, \mu}$ consists of at most one point by \Cref{lem:unramified_CM_points}. It also implies that the $\breve{F}$-valued point $x_\CM \in (\cl{M}^\ints_{\cl{M}, b, \mu})^\wedge_{/[1], \eta}$ is given by the condition $u_\alpha = 0$ for every $\alpha \in \Phi_{\mu < 0}$. Thus, $x_\CM \in \cl{U}(0)$. It remains to see $\sigma_E(\cl{U}(0)) = z_\mu \cdot \cl{U}(0)$. By \Cref{lem:zero_dimensional_LSV}, $\sigma_E(\cl{U}(0)) \cap z_\mu \cdot \cl{U}(0) \neq \emptyset$ since $x_\CM \in \cl{U}(0)$. Then, the claim follows from \cite[Proposition 5.14]{Tak25z} since $\cl{U}(0)$ is connected. 
\end{proof}

\Cref{prop:U(0)_previous_work} provides most of the information for \Cref{prop:restatement_LSV}. The only property remaining to be established is \Cref{prop:restatement_LSV} (2). 

\subsection{Computation of the Grothendieck-Messing period map}

In this section, we show that $\pi_\GM\vert_{\cl{U}(0)}$ is an open immersion, thereby establishing \Cref{prop:restatement_LSV} (2) for $\cl{U}(0)$. 

First, we review the construction of the quasi-isogeny $\iota$ to $\cl{E}^b$ on the tubular neighborhood $\Spf(R_{\cl{G}, \mu})_\eta$ (recall \Cref{defi:v_sheaf_integral_model}). Let $x \colon \Spa(C, C^+) \to \Spf(R_{\cl{G}, \mu})_\eta$ be a geometric point, i.e.\ a map from an algebraically closed perfectoid field, and let $P_x$ be the $\cl{G}$-BKF module associated to $x$ via \Cref{lem:expBKF}. By construction in \cite[Proposition 5.2.6]{Ito25a} and \cite[Lemma 2.24]{Gle26_local_Shimura}, $\iota$ is defined over
\[
    \cl{Y}_{(C, C^+), [r, \infty]} = \{ \lvert [\varpi] \rvert \leq \lvert \pi \rvert^r \neq 0 \} \subset \Spa(W_{O_F}(C^{\flat +})) - V(\pi)
\]
for sufficiently large $r$. In fact, the associated element can be explicitly described as follows. 

\begin{lem} \label{lem:quasi_isog_computation}
    Let $r, s$ be rational numbers such that $r > q^{-1}$ and $1 < s < \min(q, qr)$ and suppose that for each $\alpha \in \Phi_{\mu < 0}$, $x$ satisfies the inequality
    \[
        \lvert u_\alpha \rvert \leq \lvert \pi \rvert^r \neq 0. 
    \]
    Take $\pi^\flat, u_\alpha^\flat \in C^\flat$ so that $(\pi^\flat)^\sharp = \pi$ and $(u_\alpha^\flat)^\sharp=u_\alpha$ for $x$. Let $u = \prod_{\alpha \in \Phi_{\mu < 0}} i_\alpha([u_\alpha^\flat])$ and let $B_s = W_{O_F}(C^{\flat +})[\tfrac{[\pi^{\flat, s}]}{\pi}]^\wedge$ denote the $\pi$-adic completion. Then, there is a unique element $g \in \cl{G}(B_s)$ that is trivial modulo $[\varpi]$ for some pseudo-uniformizer $\varpi \in C^\flat$ and satisfies
    \begin{equation} \label{eq:g_condition}
        g = \Ad(\mu(-\pi)\sigma)(g) \left((\sigma\mu)\bigl(1 - \tfrac{[\pi^{\flat, q}]}{\pi}\bigr) \Ad(\mu_2(-\pi)\sigma)(u) \right)^{-1}
    \end{equation}
    in $G(B_s[\tfrac{1}{\pi}])$. Here, $\Ad(\mu(-\pi)\sigma)$ denotes the composition $\Ad(\mu(-\pi)) \circ \sigma$ and $\mu_2 = \mu + \sigma\mu$. 
\end{lem}

\begin{proof}
    Let $\mu_i = \sum_{0 \leq j < i} \sigma^j \mu$ for $i \geq 0$ and let $d = [E \colon F]$. Since $Z_M^\circ  / Z_G^\circ$ is anisotropic, $\mu_d$ is central. By applying \eqref{eq:g_condition} repeatedly, we get
    \begin{equation} \label{eq:g_condition_sigma^d}
        g = \Ad(\sigma^d)(g) \left(\prod_{1 \leq i \leq d} (\sigma^i\mu)\bigl(1 - \tfrac{[\pi^{\flat, q^{i}}]}{\pi}\bigr) \Ad(\mu_{i+1}(-\pi) \sigma^i)(u) \right)^{-1}. 
    \end{equation}
    Since $s < q$, $(\sigma^i\mu)\bigl(1 - \tfrac{[\pi^{\flat, q^{i}}]}{\pi}\bigr) \in \cl{G}(B_s)$ is trivial modulo $[\varpi]$ for some pseudo-uniformizer $\varpi \in C^\flat$. Now, for each $i \geq 1$, we have
    \[
        \Ad(\mu_{i + 1}(-\pi) \sigma^i)(u) \in \cl{G}(B_s)
    \]
    and it is trivial modulo $[\varpi]$ for some $\varpi$. It follows from $\mu_i$ being minuscule (see \Cref{lem:automatic_minuscule}) because $\tfrac{[u_\alpha^{q^i}]}{\pi} \in B_s$ and it is trivial modulo $[\varpi]$ for some $\varpi$ since $s < qr$. 

    Now, $X = \prod_{1 \leq i \leq d} (\sigma^i\mu)\bigl(1 - \tfrac{[\pi^{\flat, q^{i}}]}{\pi}\bigr) \Ad(\mu_{i+1}(-\pi) \sigma^i)(u) \in \cl{G}(B_s)$ is trivial modulo $[\varpi]$ for some $\varpi$. Since $B_s$ is $[\varpi]$-adically complete, the infinite product 
    \begin{equation} \label{eq:definition_Y}
        Y = X \sigma^d(X) \sigma^{2d}(X) \cdots = \prod_{i \geq 1} (\sigma^i\mu)\bigl(1 - \tfrac{[\pi^{\flat, q^{i}}]}{\pi}\bigr) \Ad(\mu_{i+1}(-\pi) \sigma^i)(u)
    \end{equation}
    converges in $\cl{G}(B_s)$ and $g = Y^{-1}$ satisfies \eqref{eq:g_condition}. It is also easy to see the uniqueness of $g$ from \eqref{eq:g_condition_sigma^d} since $g$ is assumed to be trivial modulo $[\varpi]$ for some $\varpi$. 
\end{proof}

\begin{thm} \label{thm:computation_period_map}
    Let $r$ be a rational number with $r > q^{-1}$ and take an open ball 
    \[
        \cl{U}_r = \{ \lvert u_\alpha \rvert \leq \lvert \pi \rvert^r \neq 0 \} \subset \Spf(R_{\cl{G}, \mu})_\eta
    \]
    of radius $r$. The period map $\pi_\GM$ induces an open immersion 
    \[
        \cl{U}_r \to \Gr_{G_b, -\mu, \breve{F}} \cong \Flag_{G, -\mu, \breve{F}}. 
    \]
    Moreover, the image in $\Flag_{G, -\mu, \breve{F}}$ is exactly equal to
    \[
        \cl{V}_r = \{ \lvert u_\alpha \rvert \leq \lvert \pi \rvert^{r - 1} \neq 0 \} \subset U_{\mu} \subset \Flag_{G, -\mu, \breve{F}}. 
    \]
\end{thm}
\begin{proof}
    For simplicity, we begin with a proof under the assumption $q^{-1} < r < 1$. First, we show that the image of $\pi_\GM$ lies in $\cl{V}_r$. Let $x \in \cl{U}_r(C, C^+)$ be a geometric point and take $g \in \cl{G}(B_s)$ as in \Cref{lem:quasi_isog_computation}. Then, the quasi-isogeny $\iota$ at $x$ from the $\cl{G}$-BKF module $P_x$ to $\cl{E}^b$ is given by an element $h$ defined by
    \[
        g = h \mu([\pi^{\flat}] - \pi) u \mu(-\pi)^{-1}
    \]  
    (see \cite[Lemma 2.24]{Gle26_local_Shimura} and \cite[Proposition 5.2.6]{Ito25a} for the characterization of $\iota$) since it is trivial modulo $[\varpi]$ for some $\varpi$ and \Cref{eq:g_condition} implies
    \[
        h \cdot \mu([\pi^\flat] - \pi)u = \mu(-\pi)\sigma(h) \in G(A_s[\tfrac{1}{[\pi^\flat] - \pi}])
    \]
    by setting $A_s = B_s[\tfrac{1}{\pi}]$ and we get a commutative  diagram 
    \begin{center}
        \begin{tikzcd}[column sep=10em]
            \varphi^*P_x \otimes A_s \lbrack \tfrac{1}{\lbrack \pi^\flat \rbrack - \pi} \rbrack \ar[r, "x \mapsto \mu(\lbrack \pi^\flat \rbrack - \pi) u x"] \ar[d, "x \mapsto \sigma(h)x"] & P_x \otimes A_s \lbrack \tfrac{1}{\lbrack \pi^\flat \rbrack - \pi} \rbrack \ar[d, "x \mapsto hx"] \\
            \varphi^* \cl{E}^b \otimes A_s \lbrack \tfrac{1}{\lbrack \pi^\flat \rbrack - \pi} \rbrack \ar[r, "x \mapsto \mu(- \pi) x"] & \cl{E}^b \otimes A_s \lbrack \tfrac{1}{\lbrack \pi^\flat \rbrack - \pi} \rbrack. 
        \end{tikzcd}
    \end{center}
    By the construction of the period map, we have
    \[
        \pi_\GM(x) = [h] = [g \mu(-\pi) u^{-1}] \in \Flag_{G, -\mu, \breve{F}}(C). 
    \]
    where $[g \mu(-\pi) u^{-1}]$ denotes the class in $G(C) / P_{-\mu}(C)$ of the reduction along $\theta_{C^+} \colon A_s \to C$. We compute this reduction as 
    \[
        \mu(-\pi) \cdot \Ad(\mu(-\pi)^{-1})(\ov{Y})^{-1} \ov{u}^{-1}
    \]
    where the reduction along $\theta_{C^+}$ is denoted by $\overline{(-)}$. By \eqref{eq:definition_Y}, $\Ad(\mu(-\pi)^{-1})(\ov{Y})^{-1} \ov{u}^{-1}$ equals
    \[
        \left( \prod_{i \geq 2} (\sigma^i\mu)\bigl(1 - \pi^{q^i-1}\bigr) \Ad((\sigma\mu_{i})(-\pi))(\ov{\sigma^iu}) \right)^{-1} \cdot \Ad((\sigma\mu)(-\pi))(\ov{\sigma u})^{-1} (\sigma \mu)(1 - \pi^{q-1})^{-1} \ov{u}^{-1}. 
    \]
    It lies in $\cl{G}(C^+)$ and is congruent to $\Ad((\sigma\mu)(-\pi))(\ov{\sigma u})^{-1} \cdot \ov{u}^{-1}$ modulo $\pi$. Now, the $\mu$-weight on $\ov{\sigma u} \in \cl{G}^{\sigma \mu < 0}$ is nonnegative since we have
    \[
        \cl{G}^{\sigma \mu < 0} \subset \cl{G}^{\sigma \lambda < 0} ,\quad 
        \cl{G}^{\mu < 0} \subset \cl{G}^{\sigma \lambda > 0}
    \]
    by $\msf{N} \cap \sigma(\ov{\msf{N}}) = \msf{U}_{\mu > 0}$ (see \Cref{defi:parabolic_mu}). Since $\Ad((\sigma\mu)(-\pi))(\ov{\sigma u})$ is trivial modulo $\pi^{qr - 1}$, 
    \[
        \Ad((\sigma\mu)(-\pi))(\ov{\sigma u})^{-1} \cdot \ov{u}^{-1} \equiv \ov{u}^{-1} \bmod{(\pi^{(q + 1) r - 1}, \cl{P}_{-\mu}(C^+))}. 
    \]
    Let $t = \min(1, (q + 1) r - 1) > r$. Then, we get
    \[
        [\Ad(\mu(-\pi)^{-1})(\ov{Y})^{-1} \ov{u}^{-1}] \equiv [\ov{u}^{-1}] \in (\cl{\breve{G}} / \cl{P}_{-\mu})(C^+ / \pi^t). 
    \]
    In particular, $\pi_\GM(x) \in \cl{V}_r$. Now, we show that $\pi_\GM \colon \cl{U}_r \to \cl{V}_r$ is an isomorphism. To distinguish the coordinates of $\cl{U}_r$ and $\cl{V}_r$, let $v_\alpha = u_\alpha$ denote the coordinate of $\cl{V}_r$ associated to $\alpha \in \Phi_{\mu < 0}$. The base change of $\pi_{\GM}$ to $C^+$ extends to formal models $\mfr{U}_{r, C^+} \to \mfr{V}_{r, C^+}$ where
    \[
        \mfr{U}_{r, C^+} = \Spf(C^+\langle \tfrac{u_\alpha}{\pi^r} \mid \alpha \in \Phi_{\mu < 0} \rangle), \;
        \mfr{V}_{r, C^+} = \Spf(C^+\langle \tfrac{v_\alpha}{\pi^{r - 1}} \mid \alpha \in \Phi_{\mu < 0} \rangle). 
    \]
    By the above computations, $\tfrac{v_\alpha}{\pi^{r - 1}} \equiv \tfrac{u_\alpha}{\pi^r} \pmod{\pi^{t - r}}$, so $\mfr{U}_{r, C^+} \to \mfr{V}_{r, C^+}$ is an isomorphism. 

    Before removing the restriction on $r$, we will show that $\mfr{U}_{r, C^+} \to \mfr{V}_{r, C^+}$ sends the origin of $\cl{U}_r$ ($u_\alpha = 0$) to the origin of $\cl{V}_r$ ($v_\alpha = 0$). Suppose that $x \in \cl{U}_r$ is the origin. Then $u = 1$, so \eqref{eq:definition_Y} implies 
    \[
        g = \left( \prod_{i \geq 1} (\sigma^i\mu)\bigl(1 - \tfrac{[\pi^{\flat, q^{i}}]}{\pi}\bigr) \right)^{-1} \in \cl{T}(B_s). 
    \]
    Thus, $[g] \in \Flag_{G, -\mu}$ is trivial and we get the claim. In particular, $\mfr{U}_{r, C^+} \to \mfr{V}_{r, C^+}$ is given on coordinates by 
    \[
        \tfrac{v_\alpha}{\pi^{r-1}} \mapsto P_\alpha \in C^+\langle \tfrac{u_\alpha}{\pi^r} \mid \alpha \in \Phi_{\mu < 0} \rangle
    \]
    such that $P_\alpha \equiv \tfrac{u_\alpha}{\pi^r} \pmod{\pi^{t-r}}$ and $P_\alpha(0) = 0$. Thus, for every $r' > r$, 
    \[
        \tfrac{v_\alpha}{\pi^{r' - 1}} \mapsto \tfrac{P_\alpha}{\pi^{r' - r}} \in C^+\langle \tfrac{u_\alpha}{\pi^{r'}} \mid \alpha \in \Phi_{\mu < 0} \rangle
    \]
    provides an isomorphism $\mfr{U}_{r', C^+} \to \mfr{V}_{r', C^+}$. By construction, the generic fiber is the restriction of $\pi_\GM$, so $\pi_\GM$ induces an isomorphism $\cl{U}_{r'} \to \cl{V}_{r'}$ for every $r' > r$. 
\end{proof}

\begin{rmk}
    One cannot expect that the restriction of $\pi_\GM$ to the whole of $\Spf(R_{\cl{G}, \mu})_\eta$ is an open immersion, which already fails in the case of Lubin-Tate spaces. 
\end{rmk}

\subsection{Verification of \Cref{thm:explicit_geometry} at depth zero}

In this section, we deduce \Cref{thm:explicit_geometry} for $n = 0$ from \Cref{prop:U(0)_previous_work} and \Cref{thm:computation_period_map}. 

\begin{prop} \label{prop:verification_at_depth_zero}
    The open ball $\cl{U}(0) \subset \cl{M}_{\cl{G}, b, \mu}$ satisfies the conditions in \Cref{prop:restatement_LSV}. 
\end{prop}
\begin{proof}
    First, condition (1) follows from \Cref{prop:U(0)_previous_work} (1) since $M(F)_x = Z_G(F) \cl{M}(O_F)$ and the diagonal action of $Z_G(F)$ is trivial (see \Cref{lem:hyperspecial_M(F)_x} and \Cref{lem:diagonal_central_action}). Next, condition (2) follows from \Cref{thm:computation_period_map} since $\cl{U}(0) \subset \cl{U}_{(1 - q^{-1})+}$ by \Cref{lem:vareps}. For condition (3), as $J_0^0 = \cl{G}(O_F)$, the existence of $\can_0$ is superfluous and the existence of $x_\CM$ follows from \Cref{prop:U(0)_previous_work} (2). It remains to deduce condition (4) from \Cref{prop:U(0)_previous_work} (3). 
    
    Let $\rho \in \Irr^\sm(M(F)_x)$ be a $(G, G_b, M)$-supergeneric smooth representation of depth $0$ and let $\chi \colon Z_G(F) \to \Qlax$ be the central character of $\rho$. We will compute the map 
    \[
        R\Gamma_c(\cl{U}(0)_{\bb{C}_p}, \cl{L}_{R^{K_0}_M(\rho)})[R^{J_{b, r}}_M(\rho)] \to R\Gamma(\cl{U}(0)_{\bb{C}_p}, \cl{L}_{R^{K_0}_M(\rho)})[R^{K_{b, 0}}_M(\rho)]. 
    \]
    Since $R_M^G(\rho) \vert_{\cl{G}(1)}$ is trivial, we have 
    \[
        R\Gamma_c(\cl{U}(0)_{\bb{C}_p}, \cl{L}_{R^{K_0}_M(\rho)}) \cong R\Gamma_c(\cl{W}(0)_{\bb{C}_p}, \Qla) \otimes_{\msf{G}} R^{K_0}_M(\rho)\vert_{\msf{G}} 
    \]
    and the same holds for $R\Gamma(\cl{U}(0)_{\bb{C}_p}, \cl{L}_{R^{K_0}_M(\rho)})$. Since $\cl{W}(0)_{\bb{C}_p}$ admits a $p$-adically smooth formal model $\mfr{W}(0)$ over $O_{\bb{C}_p}$ with reduction $\msf{W}(0)$, the nearby cycle provides a commutative diagram
    \begin{center}
        \begin{tikzcd}
            R\Gamma_c(\msf{W}(0), \Qla) \otimes_{\msf{G}} R^{K_0}_M(\rho)\vert_{\msf{G}} \ar[r, "\sim"] \ar[d] & R\Gamma_c(\cl{W}(0)_{\bb{C}_p}, \Qla) \otimes_{\msf{G}} R^{K_0}_M(\rho)\vert_{\msf{G}} \ar[d] \\
            R\Gamma(\msf{W}(0), \Qla) \otimes_{\msf{G}} R^{K_0}_M(\rho)\vert_{\msf{G}} \ar[r, "\sim"] & R\Gamma(\cl{W}(0)_{\bb{C}_p}, \Qla) \otimes_{\msf{G}} R^{K_0}_M(\rho)\vert_{\msf{G}} 
        \end{tikzcd}
    \end{center}
    by \cite[Lemma 4.5]{Tsu16}\footnote{The smooth algebraizability of $\mfr{W}(0)$ is imposed loc. cit., but this automatically follows from Elkik's algebraization (see \cite[footnote 6]{BS22} over the non-Noetherian base).}. Since $\rho\vert_{\msf{M}}$ is $(\msf{G}, \msf{M})$-superregular, it follows from \cite[Théorème 11.7, B']{BR03} that 
    \[
        R\Gamma_c(\msf{W}(0), \Qla)\lbrack \rho\vert_{\msf{M}} \rbrack \to R\Gamma(\msf{W}(0), \Qla)\lbrack \rho\vert_{\msf{M}} \rbrack
    \]
    is an isomorphism and isomorphic to $R_M^{K_0}(\rho)^\vee\vert_{\msf{G}} \otimes \rho\vert_{\msf{M}}[-d]$. Since $R^{K_{b, 0}}_M(\rho) \vert_{\cl{G}_b(O_F)}$ is the inflation of $\rho$ along $\cl{G}_b(O_F) \twoheadrightarrow \msf{G}_b \cong \msf{M}$, we see that
    \[
        R\Gamma_c(\cl{U}(0)_{\bb{C}_p}, \cl{L}_{R^{K_0}_M(\rho)})[R^{K_{b, 0}}_M(\rho) \vert_{\cl{G}_b(O_F)}] \to R\Gamma(\cl{U}(0)_{\bb{C}_p}, \cl{L}_{R^{K_0}_M(\rho)})[R^{K_{b, 0}}_M(\rho) \vert_{\cl{G}_b(O_F)}]
    \]
    is an isomorphism and isomorphic to $\rho[-d]$. Since the $Z_G(F)$-action on $R\Gamma_c(\cl{U}(0)_{\bb{C}_p}, \cl{L}_{R^{K_0}_M(\rho)})$ is given by $\chi$ and the diagonal $Z_G(F)$-action is trivial by \Cref{lem:diagonal_central_action}, the $Z_G(F)$-action on $R\Gamma_c(\cl{U}(0)_{\bb{C}_p}, \cl{L}_{R^{K_0}_M(\rho)})$ is also a scalar by $\chi$. Thus, we get condition (4) of \Cref{prop:restatement_LSV} since $K_{b, 0} = Z_G(F) \cl{G}_b(O_F)$ by \Cref{lem:hyperspecial_M(F)_x}. 
\end{proof}

\section{Special open balls for the positive-depth case} \label{sec:special_open_balls}

In this section, we introduce an open ball $\cl{U}(n) \subset \cl{M}_{\cl{G}, b, \mu}$ that we use in \Cref{thm:explicit_geometry} for each $n \geq 1$ and study its stability under group actions. From now on, we set
\[
    m = \lceil \tfrac{n}{2} \rceil. 
\]

\subsection{Radii of special open balls}

In this section, we specify the radius of $\cl{U}(n)$. The following definition is compatible with the case $n = 0$. 

\begin{defi} \label{defi:U(n)_radius}
    When $n = 2m - 1$, let
    \[
        \cl{U}(n) = \{ \lvert u_\alpha \rvert \leq \lvert \pi \rvert^{m} \neq 0 \mid \alpha \in \Phi_{\mu<0}\} \subset \Spf(R_{\cl{G}, \mu})_\eta
    \]
    and when $n = 2m$, let
    \[
        \cl{U}(n) = \{ \lvert u_\alpha \rvert \leq \lvert \pi \rvert^{m + r_\alpha} \neq 0 \mid \alpha \in \Phi_{\mu<0}\} \subset \Spf(R_{\cl{G}, \mu})_\eta. 
    \]
\end{defi}

\begin{lem}
    The CM point $x_\CM \in \cl{U}(0) \cap \cl{M}_{\cl{M}, b, \mu}$ lies in $\cl{U}(n)$. 
\end{lem}
\begin{proof}
    As in the proof of \Cref{prop:U(0)_previous_work}, the origin of $\Spf(R_{\cl{G}, \mu})_\eta$ given by $u_\alpha = 0$ lies in $\cl{M}_{\cl{M}, b, \mu}$ and it is the CM point $x_\CM$. Thus, it is easy to see $x_\CM \in \cl{U}(n)$. 
\end{proof}

Since $0 < r_\alpha < 1$ (see \Cref{lem:vareps}), $\{ \cl{U}(n) \}_{n \geq 0}$ forms a decreasing family of open balls centered at $x_\CM$. The multiplication by $\lvert \pi \rvert$ of the radii of $\cl{U}(n)$ increases the depth $n$ by $2$. Another way to describe this construction is that $\cl{U}(n + 2)$ is constructed from the blowup of the formal model of $\cl{U}(n)$ at the reduction of $x_\CM$.

\subsection{Stability of special open balls} \label{ssec:stabilizer_open_balls}

In this section, we verify most parts of \Cref{prop:restatement_LSV} (1) for $\cl{U}(n)$. We begin with the inner action. In fact, it is essentially enough to keep track of the action on $x_\CM$ thanks to the existence of formal models. 

First, we will prove $j \cdot x_\CM \in \cl{U}(2m - 1)$ for $j \in K^0_{b, 2(m-1)} = G_b(F)_{x_b, m-1} M(F)_{x, 0}$. As seen from the defining inequalities of $\cl{U}(2m - 1)$, the following lemma plays an important role. 

\begin{lem}\textup{(\cite[Definition 4.1.6, Theorem 4.1.8]{Ito25a})} \label{lem:Perfd_universal_deformation}
    The set of homomorphisms
    \[
        f\colon R_{\cl{G}, \mu} \to O_{\bb{C}_p} / \pi^m \quad \text{such that} \quad f(u_\alpha) \equiv 0 \bmod{\pi} \quad (\alpha \in \Phi_{\mu < 0})
    \]
    is bijective to the set of deformations of the prismatic $(\cl{G}, \mu)$-display
    \[
        (\cl{G} \otimes W_{O_F}(O_{\bb{C}_p}^\flat)/[\pi^{\flat}], \mu(-\pi)\sigma)
    \]
    over the $O_F$-prism
    $
        (W_{O_F}(O_{\bb{C}_p}^\flat)/[\pi^{\flat, m}], [\pi^\flat] - \pi).
    $
\end{lem}
\begin{proof}
    This is a special case of the property (Perfd) loc. cit. applied to the universal $(\cl{G}, \mu)$-display $\mfr{Q}^\univ$ over $R_{\cl{G}, \mu}$. 
\end{proof}

Here, prismatic $(\cl{G}, \mu)$-displays are a generalization of $\cl{G}$-BKF modules of type $\mu$ to $O_F$-prisms $(A, I)$. Roughly, they are pairs $(P, \varphi_P)$ where $P$ is a $\cl{G}$-torsor over $A$ and
\[
    \varphi_P \colon \varphi_A^* P[1/I] \cong P[1/I]
\]
is a Frobenius structure that is of type $\mu$ in a certain sense. For more detailed expositions in this style, we refer to \cite[Section 1.5.2]{IKY23} and \cite[Section 2.2]{Tak25z}. By \Cref{lem:Perfd_universal_deformation}, in order to see $j \cdot x_\CM \in \cl{U}(2m - 1)$, we would like to show that the deformations
\[
    (\cl{G} \otimes W_{O_F}(O_{\bb{C}_p}^\flat)/[\pi^{\flat ,m}], \mu([\pi^\flat] - \pi) \sigma) ,\;
    (\cl{G} \otimes W_{O_F}(O_{\bb{C}_p}^\flat)/[\pi^{\flat ,m}], j \mu([\pi^\flat] - \pi) \sigma(j)^{-1} \sigma)
\]
are isomorphic. An isomorphism between them is supplied by the following lemma. 

Recall that the product map $\cl{\breve{G}}^{\lambda \leq 0} \times\cl{\breve{G}}^{\lambda > 0} \to \cl{G}$ is an open immersion and the open image is called the big cell attached to $\lambda$. In the following, the composition 
\[
    \Ad(\mu([\pi^\flat] - \pi)) \circ \sigma
\]
is denoted by $\Ad(\mu([\pi^\flat] - \pi) \sigma)$. 

\begin{lem} \label{lem:AdbCM}
    For every $i \geq 0$ and $j \in G_b(F)_{x_b, m-1}$, we have 
    \[
        \Ad(\mu([\pi^\flat] - \pi) \sigma)^i(j) \bmod{[\pi^{\flat, m}]} \in \cl{G}(W_{O_F}(O_{\bb{C}_p}^\flat)/[\pi^{\flat, m}]). 
    \]
    Moreover, for each $0 \leq k \leq m$, consider an ideal
    \[
        I_k = (\pi^k, \pi^{k-1}[\pi^\flat], \cdots, [\pi^{\flat, k}]) \subset W_{O_F}(O_{\bb{C}_p}^\flat) / [\pi^{\flat, m}]. 
    \]
    Then, $\Ad(\mu([\pi^\flat]-\pi) \sigma)^i(j)$ lies in the big cell $\cl{\breve{G}}^{\lambda \leq 0} \times \cl{\breve{G}}^{\lambda > 0}$ and the associated decomposition
    \[
        \Ad(\mu([\pi^\flat]-\pi) \sigma)^i(j) = \ov{p} u ,\quad 
        \ov{p} \in \cl{\breve{G}}^{\lambda \leq 0} , \quad u \in \cl{\breve{G}}^{\lambda > 0}
    \]
    satisfies that $\ov{p}$ is trivial modulo $I_{m-1}$, and $u$ is trivial modulo $I_m$.  
\end{lem}

Here, the first claim is meaningful since $[\pi^\flat] - \pi$ is a non-zero-divisor in $W_{O_F}(O_{\bb{C}_p}^\flat)/[\pi^{\flat, m}]$. 

\begin{proof}
    We will inductively prove the claim. For $i=0$, both claims automatically follow from $j \in G_b(F)_{x_b, m - 1}$ since the facet $\mfr{f}$ containing both $x$ and $x_b$ contains $x - \varepsilon \lambda$ in its interior for a sufficiently small $\varepsilon > 0$ (see \cite[Definition 5.1]{Tak25z}) and it implies
    \[  
        G_b(\breve{F})_{x_b, m - 1} = \{ \ov{p}u \in \cl{\breve{G}}^{\lambda \leq 0}(O_{\breve{F}}) \times \cl{\breve{G}}^{\lambda > 0}(O_{\breve{F}}) \mid p \equiv 1 \bmod{\pi^{m-1}}, \; u \equiv 1 \bmod{\pi^{m}} \}. 
    \] 
    
    Now, suppose that the claim holds for some $i \geq 0$. By the induction hypothesis and $\msf{\ov{N}} \cap \sigma(\msf{N}) = \msf{U}_{\mu < 0}$ (see \Cref{defi:parabolic_mu}), the image of
    \[
        \sigma(\Ad(\mu([\pi^\flat]-\pi) \sigma)^i(j)) \in \cl{G}(W_{O_F}(O_{\bb{C}_p}^\flat)/[\pi^{\flat, m}])
    \]
    along $W_{O_F}(O_{\bb{C}_p}^\flat)/[\pi^{\flat, m}] \to O_{\bb{C}_p} / \pi^m$ lies in $\cl{G}^{\mu \geq 0}(O_{\bb{C}_p}/\pi^m)$. By applying \cite[Proposition 4.2.9]{Ito25a} to $(W_{O_F}(O_{\bb{C}_p}^\flat)/[\pi^{\flat, m}], [\pi^\flat] - \pi)$, we get the first claim
    \begin{equation} \label{eq:first_induction_conseq}
        \Ad(\mu([\pi^\flat]-\pi) \sigma)^{i+1}(j) \in \cl{G}(W_{O_F}(O_{\bb{C}_p}^\flat)/[\pi^{\flat, m}]). 
    \end{equation}
    We will prove the second claim for $i+1$. First, we have
    \[
        \Ad(\mu([\pi^\flat]-\pi) \sigma)^{i+1}(j) = \Ad(\mu([\pi^\flat]-\pi))(\sigma(\ov{p} u)). 
    \]
    Since $u$ is trivial modulo $I_m$, $\sigma(\ov{p} u) \equiv \sigma(\ov{p})$ modulo $I_m$, so it lies in the big cell $\cl{\breve{G}}^{\mu<0} \times \cl{G}^{\mu = 0} \times \cl{G}^{\mu>0}$. Then, we have a decomposition
    \[
        \sigma(\ov{p} u) = u^{-} m_\mu u^+
    \]
    where $u^- \in \cl{\breve{G}}^{\mu<0}$ is trivial modulo $I_m$ and $u^+ \in \cl{G}^{\mu>0}$ is trivial modulo $I_{m-1}$. Moreover, we have a decomposition $m_\mu = \ov{p}_\mu u_\mu$ where $\ov{p}_\mu \in \cl{\breve{G}}^{\lambda \leq 0}$ is trivial modulo $I_{m-1}$ and $u_\mu \in \cl{\breve{G}}^{\lambda > 0}$ is trivial modulo $I_m$, for $\sigma(\msf{\ov{N}}) \cap \msf{G}^{\mu = 0} \subset \msf{\ov{N}}$. 
    
    Now, consider the multiplication map 
    \[
        a_\xi \colon W_{O_F}(O_{\bb{C}_p}^\flat)/[\pi^{\flat, m}] \to W_{O_F}(O_{\bb{C}_p}^\flat)/[\pi^{\flat, m}] ,\quad 
        a \mapsto ([\pi^\flat] - \pi) a. 
    \]
    It is easy to see $a_\xi^{-1}(I_m) = I_{m-1}$ from the Teichm\"{u}ller expansion. Then, by \eqref{eq:first_induction_conseq}, we have
    \[
        \Ad(\mu([\pi^\flat]-\pi))(u^-) \equiv 1 \bmod{I_{m-1}} ,\quad
        \Ad(\mu([\pi^\flat]-\pi))(u^+) \equiv 1 \bmod{I_m}. 
    \]
    Then, we get a desired decomposition for $i + 1$
    \[
        \Ad(\mu([\pi^\flat]-\pi) \sigma)^{i+1}(j) = \Ad(\mu([\pi^\flat]-\pi))(u^-) \ov{p}_\mu \cdot u_\mu \Ad(\mu([\pi^\flat]-\pi))(u^+). 
    \]
\end{proof}

\begin{prop} \label{prop:isom_deformation}
    For every $j \in G_b(F)_{x_b, m-1}$, the deformations
    \[
        (\cl{G} \otimes W_{O_F}(O_{\bb{C}_p}^\flat)/[\pi^{\flat ,m}], \mu([\pi^\flat] - \pi) \sigma) ,\;
        (\cl{G} \otimes W_{O_F}(O_{\bb{C}_p}^\flat)/[\pi^{\flat ,m}], j \mu([\pi^\flat] - \pi) \sigma(j)^{-1} \sigma)
    \]
    are isomorphic. In other words, there is $\iota \in \cl{G}(W_{O_F}(O_{\bb{C}_p}^\flat)/[\pi^{\flat ,m}])$ that is trivial modulo $[\pi^\flat]$ such that
    $
        \iota j \mu([\pi^\flat] - \pi) \sigma(\iota j)^{-1} = \mu([\pi^\flat] - \pi). 
    $
\end{prop}

\begin{proof}
    Rewrite the desired equation as
    \[
        \iota = \Ad(\mu([\pi^\flat] - \pi)\sigma)(\iota) \cdot \Ad(\mu([\pi^\flat] - \pi) \sigma)(j) j^{-1}. 
    \]
    Let $X = \Ad(\mu([\pi^\flat] - \pi) \sigma)(j) j^{-1}$ modulo $[\pi^{\flat, m}]$. By \Cref{lem:AdbCM}, we get
    \[
        \Ad(\mu([\pi^\flat] - \pi)\sigma)^i(X) \in \cl{G}(W_{O_F}(O_{\bb{C}_p}^\flat)/[\pi^{\flat, m}])
    \]
    for every $i \geq 0$. Since $X$ is trivial modulo $[\pi^\flat]$ as $j \in G_b(F)$, $\Ad(\mu([\pi^\flat] - \pi)\sigma)^i(X)$ is trivial modulo $[\pi^{\flat, q^i}]$. Thus, $\Ad(\mu([\pi^\flat] - \pi)\sigma)^i(X)$ is trivial for sufficiently large $i$, so the infinite product
    \[
        \iota = \cdots \cdot \Ad(\mu([\pi^\flat] - \pi) \sigma)^2(X) \cdot  \Ad(\mu([\pi^\flat] - \pi) \sigma)(X) \cdot X 
    \]
    is well-defined. It satisfies $\iota = \Ad(\mu([\pi^\flat] - \pi)\sigma)(\iota) X$ and is trivial modulo $[\pi^\flat]$. Thus, we get the claim. 
\end{proof}

Now, we deduce the stability of special open balls from the action on CM points. In fact, our proof relies on the stability of $\cl{U}(0)$ and the algebraicity of $\cl{U}(n)$. 

\begin{prop} \label{prop:U(n)_stability}
    For every $n \geq 1$, $\cl{U}(n)$ is stable under $J^0_{b, n}$. 
\end{prop}
\begin{proof}
    Let $j \in G_b(F)_{x_b, m-1}$. By the description of the inner action on $\Spf(R_{\cl{G}, \mu})_\eta$ in \cite[Proposition 2.15]{Tak25z}, the deformations associated to $x_\CM$ and $j \cdot x_\CM$ over $(W_{O_F}(O_{\bb{C}_p}^\flat), [\pi^\flat] - \pi)$ are
    \[
        (\cl{G} \otimes W_{O_F}(O_{\bb{C}_p}^\flat), \mu([\pi^\flat] - \pi) \sigma) ,\;
        (\cl{G} \otimes W_{O_F}(O_{\bb{C}_p}^\flat), j \mu([\pi^\flat] - \pi) \sigma(j)^{-1} \sigma). 
    \]
    By \Cref{prop:isom_deformation}, they are congruent modulo $[\pi^{\flat, m}]$. Thus, by \Cref{lem:Perfd_universal_deformation}, the composition
    \[
        R_{\cl{G},\mu} \xrightarrow{j^*} R_{\cl{G},\mu} \xrightarrow{u_\alpha \mapsto 0} O_{\bb{C}_p} \twoheadrightarrow O_{\bb{C}_p}/\pi^m
    \]
    is equal to $R_{\cl{G},\mu} \xrightarrow{u_\alpha \mapsto 0} O_{\bb{C}_p} \twoheadrightarrow O_{\bb{C}_p}/\pi^m$. This means $j \cdot x_\CM \in \cl{U}(2m - 1)$ by the explicit description of $\cl{U}(2m - 1)$ in \Cref{defi:U(n)_radius}. As $x_\CM$ is fixed by $M(F)_{x, 0}$ by the uniqueness in \Cref{prop:U(0)_previous_work} (2), we get $j \cdot x_\CM \in \cl{U}(2m - 1)$ for $j \in K^0_{b, 2(m-1)}$. 

    Now, we first prove that $\cl{U}(2m - 1)$ is stable under $K^0_{b, 2(m-1)}$. Let $j \in K^0_{b, 2(m-1)}$. Since $j \cdot x_\CM \in \cl{U}(2m - 1)$, the homomorphism
    \[
        R_{\cl{G},\mu} \xrightarrow{j^*} R_{\cl{G},\mu}
    \]
    associated to the action of $j$ sends each $u_\alpha$ to $P_\alpha \in R_{\cl{G}, \mu}$ with $P_\alpha(0) \equiv 0 \pmod{\pi^m}$. Then, it follows that $j \cdot \cl{U}(2m - 1) \subset \cl{U}(2m - 1)$, so $\cl{U}(2m - 1)$ is stable under $J^0_{b, 2m-1} = K^0_{b, 2(m-1)}$.  

    Similarly, we prove that $\cl{U}(2m)$ is stable under $K^0_{b, 2m}$. Let $j \in K^0_{b, 2m}$. Since $\cl{U}(0)$ is stable under $\cl{G}_b(O_F)$ (see \Cref{prop:U(0)_previous_work} (1)), we have a homomorphism 
    \[  
        O_{\bb{C}_p} \langle \tfrac{u_\alpha}{\pi^{r_\alpha}} \mid \alpha \in \Phi_{\mu < 0} \rangle \xrightarrow{j^*} O_{\bb{C}_p} \langle \tfrac{u_\alpha}{\pi^{r_\alpha}} \mid \alpha \in \Phi_{\mu < 0} \rangle
    \]  
    associated to the action of $j$ on $\cl{U}(0)_{\bb{C}_p}$. Since $j \cdot x_\CM \in \cl{U}(2m + 1) \subset \cl{U}(2m)$, $j^*$ sends each $\tfrac{u_\alpha}{\pi^{r_\alpha}}$ to $P_\alpha \in O_{\bb{C}_p} \langle \tfrac{u_\alpha}{\pi^{r_\alpha}} \mid \alpha \in \Phi_{\mu < 0} \rangle$ with $P_\alpha(0) \equiv 0 \pmod{\pi^m}$. Then, it follows that $j \cdot \cl{U}(2m) \subset \cl{U}(2m)$, so $\cl{U}(2m)$ is stable under $J^0_{b, 2m} = K^0_{b, 2m}$. 
\end{proof}

For \Cref{prop:restatement_LSV} (1), we need the converse direction. In other words, for $j \in G_b(F)$ such that $j \cdot \cl{U}(n) \cap \cl{U}(n) \neq \emptyset$, we would like to show $j \in J_{b, n}^0$. In this case, one needs to show that $j \cdot x_\CM \in \cl{U}(2m)$ implies $j \in K_{b, 2m}^0$. This will be proved in \Cref{prop:action_outside_J_bn}, where we use another interpretation of $j \cdot x_\CM \in \cl{U}(2m)$ as a replacement of \Cref{lem:Perfd_universal_deformation}. 

The remaining parts for \Cref{prop:restatement_LSV} (1) are easy to verify right now since the diagonal action and the Weil descent datum are easy to calculate on $x_\CM$. 

\begin{prop} \label{prop:stability_U(n)_diagonal}
    For every $n \geq 1$ and $m \in M(F)_x$, we have 
     \[
        \Delta(m)(\cl{U}(n)) = \cl{U}(n) \quad (m \in M(F)_x) ,\quad
        \sigma_E(\cl{U}(n)) = z_\mu \cdot \cl{U}(n). 
    \]
\end{prop}
\begin{proof}
    Since $M(F)_x = Z_G(F) M(F)_{x, 0}$ and the diagonal action of $Z_G(F)$ is trivial, we may assume $m \in M(F)_{x, 0}$. Then, $\Delta(m)\cdot x_\CM = x_\CM$ by the uniqueness in \Cref{prop:U(0)_previous_work} (2). Since $\Delta(m)$ extends to the integral model $\cl{M}^\ints_{\cl{G}, b, \mu}$ and the formal completion $\Spf(R_{\cl{G}, \mu})$, $\Delta(m)$ acts on $R_{\cl{G}, \mu}$\footnote{Here, we use the full faithfulness of the $v$-sheafification functor on a suitable class of formal schemes proved in \cite[Proposition 18.4.1]{SW20}. This is implicitly used in many parts of this paper when one passes from $v$-sheaves to formal schemes.}. It also acts on $O_{\bb{C}_p} \langle \tfrac{u_\alpha}{\pi^{r_\alpha}} \mid \alpha \in \Phi_{\mu < 0} \rangle$ since $\cl{U}(0)$ is stable under $\Delta(m)$. Then, by the same argument as in \Cref{prop:U(n)_stability}, $\Delta(m)\cdot x_\CM = x_\CM$ automatically implies $\Delta(m)\cdot \cl{U}(n) = \cl{U}(n)$. Similarly, since $z_\mu^{-1} \sigma_E(x_\CM) = x_\CM$ and $z_\mu^{-1} \circ \sigma_E$ induces Weil descent data on $\Spf(R_{\cl{G}, \mu})$ and the integral model of $\cl{U}(0) \otimes F[\pi^{1/e}]$, we get $\sigma_E(\cl{U}(n)) = z_\mu \cdot \cl{U}(n)$. 
\end{proof}

\section{Canonical trivializations on special open balls} \label{sec:canonical_trivialization}

Recall that we set $J_n^0 = K_n^0$ for each $n \geq 1$. In this section, we construct a unique section 
\[
    \can_n \colon \cl{U}(n) \to \cl{M}_{G, b, \mu, K_n^0}
\]
sending $x_\CM$ to its natural lift via $x_\CM \in \cl{M}_{\cl{M}, b, \mu} \subset \cl{M}_{G, b, \mu, K_n^0}$. Our construction of $\can_n$ is of a different nature from \Cref{thm:canonical_level_structures} because the associated stacky level structure will be 
\[
    \Sht^b_{\cl{G}, -\mu}(n) = \cl{U}(n)^\diamond / \und{J_{b, n}^0} \to \Sht_{K_n, -\mu}
\]
and it takes a different form compared to $\Sht_{\cl{G}, -\mu}^{[M, \mu]}(r) = \und{K_r} \backslash \bb{B}_{x, r}^\diamond$ introduced in \Cref{sec:canonical_level_structures}. Our main tool is the explicit description of the universal $\cl{G}$-BKF module over $R_{\cl{G}, \mu}$ (see \Cref{lem:expBKF}).

\subsection{Independence of $p$-power roots} \label{ssec:indepedence_p-power}

To apply \Cref{lem:expBKF} to the construction of $\can_n$, one needs to resolve the choice of $\pi^\flat$ and $u_\alpha^\flat$. In this section, we present a geometric lemma resolving this issue via the theory of geometric quotients of $v$-sheaves developed in \cite[Section 3.3]{Tak26_rel}. We begin with the review of the basic facts on geometric quotients. 

\begin{defi}\textup{(\cite[Definition 3.9]{Tak26_rel})}
    Let $\Gamma$ be a group and let $X$ be a $v$-sheaf with a $\Gamma$-action. A map $\pi\colon X\to Y$ of $v$-sheaves is a geometric quotient of $X$ by $\Gamma$ if $\pi$ is surjective, and for every geometric point $x\in X(C,C^+)$ with $C$ an algebraically closed perfectoid field with an open and bounded valuation subring $C^+\subset C$, we have $\pi^{-1}(\pi(x)) = \{ \gamma \cdot x \vert \gamma \in \Gamma \}$. 
\end{defi}
\begin{prop}\textup{(\cite[Proposition 3.11]{Tak26_rel})}
    \label{prop:geomquotiscat}
    Let $X$ be a small $v$-sheaf with an action of a group $\Gamma$ over a small $v$-sheaf $S$. Let $\pi\colon X\to Y$ be a geometric quotient of $X$ by $\Gamma$ over $S$. If $Y$ is quasiseparated over $S$, then we have the following. 
    \begin{enumerate}
        \item For every $\gamma\in \Gamma$, we have $\gamma\cdot \pi=\pi$. 
        \item For every small $v$-sheaf $Z$ quasiseparated over $S$, a map $f\colon X \to Z$ over $S$ factors through $Y$ if and only if $\gamma\cdot f=f$ for every $\gamma\in \Gamma$. 
    \end{enumerate}
\end{prop}

The main property specific to our application is the existence of the following fixed points. 

\begin{defi}
    Let $U$ be a small $v$-sheaf and let $\Gamma$ be a finite group acting on $U$. We say that $U$ has a $\Gamma$-fixed geometric point if there is a geometric point $o\colon \Spa(C,O_C) \to U$ such that $o = \gamma \cdot o$ for every $\gamma \in \Gamma$.  
\end{defi}

Then, the following lemma enables us to automatically descend some trivializations of finite \'{e}tale covers along covers $U_\infty \to U_0$ of `infinite' type. 

\begin{prop} \label{prop:descent_finite_etale_map}
    Let $\{U_i\}_{i\geq 0}$ be a family of qcqs small $v$-sheaves with qcqs transition maps $\pi_i \colon U_{i} \to U_{i-1}$. Suppose that the following conditions hold for every $i\geq 1$. 
    \begin{enumerate}
        \item There is a finite group $\Gamma_i$ such that $\pi_i$ is a geometric quotient by $\Gamma_i$. 
        \item $U_i$ is connected and has a $\Gamma_i$-fixed geometric point. 
    \end{enumerate}
    Let $H$ be a finite group and let $X_{0}$ and $Y_{0}$ be finite \'{e}tale $H$-torsors over $U_{0}$. Let $U_\infty = \lim_{i \geq 0} U_i$ and let $X_{\infty} = X_{0} \times_{U_{0}} U_{\infty}$ (resp. $Y_{\infty} = Y_{0} \times_{U_{0}} U_{\infty}$). Then,
    \[
        \Hom_{U_{0}}^H(X_{0}, Y_{0}) \to \Hom_{U_{\infty}}^H(X_{\infty}, Y_{\infty}) 
    \]
    is bijective. Here, $\Hom^H$ denotes the set of $H$-equivariant morphisms. 
\end{prop}


\begin{proof}
    Since each $\pi_i$ is surjective and the $v$-site $\Perf$ is replete (as in \cite[Lemma 2.6]{Heu21}), $U_\infty \to U_0$ is surjective, so the map in the statement is injective. We show the surjectivity, 

    Let $f_\infty\colon X_\infty \to Y_\infty$ be an $H$-equivariant map. By \cite[Lemma 12.17]{Sch17}, $f_\infty$ can be defined over $U_i$ for some $i\geq 0$. Let $f_i\colon X_i \to Y_i$ be the descent of $f_\infty$ to $U_i$. We will show that $f_i$ can be defined over $U_{i-1}$ if $i \geq 1$. 

    Since $X_i$ and $Y_i$ are defined over $U_{i-1}$, $\Gamma_i$ naturally acts on $X_i$ and $Y_i$. For each $\gamma \in \Gamma_i$, let $f_i^\gamma = \gamma_{Y_i}^{-1} \circ f_i \circ \gamma_{X_i}$. Since the actions of $H$ and $\Gamma_i$ commute and $f_i$ is $H$-equivariant, $f_i^\gamma$ is also $H$-equivariant. Thus, the equalizer $\Eq(f_i, f_i^\gamma)\subset X_i$ is $H$-stable. It is closed and open since $Y_i$ is finite \'{e}tale over $U_i$, so it descends to a closed and open subset of $U_i$. Let $o_i\colon \Spa(C,C^+) \to U_i$ be a $\Gamma_i$-fixed geometric point. Then, $f_i = f_i^\gamma$ over fibers of $o_i$, so $\Eq(f_i, f_i^\gamma)/H \subset U_i$ is a nonempty closed and open subsheaf. Since $U_i$ is connected, we have $\Eq(f_i, f_i^\gamma)/H = U_i$, so $f_i = f_i^\gamma$. In particular, $X_i \xrightarrow{f_i} Y_i \to Y_{i-1}$ is invariant under $\Gamma_i$. 
    
    Since $\pi_i$ is a geometric quotient by $\Gamma_i$, the composition factors through $X_{i-1}$ by \cite[Proposition 3.10]{Tak26_rel}. Let $f_{i-1}\colon X_{i-1} \to Y_{i-1}$ be the map obtained in this way. Since $Y_i = Y_{i-1} \times_{U_{i-1}} U_i$, $f_i$ is the base change of $f_{i-1}$. By repeating this procedure until $i=0$, it follows that $f_\infty$ is defined over $U_0$. 
\end{proof}

\begin{cor} \label{cor:trivatinf}
    In the setting of \Cref{prop:descent_finite_etale_map}, every map $U_\infty \to X_0$ over $U_0$ factors through $U_0$. 
\end{cor}
\begin{proof}
    It follows by applying \Cref{prop:descent_finite_etale_map} to the constant $H$-torsor $U_0\times H$ and $X_0$. 
\end{proof}

Now, we will explain covers $U_\infty \to U_0$ appearing In applications. Let $\wtd{R}_{\cl{G}, \mu} = O_{\bb{C}_p}\llbracket u_\alpha \vert \alpha \in \Phi_{\mu < 0} \rrbracket$ and consider its finite cover
$
    \wtd{R}_{\cl{G}, \mu, i} = O_{\bb{C}_p}\llbracket u_\alpha^{1 / p^i} \vert \alpha \in \Phi_{\mu < 0} \rrbracket
$ for each $i \geq 0$. 

\begin{lem}
    Let $\wtd{R}_{\cl{G}, \mu, \infty} = \colim_{i}^\wedge \wtd{R}_{\cl{G}, \mu, i}$ be the completed colimit along the maximal ideal of $R_{\cl{G}, \mu}$. Then $\wtd{R}_{\cl{G}, \mu, \infty}$ is a complete adic perfectoid ring in the sense of \textup{\cite[Definition 4.4]{Tak26_rel}}.
\end{lem}
\begin{proof}
    Each $\wtd{R}_{\cl{G}, \mu, i}$ is the completion of the $p$-completed polynomial algebra 
    \[
        O_{\bb{C}_p} \langle u_\alpha^{1 / p^i} \vert \alpha \in \Phi_{\mu < 0} \rangle
    \]
    along the maximal ideal of $R_{\cl{G}, \mu}$. Thus, $\wtd{R}_{\cl{G}, \mu, \infty}$ is the completion of the $p$-completed colimit
    \[
        \colim_{i \geq 0}^\wedge O_{\bb{C}_p} \langle u_\alpha^{1 / p^i} \vert \alpha \in \Phi_{\mu < 0} \rangle, 
    \]
    which is easily seen to be perfectoid, so the claim follows from \cite[Corollary 4.11]{Tak26_rel}. 
\end{proof}

There is a canonical choice of $u_\alpha^\flat = (u_\alpha^{1/p^i})_{i \geq 0} \in \wtd{R}_{\cl{G}, \mu, \infty}^\flat$, so \Cref{lem:expBKF} can be applied to $\wtd{R}_{\cl{G}, \mu, \infty}$. \Cref{prop:descent_finite_etale_map} is designed to descend a trivialization over $\wtd{R}_{\cl{G}, \mu, \infty}$ to $\wtd{R}_{\cl{G}, \mu}$. 

\begin{lem} \label{lem:ballsatis}
    Let $q_\alpha > 0$ be a positive rational number for each $\alpha \in \Phi_{\mu<0}$ and let
    \[
        \cl{U}_{i} = \{ \lvert u_\alpha \rvert \leq \lvert \pi \rvert^{q_\alpha} \neq 0 \mid \alpha \in \Phi_{\mu<0} \} \subset \Spf(\wtd{R}_{\cl{G},\mu,i})_\eta
    \]
    for $i \geq 0$ including $i = \infty$. Then $\{\cl{U}_i^{\diamond}\}$ satisfies conditions (1) and (2) of \Cref{prop:descent_finite_etale_map}. Moreover, we have $\cl{U}_\infty^\diamond \cong \varprojlim \cl{U}_i^\diamond$. 
\end{lem}
\begin{proof}
    If we set $v_\alpha = \pi^{-q_\alpha}u_\alpha$, then $\cl{U}_i = \Spf(O_{\bb{C}_p} \langle v_\alpha^{1/p^i} \vert \alpha \in \Phi_{\mu<0} \rangle)_\eta$. Thus, $\cl{U}_i$ is a closed unit ball, so $\cl{U}_i^\diamond$ is connected by \cite[Lemma 15.6]{Sch17}. Moreover, it is easy to see $\cl{U}_\infty^\diamond \cong \varprojlim \cl{U}_i^\diamond$.
    
    Let $\mu_p \subset O_{\bb{C}_p}$ be the set of $p$-th roots of unity and let $\Gamma = (\mu_{p})^{\Phi_{\mu<0}}$. For each $i \geq 1$, $\Gamma$ acts on $\cl{U}_i$ over $\cl{U}_{i-1}$ so that
    \[
        (\gamma_\alpha)_{\alpha \in \Phi_{\mu<0}} \cdot v_{\alpha}^{1/p^i} = \gamma_\alpha v_{\alpha}^{1/p^i}
    \]
    for $(\gamma_\alpha)_{\alpha \in \Phi_{\mu<0}} \in \Gamma$. The origin of $\cl{U}_i$, the $\bb{C}_p$-valued point determined by $u_\alpha = 0$, is a $\Gamma$-fixed point. It is enough to show that $\cl{U}_i^\diamond \to \cl{U}_{i-1}^\diamond$ is a geometric quotient by $\Gamma$. First, we have 
    \[
        O_{\bb{C}_p} \langle v_\alpha^{1/p^i} \vert \alpha \in \Phi_{\mu<0} \rangle^{\Gamma} = O_{\bb{C}_p} \langle v_\alpha^{1/p^{i-1}} \vert \alpha \in \Phi_{\mu<0} \rangle. 
    \]  
    Thus, by \cite[Proposition 3.12]{Tak26_rel}, the induced map
    \[
        \Spd(O_{\bb{C}_p} \langle v_\alpha^{1/p^i} \vert \alpha \in \Phi_{\mu<0} \rangle) \to \Spd(O_{\bb{C}_p} \langle v_\alpha^{1/p^{i-1}} \vert \alpha \in \Phi_{\mu<0} \rangle)
    \]
    is a geometric quotient by $\Gamma$. Then, we get the assertion by restricting to the generic fibers (see \cite[Lemma 3.9]{Tak26_rel}). 
\end{proof}

\begin{cor} \label{prop:mapatinf}
    In the setting of \Cref{lem:ballsatis}, let $H$ be a finite group and let $X_{0}$ and $Y_{0}$ be finite \'{e}tale $H$-torsors over $\cl{U}_{0}$. Let $X_{\infty} = X_{0} \times_{\cl{U}_{0}} \cl{U}_{\infty}$ (resp. $Y_{\infty} = Y_{0} \times_{\cl{U}_{0}} \cl{U}_{\infty}$). Then,
    \[
        \Hom_{\cl{U}_{0}}^H(X_{0}, Y_{0}) \to \Hom_{\cl{U}_{\infty}}^H(X_{\infty}, Y_{\infty}), \quad \Hom_{\cl{U}_0}(\cl{U}_0, X_0) \to \Hom_{\cl{U}_0^\diamond}(\cl{U}_{\infty}^\diamond, X_0^\diamond)
    \]
    are bijective.
\end{cor}
\begin{proof}
    It follows from \Cref{prop:descent_finite_etale_map} and \Cref{cor:trivatinf} since the $v$-sheafification is fully faithful on finite \'{e}tale sites by \cite[Lemma 15.6]{Sch17}.
\end{proof}

\subsection{Trivialization of the $m$-th level structure} 

In this section, we apply \Cref{prop:mapatinf} to construct $\can_n$. In fact, we will further construct a trivialization 
\[
    \can_{m, m} \colon \cl{U}(n)_{\bb{C}_p} \to \cl{M}_{G, b, \mu, G(F)_{x, m}}. 
\]
For this, we first need to specify the image $\can_{m, m}(x_\CM) \in \cl{M}_{G, b, \mu, G(F)_{x, m}}(\bb{C}_p)$. We begin with the choice of an infinite-level lift $x_{\CM, \infty} \in \cl{M}_{M, b, \mu, \infty}(\bb{C}_p)$ of the CM point $x_\CM$.

\begin{defi}
    Let $X$ be a $v$-stack over $\Spd(F)$ and let $\cl{P}$ be a local $\cl{G}$-shtuka over $X$. An $m$-th level structure on $\cl{P}$ is a commutative diagram
    \begin{center}
        \begin{tikzcd}
            X \ar[r, "\cl{P}_m"] \ar[rd, "\cl{P}"'] & \Sht_{G(F)_{x, m}} \ar[d] \\
            & \Sht_{\cl{G}}. 
        \end{tikzcd}
    \end{center}
\end{defi}

Let $\cl{G}_m$ be the $m$-th congruence subgroup scheme of $\cl{G}$. For a $\pi$-torsion free $O_F$-algebra $A$, 
\[
    \cl{G}_m(A) = \Ker(\cl{G}(A) \to \cl{G}(A / \pi^m)). 
\]
In particular, $\cl{G}_m(O_F) = G(F)_{x, m} = \Ker(\cl{G}(O_F) \twoheadrightarrow \cl{G}(O_F / \pi^m))$. Then, an $m$-th level structure on $\cl{P}$ consists of a local $\cl{G}_m$-shtuka $\cl{P}_m$ over $X$ and an identification $\cl{P}_m \times^{\cl{G}_m} \cl{G} \cong \cl{P}$. 

\begin{lem} \label{lem:Gpsubtors}
    Let $A$ be a $\pi$-torsion free $\pi$-complete $O_F$-algebra. For a $\cl{G}$-torsor $P$ over $A$, the set of $\cl{G}_m$-subtorsors of $P$ is bijective to $P(A / \pi^m)$. 
\end{lem}
\begin{proof}
    Suppose that $s \in P(A / \pi^m)$ is given. Since $A$ is $\pi$-complete and $P$ is smooth and affine, $s$ admits a lift $x\in P(A)$. Then, the $\cl{G}_m$-subtorsor generated by $x$ is independent of the choice of $x$ because any other choice $x'\in P(A)$ can be written as $x' = x g_m$ with
    \[
        g_m \in \Ker(\cl{G}(A) \to \cl{G}(A / \pi^m A)) = \cl{G}_m(A). 
    \]
    Thus, we get one direction. For the converse, let $P_m \subset P$ be a $\cl{G}_m$-subtorsor. By localizing $\pi$-completely \'{e}tale locally, we may assume that $P_m$ can be trivialized. Then, every section of $P_m$ maps to the same element in $P(A / \pi^m)$. By descent, we get the converse map. 
\end{proof}

\begin{lem} \label{lem:finetdesc}
    Let $S = \Spa(R,R^+)$ be an affinoid perfectoid space over $F$ and let $\cl{P}$ be a local $\cl{G}$-shtuka over $S$. Let $P$ be the $\cl{G}$-torsor over $W_{O_F}(R^\flat)$ with a Frobenius $\varphi_P$ obtained by 
    \[
        P = \cl{P}\vert_{\cl{Y}_{S, [0, \varepsilon]}} \otimes_{B_{S, [0, \varepsilon]}} W_{O_F}(R^\flat), \quad
        \varphi_P = \varphi_{\cl{P}}\vert_{\cl{Y}_{S, [0, \varepsilon]}} \otimes \id \colon \varphi^*P \cong P
    \]
    for sufficiently small $\varepsilon > 0$. The set of $m$-th level structures on $\cl{P}$ is bijective to
    \[
        \{ 
            p \in P(W_{O_F}(R^\flat)/\pi^m) , \quad \varphi_P(p) = p
        \}
    \]
    and the right action of $\cl{G}(O_F / \pi^m)$ on each $p$ is given by $p \cdot g = pg$ for $g \in \cl{G}(O_F / \pi^m)$. 
\end{lem}
\begin{proof}
    Since $\cl{G}_{m, F} = \cl{G}_F$, the set of $\cl{G}_m$-subtorsors of $\cl{P}$ is equivalent to the set of $\cl{G}_m$-subtorsors of $P$ by \eqref{eq:YSn} and the Beauville-Laszlo lemma. By \Cref{lem:Gpsubtors}, it is bijective to $P(W_{O_F}(R^\flat) / \pi^m)$. 

    An $m$-th level structure on $\cl{P}$ is equivalent to a $\cl{G}_m$-subtorsor $\cl{P}_m$ of $\cl{P}$ stable under $\varphi_{\cl{P}}$. The action of $\varphi_{\cl{P}}$ on the set of $\cl{G}_m$-subtorsors is given by $\varphi_P$ via the above bijection. Thus, we get the desired bijection. 

    For each $g \in \cl{G}(O_F / \pi^m)$ with a lift $\wtd{g} \in \cl{G}(O_F)$, $g$ sends a $\cl{G}_m$-subtorsor $\cl{P}_m \to \cl{P}$ to 
    \[
        \cl{P}_m\cdot \wtd{g} \to \cl{P}. 
    \]   
    Through the above bijection, this action is given as in the statement. 
\end{proof}

Now, we can use this explicit description to specify an infinite-level lift $x_{\CM, \infty}$. 

\begin{prop} \label{defi:tinfty}
    There exists an element $t_\infty \in Z^\circ_{\cl{M}}(W_{O_F}(\bb{C}_p^\flat))$ such that
    \[
        t_\infty \sigma(t_\infty)^{-1} = \mu([\pi^\flat] - \pi) ,\quad 
        t_\infty \equiv \lambda(\pi^{\flat}) \pmod{\pi}. 
    \]
\end{prop}
\begin{proof}
    Since $\bb{C}_p^\flat$ is algebraically closed, $t_\infty \bmod \pi^i$ can be constructed inductively on $i$ by applying Lang's theorem to truncations of positive loop groups of $Z^\circ_{\cl{M}}$. 
\end{proof}

The element $t_\infty$ defines an infinite-level structure on the $\cl{M}$-BKF module
\[
    (\cl{M} \otimes W_{O_F}(O_{\bb{C}_p}^\flat), \mu([\pi^\flat]-\pi)\sigma)
\]
by \Cref{lem:finetdesc}. By \Cref{lem:expBKF}, it is the $\cl{M}$-BKF module associated to $x_\CM$. 

\begin{defi} \label{defi:CM_points_in_LSV}
    Let $x_{\CM, \infty} \in \cl{M}_{M, b, \mu, \infty}(\bb{C}_p)$ be the lift of $x_\CM$ associated to the infinite-level structure given by $t_\infty$. The projection of $x_{\CM, \infty}$ to $\cl{M}_{M, b, \mu, M(F)_{x, m}}$ is denoted by $x_{\CM, m}$. 
\end{defi}

For each open subgroup $K \subset G(F)$, the projection of $x_{\CM, \infty}$ to $\cl{M}_{G, b, \mu, K}$ is denoted by $x_{\CM, K}$. When there is no confusion, it is also denoted by $x_{\CM, m}$ if $K \cap M(F) = M(F)_{x, m}$. 

\begin{lem}
    For each $m \geq 0$, $x_{\CM, m}$ is defined over a finite extension $\breve{F}_m$ of $\breve{F}$. 
\end{lem}
\begin{proof}
    It is immediate since the fiber of $\cl{M}_{M, b, \mu, M(F)_{x, m}} \to \cl{M}_{\cl{M}, b, \mu}$ at $x_{\CM}$ is finite \'{e}tale. 
\end{proof}

Now, we will construct a solution to the equation $\varphi_P(p) = p$ in $P(W_{O_F}(R^\flat)/\pi^m)$ using the explicit description in \Cref{lem:expBKF}. 

\begin{prop} \label{lem:levelstrCp}
    Let $S = \Spa(R,R^+)$ be an affinoid perfectoid space over $\Spf(\wtd{R}_{\cl{G},\mu,\infty})_\eta$. An $m$-th level structure of $\cl{P}^\univ$ over $S$ corresponds to $g_m \in \cl{G}(W_{O_F}(R^\flat)/\pi^m)$ such that 
    \begin{equation}
        g_m \sigma(g_m)^{-1} \equiv \prod_{\alpha \in \Phi_{\mu<0}} i_\alpha(\alpha(\sigma(t_\infty)^{-1})\cdot [u_\alpha^\flat]) \pmod{\pi^m}
        \label{eq:mthlevelstructure}
    \end{equation}
    The corresponding $\varphi$-invariant section of $\cl{G}(W_{O_F}(R^\flat)/\pi^m)$ via \Cref{lem:finetdesc} is $t_\infty g_m$. In particular, the right action of $h \in \cl{G}(O_F / \pi^m)$ sends $g_m$ to $g_mh$. 
\end{prop}
\begin{proof}
    By \Cref{lem:expBKF} and \Cref{lem:finetdesc}, an $m$-th level structure over $S$ corresponds to an element of $\cl{G}(W_{O_F}(R^\flat)/\pi^m)$ fixed by $\mu([\pi^\flat]-\pi) \prod_{\alpha \in \Phi_{\mu<0}} i_\alpha([u_\alpha^\flat])\sigma$. 

    By the definition of $t_\infty$, $\mu([\pi^\flat]-\pi) \prod_{\alpha \in \Phi_{\mu<0}} i_\alpha([u_\alpha^\flat])$ is equal to 
    \[
        t_\infty \cdot\prod_{\alpha \in \Phi_{\mu<0}} i_\alpha(\alpha(\sigma(t_\infty)^{-1})\cdot [u_\alpha^\flat]) \cdot \sigma(t_\infty)^{-1}. 
    \]
    Thus, an $m$-th level structure is given by $t_\infty g_m$ with $g_m$ as in the statement. 
\end{proof}

\begin{prop}\label{lem:valsalpha}
    For each $\alpha \in \Phi_{\mu<0}$, consider the Teichm\"{u}ller expansion
    \[
        \alpha(\sigma(t_\infty)^{-1}) = \sum_{n\geq 0} [s_{\alpha,n}]\pi^n \in W_{O_F}(\bb{C}_p^\flat).
    \] 
    Then $\nu^\flat(s_{\alpha,n}) = -n - r_\alpha$ for every $n \geq 0$. 
\end{prop}
\begin{proof}
    Take $N \geq 1$ so that $\sigma^N$ acts trivially on $X_*(T)$ and consider the Teichm\"{u}ller expansion 
    \[ 
        \prod_{1 \leq i \leq N} (1-[\pi^{\flat, -q^i}]\pi)^{\langle \alpha, \sigma^i \mu \rangle} = \sum_{n \geq 0} [c_n]\pi^n. 
    \] 
    Here, $\langle \alpha, \sigma^i \mu \rangle \in \{0, \pm 1\}$. In particular, we have
    \[
        (1-[\pi^{\flat, -q^i}]\pi)^{\langle \alpha, \sigma^i \mu \rangle} \in \bigl\{1,\; 1-[\pi^{\flat, -q^i}]\pi,\; \sum_{j \geq 0} [\pi^{\flat, -q^ij}]\pi^j \bigr\}. 
    \]
    Since $\langle \alpha, \sigma^N \mu \rangle = -1$, we get $\nu^\flat(c_n) = -nq^N$ via the standard analysis of Newton slopes. Since $\mu([\pi^\flat]-\pi) = t_\infty \sigma(t_\infty)^{-1}$ and $r_\alpha = \langle \alpha, q\sigma \lambda  \rangle = -\tfrac{1}{q^N-1}\sum_{1\leq i \leq N} q^i \langle \alpha, \sigma^i \mu \rangle$, we have
    \begin{equation} \label{eq:use_of_tinfty}
        \alpha(\sigma(t_\infty)) \sigma^N(\alpha(\sigma(t_\infty)))^{-1} = \prod_{1 \leq i \leq N} ([\pi^{\flat, q^i}]-\pi)^{\langle \alpha, \sigma^i \mu \rangle} = [\pi^{\flat, -(q^N-1)r_\alpha}] \sum_{n\geq 0} [c_n]\pi^n. 
    \end{equation} 
    Then, we have 
    \begin{equation*}
        \sum_{n\geq 0} [s_{\alpha,n}^{q^N}]\pi^n = \sigma^N(\alpha(\sigma(t_\infty)))^{-1} = [\pi^{\flat, -(q^N-1)r_\alpha}] \sum_{n\geq 0} [c_n]\pi^n \cdot \sum_{n\geq 0} [s_{\alpha,n}]\pi^n. \label{eq:salpha}
    \end{equation*}
    From this, we will show $\nu^\flat(s_{\alpha,n}) = -n - r_\alpha$ inductively on $n$. The base case $n=0$ follows directly. Suppose that $\nu^\flat(s_{\alpha,i}) = - i - r_\alpha$ for $0 \leq i < n$. For $i,j \geq 0$ with $i+j \leq n$, we have
    \begin{equation*}
        \nu^\flat(c_i s_{\alpha,j}) = \left\{ 
            \begin{alignedat}{3}
                & - q^N i - j - r_\alpha  & \; & (j < n) \\
                & \nu^\flat(s_{\alpha,n}) & \; & (j = n). 
            \end{alignedat}
        \right. 
    \end{equation*}
    The minimum value is taken when $(i,j) = (n,0)$ or $(0, n)$. If $\nu^\flat(s_{\alpha,n}) \leq - n q^N -r_\alpha$, we have 
    \[
        -(q^N-1)r_\alpha + \nu^\flat(s_{\alpha,n}) \leq \nu^\flat(s_{\alpha,n}^{q^N}) \leq (q^N - 1)(- n q^N -r_\alpha) + \nu^\flat(s_{\alpha,n}), 
    \]
    which is a contradiction. Thus, $\nu^\flat(s_{\alpha,n}) > - n q^N -r_\alpha$, so we have
    \[
        \nu^\flat(s_{\alpha,n}^{q^N}) = -(q^N-1)r_\alpha + (- n q^N -r_\alpha) = -q^N(n + r_\alpha), 
    \]
    so we get the claim for $n$. Thus, the claim follows by induction. 
\end{proof}

Now, we understand the right-hand side of \eqref{eq:mthlevelstructure}, so we may construct the canonical trivialization of the $m$-th level structure. 

\begin{lem} \label{lem:g<m}
    Let $S = (R,R^+)$ be an affinoid perfectoid space over a rational subdomain
    \[
        \cl{U} = \{ \lvert u_\alpha \rvert \leq \lvert \pi \rvert^{(m-1 + r_\alpha) + \varepsilon} \neq 0 \mid \alpha \in \Phi_{\mu<0} \} \subset \Spf(\wtd{R}_{\cl{G},\mu,\infty})_\eta
    \]
    for some rational number $\varepsilon > 0$. Let 
    \begin{equation}
        h_{<m} = \prod_{\alpha \in \Phi_{\mu<0}} i_\alpha\left( \sum_{0 \leq i < m} [s_{\alpha,i}u_\alpha^\flat] \pi^i \right) \in \breve{\cl{G}}^{\mu < 0}(W_{O_F}(R^{\flat +})).  \label{eq:h<n/2}
    \end{equation}
    Then, the infinite product 
    \begin{equation}
        g_{<m} = h_{<m} \sigma(h_{<m}) \sigma^2(h_{<m}) \cdots \label{eq:g<m}
    \end{equation}
    converges in $\cl{G}(W_{O_F}(R^{\flat+}))$ and is trivial modulo $[\pi^{\flat, \varepsilon}]$.
\end{lem}
\begin{proof}
    By \Cref{lem:valsalpha}, 
    $
        s_{\alpha,i}u_\alpha^\flat \in \pi^{\flat, \varepsilon} R^{\flat+}
    $
    for $0 \leq i < m$. Thus, $h_{<m}$ is trivial modulo $[\pi^{\flat, \varepsilon}]$, so $\sigma^i(h_{<m})$ is trivial modulo $[\pi^{\flat, \varepsilon q^i}]$. Thus, $g_{<m}$ is well-defined and trivial modulo $[\pi^{\flat, \varepsilon}]$. 
\end{proof}

\begin{prop}  \label{prop:canmthtriv}
    Let $q_\alpha > m - 1 + r_\alpha$ be a rational number for each $\alpha \in \Phi_{\mu < 0}$ and let 
    \[
        \cl{U} = \{ \lvert u_\alpha \rvert \leq \lvert \pi \rvert^{q_\alpha} \neq 0 \mid \alpha \in \Phi_{\mu<0} \} \subset \Spf(R_{\cl{G},\mu})_\eta \times_{\Spd(\breve{F})} \Spd(\breve{F}_m). 
    \]
    Then, there is a unique section
    $
        \can_{m, m}\colon \cl{U}\to \cl{M}_{G, b, \mu, G(F)_{x, m}}
    $
    sending $x_{\CM}$ to $x_{\CM, m}$. 
\end{prop}
\begin{proof}
    Let $\cl{U}_\infty = \{ \lvert u_\alpha \rvert \leq \lvert \pi \rvert^{q_\alpha} \neq 0 \mid\alpha \in \Phi_{\mu<0} \} \subset \Spf(\wtd{R}_{\cl{G},\mu,\infty})_\eta$. For each affinoid perfectoid space $S = \Spa(R,R^+)$ over $\cl{U}_\infty$, $g_{<m}$ as in \eqref{eq:g<m} defines an $m$-th level structure of $\cl{P}^\univ$ over $S$ by \Cref{lem:levelstrCp} and
    \[
        g_{<m} \sigma(g_{<m})^{-1} = h_{<m} \equiv \prod_{\alpha \in \Phi_{\mu<0}} i_\alpha(\alpha(\sigma(t_\infty)^{-1})\cdot [u_\alpha^\flat]) \pmod{\pi^m}. 
    \]
    This amounts to a map $\cl{U}_\infty^\diamond \to \cl{M}_{G, b, \mu, G(F)_{x, m}}^\diamond$. By applying \Cref{prop:mapatinf} to $X_0 = \cl{U}_{\bb{C}_p} \times_{\cl{M}_{\cl{G}, b, \mu}} \cl{M}_{G, b, \mu, G(F)_{x, m}}$, it factors through a trivialization 
    \[
        \can_{m, m}\colon \cl{U}_{\bb{C}_p} \to \cl{M}_{G, b, \mu, G(F)_{x, m}}. 
    \]
    When $u_{\alpha} = 0$ for every $\alpha \in \Phi_{\mu < 0}$, $g_{<m} = 1$ and the $\varphi$-invariant section associated via \Cref{lem:levelstrCp} is $t_\infty$. Thus, $\can_{m, m}(x_\CM) = x_{\CM, m}$. This property uniquely characterizes $\can_{m, m}$ since $\cl{U}_{\bb{C}_p}$ is connected and $\can_{m, m}$ induces a closed and open immersion $\cl{U}_{\bb{C}_p} \subset \cl{U}_{\bb{C}_p} \times_{\cl{M}_{\cl{G}, b, \mu}} \cl{M}_{G, b, \mu, G(F)_{x, m}}$. It remains to show that this inclusion is defined over $\breve{F}_m$. 
    
    By \cite[Lemma 11.22]{Sch17}, it can be defined over a finite extension of $\breve{F}_m$. Then, it descends to $\breve{F}_m$ by the Galois descent since $x_{\CM, m}$ is defined over $\breve{F}_m$ and $\can_{m, m}$ is uniquely characterized by the condition $\can_{m, m}(x_\CM) = x_{\CM, m}$.
\end{proof}

By taking a quotient under $M(F)_{x, 0}$, it descends to $\breve{F}$ and we get a desired map $\can_n$. 

\begin{cor} \label{cor:canonical_CM_structure}
    Let $q_\alpha > m - 1 + r_\alpha$ be a rational number for each $\alpha \in \Phi_{\mu < 0}$ and let 
    \[
        \cl{U} = \{ \lvert u_\alpha \rvert \leq \lvert \pi \rvert^{q_\alpha} \neq 0 \mid\alpha \in \Phi_{\mu<0} \} \subset \Spf(R_{\cl{G},\mu})_\eta. 
    \]
    Then, there is a unique section
    $
        \can_n\colon \cl{U}\to \cl{M}_{G, b, \mu, K_n^0}
    $
    sending $x_\CM$ to $x_{\CM, K_n^0}$. 
\end{cor}
\begin{proof}
    By \Cref{prop:canmthtriv}, $\can_n$ can be defined over $\breve{F}_m$. Since $x_{\CM, K_n^0}$ lies in the image of $\cl{M}_{\cl{M}, b, \mu} \subset \cl{M}_{G, b, \mu, K_n^0}$ and is defined over $\breve{F}$, $\can_n$ descends to $\breve{F}$ by the same argument as in the last part of \Cref{prop:canmthtriv}. 
\end{proof}

From the explicit radius of $\cl{U}(n)$, we see that \Cref{cor:canonical_CM_structure} can be applied to $\cl{U}(n)$.

\subsection{Stabilizer of special open balls}

In this section, we verify the remaining part of \Cref{prop:restatement_LSV} (1), i.e.\ for the action of $j \in \cl{G}_b(O_F) - J_{b, n}^0$ on $\cl{U}(n)$, by developing the following property of $\can_{m, m}$. 

\begin{lem} \label{lem:Adjoint_of_g<m}
    Keep the notation in \Cref{lem:g<m} and suppose $\cl{U} = \cl{U}(n)$. 
    \begin{enumerate}
        \item When $n = 2m - 1$, $\Ad(t_\infty)(g_{<m}) \bmod \pi^m \in \cl{G}(W_{O_F}(R^{\flat+})/\pi^{m})$ and it is trivial modulo $(\pi^{m-1}, [\pi^{\flat}])$. Moreover, 
        \[
            \Ad(t_\infty)(g_{<m}) \bmod{(\pi^m, [\varpi])} \in \breve{\cl{G}}^{\lambda < 0}(W_{O_F}(R^{\flat+})/(\pi^m, [\varpi]))
        \]
        for some pseudo-uniformizer $\varpi \in R^\flat$ via the Moy-Prasad isomorphism. 
        \item When $n = 2m$, $\Ad(t_\infty)(g_{<m}) \bmod \pi^m \in \cl{G}(W_{O_F}(R^{\flat+})/\pi^{m})$ and it is trivial modulo $[\varpi]$ for some pseudo-uniformizer $\varpi \in R^\flat$. 
    \end{enumerate}
\end{lem}

\begin{proof}
    First, suppose $n = 2m - 1$. Since $h_{<m} \equiv \Ad(\sigma(t_\infty)^{-1})(\prod_{\alpha \in \Phi_{\mu<0}} i_\alpha([u_\alpha^\flat])) \pmod{\pi^m}$, 
    \[
        \Ad(t_\infty)(\sigma^j(h_{<m})) \equiv \Ad(t_\infty \sigma^{j+1}(t_\infty)^{-1})\Bigl(\prod_{\alpha \in \Phi_{\mu<0}} i_{\sigma^j\alpha}([u_\alpha^{\flat, q^j}])\Bigr) \pmod{\pi^m}
    \]
    for every $j \geq 0$ by the choice in \Cref{lem:choiceia}. By \Cref{defi:tinfty}, 
    \[
        t_\infty \sigma^{j+1}(t_\infty)^{-1}  = \prod_{0 \leq t \leq j} ([\pi^{\flat, q^t}] - \pi)^{\sigma^t \mu}. 
    \]
    Thus, for each $\alpha \in \Phi_{\mu<0}$, $(\sigma^j\alpha)( t_\infty \sigma^{j+1}(t_\infty)^{-1}) = \prod_{0 \leq t \leq j} ([\pi^{\flat, q^t}] - \pi)^{\langle \sigma^{j - t} \alpha, \mu \rangle}$. 
    
    Let us write the right-hand side as $\sum_{d \geq 0} [c_{\alpha,j}^d] \pi^d$. Since $\langle \mu, \alpha \rangle = -1$, we have
    \[
        \nu^\flat(c_{\alpha,j}^d) =  - q^j(d + 1) + \sum_{0 \leq t < j} \langle \sigma^{j - t}  \alpha, \mu \rangle q^t \geq -q^j(d+1). 
    \]
    Here, the first equality can be proved as in the proof of \Cref{lem:valsalpha} and the second inequality follows from \Cref{lem:vareps}. 
    
    Since $u_\alpha^\flat \in \pi^{\flat, m} R^{\flat+}$, $c_{\alpha,j}^d u_{\alpha}^{\flat, q^j} \in R^{\flat +}$ for every $d < m$. Moreover, 
    \[
        c_{\alpha,j}^d u_{\alpha}^{\flat, q^j} \in \pi^{\flat} R^{\flat +}
    \]
    for $d < m-1$, and 
    \[
        c_{\alpha,j}^d u_{\alpha}^{\flat, q^j} \in \varpi \cdot R^{\flat +}
    \]
    for some pseudo-uniformizer $\varpi \in R^\flat$ if $\langle \sigma^{t} \alpha, \mu \rangle = 1$ for some $1 \leq t \leq j$. As a result,
    \[
        \Ad(t_\infty)(\sigma^j(h_{<m})) \bmod{\pi^m} \in \cl{G}(W_{O_F}(R^{\flat+})/\pi^{m}) ,
    \]
    \[
        \Ad(t_\infty)(\sigma^j(h_{<m})) \equiv 1 \pmod{(\pi^{m-1}, [\pi^{\flat}])}. 
    \]
    Moreover, we have 
    \[
        \Ad(t_\infty)(\sigma^j(h_{<m})) \bmod{(\pi^m, [\varpi])} \in \breve{\cl{G}}^{\lambda < 0}(W_{O_F}(R^{\flat+})/(\pi^m, [\varpi]))
    \]
    for some pseudo-uniformizer $\varpi \in R^\flat$ because $\langle \sigma^{t} \alpha, \mu \rangle \neq 1$ for every $1 \leq t \leq j$ implies $\langle \sigma^j \alpha, \lambda \rangle < 0$ since $\msf{N} \cap \sigma(\msf{\ov{N}}) = \msf{U}_{\mu > 0}$ (see \Cref{defi:parabolic_mu}). 
    
    Next, suppose $n = 2m$. Since $\cl{U}(2m) \subset \cl{U}(2m - 1)$, the first claim follows from the previous case. Since $u_\alpha^\flat \in \pi^{\flat, m + r_\alpha} R^{\flat+}$, $c_{\alpha,j}^d u_{\alpha}^{\flat, q^j} \in \pi^{\flat, r_\alpha} R^{\flat +}$ for $d < m$. Thus, if we set $\varepsilon = \min r_\alpha$, then $\Ad(t_\infty)(\sigma^j(h_{<m}))$ modulo $\pi^m$ is trivial modulo $[\pi^{\flat, \varepsilon}]$, so the same holds for $\Ad(t_\infty)(g_{<m})$. 
\end{proof}

As in \Cref{ssec:stabilizer_open_balls}, it is essentially enough to keep track of the image of $x_\CM$. 

\begin{lem} \label{lem:image_CM_outside}
    Let $j \in G_b(F)$ be an element such that $j \cdot x_\CM \in \cl{U}(n)$. Then, $j \in J_{b, n}^0$. 
\end{lem}
\begin{proof}
    Since $J_{b, 2m}^0 = J_{b, 2m + 1}^0$ and $\cl{U}(2m + 1) \subset \cl{U}(2m)$, we may assume $n = 2m$. The case $m = 0$ follows from \Cref{prop:verification_at_depth_zero}, so we may assume $m \geq 1$. For each $\alpha \in \Phi_{\mu < 0}$, take $u_\alpha^\flat \in O_{\bb{C}_p}^\flat$ so that the coordinate of $j \cdot x_\CM \in \cl{U}(n)(\bb{C}_p)$ is given by $u_\alpha = (u_\alpha^\flat)^\sharp$. By \Cref{lem:expBKF} and the inner action on $x_\CM$ given in \cite[Remark 2.11, Proposition 2.15]{Tak25z}, there is $\iota \in \cl{G}(W_{O_F}(O_{\bb{C}_p}^\flat))$ that is trivial modulo $[\varpi]$ for some pseudo-uniformizer $\varpi \in \bb{C}_p^\flat$ and satisfies 
    \[
        \iota j \mu([\pi^\flat] - \pi) \sigma(\iota j)^{-1} = \mu([\pi^\flat]-\pi) \prod_{\alpha \in \Phi_{\mu<0}} i_\alpha([u_\alpha^\flat]). 
    \]
    Since both $\iota j t_\infty \bmod \pi^m$ and $t_\infty g_{<m}$ define $m$-th level structures at $j \cdot x_\CM$, there is an element $g \in \cl{G}(O_F / \pi^m)$ such that 
    \begin{equation} \label{eq:jgrel}
        \iota j t_\infty = t_\infty g_{<m} g \Leftrightarrow \Ad(t_\infty)(g_{<m})^{-1} \iota j = \Ad(t_\infty)(g).  
    \end{equation}
    We would like to prove $g \in \cl{M}(O_F / \pi^m)$ and $j \in K_{b, 2m}^0$. We will inductively show $g \bmod \pi^i \in \cl{M}(O_F / \pi^i)$ for each $1 \leq i \leq m$. 
    
    First, we prove the base case $i = 1$. By \Cref{lem:Adjoint_of_g<m} (2), the left-hand side of \eqref{eq:jgrel} lies in $\cl{G}(W_{O_F}(O_{\bb{C}_p}^\flat) / \pi^m)$ and is congruent to $j$ modulo $[\varpi]$ for some $\varpi$. Then, its reduction modulo $(\pi, [\varpi])$ lies in $\msf{\ov{P}}$ by the description of $j \bmod \pi$ in \cite[Lemma 5.6]{Tak25z}, so $\Ad(t_\infty)(g)$ lies in the big cell $\cl{\breve{G}}^{\lambda \leq 0} \times \cl{\breve{G}}^{\lambda > 0}$. Since $t_\infty \in Z_{\cl{M}}^\circ$ and $t_\infty \equiv \lambda(\pi^\flat) \pmod{\pi}$, it follows that $g$ also lies in the big cell $\cl{\breve{G}}^{\lambda \leq 0} \times \cl{\breve{G}}^{\lambda > 0}$ and $g \bmod \pi \in \msf{P}$. Since $g \bmod \pi$ is stable under $\sigma$ and $\bigcap_{i \geq 0} \sigma^i(\msf{P}) = \msf{M}$, we get $g \bmod \pi \in \msf{M}$. 

    Next, suppose $g \bmod \pi^i \in \cl{M}(O_F / \pi^i)$ for some $1 \leq i < m$. If one replaces $j$ by $jm$ for $m \in \cl{M}(O_F)$, then $g$ would be replaced by $gm$ as $t_\infty \in Z_\cl{M}^\circ$. Thus, we may assume that $g$ is trivial modulo $\pi^i$. By \Cref{lem:Adjoint_of_g<m} (2) and \eqref{eq:jgrel}, $j$ is trivial modulo $\pi^i$ as an element of $\cl{G}(O_{\breve{F}})$. Since $b = \mu(-\pi)$, we have 
    \[
        j \in G_b(F) \cap G(F)_{x, i} = G_b(F)_{x_b, i}. 
    \]  
    In other words, $j \bmod \pi^{i+1} \in \Lie(\msf{\ov{P}})$ via the Moy-Prasad isomorphism. Then, \eqref{eq:jgrel} implies $g \bmod \pi^{i+1} \in \Lie(\msf{\ov{P}})$. Since $g \bmod \pi^{i+1}$ is stable under $\sigma$, we get $g \bmod \pi^{i+1} \in \Lie(\msf{M})$. 

    Thus, we get $g \in \cl{M}(O_F / \pi^m)$ by induction on $i$. As previously, by replacing $g$ and $j$ under $\cl{M}(O_F)$, we may assume $g = 1$. We also have $j \in G_b(F)_{x_b, m}$ by the previous argument. Since $G_b(F)_{x_b, m} \subset K_{b, 2m}^0$, we get $j \in J_{b, n}^0$. 
\end{proof}

\begin{prop} \label{prop:action_outside_J_bn}
    For every $j \in G_b(F) - J_{b, n}^0$, we have $\cl{U}(n) \cap j \cdot \cl{U}(n) = \emptyset$. 
\end{prop}
\begin{proof}
    Suppose $\cl{U}(n) \cap j \cdot \cl{U}(n) \neq \emptyset$. Since $\cl{U}(n) \subset \cl{U}(0)$, $j \in \cl{G}_b(O_F)$ by \Cref{prop:U(0)_previous_work} (1). Take a geometric point $x \in \cl{U}(n)(\bb{C}_p)$ with $j \cdot x \in \cl{U}(n)$ and let $x_\alpha \in O_{\bb[C]_p}$ be the $u_\alpha$-coordinate of $x$ for each $\alpha \in \Phi_{\mu < 0}$. 

    First, we treat the case $n = 2m - 1$. The homomorphism
    \[
        R_{\cl{G},\mu} \xrightarrow{j^*} R_{\cl{G},\mu}
    \]
    associated to the action of $j$ sends each $u_\alpha$ to $P_\alpha \in R_{\cl{G}, \mu}$ with $P_\alpha(x) \equiv 0 \pmod{\pi^m}$. Since $x \in \cl{U}(n)$, each $x_\alpha$ is trivial modulo $\pi^m$. Thus, $P_\alpha(0) \equiv P_\alpha(x) \equiv 0 \pmod{\pi^m}$. In particular, $j \cdot x_\CM \in \cl{U}(n)$, so we get $j \in J_{b, n}^0$ by \Cref{lem:image_CM_outside}. 

    Similarly, we treat the case $n = 2m$. The homomorphism 
    \[  
        O_{\bb{C}_p} \langle \tfrac{u_\alpha}{\pi^{r_\alpha}} \mid \alpha \in \Phi_{\mu < 0} \rangle \xrightarrow{j^*} O_{\bb{C}_p} \langle \tfrac{u_\alpha}{\pi^{r_\alpha}} \mid \alpha \in \Phi_{\mu < 0} \rangle
    \]  
    associated to the action of $j$ on $\cl{U}(0)_{\bb{C}_p}$ sends each $\tfrac{u_\alpha}{\pi^{r_\alpha}}$ to $P_\alpha \in O_{\bb{C}_p} \langle \tfrac{u_\alpha}{\pi^{r_\alpha}} \mid \alpha \in \Phi_{\mu < 0} \rangle$ with $P_\alpha(x) \equiv 0 \pmod{\pi^m}$. Since $x \in \cl{U}(n)$, each $\tfrac{x_\alpha}{\pi^{r_\alpha}}$ is trivial modulo $\pi^m$. Thus, $P_\alpha(0) \equiv P_\alpha(x) \equiv 0 \pmod{\pi^m}$. In particular, $j \cdot x_\CM \in \cl{U}(n)$, so we get $j \in J_{b, n}^0$ by \Cref{lem:image_CM_outside}.  
\end{proof}

Next, we study the inner action on $\Img(\can_{m ,m}) \subset \cl{M}_{G, b, \mu, G(F)_{x, m}}$ to deduce the required stability of $\Img(\can_n)$ in \Cref{prop:restatement_LSV} (3). By \Cref{prop:U(n)_stability}, $J_{b, n}^0$ acts on the inverse image of $\cl{U}(n)$ in $\cl{M}_{G, b, \mu, G(F)_{x, m}}$, which is the disjoint union 
\begin{equation} \label{eq:inverse_image_U(n)_first}
    \bigsqcup_{g \in G(F)_{x, 0} / G(F)_{x, m}} g \cdot \Img(\can_{m, m}).
\end{equation}
Since $\Img(\can_{m, m})$ is connected, it is essentially enough to keep track of the CM point to determine the stabilizer of $\Img(\can_{m, m})$. Now, we will use the following index for the Moy-Prasad filtration
\[
    m_b = \left\{ \begin{alignedat}{4}
        & m & \quad & (n = 2m) \\
        & (m - 1)+ & \quad & (n = 2m - 1). 
    \end{alignedat}
    \right.
\]

\begin{lem} \label{lem:intermediate_image_x_CM}
    For each $j \in G_b(F)_{x_b, m_b}$, we have $j \cdot x_{\CM, m} \in \Img(\can_{m, m})$. 
\end{lem}
\begin{proof}
    For each $\alpha \in \Phi_{\mu < 0}$, take $u_\alpha^\flat \in O_{\bb{C}_p}^\flat$ so that the coordinate of $j \cdot x_\CM \in \cl{U}(n)(\bb{C}_p)$ is given by $u_\alpha = (u_\alpha^\flat)^\sharp$. By \Cref{lem:expBKF} and the inner action on $x_\CM$ given in \cite[Remark 2.11, Proposition 2.15]{Tak25z}, there is $\iota \in \cl{G}(W_{O_F}(O_{\bb{C}_p}^\flat))$ that is trivial modulo $[\varpi]$ for some pseudo-uniformizer $\varpi \in \bb{C}_p^\flat$ and satisfies 
    \[
        \iota j \mu([\pi^\flat] - \pi) \sigma(\iota j)^{-1} = \mu([\pi^\flat]-\pi) \prod_{\alpha \in \Phi_{\mu<0}} i_\alpha([u_\alpha^\flat]). 
    \]
    As in the proof of \Cref{lem:image_CM_outside}, there is an element $g \in \cl{G}(O_F / \pi^m)$ such that 
    \begin{equation} \label{eq:jgrel_2}
        \iota j t_\infty = t_\infty g_{<m} g \Leftrightarrow \Ad(t_\infty)(g) = \Ad(t_\infty)(g_{<m})^{-1} \iota j.  
    \end{equation}

    We would like to show $g = 1$ to deduce $j \cdot x_\CM \in \Img(\can_{m, m})$. First, we treat the case $n = 2m$. Then $j \in G_b(F)_{x_b, m}$ implies $j \equiv 1 \pmod{\pi^m}$ in $\cl{G}(O_{\breve{F}})$ since $x$ and $x_b$ lie in the same facet $\mfr{f}$. By \Cref{lem:Adjoint_of_g<m} (2), \eqref{eq:jgrel_2} implies that $\Ad(t_\infty)(g)$ lies in $\cl{G}(W_{O_F}(O_{\bb{C}_p}^\flat) / \pi^m)$ and is trivial modulo $[\varpi]$ for some pseudo-uniformizer $\varpi \in \bb{C}_p^\flat$. Since $g \in \cl{G}(O_F / \pi^m)$, it follows that $g \in \cl{\breve{G}}^{\lambda \leq 0} \times \cl{\breve{G}}^{\lambda > 0}$ and its component on $\cl{\breve{G}}^{\lambda \leq 0}$ vanishes. Since $g$ is stable under the Frobenius and $\bigcap_{i \geq 0} \sigma^i(\cl{\breve{G}}^{\lambda > 0}) = \{ 1 \}$, we get $g = 1$. 

    Next, we treat the case $n = 2m - 1$. By $j \in G_b(F)_{x_b, (m-1)+}$ and \Cref{lem:Adjoint_of_g<m} (1), \eqref{eq:jgrel_2} implies that $\Ad(t_\infty)(g)$ lies in $\cl{G}(W_{O_F}(O_{\bb{C}_p}^\flat) / \pi^m)$, lies in $\cl{\breve{G}}^{\lambda < 0}$ modulo $[\varpi]$ and is trivial modulo $(\pi^{m-1}, [\varpi])$ for some pseudo-uniformizer $\varpi \in \bb{C}_p^\flat$. Then, it follows that $g \in \cl{\breve{G}}^{\lambda \leq 0} \times \cl{\breve{G}}^{\lambda > 0}$ and its component on $\cl{\breve{G}}^{\lambda \leq 0}$ vanishes. Then we get $g = 1$ as previously. 
\end{proof}

\begin{prop} \label{prop:stability_Img_can_n}
    The open image $\Img(\can_{m, m}) \subset \cl{M}_{G, b, \mu, G(F)_{x, m}}$ is stable under $G_b(F)_{x_b, m_b}$. Moreover, $\Img(\can_n) \subset \cl{M}_{G, b, \mu, K_n^0}$ is stable under $J_{b, n}^0$. 
\end{prop}
\begin{proof}
    Since $G_b(F)_{x_b, m_b}$ acts on \eqref{eq:inverse_image_U(n)_first}, it stabilizes $\Img(\can_{m, m})$ by \Cref{lem:intermediate_image_x_CM}. Then, $\Img(\can_n)$ is stable under $G_b(F)_{x_b, m_b}$. Since $x_\CM$ is fixed under the inner action of $M(F)_{x, 0}$, the uniqueness of $\can_n$ implies that $\Img(\can_n)$ is also stable under $M(F)_{x, 0}$. Since $J_{b, n}^0 = G_b(F)_{x_b, m_b} M(F)_{x, 0}$, we get the second claim. 
\end{proof}

\subsection{Refinement of canonical trivializations}

In this section, we refine $\can_{m, m}$ to a map
\[
    \can_{n, m} \colon \cl{U}(n)_{\breve{F}_n} \to \cl{M}_{G, b, \mu, G(F)_{x, n, m}}. 
\]
Here, $G(F)_{x, n, m}$ denotes a Moy-Prasad subgroup for $M \subset G$ used in Yu's construction (see \Cref{ssec:representation_theory_preparation} for our notation). This refinement is needed in the analysis of cohomology for \Cref{prop:restatement_LSV} (4).

Since we will frequently use the Moy-Prasad isomorphism, we first set up notation for Lie algebras. Let $\mfr{g} = \Lie(\cl{G})$ and $\mfr{m} = \Lie(\cl{M})$. The orthogonal complement of $\mfr{m} \subset \mfr{g}$ with respect to the toral action of $Z_{\cl{M}}^\circ$ is denoted by $\mfr{m}^\perp$ and let $\pi_{\mfr{m}} \colon \mfr{g} \to \mfr{m}$ denote the natural projection. Then, let $du_\alpha \in \breve{\mfr{m}}^\perp$ denote the image of $dt \in \Lie(\bb{A}^1)$ under $\Lie(i_\alpha) \colon \Lie(\bb{A}^1) \cong \Lie(\cl{U}_\alpha) \subset \mfr{\breve{g}}$. Here, the breve accent denotes the base change to $O_{\breve{F}}$. 

\begin{prop} \label{lem:levelVn}
    Let $\cl{U}_\infty(n) = \cl{U}(n)\times_{\Spf(R_{\cl{G}, \mu})_\eta} \Spf(\wtd{R}_{\cl{G},\mu,\infty})_\eta$ and let $S = \Spa(R,R^+)$ be an affinoid perfectoid space over $\cl{U}_\infty(n)$. Let
    \begin{equation} \label{eq:definition_h_mn}
        h_{[m, n)} = \sum_{\alpha \in \Phi_{\mu<0}} \sum_{m \leq i < n} [s_{\alpha,i}u_\alpha^\flat] \pi^{i-m} du_\alpha \in \mfr{g} \otimes W_{O_F}(R^\flat)/\pi^{n - m}. 
    \end{equation}
    Then, the set of lifts $S \to \cl{M}_{G, b, \mu, G(F)_{x, n, m}}$ of $\can_{m, m}$ is bijective to the set of elements $x \in \mfr{m} \otimes W_{O_F}(R^\flat) / \pi^{n-m}$ such that 
    \begin{equation}
        x - \sigma(x) = \pi_{\mfr{m}}(\Ad(\sigma(g_{<m})^{-1})(h_{[m, n)})) \label{eq:levelVn}
    \end{equation}
    where $g_{<m}$ is defined as in \eqref{eq:g<m}. 
\end{prop}
\begin{proof}
    By \Cref{lem:levelstrCp}, the set of lifts $S \to \cl{M}_{G, b, \mu, G(F)_{x, n}}$ of $\can_{m, m}$ is bijective to the set of elements $g_{\geq m}\in \cl{G}(W_{O_F}(R^\flat)/\pi^n)$ that are trivial modulo $\pi^{m}$ and satisfy
    \[
        g_{<m} g_{\geq m} = \prod_{\alpha \in \Phi_{\mu<0}} i_\alpha\left(\sum_{0 \leq i < n} [s_{\alpha,i}u_\alpha^\flat]\pi^i\right) \cdot \sigma(g_{<m}g_{\geq m}). 
    \]
    Since $\mu$ is minuscule and $\breve{\cl{G}}^{\mu < 0}$ is abelian, $g_{<m} \sigma(g_{<m})^{-1} = h_{<m}$ implies that the above equation is equivalent to
    \[
        g_{\geq m} = \Ad(\sigma(g_{<m})^{-1})\left(\prod_{\alpha \in \Phi_{\mu<0}} i_\alpha\left(\sum_{m \leq i < n} [s_{\alpha,i}u_\alpha^\flat]\pi^i\right) \right)\cdot \sigma(g_{\geq m}). 
    \]
    Since every term is trivial modulo $\pi^{m}$, $g_{\geq m}$ uniquely corresponds to an element $x \in \mfr{\breve{g}} \otimes W_{O_F}(R^\flat)/\pi^{n -m}$ such that 
    \[
        x - \sigma(x) = \Ad(\sigma(g_{<m})^{-1})(h_{[m, n)})
    \]
    via the Moy-Prasad isomorphism. Now, recall  
    \[
        \cl{M}_{G, b, \mu, G(F)_{x, n, m}} = \cl{M}_{G, b, \mu, G(F)_{x, n}} / (\mfr{m}^\perp/\pi^{n-m}). 
    \]
    When $S$ is strictly totally disconnected, any finite \'{e}tale cover of $S$ splits, so the map $x \mapsto \pi_{\mfr{m}}(x)$ realizes the quotient by $\underline{\mfr{m}^\perp/\pi^{n-m}}(S)$ and the set of $\pi_{\mfr{m}}(x)$ is bijective to the set of lifts $S \to \cl{M}_{G, b, \mu, G(F)_{x, n, m}}$ of $\can_{m, m}$. It is also bijective to the solutions of \eqref{eq:levelVn}, so the claim holds for strictly totally disconnected spaces. Then it extends to every $S$ by $v$-descent. 
\end{proof}

Now, we analyze the right-hand side of \eqref{eq:levelVn} for each summand of $h_{[m, n)}$. 

\begin{lem} \label{lem:compute_h_mn}
    For every $\alpha \in \Phi_{\mu < 0}$ and $m \leq i < n$, 
    \[
        \pi_{\mfr{m}}(\Ad(\sigma(g_{<m})^{-1})([s_{\alpha,i}u_\alpha^\flat]\pi^{i-m}du_\alpha)) \bmod \pi^{n-m} \in \mfr{m} \otimes W_{O_F}(R^{\flat +}) / \pi^{n-m}. 
    \]
    Moreover, it is trivial modulo $[\pi^{\flat, 1/N}]$ for sufficiently large $N$. 
\end{lem}
\begin{proof}
    The adjoint action on $[s_{\alpha,i}u_\alpha^\flat]\pi^{i-m}du_\alpha$ factors through 
    \[
        \cl{G}(W_{O_F}(R^{\flat +})/\pi^{m}) \twoheadrightarrow \cl{G}(W_{O_F}(R^{\flat +})/\pi^{n-i}). 
    \]
    By \Cref{defi:U(n)_radius} and \Cref{lem:valsalpha}, $h_{<m}$ is trivial modulo $(\pi^{n - i}, [\pi^{\flat, \nu_\alpha}])$ with
    \[
        \nu_\alpha = \left\{
            \begin{alignedat}{4}
                & i-(n-m) & \quad & (n = 2m - 1) \\
                & i-(n-m)+1 & \quad & (n = 2m).  
            \end{alignedat}
        \right.
    \]
    By \eqref{eq:g<m}, $\sigma(g_{<m})^{-1}$ is trivial modulo $(\pi^{n - i}, [\pi^{\flat, q \nu_\alpha}])$. Since $s_{\alpha,i}u_\alpha^\flat \in \pi^{\flat, \tau_\alpha} \cdot R^{\flat +}$ with 
    \[
        \tau_\alpha = \left\{
            \begin{alignedat}{4}
                & m - i - r_\alpha & \quad & (n = 2m - 1) \\
                & m - i & \quad & (n = 2m), 
            \end{alignedat}
        \right.
    \]
    it follows that 
    \[
        \Ad(\sigma(g_{<m})^{-1})([s_{\alpha,i}u_\alpha^\flat]\pi^{i-m}du_\alpha) - [s_{\alpha,i}u_\alpha^\flat]\pi^{i-m}du_\alpha
    \]
    lies in 
    \[
        [\pi^{\flat, q \nu_\alpha}] \cdot [\pi^{\flat, \tau_{\alpha}}] \cdot \mfr{\breve{g}}\otimes W_{O_F}(R^{\flat +})/\pi^{n - m}. 
    \]
    Since $\pi_{\mfr{m}}(du_\alpha) = 0$, the claim follows from $q\nu_\alpha + \tau_\alpha > 0$, which can be directly verified from the explicit descriptions. 
\end{proof}

\begin{thm} \label{thm:refined_trivialization}
    There is a unique section 
    $
        \can_{n, m} \colon \cl{U}(n)_{\breve{F}_n} \to \cl{M}_{G, b, \mu, G(F)_{x, n, m}}
    $
    sending $x_\CM$ to $x_{\CM, n}$. 
\end{thm}
\begin{proof}
    For each affinoid perfectoid space $S = \Spa(R,R^+)$ over $\cl{U}_\infty(n)$, the infinite sum
    \[
        x = \sum_{i \geq 0} \sigma^i(\pi_{\mfr{m}}(\Ad(\sigma(g_{<m})^{-1})(h_{[m, n)})))
    \]
    is convergent in $\mfr{m} \otimes W_{O_F}(R^{\flat +})/\pi^{n - m}$ by \Cref{lem:compute_h_mn}. This is a functorial solution to \eqref{eq:levelVn}, so this amounts to a map $\cl{U}_\infty(n)^\diamond \to \cl{M}_{G, b, \mu, G(F)_{x, n, m}}^\diamond$ over $\can_{m, m}$. By applying \Cref{prop:mapatinf} to $X_0 = \Img(\can_{m, m}) \times_{\cl{M}_{G, b, \mu, G(F)_{x, m}}} \cl{M}_{G, b, \mu, G(F)_{x, n, m}}$, it factors through a trivialization 
    \[
        \can_{n, m}\colon \cl{U}(n)_{\bb{C}_p} \to \cl{M}_{G, b, \mu, G(F)_{x, n, m}}. 
    \]
    When $u_{\alpha} = 0$ for every $\alpha \in \Phi_{\mu < 0}$, we have $h_{[m, n)} = 0$. Thus, $\can_{n, m}$ maps $x_\CM$ to $x_{\CM, n}$ and this property uniquely characterizes $\can_{n, m}$. Then, the claim follows from the same argument as in \Cref{prop:canmthtriv} since $x_{\CM, n}$ is defined over $\breve{F}_n$. 
\end{proof}

Next, we study the inner action on $\Img(\can_{n, m}) \subset \cl{M}_{G, b, \mu, G(F)_{x, n, m}}$. By \Cref{prop:stability_Img_can_n}, $G_b(F)_{x_b, m_b}$ acts on the inverse image of $\Img(\can_{m, m})$ in $\cl{M}_{G, b, \mu, G(F)_{x, n, m}}$, which is the disjoint union 
\begin{equation} \label{eq:inverse_image_U(n)}
    \bigsqcup_{g \in M(F)_{x, m} / M(F)_{x, n}} g \cdot \Img(\can_{n, m}).
\end{equation}
Here, $M(F)_{x, m}$ acts via the inclusion into $G(F)_{x, 0}$. Since $\Img(\can_{n, m})$ is connected, it is essentially enough to keep track of the CM point to determine the stabilizer of $\Img(\can_{n, m})$. For example, it is easy to see the stability under the diagonal action. 

\begin{prop} \label{prop:stability_under_diagonal_action}
    The open subspace $\Img(\can_{n, m}) \subset \cl{M}_{G, b, \mu, G(F)_{x, n, m}, \breve{F}_n}$ is stable under the diagonal action of $M(F)_x$. 
\end{prop}
\begin{proof}
    By \Cref{lem:hyperspecial_M(F)_x}, it is enough to treat the action of $M(F)_{x, 0}$ (see \Cref{lem:diagonal_central_action}). By \Cref{prop:stability_U(n)_diagonal}, $\cl{U}(n)$ is stable under the diagonal action of $M(F)_{x, 0}$, so it acts diagonally on \eqref{eq:inverse_image_U(n)}. Since $\cl{U}(n)$ is connected, it is enough to show $\Delta(h) \cdot x_{\CM, n} \in \Img(\can_{n, m})$ for $h \in M(F)_{x, 0}$. Recall that $x_\CM$ is associated to the $\cl{M}$-BKF module
    \[
        (\cl{M} \otimes W_{O_F}(O_{\bb{C}_p}^\flat), \mu([\pi^\flat]-\pi)\sigma)
    \]
    and $x_{\CM, n} \in \cl{M}_{M, b, \mu, M(F)_{x, n}}$ is the lift of $x_\CM$ associated to $t_\infty \bmod \pi^n$. Since 
    \[
        h \cdot \mu([\pi^\flat]-\pi)\sigma \cdot h^{-1} = \mu([\pi^\flat]-\pi)\sigma ,\quad 
        h \cdot t_\infty \cdot h^{-1} = t_\infty
    \]
    as $t_\infty \in Z_{\cl{M}}^\circ$, we get $\Delta(h) \cdot x_{\CM, n} = x_{\CM, n}$ (see \cite[Proposition 2.15]{Tak25z} for the inner action and \Cref{lem:levelstrCp} for the other action). Thus, we get the claim. 
\end{proof}

Recall that we set
\[
    m_b = \left\{ \begin{alignedat}{4}
        & m & \quad & (n = 2m) \\
        & (m - 1)+ & \quad & (n = 2m - 1). 
    \end{alignedat}
    \right.
\]
Then, we may set $G_b(F)_{x_b, n, m_b} = G_b(F)_{x_b, 2m_b, m_b}$ since $M(F)_{x, 2m_b} = M(F)_{x, n}$. Our proof consists of several computational lemmas on $\cl{G}$-BKF modules. 

\begin{prop} \label{prop:image_x_CM_n_mb}
    For each $j \in G_b(F)_{x_b, n, m_b}$, we have $j \cdot x_{\CM, n} \in \Img(\can_{n, m})$. 
\end{prop}
\begin{proof}
    We keep the notation in the proof of \Cref{lem:intermediate_image_x_CM}. Let $g_n \in \cl{G}(W_{O_F}(\bb{C}_p^\flat) / \pi^n)$ be the element representing the $n$-th level structure of $j \cdot x_\CM$ via \Cref{lem:levelstrCp}. As in the proof of \Cref{lem:intermediate_image_x_CM}, we have 
    \begin{equation} \label{eq:characterization_t_infty}
        \iota j t_\infty = t_\infty g_n \Leftrightarrow g_n = \Ad(t_\infty^{-1})(\iota j).
    \end{equation}
    By \Cref{lem:intermediate_image_x_CM}, $g_n \equiv g_{<m} \pmod{\pi^m}$. In particular, $g_n$ lies in the big cell $\cl{\breve{G}}^{\lambda < 0} \times \cl{\breve{M}} \times \cl{\breve{G}}^{\lambda > 0}$, so that we get a decomposition 
    \[
        g_n = \ov{u}_- \ov{u}_+ \cdot m_n \cdot u_- u_+
    \]
    where $u_{\pm} \in \cl{\breve{G}}^{\lambda > 0}$, $m_n \in \cl{\breve{M}}$ and $\ov{u}_{\pm} \in \cl{\breve{G}}^{\lambda < 0}$. Moreover, we may suppose that
    \begin{enumerate}
        \item $u_+ \equiv 1 \pmod{\pi^m}$ and $\ov{u}_+ \equiv 1 \pmod{\pi^m}$. We will identify them as 
        \[
            u_+ \in \mfr{\breve{g}}^{\lambda > 0} \otimes W_{O_F}(\bb{C}_p^\flat) / \pi^{n - m} ,\quad 
            \ov{u}_+ \in \mfr{\breve{g}}^{\lambda < 0} \otimes W_{O_F}(\bb{C}_p^\flat) / \pi^{n - m}
        \]
        via the Moy-Prasad isomorphism. 
        \item $u_-$, $m_n$ and $\ov{u}_-$ have coefficients in $W_{O_F}(O_{\bb{C}_p}^\flat) / \pi^n$, and they are trivial modulo $[\varpi]$ for some pseudo-uniformizer $\varpi \in \bb{C}_p^\flat$. 
    \end{enumerate}
    This choice is possible by using $g_n \equiv g_{<m} \pmod{\pi^m}$ and \eqref{eq:characterization_t_infty}, for $g_{<m}$ is trivial modulo $[\varpi]$ for some $\varpi$ by \Cref{lem:g<m} and $j \in G_b(F)_{x_b, n, m_b}$. 
    
    Let $h \in M(F)_{x, m} / M(F)_{x, n} \cong \mfr{m} / \pi^{n - m}$ be the element such that $g_n \bmod G(F)_{x, n, m} = h \cdot \can_{n, m}(j \cdot x_\CM)$ in $\cl{M}_{G, b, \mu, G(F)_{x, n, m}}$. By the construction of $\can_{n, m}$ in the proof of \Cref{thm:refined_trivialization}, 
    \[
        \pi_{\mfr{m}}(g_{<m}^{-1} g_n) \in \mfr{m} \otimes W_{O_F}(O_{\bb{C}_p}^\flat) / \pi^{n - m}
    \]
    and it is congruent to $h$ modulo $[\varpi]$ for some $\varpi \in \bb{C}_p^\flat$. Now, we have
    \begin{equation} \label{eq:restatement_u+}
        \pi_{\mfr{m}}(g_{<m}^{-1} g_n) \equiv \pi_{\mfr{m}}(\Ad(m_n u_-)^{-1}(\ov{u}_+) - \ov{u}_+) \pmod{[\varpi]}
    \end{equation}
    for some $\varpi \in \bb{C}_p^\flat$. Thus, it is enough to show that the right-hand side of \eqref{eq:restatement_u+} is trivial modulo $[\varpi]$ for some $\varpi$. 
    
\begin{lem} \label{lem:valmn<m}
    For each $0 \leq i < n - m$, $m_n u_-$ is trivial modulo $(\pi^{n - m - i}, [\pi^{\flat, q(i+1)}])$. 
\end{lem}
\begin{proof}
    First, $m_n u_- \bmod \pi^{n - m - i}$ is the projection of $g_{< m} \bmod \pi^{n-m-i}$ along 
    \[
        \pr_2 \colon \cl{\breve{G}}^{\lambda < 0} \times \cl{\breve{G}}^{\lambda \geq 0} \to \cl{\breve{G}}^{\lambda \geq 0}
    \]
    as $g_n \equiv g_{<m} \pmod{\pi^m}$. Recall from \eqref{eq:h<n/2} that $h_{<m} \in \cl{\breve{G}}^{\lambda < 0}$. Thus, $\pr_2(g_{<m}) = \pr_2(h_{<m}^{-1}g_{<m})$. By \eqref{eq:g<m}, we have 
    \[
        h_{<m}^{-1} g_{<m} = \sigma(h_{<m}) \sigma^2(h_{<m}) \cdots. 
    \]
    Now, $h_{<m}$ is trivial modulo  $(\pi^{n - m - i}, [\pi^{\flat, i+1}])$ because \Cref{lem:valsalpha} and $j \cdot x_{\CM} \in \cl{U}(n)$ imply that for $0 \leq t < n - m - i$, we have 
    \[
        \nu^\flat(s_{\alpha, t}u_\alpha^\flat) \geq m - t - r_\alpha > m - t - 1 \geq i + 1
    \]
    when $n = 2m - 1$ and
    \[
        \nu^\flat(s_{\alpha, t}u_\alpha^\flat) \geq m - t \geq i + 1
    \]
    when $n = 2m$. Thus, $h_{<m}^{-1} g_{<m}$ is trivial modulo $(\pi^{n - m - i}, [\pi^{\flat, q(i+1)}])$ and we get the claim. 
\end{proof}
    Let $\Phi_{\msf{\ov{N}}} \subset \Phi$ be the set of roots inside $\msf{\ov{N}}$ and write 
    \[
        \ov{u}_+ = \sum_{\alpha \in \Phi_{\msf{\ov{N}}}} \sum_{0 \leq i < n - m} [\ov{u}_{\alpha,i}] \pi^i du_\alpha \in \mfr{\breve{g}}^{\lambda < 0} \otimes W_{O_F}(\bb{C}_p^\flat)/\pi^{n-m}. 
    \]
    For the right-hand side of \eqref{eq:restatement_u+} to be trivial modulo $[\varpi]$ for some $\varpi$, it is enough to show $\nu^\flat(\ov{u}_{\alpha, i}) > -q(i+1)$ for $0 \leq i < n - m$ by \Cref{lem:valmn<m}. Now, consider the element
    \[
        \wtd{u} = \Ad(t_\infty)(\ov{u}_-\ov{u}_+) \in \cl{\breve{G}}^{\lambda < 0 }(W_{O_F}(O_{\bb{C}_p}^\flat) / \pi^n)
    \]
    (see \eqref{eq:characterization_t_infty}). The key identity to evaluate $\nu^\flat(\ov{u}_{\alpha, i})$ is 
    \begin{equation} \label{eq:expression_u+}
        \ov{u}_+ = \Ad(t_\infty)^{-1}(\Ad(t_\infty)(\ov{u}_-)^{-1} \wtd{u}). 
    \end{equation}
    We begin with the estimation of the valuation of each coefficient of $\Ad(t_\infty)(\ov{u}_-)^{-1}$. 

\begin{lem} \label{lem:valuation_on_calpha+-}
    For each $\alpha \in \Phi_{\ov{\msf{N}}}$, let $d_\alpha \geq 1$ be the minimum positive integer such that $\langle \alpha, \sigma^{-d_\alpha} \mu \rangle = 1$. Then, the coefficients of
    \[
        \alpha(t_\infty^{-1}) = \sum_{i \geq 0} [c_{\alpha, i}^{-}] \pi^i, \quad
        \alpha(t_\infty) = \sum_{i \geq 0} [c_{\alpha, i}^{+}] \pi^i
    \]
    satisfy $\nu^\flat(c_{\alpha, i}^{-}) = - \langle \lambda, \alpha \rangle - q^{-d_\alpha} i$ and $\nu^\flat(c_{\alpha, i}^{+}) \geq \langle \lambda, \alpha \rangle - q^{-d_\alpha} i$ for every $i \geq 0$. 
\end{lem}
\begin{proof}
    The second inequality follows from the first one since $\alpha(t_\infty) = \alpha(t_\infty^{-1})^{-1}$. We will prove the first one as in the proof of \Cref{lem:valsalpha}. 

    Take $N \geq 1$ so that $\sigma^N$ acts trivially on $X_*(T)$. Then, we have (cf.\ \eqref{eq:use_of_tinfty})
    \[
        \alpha(t_\infty^{-1}) \sigma^N(\alpha(t_\infty)) = \prod_{0 \leq i < N} ([\pi^{\flat, q^i}]-\pi)^{\langle \alpha, - \sigma^i \mu \rangle}. 
    \]
    By \cite[Lemma 3.2]{Tak25z}, $\langle \alpha, \sigma^i \mu \rangle = 0$ for $i > N - d_\alpha$ and $\langle \alpha, \sigma^i \mu \rangle = 1$ for $i = N - d_\alpha$. Thus, the claim follows by induction as in the proof of \Cref{lem:valsalpha}. 
\end{proof}

For later use, fix an ordering of $\Phi_{\msf{\ov{N}}}$ and take a product $\cl{\breve{G}}^{\lambda < 0} \cong \prod_{\alpha \in \Phi_{\ov{\msf{N}}}} \cl{U}_\alpha$ in this order. 

\begin{lem} \label{lem:valuation_on_n-<m}
    Let
    \[
        \ov{u}_- = \prod_{\alpha \in \Phi_{\ov{\msf{N}}}} i_{\alpha}\Bigl(\sum_{i \geq 0} [\ov{u}_{-, \alpha, i}]\pi^i \Bigr). 
    \]
    be the decomposition of $\ov{u}_-$. Then, $\nu^\flat(\ov{u}_{-, \alpha, i}) \geq \max(n - m - i, 0)$ for every $i \geq 0$. 
\end{lem}
\begin{proof}
    By the proof of \Cref{lem:valmn<m}, $h_{<m}$ and $g_{<m}$ are trivial modulo $(\pi^i, [\pi^{\flat, n - m - i + 1}])$ for every $0 \leq i \leq m$. Since $g_n \equiv g_{<m} \pmod{\pi^m}$, $\ov{u}_-$ is also trivial modulo $(\pi^i, [\pi^{\flat, n - m - i + 1}])$ for every $0 \leq i \leq m$. Thus, we get the claim. 
\end{proof} 

\begin{lem} \label{lem:coefmodm+i+1}
    For every $i \geq 0$, the $u_\alpha$-coordinate of
    \[
        \Ad(t_\infty)(\ov{u}_-) \bmod{\pi^{m + i + 1}} \in \breve{\cl{G}}^{\lambda < 0}(W_{O_F}(\bb{C}_p^\flat) / \pi^{m + i +1})
    \]
    lies in $[\pi^{\flat, - q^{-d_\alpha}(i + 1) - 1}] \cdot W_{O_F}(O_{\bb{C}_p}^\flat) / \pi^{m + i +1}$ for each $\alpha \in \Phi_{\msf{\ov{N}}}$. 
\end{lem}
\begin{proof}
    We have 
    \[
        \Ad(t_\infty)(\ov{u}_-) = \prod_{\alpha \in \Phi_{\ov{\msf{N}}}} i_{\alpha}\Bigl(\sum_{i_1 \geq 0} [c_{\alpha, i_1}^{+}] \pi^{i_1} \cdot \sum_{i_2 \geq 0} [\ov{u}_{-, \alpha, i_2}]\pi^{i_2} \Bigr). 
    \]
    For every $i_1, i_2 \geq 0$, \Cref{lem:valuation_on_calpha+-} and \Cref{lem:valuation_on_n-<m} imply
    \[
        \nu^\flat(c_{\alpha, i_1}^+ \ov{u}_{-, \alpha, i_2}) \geq \langle \lambda, \alpha \rangle - q^{-d_\alpha} i_1 + \max(n - m - i_2, 0) \geq - q^{-d_\alpha} (i_1 + i_2 - n + m) - 1. 
    \]
    Here, we use $\lvert  \langle \lambda, \alpha \rangle \rvert \leq 1$, which follows from the fact that $(q^N - 1)\lambda = - \sum_{0 \leq k < N} (q\sigma)^k\mu$ for some $N\geq 1$ and $\mu$ is minuscule. Then, the claim follows since $n \geq 2m - 1$. 
\end{proof}

\begin{cor} \label{cor:valuation_last_step}
    Let us write 
    \[
        \Ad(t_\infty)(\ov{u}_-)^{-1} \wtd{u} = \sum_{\alpha \in \Phi_{\msf{\ov{N}}}} \sum_{0 \leq i < n - m} [\wtd{u}_{\alpha,i}] \pi^i du_\alpha \in \mfr{\breve{g}}^{\lambda < 0} \otimes W_{O_F}(\bb{C}_p^\flat)/\pi^{n-m}
    \]
    via the Moy-Prasad isomorphism. Then, we have $\nu^\flat(\wtd{u}_{\alpha, i}) \geq - q^{-1}(i + 1) - 1$.
\end{cor}
\begin{proof}
    Consider the subset $S \subset \cl{\breve{G}}^{\lambda < 0}(W_{O_F}(\bb{C}_p^\flat) / \pi^n)$ consisting of elements $A$ admitting a decomposition $A = A_{<m} A_{\geq m}$ such that $A_{\geq m}$ is trivial modulo $\pi^m$ and we have 
    \[
        A_{<m} \in \cl{\breve{G}}^{\lambda < 0}(W_{O_F}(O_{\bb{C}_p}^\flat) / \pi^n)
    \]
    \[
        A_{\geq m} \bmod{\pi^{m + i +1}} \in [\pi^{\flat, - q^{-1}(i + 1) - 1}]\cdot \mfr{\breve{g}}^{\lambda < 0} \otimes W_{O_F}(O_{\bb{C}_p}^\flat)/\pi^{i + 1}
    \]
    for every $0 \leq i < n-m$. 
    Then, $S$ is closed under multiplication since for $A , B \in S$, 
    \[
        AB = A_{<m} B_{<m} \cdot \Ad(B_{<m}^{-1})(A_{\geq m}) B_{\geq m}
    \]
    is a desired decomposition. Moreover, $S$ is also closed under the inverse. 

    Since $\wtd{u} \in \cl{\breve{G}}^{\lambda < 0}(W_{O_F}(O_{\bb{C}_p}^\flat) / \pi^n)$, we have $\wtd{u} \in S$. Moreover, \eqref{eq:characterization_t_infty} implies 
    \[
        \Ad(t_\infty)(\ov{u}_-) \bmod \pi^m \in \breve{\cl{G}}^{\lambda < 0}(W_{O_F}(O_{\bb{C}_p}^\flat) / \pi^m). 
    \]
    Together with \Cref{lem:coefmodm+i+1}, we have $\Ad(t_\infty)(\ov{u}_-) \in S$. Thus, $\Ad(t_\infty)(\ov{u}_-)^{-1} \wtd{u} \in S$. 
\end{proof}

    Now, we return to the proof of $\nu^\flat(\ov{u}_{\alpha, i}) > -q(i+1)$ for $0 \leq i < n - m$. 
    By \eqref{eq:expression_u+}, we have    
    \[
        \ov{u}_+ = \sum_{\alpha \in \Phi_{\msf{\ov{N}}}} \Bigl(\sum_{0 \leq i_1 < n - m} [c_{\alpha, i_1}^{-}] \pi^{i_1} \cdot \sum_{0 \leq i_2 < n-m} [\wtd{u}_{\alpha, i_2}]\pi^{i_2}\Bigr) du_\alpha. 
    \]
    For every $0 \leq i_1, i_2 < n-m$, it follows from \Cref{lem:valuation_on_calpha+-} and \Cref{cor:valuation_last_step} that 
    \[
        \nu^\flat(c_{\alpha, i_1}^- \wtd{u}_{\alpha, i_2}) \geq - \langle \lambda, \alpha \rangle - q^{-d_\alpha} i_1 - q^{-1}(i_2 +1) - 1 \geq - q^{-1} (i_1 + i_2 + 1) - 1. 
    \]
    Thus, $\ov{u}_+$ is trivial modulo $(\pi^{i+1}, [\pi^{\flat, - q^{-1}(i + 1) -1}])$ for every $0 \leq i < n - m$. Then, the claim follows since
    $
        \nu^\flat(\ov{u}_{\alpha,i}) \geq - q^{-1}(i + 1) -1 > -q(i+1). 
    $
\end{proof}

Then, we get the desired stability of $\Img(\can_{n, m})$ from \eqref{eq:inverse_image_U(n)}.

\begin{cor} \label{cor:stability_Img_can_nm}
    The subspace $\Img(\can_{n, m}) \subset \cl{M}_{G, b, \mu, G(F)_{x, n, m}, \breve{F}_n}$ is stable under $G_b(F)_{x_b, n, m_b}$. 
\end{cor}
\begin{proof}
    Recall from \eqref{eq:inverse_image_U(n)} that $G_b(F)_{x_b, n, m_b}$ acts on the disjoint union 
    \[
        \bigsqcup_{g \in M(F)_{x, m} / M(F)_{x, n}} g \cdot \Img(\can_{n, m})
    \]
    since $\cl{U}(n)$ is stable under $G_b(F)_{x_b, n, m_b}$ by \Cref{prop:U(n)_stability}. Since $\cl{U}(n)$ is connected, it follows from \Cref{prop:image_x_CM_n_mb} that $j \cdot \Img(\can_{n, m}) = \Img(\can_{n, m})$ for every $j \in G_b(F)_{x_b, n, m_b}$. 
\end{proof}

\section{Cohomology of local systems on special open balls} \label{sec:cohomology_computation}

In this section, we introduce special affinoids $\cl{W}(n)_\phi$ using the canonical trivialization $\can_{n, m}$ (see \Cref{defi:special_affinoids_positive_depth}). We show that $\cl{W}(n)_\phi$ has good reduction in \Cref{ssec:reduction_odd} and \Cref{ssec:reduction_even} (separately according to the parity of $n$). Then, we compute the \'{e}tale cohomology of reductions in \Cref{ssec:cohomology_reduction_even} and \Cref{ssec:cohomology_reduction_odd} and complete the transfer computation for \Cref{prop:restatement_LSV} (4) via the nearby cycles functor. 


\subsection{Special affinoids at positive depth} \label{ssec:special_affinoids_at_positive_depth}

In this section, we set up our notation and introduce special affinoids $\cl{W}(n)_\phi$. 

Let $\rho \in \Irr^\sm(M(F)_x)$ be a $(G, G_b, M)$-supergeneric irreducible representation of depth $n$. By definition (see \Cref{defi:supergeneric_representation}), we can take a smooth character
\[
    \phi \colon M(F) \to \Qlax
\]
that is $(G, M)$-supergeneric and $(G_b, M)$-supergeneric of depth $n$. Note that the same $\phi$ can be used for both inclusions $M \subset G$ and $M\subset G_b$ since the supergenericity only depends on $\phi \vert_{M(F)_{x, n} / M(F)_{x, n+}}$ (see \cite[Definition 3.1]{Tak26_Yu}). We fix a nontrivial additive character $\psi \colon k \to \Qlax$ and an $M$-stable $k$-linear map
\[
    X \colon \msf{m}_{x, n} \cong M(F)_{x, n} / M(F)_{x, n+} \to k
\]
with $\phi \vert_{\msf{m}_{x, n}} = \psi \circ X$. Precisely, we assume that $X$ admits a lift $\wtd{X} \in \Lie^*(M)^M(F)$ so that $\wtd{X}\vert_{\Lie(M)_{x, n}}$ takes values in $O_F$ and $X = \wtd{X}\vert_{\Lie(M)_{x, n}} \bmod \pi$ (see \cite[Definition 4.1.1]{FS25}). 

\begin{lem} \label{lem:vanish_X}
    For each $\alpha \in \Phi$, we have $(X \circ \pi_{\mfr{m}})(du_\alpha) = 0$.  
\end{lem}
\begin{proof}
    As $\wtd{X} \in \Lie^*(M)^T$, $\wtd{X}$ vanishes on the weight space for each nontrivial weight.  
\end{proof}

\begin{defi}
    Let $K_\phi \subset G(F)_{x, n, n/2}$ be a subgroup containing $G(F)_{x, n+, n/2 +}$ such that 
    \[
        K_\phi / G(F)_{x, n+, n/2+} \cong \Ker(X) \subset \msf{m}_{x, n} \subset G(F)_{x, n, n/2} / G(F)_{X, n+, n/2+}. 
    \]
\end{defi}

Then, $R^{K_n}_M(\rho)$ is trivial on $K_\phi$. 
In particular, the computation of \eqref{eq:map_from_compact_to_usual} can be realized as an isotypic part of the cohomology of the constant sheaf on the following affinoid. 

\begin{defi} \label{defi:special_affinoids_positive_depth}
    Let $\cl{W}(n)_\phi \subset \cl{M}_{G, b, \mu, K_\phi, \breve{F}_n}$ be the inverse image of $\Img(\can_{n, m})$. 
\end{defi}

We regard $\cl{W}(n)_\phi$ as a \textit{special affinoid} of the local Shimura variety at positive depth, i.e.\ a positive-depth analogue of $\cl{W}(0)$ (see \Cref{prop:U(0)_previous_work}). As in the proof of \Cref{prop:verification_at_depth_zero}, we will apply the nearby cycles functor to compute the map
\[
    R\Gamma_c(\cl{W}(n)_{\phi, \bb{C}_p}, \Qla) \to R\Gamma(\cl{W}(n)_{\phi, \bb{C}_p}, \Qla). 
\]
For this, we will show that $\cl{W}(n)_{\phi, \bb{C}_p}$ has good reduction and explicitly describe its reduction $\msf{W}(n)_\phi$. As in the Lubin-Tate case (see \cite[Section 3.8, 3.9]{BW16}), the reduction $\msf{W}(n)_\phi$ takes different forms according to whether
\[
    n = 1, \quad n = 2m - 1 \ (m \geq 2) \ \text{or} \ n = 2m. 
\]


\subsection{Reductions of special affinoids for $n = 2m - 1$} \label{ssec:reduction_odd}

In this section, we treat the case $n = 2m - 1$. Later, we will treat $m = 1$ and $m \geq 2$ separately. In this case, $G(F)_{x, 2m, m} \subset K_{\phi}$, so we may apply \Cref{lem:levelVn} to study the fiber of $\cl{W}(n)_{\phi} \to \Img(\can_{n, m}) \cong \cl{U}(n)$. We use the same notation as there. 

\begin{prop} \label{lem:levelVn_odd}
    Let $S = \Spa(R,R^+)$ be an affinoid perfectoid space over $\cl{U}_\infty(n)$. The set of lifts $S \to \cl{M}_{G, b, \mu, K_\phi}$ of $\can_{n, m}$ is bijective to the set of elements
    \[
        x \in (\mfr{m} \otimes W_{O_F}(R^\flat) / \pi^m) / \Ker(X) \otimes R^\flat
    \]
    such that $x \bmod \pi^{m - 1} \in \mfr{m} \otimes W_{O_F}(R^{\flat+}) / \pi^{m - 1}$ and it is trivial modulo $[\varpi]$ for some pseudo-uniformizer $\varpi \in R^\flat$ and satisfies
    \begin{equation} \label{eq:levelVn_odd}
        x - \sigma(x) \equiv \pi_{\mfr{m}}(\Ad(\sigma(g_{<m})^{-1})(h_{[m, 2m)})) \pmod{\Ker(X) \otimes R^{\flat}}. 
    \end{equation}
\end{prop}
\begin{proof}
    It follows by the same argument as \Cref{lem:levelVn}: \eqref{eq:levelVn_odd} is the condition for $x$ to be a lift of $\can_{m, m}$, and the condition on $x \bmod \pi^{m-1}$ ensures that $x$ is a lift of $\can_{n, m}$. 
\end{proof}

Now, we will analyze the right-hand side of \eqref{eq:levelVn_odd} for each summand of $h_{[m, 2m)}$. 

\begin{lem} \label{lem:computepit_odd}
    For every $\alpha \in \Phi_{\mu < 0}$ and $m < i < 2m$, 
    \[
        \pi_{\mfr{m}}(\Ad(\sigma(g_{<m})^{-1})([s_{\alpha,i}u_\alpha^\flat]\pi^{i-m}du_\alpha)) \bmod \pi^m \in \mfr{m}\otimes W_{O_F}(R^{\flat +}) / \pi^m. 
    \]
    Moreover, it is trivial modulo $[\pi^{\flat, 1/N}]$ for sufficiently large $N$. 
\end{lem}
\begin{proof}
    It follows by the same argument as \Cref{lem:compute_h_mn}: in this case, $\sigma(g_{<m})^{-1}$ is trivial modulo $(\pi^{2m - i}, [\pi^{\flat, q(i-m)}])$ and since $s_{\alpha,i}u_\alpha^\flat \in \pi^{\flat, m - i - r_{\alpha}} R^{\flat +}$, 
    \[
        \Ad(\sigma(g_{<m})^{-1})([s_{\alpha,i}u_\alpha^\flat]\pi^{i-m}du_\alpha) - [s_{\alpha,i}u_\alpha^\flat]\pi^{i-m}du_\alpha 
    \]
    modulo $\pi^m$ lies in 
    \[
        [\pi^{\flat, q(i-m) + (m - i - r_{\alpha})}] \cdot \mfr{\breve{g}}\otimes W_{O_F}(R^{\flat +})/\pi^{m}. 
    \]
    The claim follows since 
    $
        q(i-m) + (m - i - r_{\alpha}) > 0. 
    $
\end{proof}


\begin{lem} \label{lem:redtoh_odd}
    Let 
    \[
        h_{m-1} = \prod_{\alpha \in \Phi_{\mu<0}} i_\alpha(\pi^{m-1}[s_{\alpha,m-1}u_\alpha^\flat]). 
    \]
    Then, the infinite product
    \[
        g_{m-1} = h_{m-1} \sigma(h_{m-1}) \sigma^2(h_{m-1}) \cdots
    \]
    converges in $\cl{G}(W_{O_F}(R^{\flat+}))$. For every $\alpha \in \Phi_{\mu<0}$, 
    \begin{equation} \label{eq:main_term_sigma_g_du_alpha}
        \pi_{\mfr{m}}(\Ad(\sigma(g_{<m})^{-1})([s_{\alpha,m}u_\alpha^\flat]du_\alpha)) \equiv \pi_{\mfr{m}}(\Ad(\sigma(g_{m-1})^{-1})([s_{\alpha,m}u_\alpha^\flat]du_\alpha))
    \end{equation}
    modulo $[\pi^{\flat}]\cdot \mfr{m} \otimes W_{O_F}(R^{\flat+})$. 
\end{lem}
\begin{proof}
    The first claim follows by the same argument as \Cref{lem:g<m}. Since $h_{m-1}\equiv h_{< m}$ modulo $[\pi^{\flat}]$, $g_{<m} \equiv g_{m-1}$ modulo $[\pi^{\flat}]$. Thus, $\sigma(g_{m-1}) \cdot \sigma(g_{<m})^{-1}$ is trivial modulo $[\pi^{\flat, q}]$. Since $s_{\alpha,m}u_\alpha^\flat \in \pi^{\flat, - r_{\alpha}} R^{\flat+}$, 
    \[
        \Ad(\sigma(g_{<m})^{-1})([s_{\alpha,m}u_\alpha^\flat]du_\alpha) - \Ad(\sigma(g_{m-1})^{-1})([s_{\alpha,m}u_\alpha^\flat]du_\alpha)
    \]
    lies in $[\pi^{\flat, q - r_{\alpha}}] \cdot \mfr{\breve{g}} \otimes W_{O_F}(R^{\flat+})/\pi^{m}$. The claim follows from $r_{\alpha} < 1$. 
\end{proof}

The right-hand side of \eqref{eq:levelVn_odd} is now approximated by $\pi_{\mfr{m}}(\Ad(\sigma(g_{m-1})^{-1})([s_{\alpha,m}u_\alpha^\flat]du_\alpha))$ in \eqref{eq:main_term_sigma_g_du_alpha}. For this computation, we need to treat the cases $m = 1$ and $m \geq 2$ separately. 

\subsubsection{The case $m \geq 2$}

First, we treat the simpler case $m \geq 2$. In this case, we may apply the Moy-Prasad isomorphism: since $h_{m-1}$ is trivial modulo $\pi^{m-1}$, 
\[
    \cl{G}_{m-1}(W_{O_F}(R^{\flat+}))/\cl{G}_{m}(W_{O_F}(R^{\flat+})) \cong \mfr{g} \otimes R^{\flat +}
\]
sends $\sigma(g_{m-1})^{-1}$ to $- \sum_{1\leq i < \infty} \sum_{\beta \in \Phi_{\mu<0}} \sigma^i([s_{\beta,m-1}u_\beta^\flat] du_\beta)$. 
Then, 
\[
    \Ad(\sigma(g_{m-1})^{-1})([s_{\alpha,m}u_\alpha^\flat]du_\alpha) - [s_{\alpha,m}u_\alpha^\flat]du_\alpha
\]
lies in $\mfr{g} \otimes \pi^{m-1} W_{O_F}(R^\flat)/\pi^{m} W_{O_F}(R^\flat) \cong \mfr{g} \otimes R^{\flat}$ and can be expressed as
\begin{equation}\label{eq:breven}
    - \left[\sum_{1\leq i < \infty} \sum_{\beta \in \Phi_{\mu<0}} [s_{\beta,m-1}u_\beta^\flat]^{q^i} du_{\sigma^i\beta}, [s_{\alpha,m}u_\alpha^\flat]du_\alpha \right]_{\mfr{g}} 
\end{equation}
using the Lie bracket $[-,-]_{\mfr{g}}$ thanks to the choice of \Cref{lem:choiceia}. 

\begin{lem} \label{lem:expcompeven}
    The projection of \eqref{eq:breven} along $X \circ \pi_{\mfr{m}}$ lies in $R^{\flat +}$ and is congruent to
    \[
        s_{\alpha,m}u_\alpha^\flat (s_{\beta_\alpha,m-1}u_{\beta_\alpha}^\flat)^{q^{n_\alpha}} \cdot X([du_\alpha, du_{-\alpha}]_{\mfr{g}}) 
    \]
    modulo $\pi^{\flat, 1/N}$ for sufficiently large $N$ for every $\alpha \in \Phi_{\mu<0}$. 
\end{lem}
\begin{proof}
    If $\sigma^i\beta \neq - \alpha$, $[du_{\sigma^i\beta}, du_\alpha]_{\mfr{g}}$ is a linear combination of $du_\gamma$ over linear combinations $\gamma \in \Phi$ of $\alpha$ and $\sigma^i\beta$. Thus, the terms for such $(i, \beta)$ vanish after the projection along $X \circ \pi_{\mfr{m}}$ by \Cref{lem:vanish_X}. Thus, it is enough to treat terms for $(i, \beta)$ with $\sigma^i\beta = - \alpha$. 

    Recall the notation in \Cref{lem:vareps}. We show that if $\beta=-\sigma^{-i}\alpha \in \Phi_{\mu < 0}$ and $i>n_\alpha$, then
    \[
        (s_{\beta,m-1}u_\beta^\flat)^{q^i} \cdot (s_{\alpha,m}u_\alpha^\flat) \in \pi^{\flat, 1/N} \cdot R^{\flat+}
    \]
    for sufficiently large $N$. Since $i > n_\alpha$, $\beta = \beta_\gamma$ for some $\gamma = \sigma^{-j}\alpha \in \Phi_{\mu<0}$ with $n_\alpha < j < i$. 
    By \Cref{lem:vareps}, we have
    \[
        q^{i}(1-r_\beta) = q^{j}r_\gamma \geq q^{n_\alpha+1} r_\gamma > 1. 
    \]
    Since $\pi^{\flat, -(1-r_{\beta})} s_{\beta,m-1}u_{\beta}^\flat, \pi^{\flat, r_{\alpha}} s_{\alpha,m}u_\alpha^\flat \in R^{\flat +}$ and $r_{\alpha} < 1$, the claim follows from 
    \[
        (s_{\beta,m-1}u_\beta^\flat)^{q^i} \cdot (s_{\alpha,m}u_\alpha^\flat) \in \pi^{\flat, q^i(1-r_\beta) - r_{\alpha}} \cdot R^{\flat+}. 
    \]
    Thus, the nontrivial contribution comes only from $(i, \beta) = (n_\alpha, \beta_{\alpha})$. By \Cref{lem:vareps}, 
    \[
        (s_{\beta_\alpha,m-1}u_{\beta_\alpha}^\flat)^{q^{n_\alpha}} \cdot (s_{\alpha,m}u_\alpha^\flat) \in \pi^{\flat, q^{n_\alpha}(1-r_{\beta_\alpha}) - r_{\alpha}} \cdot R^{\flat+} = R^{\flat +}, 
    \]
    so the claim follows. 
\end{proof}

\begin{prop} \label{prop:levelVn_odd_modulo}
    For $m \geq 2$, the right-hand side of \eqref{eq:levelVn_odd} lies in 
    \[
        (\mfr{m} \otimes W_{O_F}(R^{\flat+}) / \pi^{m}) / \Ker(X) \otimes R^{\flat +}
    \]
    and is congruent to 
    \[
        \sum_{\alpha \in \Phi_{\mu<0}} s_{\alpha,m}u_\alpha^\flat (s_{\beta_\alpha,m-1}u_{\beta_\alpha}^\flat)^{q^{n_\alpha}} \cdot X([du_\alpha, du_{-\alpha}]_{\mfr{g}}) \in R^{\flat +}
    \]
    modulo $[\pi^{\flat,1/N}]$ for sufficiently large $N$. 
\end{prop}
\begin{proof}
    It follows from the combination of \Cref{lem:computepit_odd}, \Cref{lem:redtoh_odd} and \Cref{lem:expcompeven}. 
\end{proof}

For the computation of the reduction of $\cl{W}(n)_\phi$, we introduce the relevant polynomial $f_n$. 

\begin{defi} \label{defi:formal_model_U(n)}
    Let $u_{\alpha,n} = \pi^{-m} u_\alpha$ and let 
    \[
        \mfr{U}(n) = \Spf(O_{\bb{C}_p}\langle u_{\alpha, n} \vert \alpha \in \Phi_{\mu < 0} \rangle) \to \Spf(R_{\cl{G}, \mu}) 
    \]
    be the smooth formal model of $\cl{U}(n)_{\bb{C}_p}$. Let $\msf{U}(n) = \mfr{U}(n)_\red = \Spec(\ov{k}[u_{\alpha, n} \vert \alpha \in \Phi_{\mu < 0}])$. 
\end{defi}

\begin{defi} \label{defi:constant_c_alpha_odd}
    For $\alpha \in \Phi_{\mu<0}$, let 
    \[
        c_{\alpha,n} = - \pi^{\flat, (q^{n_\alpha} + 1)m} s_{\alpha,m}s_{\beta_\alpha,m-1}^{q^{n_\alpha}} \in O_{\bb{C}_p^\flat}^\times. 
    \]
    Here, $c_{\alpha, n} \in O_{\bb{C}_p^\flat}^\times$ follows from \Cref{lem:valsalpha}. 
    Let
    \[
        f_{n} = \sum_{\alpha \in \Phi_{\mu<0}} c_{\alpha,n} X([du_\alpha, du_{-\alpha}]_{\mfr{g}}) \cdot u_{\alpha,n} u_{\beta_\alpha,n}^{q^{n_\alpha}} \colon \mfr{U}(n) \otimes (O_{\bb{C}_p}/\pi) \to \bb{A}^1. 
    \]
    Here, we use the identification $O_{\bb{C}_p}/\pi \cong O_{\bb{C}_p^\flat}/\pi^\flat$ to have $c_{\alpha, n} \in O_{\bb{C}_p} / \pi$. 
\end{defi}

\subsubsection{The case $m = 1$}

Next, we treat the case $m = 1$. In this case, we directly work with 
\[
    h_0 = \prod_{\alpha \in \Phi_{\mu < 0}} i_\alpha(s_{\alpha, 0} u_\alpha^\flat) , \quad g_0 = h_0 \sigma(h_0) \cdots \in \msf{G}(R^{\flat +})
\]
and compute $\pi_{\msf{m}}(\Ad(\sigma(g_0)^{-1})(s_{\alpha, 1} u_\alpha^\flat du_\alpha))$ for each $\alpha \in \Phi_{\mu < 0}$. To simplify our notation, the condition $s \in \pi^{\flat, r} R^{\flat +}$ is denoted by $\nu^\flat(s) \geq r$ for $s \in R^\flat$ and $r \in \bb{Q}$. For each $\alpha \in \Phi_{\mu < 0}$, let $\gamma_\alpha \in \Phi_{\mu < 0}$ be the element with $\beta_{\gamma_\alpha} = \alpha$. 

\begin{lem} \label{lem:val_alpha_s_alpha}
    For each $i \geq 1$ and $\alpha \in \Phi_{\mu < 0}$, we have $\nu^\flat(\sigma^i(s_{\alpha, 0} u_\alpha^\flat)) \geq \langle \sigma^i \alpha, -q \sigma \lambda \rangle$. Moreover, if $i > n_{\gamma_\alpha}$, we have
    \[
        \nu^\flat(\sigma^i(s_{\alpha, 0} u_\alpha^\flat)) \geq \langle \sigma^i \alpha, -q \sigma \lambda \rangle + 1. 
    \]
\end{lem}
\begin{proof}
    By \Cref{defi:U(n)_radius}, it follows from $q\sigma \lambda + \mu = \lambda$ that 
    \[
        \nu^\flat(\sigma^i(s_{\alpha, 0} u_\alpha^\flat)) \geq q^i (1 - \langle \alpha, q \sigma \lambda \rangle) = - q^i \langle \alpha, \lambda \rangle = - \langle \sigma^i \alpha, (q\sigma)^i \lambda \rangle. 
    \]
    First, suppose $i \leq n_{\gamma_\alpha}$. Then, $\langle \sigma^j \alpha, \mu \rangle = 0$ for $1 \leq j < i$, so we inductively get $\langle \sigma^i \alpha, (q\sigma)^i \lambda \rangle = \langle \sigma^i \alpha, q\sigma \lambda \rangle$ using $q\sigma \lambda = \lambda - \mu$. Then, for $i > n_{\gamma_\alpha}$, the claim follows from \Cref{lem:vareps} since 
    \[
        q^i (1 - r_\alpha) \geq q r_{\gamma_\alpha} > 1. 
    \]
    \vspace{-\topsep}
\end{proof}

Let $\msf{G}(R^{\flat +})_1 = \Ker(\msf{G}(R^{\flat +}) \to \msf{G}(R^{\flat +} / \pi^\flat))$ and let
\[
    \msf{G}(R^{\flat})^{-q \sigma \lambda}_1 = \Ad(\pi^{\flat, -q \sigma \lambda})(\msf{G}(R^{\flat +})_1) \subset \Ad(\pi^{\flat, -q \sigma \lambda})(\msf{G}(R^{\flat +})) = \msf{G}(R^{\flat})^{-q \sigma \lambda}_0. 
\]
Then, \Cref{lem:val_alpha_s_alpha} implies $\sigma(g_0) \in \msf{G}(R^{\flat})_0^{-q \sigma \lambda}$. Moreover, if $i > n_{\gamma_\alpha}$, we have 
\[
    \sigma^i(i_\alpha(s_{\alpha, 0} u_\alpha^\flat)) \in \msf{G}(R^{\flat})^{-q \sigma \lambda}_1. 
\]
This enables us to approximate $\pi_{\msf{m}}(\Ad(\sigma(g_0)^{-1})(s_{\alpha, 1} u_\alpha^\flat du_\alpha))$. 

\begin{lem} \label{lem:reduced_term_m=1}
    For each $i \geq 1$, we set
    \[
        h_0^{(i)} = \prod_{\substack{\alpha \in \Phi_{\mu < 0} \\ i \leq n_{\gamma_\alpha}}} \sigma^i(i_\alpha(\pi^{\flat,- 1} u_\alpha^\flat)) \in \sigma(\msf{\ov{N}})(R^{\flat+}), \quad g^+_0 = \prod_{i \geq 1} h_0^{(i)}, \quad g_0^- = (g^+_0)^{-1}. 
    \]
    Then $\pi_{\msf{m}}(\Ad(\sigma(g_0)^{-1})(s_{\alpha, 1} u_\alpha^\flat du_\alpha)) \in \msf{m} \otimes R^{\flat +}$ and it is congruent to
    \[
        \pi_{\msf{m}}(\Ad(g^-_0)(\pi^{\flat, r_\alpha} s_{\alpha, 1} u_\alpha^\flat du_\alpha))
    \]
    modulo $\pi^\flat$. Here, $\Ad(g^-_0)(\pi^{\flat, r_\alpha} s_{\alpha, 1} u_\alpha^\flat du_\alpha) \in \msf{g} \otimes R^{\flat +}$. 
\end{lem}
\begin{proof}
    Since $q\sigma \lambda$ is central in $\msf{M}$, we have 
    \[
        \pi_{\msf{m}}(\Ad(\sigma(g_0)^{-1})(s_{\alpha, 1} u_\alpha^\flat du_\alpha)) = \pi_{\msf{m}}(\Ad(\Ad(\pi^{\flat, q \sigma \lambda})(\sigma(g_0)^{-1}))(\pi^{\flat, r_\alpha} s_{\alpha, 1} u_\alpha^\flat du_\alpha)). 
    \]
    By the above discussions, $\Ad(\pi^{\flat, q \sigma \lambda})(\sigma(g_0)^{-1}) \in \msf{G}(R^{\flat +})$ and 
    \[
        \Ad(\pi^{\flat, q \sigma \lambda})(\sigma(g_0)^{-1}) \equiv g_0^- \bmod{\msf{G}(R^{\flat +})_1}
    \]
    since $s_{\alpha, 0} = \pi^{\flat, -r_\alpha}$. Since $\nu^\flat(s_{\alpha, 1} u_\alpha^\flat) \geq - r_\alpha$, we get the claim. 
\end{proof}

\begin{prop} \label{prop:levelVn_m=1_modulo}
    For $m = 1$, the right-hand side of \eqref{eq:levelVn_odd} is in 
    $
        (\msf{m} / \Ker(X)) \otimes R^{\flat +} \cong R^{\flat +}
    $
    and congruent modulo $\pi^{\flat}$ to 
    \[
        (X \circ \pi_{\msf{m}})\biggl(\sum_{\alpha \in \Phi_{\mu < 0}} \Ad(g^-_0)(\pi^{\flat, r_\alpha} s_{\alpha, 1} u_\alpha^\flat du_\alpha)\biggr) \in R^{\flat +}. 
    \]
\end{prop}
\begin{proof}
    It follows directly from \Cref{lem:reduced_term_m=1} by summing up over $\Phi_{\mu < 0}$. 
\end{proof}

As in the case $m \geq 2$, we introduce a polynomial $f_1$ relevant to the computation of the reduction. We define the formal model $\mfr{U}(1)$ and the parameter $u_{\alpha, 1}$ as in \Cref{defi:formal_model_U(n)}. 

\begin{defi} \label{defi:constant_c_alpha_m=1}
    For $\alpha \in \Phi_{\mu<0}$, let $c_{\alpha, 1} = - \pi^{\flat, r_\alpha + 1} s_{\alpha, 1}$ be elements in $O_{\bb{C}_p}^\flat$ and let 
    \[
        h_0^{(i)} = \prod_{\substack{\alpha \in \Phi_{\mu < 0} \\ i \leq n_{\gamma_\alpha}}} \sigma^i(i_\alpha(u_{\alpha, 1})) \in \sigma(\msf{\ov{N}})(O_{\bb{C}_p} / \pi), \quad g^+_0 = \prod_{i \geq 1} h_0^{(i)}, \quad g_0^- = (g_0^+)^{-1}
    \]
    via $O_{\bb{C}_p}/\pi \cong O_{\bb{C}_p^\flat}/\pi^\flat$. Then, we define
    \[
        f_1 = (X \circ \pi_{\msf{m}})\biggl(\sum_{\alpha \in \Phi_{\mu < 0}} \Ad(g^-_0)(c_{\alpha, 1} u_{\alpha, 1} du_\alpha)\biggr) \colon \mfr{U}(1) \otimes (O_{\bb{C}_p}/\pi) \to \bb{A}^1. 
    \]
\end{defi}

\subsubsection{Good reductions}

Now, we will compute the reduction of $\cl{W}(n)_\phi$ using \Cref{prop:levelVn_odd_modulo} and \Cref{prop:levelVn_m=1_modulo}. Here, we treat $m = 1$ and $m \geq 2$ simultaneously. Following the depth-zero case \cite[Theorem 3.16]{Tak25z}, we begin with the construction of formal models. 

\begin{defi} \label{defi:reduction_of_specaff_odd}
    Let $\msf{W}(n)_\phi$ be a finite \'{e}tale cover of $\msf{U}(n)$ given by the Cartesian diagram
    \begin{center}
        \begin{tikzcd}
            \msf{W}(n)_\phi \ar[r] \ar[d] & \bb{A}^1 \ar[d, "t \mapsto t^q - t"] \\
            \msf{U}(n) \ar[r, "f_n"] & \bb{A}^1. 
        \end{tikzcd}
    \end{center}
    By the topological invariance of finite \'{e}tale sites, there is a unique $k$-torsor lift
    \[
        \mfr{W}(n)_\phi \to \mfr{U}(n)
    \]
    of the $k$-torsor $\msf{W}(n)_\phi \to \msf{U}(n)$. By the uniqueness, we have a Cartesian diagram
    \begin{center}
        \begin{tikzcd}
            \mfr{W}(n)_\phi \otimes (O_{\bb{C}_p}/\pi) \ar[r] \ar[d] & \msf{t} \ar[d, "t \mapsto t^q - t"] \\
            \mfr{U}(n) \otimes (O_{\bb{C}_p}/\pi) \ar[r, "f_n"] & \msf{t}. 
        \end{tikzcd}
    \end{center}
\end{defi}

\begin{prop} \label{prop:good_reduction_odd}
    There is a $k$-equivariant isomorphism $\mfr{W}(n)_{\phi, \eta} \cong \cl{W}(n)_{\phi, \bb{C}_p}$ over $\cl{U}(n)$. 
\end{prop}
\begin{proof}
    Let $\mfr{U}_\infty(n) = \Spf(O_{\bb{C}_p}\langle u_{\alpha,n}^{1/p^\infty} \vert \alpha \in \Phi_{\mu < 0} \rangle )$ be a formal model of $\cl{U}_\infty(n)$. Let 
    \[
        \mfr{W}_\infty(n)_\phi = \mfr{U}_\infty(n)\times_{\mfr{U}(n)} \mfr{W}(n)_\phi
    \]
    be a finite \'{e}tale cover of $\mfr{U}_\infty(n)$. By \Cref{prop:mapatinf}, it is enough to construct a $k$-equivariant map $\mfr{W}_\infty(n)_{\phi, \eta}^\diamond \to \cl{W}(n)^\diamond_\phi$. Let $S = \Spa(R,R^+)$ be an affinoid perfectoid space over $\mfr{W}_\infty(n)_{\phi, \eta}$. Since $\mfr{W}(n)_\phi$ is affine, $S$ admits a unique extension $\Spf(R^+) \to \mfr{W}(n)_\phi$. Its restriction to $\Spec(R^+/\pi)$ gives an element $t_0 \in R^+/\pi$ such that
    \[
        t_0 - t_0^q = - f_n(u_{\alpha, n}(S)) \in R^+ / \pi \cong R^{\flat +} / \pi^\flat. 
    \]
    Now, the right-hand side admits a lift
    \[
        \sum_{\alpha \in \Phi_{\mu<0}} s_{\alpha,m}s_{\beta_\alpha,m-1}^{q^{n_\alpha}} u_\alpha^\flat u_{\beta_\alpha}^{\flat, q^{n_\alpha}}X([du_\alpha, du_{-\alpha}]_{\mfr{g}}) ,\quad (X \circ \pi_{\msf{m}})\biggl(\sum_{\alpha \in \Phi_{\mu < 0}} \Ad(g^-_0)(\pi^{\flat, r_\alpha} s_{\alpha, 1} u_\alpha^\flat du_\alpha)\biggr)
    \]
    to $R^{\flat +}$ depending on $m \geq 2$ or $m = 1$. By \Cref{prop:levelVn_odd_modulo} and \Cref{prop:levelVn_m=1_modulo}, $t_0 - t_0^q$ is congruent to the right-hand side of \eqref{eq:levelVn_odd} modulo $[\pi^{\flat, 1/N}]$ for sufficiently large $N$. In particular, $t_0$ is a solution of \eqref{eq:levelVn_odd} modulo $[\pi^{\flat, 1/N}]$. Since $W_{O_F}(R^{\flat +})/\pi^{m}$ is $[\pi^{\flat}]$-adically complete, $t_0$ lifts uniquely to a solution $x$ of \eqref{eq:levelVn_odd} by the same argument as in \cite[Lemma 3.5]{Tak25z}. By construction, $x \bmod \pi^{m - 1} \in \mfr{m} \otimes W_{O_F}(R^{\flat+}) / \pi^{m - 1}$ and it is trivial modulo $[\varpi]$ for some pseudo-uniformizer $\varpi \in R^\flat$. Thus, we get a map
    \[
        x \colon S \to \cl{W}(n)_\phi. 
    \]
    Since $x$ is functorially defined over $\mfr{W}_\infty(n)_{\phi, \eta}$, this gives $\mfr{W}_\infty(n)_{\phi, \eta}^\diamond \to \cl{W}(n)^\diamond_\phi$. The map is $k$-equivariant since the action $t \mapsto t + a$ for $a \in k$ sends the solution $x$ to $x + a$.  
\end{proof}

\subsection{Reductions of special affinoids for $n = 2m$} \label{ssec:reduction_even}

In this section, we treat the case $n = 2m$. In this case, $G(F)_{x, m} / K_\phi$ is non-abelian and one needs to modify \Cref{lem:levelVn}. 

First, there are natural inclusions
\[
    \mfr{m}^\perp \otimes W_{O_F}(R^\flat) / \pi^m, \Ker(X) \subset \mfr{g} \otimes W_{O_F}(R^\flat) / \pi^m \subset \cl{G}(W_{O_F}(R^\flat) / \pi^{2m+1})
\]
via the Moy-Prasad isomorphism. Moreover, the subgroup
\[
    \mfr{g} \otimes W_{O_F}(R^\flat) / \pi^m \subset \cl{G}_m(W_{O_F}(R^\flat) / \pi^{m + 1})
\]
is central. Here, $\cl{G}_m$ is the $m$-th congruence subgroup scheme of $\cl{G}$ and the right-hand side consists of elements of $\cl{G}(W_{O_F}(R^\flat) / \pi^{2m+1})$ that are trivial modulo $\pi^m$. 

\begin{prop} \label{lem:levelVn_even}
    Let $S = \Spa(R,R^+)$ be an affinoid perfectoid space over $\cl{U}_\infty(n)$. Let
    \begin{equation} \label{eq:definition_h_m_2m}
        h_{[m, 2m]} =  \prod_{\alpha \in \Phi_{\mu<0}} i_\alpha\bigl(\sum_{m \leq i \leq 2m} [s_{\alpha,i}u_\alpha^\flat] \pi^i \bigr) \in \cl{G}_m(W_{O_F}(R^{\flat}) / \pi^{m + 1}). 
    \end{equation}
    The set of lifts $S \to \cl{M}_{G, b, \mu, K_\phi}$ of $\can_{n, m}$ is bijective to the set of elements
    \[
        x \in \cl{G}_m(W_{O_F}(R^\flat) / \pi^{m + 1}) / (\mfr{m}^\perp \otimes W_{O_F}(R^\flat) / \pi^m \oplus \Ker(X) \otimes R^{\flat})
    \]
    such that $\pi_{\mfr{m}}(x \bmod \pi^m) \in \mfr{m} \otimes W_{O_F}(R^{\flat+}) / \pi^{m}$ and it is trivial modulo $[\varpi]$ for some pseudo-uniformizer $\varpi \in R^\flat$ and satisfies
    \begin{equation} \label{eq:levelVn_even}
        x \sigma(x)^{-1} \equiv \Ad(\sigma(g_{<m})^{-1})(h_{[m, 2m]}) \pmod{\mfr{m}^\perp \otimes W_{O_F}(R^\flat) / \pi^m \oplus \Ker(X) \otimes R^{\flat}}. 
    \end{equation}
\end{prop}
\begin{proof}
    It follows by the same argument as \Cref{lem:levelVn}. By \Cref{lem:levelstrCp}, the set of lifts $S \to \cl{M}_{G, b, \mu, G(F)_{x, 2m + 1}}$ of $\can_{m, m}$ is bijective to the set of elements $g_{\geq m}\in \cl{G}_m(W_{O_F}(R^\flat)/\pi^{m+1})$ satisfying
    \[
        g_{<m} g_{\geq m} = \prod_{\alpha \in \Phi_{\mu<0}} i_\alpha\left(\sum_{0 \leq i \leq 2m} [s_{\alpha,i}u_\alpha^\flat]\pi^i\right) \cdot \sigma(g_{<m}g_{\geq m}). 
    \]
    Then $g_{<m} \sigma(g_{<m})^{-1} = h_{<m}$ implies that the above equation is equivalent to
    \[
        g_{\geq m} \sigma(g_{\geq m})^{-1} = \Ad(\sigma(g_{<m})^{-1})(h_{[m, 2m]}). 
    \]
    Now, recall  
    \[
        \cl{M}_{G, b, \mu, K_\phi} = \cl{M}_{G, b, \mu, G(F)_{x, 2m + 1}} / (\mfr{m}^\perp/\pi^m \oplus \Ker(X)). 
    \]
    When $S$ is strictly totally disconnected, the map
    \[
        g_{\geq m} \mapsto x = g_{\geq m} \bmod \mfr{m}^\perp \otimes W_{O_F}(R^\flat) / \pi^m \oplus \Ker(X) \otimes R^{\flat}
    \]
    realizes the quotient by $\underline{\mfr{m}^\perp/\pi^m \oplus \Ker(X)}(S)$ and the set of $x$ is bijective to the set of lifts $S \to \cl{M}_{G, b, \mu, K_\phi}$ of $\can_{m, m}$. Each $x$ corresponds to a solution of \eqref{eq:levelVn_even} and it is a lift $\can_{n, m}$ exactly when $\pi_{\mfr{m}}(x \bmod \pi^m) \in \mfr{m} \otimes W_{O_F}(R^{\flat+}) / \pi^{m}$ and it is trivial modulo $[\varpi]$ for some pseudo-uniformizer $\varpi \in R^\flat$. Thus, the claim holds for strictly totally disconnected spaces, and it extends to every $S$ by $v$-descent. 
\end{proof}

Now, we will analyze the right-hand side of \eqref{eq:levelVn_even} for each term of \eqref{eq:definition_h_m_2m}. 

\begin{lem} \label{lem:computepit_even}
    For every $\alpha \in \Phi_{\mu < 0}$ and $m < i \leq 2m$, 
    \[
        \pi_{\mfr{m}}(\Ad(\sigma(g_{<m})^{-1})([s_{\alpha,i}u_\alpha^\flat]\pi^{i-m-1}du_\alpha)) \bmod \pi^m \in \mfr{m}\otimes W_{O_F}(R^{\flat +}) / \pi^m. 
    \]
    Moreover, it is trivial modulo $[\pi^{\flat, 1/N}]$ for sufficiently large $N$. 
\end{lem}
\begin{proof}
    It follows by the same argument as \Cref{lem:compute_h_mn}: in this case, $\sigma(g_{<m})^{-1}$ is trivial modulo $(\pi^{2m - i + 1}, [\pi^{\flat, q(i-m)}])$ and since $s_{\alpha,i}u_\alpha^\flat \in \pi^{\flat, m - i} R^{\flat +}$, 
    \[
        \Ad(\sigma(g_{<m})^{-1})([s_{\alpha,i}u_\alpha^\flat]\pi^{i-m-1}du_\alpha) - [s_{\alpha,i}u_\alpha^\flat]\pi^{i-m-1}du_\alpha 
    \]
    modulo $\pi^m$ lies in 
    \[
        [\pi^{\flat, q(i-m) + (m - i)}] \cdot \mfr{\breve{g}}\otimes W_{O_F}(R^{\flat +})/\pi^{m}. 
    \]
    The claim follows since 
    $
        q(i-m)  + (m - i) > 0. 
    $
\end{proof}

\begin{lem} \label{lem:redtoh_even}
    Let 
    $
        h_m = \prod_{\alpha \in \Phi_{\mu<0}} i_\alpha([s_{\alpha,m}u_\alpha^\flat]\pi^m) \in \cl{G}_m(W_{O_F}(R^{\flat +}) / \pi^{m + 1}). 
    $
    Then
    \[
        \Ad(\sigma(g_{<m})^{-1})(h_m) \equiv h_m \pmod{[\pi^{\flat, q}]}. 
    \]
\end{lem}
\begin{proof}
    The claim follows from \Cref{lem:valsalpha} since $g_{<m}$ is trivial modulo $[\pi^{\flat}]$. 
\end{proof}

Now, the functor on perfect $k$-algebras
\[
    H_{[m, 2m], \phi} \colon A \mapsto \cl{G}_m(W_{O_F}(A) / \pi^{m + 1}) / (\mfr{m}^\perp \otimes W_{O_F}(A) / \pi^m \oplus \Ker(X) \otimes A)
\]
is a perfect group $k$-scheme perfectly of finite type (cf.\ \cite[Section 2.2]{Tak26_Yu}). By \cite[Lemma A.26]{Zhu17}, there is a smooth $k$-group scheme $H_{[m, 2m], \phi}^\pre$ with $H_{[m, 2m], \phi} = (H_{[m, 2m], \phi}^\pre)^\perf$. This deperfection is technically needed to consider the congruence on this functor. 

\begin{prop} \label{prop:levelVn_even_modulo}
    The right-hand side of \eqref{eq:levelVn_even} lies in 
    \[
        \cl{G}_m(W_{O_F}(R^{\flat +}) / \pi^{m + 1}) / (\mfr{m}^\perp \otimes W_{O_F}(R^{\flat +}) / \pi^m \oplus \Ker(X) \otimes R^{\flat +})
    \]
    and is congruent to $h_m$ modulo $[\pi^{\flat,1/N}]$ for sufficiently large $N$. 
\end{prop}
\begin{proof}
    It follows from the combination of \Cref{lem:computepit_even} and \Cref{lem:redtoh_even}. Precisely, the second claim is an equality in $H_{[m, 2m], \phi}^\pre(R^{\flat+} / \pi^{\flat, 1 / N})$ and the claim follows from the equality in $H_{[m, 2m], \phi}(R^{\flat+} / R^{\flat, \circ \circ})$ since $H_{[m, 2m], \phi}^\pre$ is finitely presented. 
\end{proof}

Now, we will compute the reduction of $\cl{W}(n)_\phi$ using \Cref{prop:levelVn_even_modulo}. Following \cite[Theorem 3.16]{Tak25z} in the depth-zero case, we begin with the construction of formal models. 

\begin{defi}
    Let $u_{\alpha,n} = - s^\sharp_{\alpha, m} u_\alpha$ and let 
    \[
        \mfr{U}(n) = \Spf(O_{\bb{C}_p}\langle u_{\alpha, n} \vert \alpha \in \Phi_{\mu < 0} \rangle) \to \Spf(R_{\cl{G}, \mu}) 
    \]
    be the smooth formal model of $\cl{U}(n)_{\bb{C}_p}$. Let $\msf{U}(n) = \mfr{U}(n)_\red = \Spec(\ov{k}[u_{\alpha, n} \vert \alpha \in \Phi_{\mu < 0}])$. 
\end{defi}

As in \Cref{defi:reduction_of_specaff_odd}, we first introduce a finite \'{e}tale cover of $\msf{U}(n)$. For this, we use the Heisenberg torsor of $\msf{U}(n)$ introduced in \cite[Definition 5.2]{Tak26_Yu}. We review this construction. 

\begin{defi}\textup{(\cite[Definition 2.11]{Tak26_Yu})} \label{defi:Heisenberg_group_scheme}
    Let $H_{x, n, \phi}$ be the perfect group scheme perfectly of finite type over $k$ defined as the quotient
    \[
        H_{x, n, \phi} = (L^+\cl{G}_{x, n, m} / L^+ \cl{G}_{x, n+, m+}) / \Ker(X).     
    \]
    Here, $\cl{G}_{x, n, m}$ and $\cl{G}_{x, n+, m+}$ are Moy-Prasad group schemes and $L^+(-)$ denotes the positive loop group functor (see \cite[Definition 2.4]{Tak26_Yu}). We set $\msf{H}_{x, n, \phi} = H_{x, n, \phi}(k)$. 
\end{defi}

By construction, $\msf{H}_{x, n, \phi} \cong G(F)_{x, n, m} / K_\phi$. Now, consider the decomposition
\[
    \Phi = \Phi_{\msf{N}} \sqcup \Phi_{\msf{M}} \sqcup \Phi_{\msf{\ov{N}}}
\]
in terms of roots inside $\msf{N}$, $\msf{M}$ and $\msf{\ov{N}}$ respectively. Let $\Phi^{\msf{M}} = \Phi - \Phi_{\msf{M}}$. To apply the Deligne-Lusztig construction of $H_{x, r, \phi}$ (see \cite[Definition 5.2]{Tak26_Yu}), one needs a polarization $\Phi^{\msf{M}} = \Phi^{\msf{M}+} \sqcup \Phi^{\msf{M}-}$. Here, we use the opposite decomposition 
\[
    \Phi^{\msf{M}+} = \Phi_{\msf{\ov{N}}}, \quad \Phi^{\msf{M}-} = \Phi_{\msf{N}}. 
\] 
Since $M$ is elliptic, the characteristic index set (see \cite[Definition 4.10]{Tak26_Yu}) is $\Phi_{\msf{\ov{N}}} \cap \sigma(\Phi_{\msf{N}}) = \Phi_{\mu < 0}$ (cf.\ \cite[Lemma 6.1]{Tak26_Yu}). The inclusion $\breve{\cl{G}}_{x, m}^{\mu < 0} \subset \cl{\breve{G}}_{x, n, m}$ induces a closed immersion 
\[
    \msf{U}(n)^\perf \subset H_{x, r, \phi, \ov{k}}. 
\]
Then, we use the Lang torsor of $H_{x, r, \phi}^\pre$ to construct a finite \'{e}tale cover. 

\begin{defi} \label{defi:reduction_of_specaff_even}
    Let $\msf{W}(n)_\phi$ be a finite \'{e}tale cover of $\msf{U}(n)$ filling the Cartesian diagram
    \begin{center}
        \begin{tikzcd}
            \msf{W}(n)_\phi^\perf \ar[r] \ar[d] & H_{x, n, \phi, \ov{k}}^\pre \ar[d, "h \mapsto \sigma(h) h^{-1}"] \\
            \msf{U}(n)^\perf \ar[r, hook] & H_{x, n, \phi, \ov{k}}^\pre. 
        \end{tikzcd}
    \end{center}
    By the topological invariance of finite \'{e}tale sites, $\msf{W}(n)_\phi$ is uniquely determined and $\msf{W}(n)_\phi \to \msf{U}(n)$ uniquely lifts to a finite \'{e}tale $\msf{H}_{x, n, \phi}$-torsor
    \[
        \mfr{W}(n)_\phi \to \mfr{U}(n). 
    \]
\end{defi}

\begin{prop} \label{prop:good_reduction_even}
    There is an $\msf{H}_{x, n, \phi}$-equivariant isomorphism $\mfr{W}(n)_{\phi, \eta} \cong \cl{W}(n)_{\phi, \bb{C}_p}$ over $\cl{U}(n)$. 
\end{prop}
\begin{proof}
    Let $\mfr{U}_\infty(n) = \Spf(O_{\bb{C}_p}\langle u_{\alpha,n}^{1/p^\infty} \vert \alpha \in \Phi_{\mu < 0} \rangle )$ be a formal model of $\cl{U}_\infty(n)$. Let 
    \[
        \mfr{W}_\infty(n)_\phi = \mfr{U}_\infty(n)\times_{\mfr{U}(n)} \mfr{W}(n)_\phi
    \]
    be a finite \'{e}tale cover of $\mfr{U}_\infty(n)$. By \Cref{prop:mapatinf}, it is enough to construct an $\msf{H}_{x, n, \phi}$-equivariant map $\mfr{W}_\infty(n)_{\phi, \eta}^\diamond \to \cl{W}(n)^\diamond_\phi$. Let $S = \Spa(R,R^+)$ be an affinoid perfectoid space over $\mfr{W}_\infty(n)_{\phi, \eta}$. Since $\mfr{W}(n)_\phi$ is affine, $S$ admits a unique extension $\Spf(R^+) \to \mfr{W}(n)_\phi$. The restriction to $\Spec(R^+/R^{\circ \circ})$ provides $h \in H_{x, n, \phi}(R^+/R^{\circ \circ})$ with
    \[
       \sigma(h) h^{-1} \equiv h_m^{-1} \bmod R^{\circ \circ}
       \in H_{x, n, \phi}(R^{\flat +}/R^{\circ \circ}). 
    \]
    For sufficiently large $N$, $h$ admits a lift $h_N \in H_{[m, 2m], \phi}^\pre(R^{\flat +} / \pi^{\flat, 1 / N})$ with 
    \[
        h_N \sigma(h_N)^{-1} \equiv h_m \bmod \pi^{\flat, 1 / N}  \in H_{[m, 2m], \phi}^\pre(R^{\flat +} / \pi^{\flat, 1 / N}). 
    \]
    By \Cref{prop:levelVn_even_modulo}, $h_N \sigma(h_N)^{-1}$ is congruent to the right-hand side of \eqref{eq:levelVn_even} modulo $[\pi^{\flat, 1/N}]$ for sufficiently large $N$. In particular, $h_N$ is a solution of \eqref{eq:levelVn_even} modulo $[\pi^{\flat, 1/N}]$. Since $R^{\flat +}$ is $\pi^{\flat}$-adically complete, $h_N$ lifts uniquely to a solution 
    \[
        x \in H_{[m, 2m], \phi}^\pre(R^{\flat +}) = H_{[m, 2m], \phi}(R^{\flat +}). 
    \]
    of \eqref{eq:levelVn_even} by the same argument as in \cite[Lemma 3.5]{Tak25z}. By construction, $\pi_{\mfr{m}}(x \bmod \pi^m) \in \mfr{m} \otimes W_{O_F}(R^{\flat+}) / \pi^{m}$ and it is trivial modulo $[\pi^{\flat, 1 / N}]$ for sufficiently large $N$. Thus, we get
    \[
        x \colon S \to \cl{W}(n)_\phi. 
    \]
    Since $x$ is functorially defined over $\mfr{W}_\infty(n)_{\phi, \eta}$, this gives $\mfr{W}_\infty(n)_{\phi, \eta}^\diamond \to \cl{W}(n)_\phi^\diamond$. The map is $\msf{H}_{x, r, \phi}$-equivariant since the action $h \mapsto h \cdot a$ for $a \in \msf{H}_{x, r, \phi}$ sends the solution $x$ to $x \cdot a$.  
\end{proof}

\subsection{Generalities on transfer computations}

In this section, we set up a general way to compute \eqref{eq:map_from_compact_to_usual} using the description $\cl{W}(n)_{\phi, \bb{C}_p} \cong \mfr{W}(n)_{\phi, \eta}$. 

By \Cref{prop:stability_under_diagonal_action} and the definition of $\cl{W}(n)_\phi$ (see \Cref{defi:special_affinoids_positive_depth}), $\cl{W}(n)_\phi$ is stable under the diagonal action of $M(F)_x$. Combining with \Cref{cor:stability_Img_can_nm}, $\cl{W}(n)_\phi$ admits an action of 
\[
    \cl{S}_n = (G(F)_{x, n, m} \times G_b(F)_{x_b, n, m_b}) \cdot \Delta(M(F)_{x, 0}) \subset J_n^0 \times J_{b, n}^0. 
\]
Moreover, the $(J_n^0 \times J_{b, n}^0)$-orbit of $\cl{W}(n)_\phi$ equals
\begin{equation} \label{eq:inverse_image_W(n)}
    \bigsqcup_{g \in M(F)_{x, 0} / M(F)_{x, n}} g \cdot \cl{W}(n)_\phi, 
\end{equation}
which is the inverse image of $\Img(\can_n)$ along $\cl{M}_{G, b, \mu, K_\phi} \to \cl{M}_{G, b, \mu, J_n^0}$ (cf.\ \eqref{eq:inverse_image_U(n)}). 

Recall from \Cref{defi:R_induction} (2) that a Heisenberg-Weil representation $\kappa_{x, n, \phi}$ is an irreducible representation of $(G(F)_{x, n, m} / K_\phi) \rtimes M(F)_x$ whose restriction to $G(F)_{x, n, m} / K_\phi \cong \msf{H}_{x, n, \phi}$ is a Heisenberg representation with central character $\psi$. Note that $\msf{H}_{x, n, \phi} \cong k$ when $n = 2m - 1$. Now, $K_\phi \subset J_n^0$ is normal and we have 
\[
    R^{J_n}_M(\rho)\vert_{J_n^0} = \kappa_{x, n, \phi} \otimes \rho
\]
by construction. In particular, it follows formally that 
\begin{align*}
    R\Gamma_c(\cl{U}(n)_{\bb{C}_p}, \can_n^*\cl{L}_{R^{K_n}_M(\rho)}) &\cong R\Gamma_c(\sqcup_{g \in M(F)_{x, 0} / M(F)_{x, n}} g \cdot \cl{W}(n)_{\phi, \bb{C}_p}, \Qla) \otimes_{J_n^0} (\kappa_{x, n, \phi} \otimes \rho) \\
    & \cong (R\Gamma_c(\cl{W}(n)_{\phi, \bb{C}_p}, \Qla) \otimes_{G(F)_{x, n, m}} \kappa_{x, n, \phi}) \otimes \rho, 
\end{align*}
where we regard the $J_n^0$-action on \eqref{eq:inverse_image_W(n)} as a right action. This isomorphism is equivariant under $G_b(F)_{x_b, n, m_b}$. Moreover, for each $j \in M(F)_{x, 0}$, the inner action of $j$ is given diagonally on each factor, i.e.\ 
\[
    j \cdot (a \otimes b \otimes c) = \Delta(j)\cdot a \otimes jb \otimes jc
\]
for $a \in H^\bullet_c(\cl{W}(n)_\phi, \Qla)$, $b \in \kappa_{x, n, \phi}$ and $c \in \rho$. As in the proof of \Cref{prop:verification_at_depth_zero}, the following holds via the nearby cycle functor. 

\begin{prop} \label{prop:nearby_cycle_positive_depth}
    The natural map 
    \[
        R\Gamma_c(\cl{U}(n)_{\bb{C}_p}, \can_n^*\cl{L}_{R^{K_n}_M(\rho)}) \to R\Gamma(\cl{U}(n)_{\bb{C}_p}, \can_n^*\cl{L}_{R^{K_n}_M(\rho)})
    \]
    is isomorphic (as complexes of $J_{b, n}$-representations) to 
    \[
        (R\Gamma_c(\msf{W}(n)_\phi, \Qla) \otimes_{G(F)_{x, n, m}} \kappa_{x, n, \phi}) \otimes \rho \to (R\Gamma(\msf{W}(n)_\phi, \Qla) \otimes_{G(F)_{x, n, m}} \kappa_{x, n, \phi}) \otimes \rho, 
    \]
    where the $M(F)_x$-action on each term is given diagonally. 
\end{prop}
\begin{proof}
    The previous argument works as well to $R\Gamma(\cl{U}(n)_{\bb{C}_p}, \can_n^*\cl{L}_{R^{K_n}_M(\rho)})$, so the claim follows as in the proof of \Cref{prop:verification_at_depth_zero}, using \cite[Lemma 4.5]{Tsu16} to get a commutative diagram
    \begin{center}
        \begin{tikzcd}
            R\Gamma_c(\msf{W}(n)_\phi, \Qla) \ar[r, "\sim"] \ar[d] & R\Gamma_c(\cl{W}(n)_{\phi, \bb{C}_p}, \Qla)\ar[d] \\
            R\Gamma(\msf{W}(n)_\phi, \Qla) \ar[r, "\sim"] & R\Gamma(\cl{W}(n)_{\phi, \bb{C}_p}, \Qla). 
        \end{tikzcd}
    \end{center}
    equivariant under $\cl{S}_n$. Here, we use \Cref{prop:good_reduction_odd} and \Cref{prop:good_reduction_even}, which imply that $\cl{W}(n)_{\phi, \bb{C}_p}$ has good reduction and $\msf{W}(n)_\phi$ is its reduction. Note that we only keep track of the $J_{b, n}^0$-action previously, but it is enough for the claim since $J_{b, n} = Z_G(F)J_{b, n}^0$ by \Cref{lem:hyperspecial_M(F)_x} (see also \Cref{lem:diagonal_central_action}). 
\end{proof}

Now, let $\kappa_{x_b, n, \phi}$ be the Heisenberg-Weil representation of
\[
    (G_b(F)_{x_b, n, n/2} / G_b(F)_{x_b, n+, n/2+}) \rtimes M(F)_x. 
\]
Since it might be $G_b(F)_{x_b, n, m_b} \neq G_b(F)_{x_b, n, n/2}$ when $n = 2m - 1$, one needs the following variant of $\kappa_{x_b, n, \phi}$. 

\begin{defi}
    Let $\kappa_{b, n, \phi}$ be a smooth representation of $G_b(F)_{x_b, n, m_b} \rtimes M(F)_x$ such that 
    \[
        \kappa_{b, n, \phi} = \cInd_{G_b(F)_{x_b, n, n/2}}^{G_b(F)_{x_b, n, m_b}} \kappa_{x_b, n, \phi}. 
    \]
\end{defi}

By construction, $R_M^{J_{b, n}}(\rho) = \kappa_{b, n, \phi} \otimes \rho$. Since $R_M^{J_{b, n}}$ is irreducible, $\kappa_{b, n, \phi}$ is also irreducible in this context. By \Cref{prop:nearby_cycle_positive_depth}, \Cref{prop:restatement_LSV} (4) is now reduced to the following. 

\begin{prop} \label{prop:cohomology_reduction}
    The natural map 
    \[
        R\Gamma_c(\msf{W}(n)_\phi, \Qla) \otimes_{G(F)_{x, n, m}} \kappa_{x, n, \phi} \to R\Gamma(\msf{W}(n)_\phi, \Qla) \otimes_{G(F)_{x, n, m}} \kappa_{x, n, \phi}
    \]
    is an isomorphism and isomorphic to $\kappa_{b, n, \phi}[-d]$ as representations of $G_b(F)_{x_b, n, m_b} \rtimes M(F)_x$. 
\end{prop}

\begin{proof}
    As previously, it is enough to keep track of the action of $G_b(F)_{x_b, n, m_b} \rtimes M(F)_{x, 0}$ by \Cref{lem:hyperspecial_M(F)_x}. The claim will be proved in \Cref{prop:cohmology_reduction_even} when $n = 2m$ and in \Cref{prop:cohomology_reduction_odd} and \Cref{prop:cohomology_reduction_n=1} when $n = 2m - 1$.
\end{proof}

Note that $d = \dim \cl{M}_{\cl{G}, b, \mu} = \dim \msf{U}(n) = \dim \msf{W}(n)_\phi$. We devote the following sections to the proof of \Cref{prop:cohomology_reduction}. 

\subsection{Group actions on reductions of special open balls}

In this section, we study the group action on $\msf{U}(n)$. Let $\msf{u}_{\mu < 0} = \Lie(\msf{U}_{\mu < 0})$. Then, there is a natural isomorphism
\begin{equation} \label{eq:U(n)_isomorphism}
    \msf{U}(n) = \Spec(\ov{k}[u_{\alpha, n} \vert \alpha \in \Phi_{\mu < 0}]) \to \msf{u}_{\mu < 0}, \quad (u_{\alpha, n}) \mapsto \sum_{\alpha \in \Phi_{\mu < 0}} u_{\alpha, n}du_\alpha.
\end{equation}
When $n = 0$, we set $u_{\alpha, 0} = \pi^{-r_\alpha} u_\alpha$ and it is conceptually better to interpret \eqref{eq:U(n)_isomorphism} as 
\begin{equation} \label{eq:U(0)_isomorphism}
    \msf{U}(0) = \Spec(\ov{k}[u_{\alpha, 0} \vert \alpha \in \Phi_{\mu < 0}]) \to \msf{U}_{\mu < 0}, \quad (u_{\alpha, 0}) \mapsto \prod_{\alpha \in \Phi_{\mu < 0}} i_\alpha(u_{\alpha, 0}).
\end{equation}
The two versions are the same as $\msf{U}_{\mu < 0}$ is abelian. The first objective is to compare the $M(F)_x$-action on both sides. Here, the action on $\msf{u}_{\mu < 0}$ (resp.\ $\msf{U}_{\mu < 0}$) is the adjoint action and the action on $\msf{U}(n)$ is the reduction of the action on the generic fiber $\cl{U}(n)_{\bb{C}_p}$.

\begin{lem} \label{lem:equivariance_M_x_depth_zero}
    The isomorphism \eqref{eq:U(0)_isomorphism} is equivariant under $M(F)_x$.  
\end{lem}
\begin{proof}
    Since $M(F)_x = M(F)_{x, 0} Z_G(F)$ (see \Cref{lem:hyperspecial_M(F)_x}), it is enough to keep track of the action of $M(F)_{x, 0}$. Recall that the reduction $\msf{W}(0)$ of $\cl{W}(0)_{\bb{C}_p}$ is a finite \'{e}tale cover 
    \[
        \msf{W}(0) = \{ g \in \msf{G} \mid g^{-1} \sigma(g) \in \msf{U}_{\mu < 0} \} \to \msf{U}_{\mu < 0} \cong \msf{U}(0) ,\quad g \mapsto g^{-1} \sigma(g)
    \]
    (see \cite[Proposition 3.10]{Tak25z}\footnote{The absence of $w$ stems from the difference in the normalization of $M$ (see the proof of \Cref{lem:expBKF}).}). Here, the second isomorphism is the inverse to \eqref{eq:U(0)_isomorphism}. The description of the inner action on $\msf{W}(0)$ in \cite[Proposition 5.7]{Tak25z} implies that the diagonal $M(F)_{x, 0}$-action on $\msf{W}(0)$ is given by 
    \[
        \Delta(h) \cdot g = hgh^{-1} ,\quad h \in M(F)_{x, 0},\quad g \in \msf{W}(0). 
    \]
    Since $\msf{W}(0) \to \msf{U}(0)$ is surjective and we have
    \[
        (\Delta(h) \cdot g)^{-1} \sigma(\Delta(h) \cdot g) = \Ad(h)(g^{-1} \sigma(g)), 
    \]
    the action on $\msf{U}(0)$ is identified with the adjoint action on $\msf{U}_{\mu < 0}$ via \eqref{eq:U(0)_isomorphism}.
\end{proof}

Now, the case $n = 2m$ almost formally follows from the depth-zero case. 

\begin{prop} \label{prop:equivariance_reduction_even}
    The isomorphism \eqref{eq:U(n)_isomorphism} is equivariant under $M(F)_x$ when $n = 2m$. 
\end{prop}
\begin{proof}
    By \Cref{lem:equivariance_M_x_depth_zero}, the action of $h \in M(F)_x$ on the coordinate ring of $\mfr{U}(0)_{\bb{C}_p}$ is 
    \[
        h \cdot u_{\alpha, 0} = \Ad(h)(u_{\alpha, 0}) + \varpi P_{\alpha, h}
    \]
    for some $P_{\alpha, h} \in O_{\bb{C}_p} \langle u_{\alpha, 0} \vert \alpha \in \Phi_{\mu < 0} \rangle$ and a pseudo-uniformizer $\varpi \in \bb{C}_p$. Then, on the coordinate ring of $\cl{U}(n)_{\bb{C}_p}$, we have
    \[
        h \cdot \pi^{-m} u_{\alpha, 0} = \Ad(h)(\pi^{-m}u_{\alpha, 0}) + \varpi (\pi^{-m} P_{\alpha, h}). 
    \]
    Now, $P_{\alpha, h}(0) = 0$ since $x_\CM$ is fixed by $h$ by the uniqueness in \Cref{prop:U(0)_previous_work} (2). Thus $\pi^{-m} P_{\alpha, h} \in O_{\bb{C}_p} \langle u_{\alpha, n} \vert \alpha \in \Phi_{\mu < 0} \rangle$ and the action of $h$ on $\msf{U}(n)$ is given by 
    \begin{equation} \label{eq:coordinate_pi_power_to_s_alpha}
        \pi^{-m} u_{\alpha, 0} \mapsto \Ad(h)(\pi^{-m}u_{\alpha, 0}).
    \end{equation}
    By \Cref{lem:valsalpha}, $u_{\alpha, n}$ and $\pi^{-m}u_{\alpha, 0}$ differ only by a constant $c_{\alpha, n} = - \pi^{m + r_\alpha} s^\sharp_{\alpha, m} \in O_{\bb{C}_p}^\times$, and since $t_\infty \in Z_{\cl{M}}^\circ$, $c_{\alpha, n}$ depends only on the restriction of $\alpha$ to $Z_M^\circ$. Since the weight decomposition with respect to $Z_M^\circ$ 
    \[
        \msf{u}_{\mu < 0} = \bigoplus_{\nu \in X^*(Z_M^\circ)} \msf{u}_{\mu < 0}^\nu
    \]
    is stable under the adjoint action of $M(F)_{x, 0}$ and $t_\infty$ acts by a scalar on each $\msf{u}_{\mu < 0}^\lambda$, \eqref{eq:coordinate_pi_power_to_s_alpha} implies that the action of $h$ on $\msf{U}(n)$ is $\Ad(h)$ via the identification $u_{\alpha, n} = du_\alpha$. 
\end{proof}

Next, we treat the case $n = 2m - 1$. For this, we carry out a concrete computation of the inner action on explicit $\cl{G}$-BKF modules via \Cref{lem:expBKF} and \cite[Proposition 2.15]{Tak25z}. 

\begin{prop} \label{prop:equivariance_reduction_odd}
    The isomorphism \eqref{eq:U(n)_isomorphism} is equivariant under $M(F)_x$ when $n = 2m - 1$. 
\end{prop}
\begin{proof}
    We study the action of $h \in M(F)_{x, 0}$. Take an arbitrary geometric point $x \in \cl{U}(n)(\bb{C}_p)$ and let $u_\alpha^\flat, v_\alpha^\flat \in \bb{C}_p^\flat$ be elements such that $u_\alpha(x) = (u_\alpha^\flat)^\sharp$ and $u_\alpha(hx) = (v_\alpha^\flat)^\sharp$. Then, there is a unique isomorphism between 
    \[
        (W_{O_F}(O_{\bb{C}_p}^\flat) \otimes \cl{G}, h \mu([\pi^\flat] - \pi) \prod_{\alpha \in \Phi_{\mu < 0}} i_\alpha([u_\alpha^\flat]) h^{-1} \cdot \sigma)
    \]
    and
    \[
        (W_{O_F}(O_{\bb{C}_p}^\flat) \otimes \cl{G}, \mu([\pi^\flat] - \pi) \prod_{\alpha \in \Phi_{\mu < 0}} i_\alpha([v_\alpha^\flat]) \cdot \sigma)
    \]
    as deformations of the prismatic $(\cl{G}, \mu)$-display $(\breve{\cl{G}}, \mu(-\pi) \sigma)$. It is given by a unique element $\iota \in \cl{G}(W_{O_F}(O_{\bb{C}_p}^\flat))$ that is trivial modulo $[\varpi]$ for some pseudo-uniformizer $\varpi \in \bb{C}_p^\flat$ and satisfies
    \begin{equation} \label{eq:inner_action_for_reduction}
        \iota h \mu([\pi^\flat] - \pi) \prod_{\alpha \in \Phi_{\mu < 0}} i_\alpha([u_\alpha^\flat]) h^{-1} \sigma(\iota)^{-1} = \mu([\pi^\flat] - \pi) \prod_{\alpha \in \Phi_{\mu < 0}} i_\alpha([v_\alpha^\flat]). 
    \end{equation}
    In fact, $\iota$ is trivial modulo $[\pi^{\flat, m}]$ by the uniqueness of $\iota$ and $u_\alpha^\flat, v_\alpha^\flat \in \pi^{\flat, m} \cdot O_{\bb{C}_p}^\flat$ (see \cite[Remark 2.11]{Tak25z}). Now, we rewrite \eqref{eq:inner_action_for_reduction} after taking the congruence modulo $[\pi^{\flat, m + 1}]$. Then, since $h \in M(F)_{x, 0}$ commutes with $\mu([\pi^\flat] - \pi)$, we have 
    \begin{equation} \label{eq:inner_action_modulo_pi_flat_m+1}
        \mu([\pi^\flat] - \pi)^{-1} \iota \mu([\pi^\flat] - \pi) \cdot \sigma(\iota)^{-1} \equiv \prod_{\alpha \in \Phi_{\mu < 0}} i_\alpha([v_\alpha^\flat]) \Ad(h)\bigl(\prod_{\alpha \in \Phi_{\mu < 0}} i_\alpha([u_\alpha^\flat])\bigr)^{-1} \hspace{-5pt}\pmod{[\pi^{\flat, m + 1}]}. 
    \end{equation}
    Note that this equation can be regarded as an equation in the abelian group
    \[
        \cl{G}([\pi^{\flat, m}]) / \cl{G}([\pi^{\flat, m + 1}]) \cong \mfr{g} \otimes [\pi^{\flat, m}] \cdot W_{O_F}(O_{\bb{C}_p}^\flat) / [\pi^{\flat, m + 1}] \cdot W_{O_F}(O_{\bb{C}_p}^\flat)
    \]
    where $\cl{G}([\varpi]) = \Ker(\cl{G}(W_{O_F}(O_{\bb{C}_p}^\flat)) \to \cl{G}(W_{O_F}(O_{\bb{C}_p}^\flat) / [\varpi]))$ for $\varpi \in O_{\bb{C}_p}$. Then, consider the projection 
    \[
        \pr \colon \mfr{g} \otimes [\pi^{\flat, m}] \cdot W_{O_F}(O_{\bb{C}_p}^\flat) / [\pi^{\flat, m + 1}] \cdot W_{O_F}(O_{\bb{C}_p}^\flat) \to \mfr{g} \otimes \pi^{m} O_{\bb{C}_p} / \pi^{m + 1} O_{\bb{C}_p}
    \]
    modulo $[\pi^\flat] - \pi$. By \eqref{eq:inner_action_modulo_pi_flat_m+1}, it is easy to see 
    \[
        \pr(\mu([\pi^\flat] - \pi)^{-1} \iota \mu([\pi^\flat] - \pi)) \in \breve{\mfr{g}}^{\mu \geq 0}, \quad 
        \pr(\iota) \in \breve{\mfr{g}}^{\mu \leq 0}. 
    \]
    In particular, the image of the left-hand side of \eqref{eq:inner_action_modulo_pi_flat_m+1} under $\pr$ lies in $\mfr{\breve{g}}^{\mu \geq 0}$ since $\msf{N} \cap \sigma(\ov{\msf{N}}) = \msf{U}_{\mu > 0}$ (see \Cref{defi:parabolic_mu}). It follows that \eqref{eq:inner_action_modulo_pi_flat_m+1} implies 
    \[
        \sum_{\alpha \in \Phi_{\mu < 0}} u_\alpha(hx) du_\alpha \equiv \Ad(h)\bigl(\sum_{\alpha \in \Phi_{\mu < 0}} u_\alpha(x) du_\alpha\bigr) \pmod{\pi^{m + 1}}
    \]
    in $\breve{\mfr{g}}^{\mu < 0} \otimes \pi^{m} O_{\bb{C}_p} / \pi^{m + 1} O_{\bb{C}_p}$. Since $u_{\alpha, n} = \pi^{-m}u_\alpha$, we have 
    \[
        \bigl(\sum_{\alpha \in \Phi_{\mu < 0}} u_{\alpha, n} du_\alpha\bigr) \circ h \equiv \Ad(h) \circ \bigl(\sum_{\alpha \in \Phi_{\mu < 0}} u_{\alpha, n} du_\alpha \bigr) \pmod{\pi}
    \]
    as maps from $\mfr{U}(n) \otimes (O_{\bb{C}_p} / \pi)$ to $\msf{u}_{\mu < 0}$. Thus, we get the desired equivariance. 
\end{proof}

Next, we study the inner action on $\msf{U}(n)$. For our purpose, it is enough to consider the inner action of $G_b(F)_{x_b, m}$. We will deduce the triviality of the inner action from the following computation of explicit $\cl{G}$-BKF modules. 

\begin{lem} \label{lem:isom_deformation_arbit}
    Take an element $u_\alpha^\flat \in \pi^{\flat, m} O_{\bb{C}_p}^\flat$ for each $\alpha \in \Phi_{\mu < 0}$ and let
    \[
        b_u = \mu([\pi^\flat] - \pi) \prod_{\alpha \in \Phi_{\mu < 0}} i_\alpha([u_\alpha^\flat]). 
    \]
    Then, for every $j \in G_b(F)_{x_b, m}$, the deformations
    \[
        (\cl{G} \otimes W_{O_F}(O_{\bb{C}_p}^\flat)/[\pi^{\flat ,m + 1}],b_u \sigma) ,\quad
        (\cl{G} \otimes W_{O_F}(O_{\bb{C}_p}^\flat)/[\pi^{\flat ,m + 1}], jb_u\sigma(j)^{-1} \sigma)
    \]
    are isomorphic. In other words, there is $\iota \in \cl{G}(W_{O_F}(O_{\bb{C}_p}^\flat)/[\pi^{\flat ,m + 1}])$ that is trivial modulo $[\pi^\flat]$ such that
    $
        \iota j b_u \sigma(\iota j)^{-1} = b_u. 
    $
\end{lem}
\begin{proof}
    Let $\Ad(b_u \sigma) = \Ad(b_u) \circ \sigma$. Then, the desired equation is rewritten as 
    \[
        \iota = \Ad(b_u\sigma)(\iota) \cdot \Ad(b_u \sigma)(j) j^{-1}. 
    \]
    Let $X = \Ad(b_u \sigma)(j) j^{-1} \in G(W_{O_F}(O_{\bb{C}_p}^\flat)[\tfrac{1}{\pi}]/[\pi^{\flat, m + 1}])$. Since $j \in G_b(F)$, $X$ is trivial modulo $[\pi^\flat]$. Then, $\Ad(b_u\sigma)^i(X)$ is trivial modulo $[\pi^{\flat, q^i}]$, so it is trivial for sufficiently large $i$. Now, we will show that for every $i \geq 0$, we have
    \begin{equation} \label{eq:Ad(bu_sigma)}
        \Ad(b_u\sigma)^i(X) \in \cl{G}(W_{O_F}(O_{\bb{C}_p}^\flat)/[\pi^{\flat, m + 1}]). 
    \end{equation}
    Then, the infinite product
    \[
        \iota = \cdots \Ad(b_u \sigma)^2(X) \cdot  \Ad(b_u \sigma)(X) \cdot X 
    \]
    is well-defined and satisfies the desired conditions. Thus, it is enough to prove \eqref{eq:Ad(bu_sigma)}. 

    First, suppose $i = 0$. Then, it is enough to show $\Ad(b_u \sigma)(j) \in \cl{G}(W_{O_F}(O_{\bb{C}_p}^\flat)/[\pi^{\flat, m + 1}])$. Let
    \[
        A = \Ad\biggl(\prod_{\alpha \in \Phi_{\mu < 0}} i_\alpha([u_\alpha^\flat]) \sigma\biggr)(j) \in \cl{G}(W_{O_F}(O_{\bb{C}_p}^\flat))
    \]
    and let $\bar{A}$ denote the pushforward of $A$ along $W_{O_F}(O_{\bb{C}_p}^\flat) \to O_{\bb{C}_p} / \pi^{m + 1}$. By applying \cite[Proposition 4.2.9]{Ito25a} to $(W_{O_F}(O_{\bb{C}_p}^\flat)/[\pi^{\flat, m + 1}], [\pi^\flat] - \pi)$, it is enough to show that $\bar{A}$ lies in $\cl{\breve{G}}^{\mu \geq 0}(O_{\bb{C}_p}/\pi^{m + 1})$. Since $j \in G_b(F)_{x_b, m}$ and $u_\alpha^\flat \in \pi^{\flat, m} O_{\bb{C}_p}^\flat$, $\bar{A}$ is trivial modulo $\pi^m$ and $\bar{A} = \sigma(j)$. We also have $\sigma(j) \bmod \pi^{m + 1} \in \breve{\cl{G}}^{\mu \geq 0}(O_{\breve{F}}/\pi^{m + 1})$, so we get the claim. 

    Next, suppose $i > 0$. Since $b_u \equiv \mu([\pi^\flat] - \pi)$ modulo $[\pi^{\flat, m}]$, 
    \[
        \Ad(b_u\sigma)^i(X) \equiv \Ad(b_u \sigma)(\Ad(\mu([\pi^\flat] - \pi) \sigma)^i(j) \Ad(\mu([\pi^\flat] - \pi) \sigma)^{i-1}(j)^{-1} ) \bmod{[\pi^{\flat, m+1}]}. 
    \]
    By \Cref{lem:AdbCM}, we have
    \[
        X_i = \Ad(\mu([\pi^\flat] - \pi) \sigma)^i(j) \Ad(\mu([\pi^\flat] - \pi) \sigma)^{i-1}(j)^{-1} \in \cl{G}(W_{O_F}(O_{\bb{C}_p}^\flat) / [\pi^{\flat, m+1}]). 
    \]
    Let
    \[
        A = \Ad\biggl(\prod_{\alpha \in \Phi_{\mu < 0}} i_\alpha([u_\alpha^\flat]) \sigma \biggr)(X_i) \in \cl{G}(W_{O_F}(O_{\bb{C}_p}^\flat) / [\pi^{\flat, m+1}])
    \]
    and let $\bar{A}$ denote the pushforward of $A$ along $W_{O_F}(O_{\bb{C}_p}^\flat) / [\pi^{\flat, m+1}] \to O_{\bb{C}_p} / \pi^{m + 1}$. As previously,  it is enough to show that $\bar{A}$ lies in $\cl{\breve{G}}^{\mu \geq 0}(O_{\bb{C}_p}/\pi^{m + 1})$. By \Cref{lem:AdbCM}, the pushforward $\bar{X}_i$ of $X_i$ along $W_{O_F}(O_{\bb{C}_p}^\flat) / [\pi^{\flat, m+1}] \to O_{\bb{C}_p} / \pi^{m + 1}$ lies in $\breve{\cl{G}}^{\lambda \leq 0}(O_{\bb{C}_p} / \pi^{m + 1})$. Since $u_\alpha^\flat \in \pi^{\flat, m} O_{\bb{C}_p}^\flat$, we get $\bar{A} = \sigma(X_i)$, so the claim follows from $\msf{\ov{N}} \cap \sigma(\msf{N}) = \msf{U}_{\mu < 0}$ (see \Cref{defi:parabolic_mu}). 
\end{proof}

\begin{prop} \label{prop:inner_action_U(n)_reduction}
    The action of $G_b(F)_{x_b, m}$ on $\msf{U}(n)$ is trivial. 
\end{prop}
\begin{proof}
    First, suppose $n = 2m - 1$. Take a geometric point $x \in \cl{U}(n)(\bb{C}_p)$ and choose $u_\alpha^\flat \in \pi^{\flat, m}O_{\bb{C}_p}^\flat$ so that $u_\alpha(x) = (u_\alpha^\flat)^\sharp$. By \Cref{lem:Perfd_universal_deformation} and \Cref{lem:isom_deformation_arbit}, the composition
    \[
        R_{\cl{G},\mu} \xrightarrow{j^*} R_{\cl{G},\mu} \xrightarrow{x} O_{\bb{C}_p} \twoheadrightarrow O_{\bb{C}_p}/\pi^{m + 1}
    \]
    is equal to $R_{\cl{G},\mu} \xrightarrow{x} O_{\bb{C}_p} \twoheadrightarrow O_{\bb{C}_p}/\pi^{m + 1}$ for every $j \in G_b(F)_{x_b, m}$. In other words,
    \[
        O_{\bb{C}_p} \langle u_{\alpha, n} \rangle \xrightarrow{j^*} O_{\bb{C}_p} \langle u_{\alpha, n} \rangle \xrightarrow{x} O_{\bb{C}_p} \twoheadrightarrow O_{\bb{C}_p}/\pi
    \]
    is equal to $ O_{\bb{C}_p} \langle u_{\alpha, n} \rangle \xrightarrow{x} O_{\bb{C}_p} \twoheadrightarrow O_{\bb{C}_p}/\pi$. It follows that $G_b(F)_{x_b, m}$ acts trivially on $\msf{U}(n)$ since the choice of $x$ is arbitrary. 

    The case $n = 2m$ can be proved in the same way. Take a geometric point $x \in \cl{U}(n)(\bb{C}_p)$ and choose $u_\alpha^\flat \in \pi^{\flat, m + r_\alpha}O_{\bb{C}_p}^\flat$ so that $u_\alpha(x) = (u_\alpha^\flat)^\sharp$.  By \Cref{lem:Perfd_universal_deformation} and \Cref{lem:isom_deformation_arbit}, 
    \[
        R_{\cl{G},\mu} \xrightarrow{j^*} R_{\cl{G},\mu} \xrightarrow{x} O_{\bb{C}_p} \twoheadrightarrow O_{\bb{C}_p}/\pi^{m + 1}
    \]
    is equal to $R_{\cl{G},\mu} \xrightarrow{x} O_{\bb{C}_p} \twoheadrightarrow O_{\bb{C}_p}/\pi^{m + 1}$ for every $j \in G_{b, \mfr{f}, m}$. In other words,
    \[
        O_{\bb{C}_p} \langle u_{\alpha, n} \rangle \xrightarrow{j^*} O_{\bb{C}_p} \langle u_{\alpha, n} \rangle \xrightarrow{x} O_{\bb{C}_p} \twoheadrightarrow O_{\bb{C}_p}/\pi^\varepsilon
    \]
    is equal to $ O_{\bb{C}_p} \langle u_{\alpha, n} \rangle \xrightarrow{x} O_{\bb{C}_p} \twoheadrightarrow O_{\bb{C}_p}/\pi^\varepsilon$ for sufficiently small rational $\varepsilon > 0$. It follows that $G_b(F)_{x_b, m}$ acts trivially on $\msf{U}(n)$ since the choice of $x$ is arbitrary. 
\end{proof}

\subsection{\'{E}tale cohomology of reductions for $n = 2m$} \label{ssec:cohomology_reduction_even}

In this section, we treat the case $n = 2m$. In this case, $m_b = m$ and $\kappa_{b, n, \phi} = \kappa_{x_b, n, \phi}$ is a character that is trivial on $M(F)_x$ and given by $\psi \circ X$ on $G_b(F)_{x, n, m} / G_b(F)_{x, n+, m+} \cong \msf{m}_{x, n}$ (see \cite[Theorem 3]{Tak26_Yu}). 

In fact, \Cref{prop:cohomology_reduction} almost follows from \cite[Theorem 6]{Tak26_Yu} since $\msf{W}(n)_\phi \to \msf{U}(n)$ can be identified with the Heisenberg-torsor introduced there. We briefly review the construction in \cite{Tak26_Yu}. 

\begin{defi}\textup{(\cite[Definition 5.2]{Tak26_Yu})}
    The inclusion $\breve{\cl{G}}_{x, m}^{\mu < 0} \subset \cl{\breve{G}}_{x, n, m}$ induces a closed immersion $\msf{u}_{\mu < 0}^\perf \subset H_{x, r, \phi, \ov{k}}$ via the Moy-Prasad isomorphism. Let $\msf{W}_{\Phi_{\mu < 0}, n}$ be a finite \'{e}tale cover of $\msf{u}_{\mu < 0}^\perf$ filling the Cartesian diagram
    \begin{center}
        \begin{tikzcd}
            \msf{W}_{\Phi_{\mu < 0}, n} \ar[r] \ar[d] & H_{x, n, \phi, \ov{k}} \ar[d, "h \mapsto \sigma(h) h^{-1}"] \\
            \msf{u}_{\mu < 0}^\perf \ar[r, hook] & H_{x, n, \phi, \ov{k}}. 
        \end{tikzcd}
    \end{center}
    Here, the $M(F)_x$-actions on $\msf{u}_{\mu < 0}$ and $\msf{W}_{\Phi_{\mu < 0}, n}$ are induced from the adjoint action on $H_{x, n, \phi}$. 
\end{defi}

This definition is almost the same as \Cref{defi:reduction_of_specaff_even}: we have an isomorphism $\msf{W}(n)_\phi^\perf \cong \msf{W}_{\Phi_{\mu < 0}, n}$ by the base change along \eqref{eq:U(n)_isomorphism}. The subtlety is that the diagonal $M(F)_x$-action on $\msf{W}(n)_\phi$ comes from the generic fiber $\cl{W}(n)_{\phi, \bb{C}_p}$, while the action on $\msf{W}_{\Phi_{\mu < 0}, n}$ comes from $H_{x, n, \phi}$. We identify the two actions via \Cref{prop:equivariance_reduction_even}.

\begin{prop} \label{prop:equivariance_under_HW_even}
    We have $\msf{W}(n)_\phi^\perf \cong \msf{W}_{\Phi_{\mu < 0}, n}$ equivariantly under $G(F)_{x, n, m} \rtimes M(F)_x$. 
\end{prop}
\begin{proof}
    We already have a $G(F)_{x, n, m}$-equivariant isomorphism $\iota \colon \msf{W}(n)_\phi^\perf \cong \msf{W}_{\Phi_{\mu < 0}, n}$ by the above discussions, and it is enough to show the equivariance under $M(F)_{x, 0}$ by \Cref{lem:hyperspecial_M(F)_x}. For every $h \in M(F)_{x, 0}$, consider the endomorphism 
    \[
        f_h = \iota^{-1} \circ h^{-1} \circ \iota \circ h \colon \msf{W}(n) \to \msf{W}(n). 
    \]
    We would like to show $f_h = \id$. Since $\iota$ is $G(F)_{x, n, m}$-equivariant, $f_h$ is also equivariant under $G(F)_{x, n, m}$. Moreover, $f_h$ is a morphism over $\msf{U}(n)$ by \Cref{prop:equivariance_reduction_even}. Since $\msf{W}(n) \to \msf{U}(n)$ is finite \'{e}tale, the equalizer $\Eq(f_h, \id) \subset \msf{W}(n)$ is a $G(F)_{x, n, m}$-equivariant closed and open subset of $\msf{W}(n)$. Since $\msf{U}(n)$ is connected, $\Eq(f_h, \id)$ is $\emptyset$ or $\msf{W}(n)$, so it is enough to show $\Eq(f_h, \id) \neq \emptyset$. 

    Now, consider the CM point $x_{\CM, K_\phi} \in \cl{W}(n)_\phi$. As in the proof of \Cref{prop:stability_under_diagonal_action}, $x_{\CM, K_\phi}$ is fixed by the diagonal action of $M(F)_{x, 0}$. Then, its reduction $\msf{x}_\phi \in \msf{W}(n)_\phi$ is also fixed by $M(F)_{x, 0}$, so it is enough to show that $\iota(\msf{x}_\phi)$ is a fixed point under $M(F)_{x, 0}$. 
    Now, recall from \Cref{defi:CM_points_in_LSV} that the $\cl{G}$-BKF module and the level structure at $x_{\CM, K_\phi}$ are explicitly given by 
    \[
        (\cl{G} \otimes W_{O_F}(O_{\bb{C}_p}^\flat), \mu([\pi^\flat] - \pi) \sigma) ,\quad t_\infty \in Z_{\cl{M}}^\circ(W_{O_F}(\bb{C}_p^\flat)). 
    \]
    Thus, the $K_\phi$-level structure of $x_{\CM, K_\phi}$ associated via \Cref{lem:levelstrCp} and \Cref{lem:levelVn_even} is given by $x = \id$. Thus, $\iota(\msf{x}_\phi) = \id \in \msf{W}_{\Phi_{\mu < 0}, n} \subset H_{x, n, \phi, \ov{k}}$ and it is fixed under the adjoint action of $M(F)_{x, 0}$. 
\end{proof}

Similarly, we study the inner action on $\msf{W}(n)$ via \Cref{prop:inner_action_U(n)_reduction}. 

\begin{prop} \label{prop:trivial_inner_action_even}
    The inner action of $G_b(F)_{x_b, n+, m+}$ on $\msf{W}(n)_\phi$ is trivial when $n = 2m$. 
\end{prop}
\begin{proof}
    By \Cref{prop:inner_action_U(n)_reduction}, the action of each $j \in G_b(F)_{x_b, n+, m+}$ on $\msf{W}(n)_\phi$ is an action over $\msf{U}(n)$. Since the inner action commutes with the action of $G(F)_{x, n, m}$, the action of $j$ is trivial once we have $\Eq(j \vert_{\msf{W}(n)_\phi}, \id_{\msf{W}(n)_\phi}) \neq \emptyset$ as in the proof of \Cref{prop:equivariance_under_HW_even}. 

    For this, consider the sequence 
    \[
        \cl{M}_{G, b, \mu, G(F)_{x, 2m + 1, m + 1}} \to \cl{M}_{G, b, \mu, K_\phi} \to \cl{M}_{G, b, \mu, G(F)_{x, 2m, m}}. 
    \]
    It induces a sequence $\Img(\can_{n + 1, m + 1}) \to \cl{W}(n)_\phi \to \Img(\can_{n, m})$ since $\can_{n, m}$ and $\can_{n + 1, m + 1}$ are compatible thanks to the matching of the image of $x_\CM$. By \Cref{cor:stability_Img_can_nm}, this sequence is equivariant under $G_b(F)_{x_b, n+, m+}$. By passing to the reduction, we get a map 
    \[
        \msf{U}(n + 1) \to \msf{W}(n)_\phi \to \msf{U}(n)
    \]
    equivariant under $G_b(F)_{x_b, n+, m+}$. It is easy to see from the definition that $\msf{U}(n + 1) \to \msf{U}(n)$ maps to the origin $u_{\alpha, n} = 0$, so the image of $\msf{U}(n + 1) \to \msf{W}(n)_\phi$ is finite since $\msf{W}(n)_\phi \to \msf{U}(n)$ is finite \'{e}tale. Since $\msf{U}(n + 1)$ is connected, $\msf{U}(n + 1) \to \msf{W}(n)_\phi$ maps to a single point $\msf{x} \in \msf{W}(n)_\phi$. Then $\msf{x}$ is fixed by $G_b(F)_{x_b, n+, m+}$ since $\msf{U}(n + 1) \to \msf{W}(n)_\phi$ is equivariant under $G_b(F)_{x_b, n+, m+}$, so we get the claim. 
\end{proof}

Now, the desired cohomological property follows from the computation in \cite{Tak26_Yu}. 

\begin{prop} \label{prop:cohmology_reduction_even}
    \Cref{prop:cohomology_reduction} holds for $n = 2m$. 
\end{prop}
\begin{proof}
    By \Cref{prop:equivariance_under_HW_even}, \cite[Theorem 6]{Tak26_Yu} implies that we have
    \[
        H_c^i(\msf{W}(n)_\phi, \Qla)[\psi^{-1}] \cong 
        \left\{ \begin{alignedat}{4}
            & \kappa_{x, n, \phi}^* & \quad & (i = d) \\
            & 0 & \quad & (i \neq d)
        \end{alignedat}
        \right.
    \]
    as a representation of $\msf{H}_{x, r, \phi} \rtimes M(F)_x$ and the natural map  
    \[
        H_c^d(\msf{W}(n)_\phi, \Qla)[\psi^{-1}] \to H^d(\msf{W}(n)_\phi, \Qla)[\psi^{-1}]
    \]
    is an isomorphism (see also \cite[Lemma 6.1]{Tak26_Yu} for the middle concentration). Thus
    \[
        R\Gamma_c(\msf{W}(n)_\phi, \Qla) \otimes_{G(F)_{x, n, m}} \kappa_{x, n, \phi} \to R\Gamma(\msf{W}(n)_\phi, \Qla) \otimes_{G(F)_{x, n, m}} \kappa_{x, n, \phi}
    \]
    is an isomorphism and both sides are concentrated in degree \(d\), where the cohomology is one-dimensional. Moreover, the diagonal action of $M(F)_x$ on each side is trivial. Then, it remains to show that $G_b(F)_{x_b, n, m}$ acts on $H_c^d(\msf{W}(n)_\phi, \Qla) \otimes_{G(F)_{x, n, m}} \kappa_{x, n, \phi}$ via the character 
    \[
        G_b(F)_{x_b, n, m} \to G_b(F)_{x_b, n, m} / G_b(F)_{x_b, n+, m+} \cong \msf{m}_{x, n} \xrightarrow{\psi \circ X} \Qlax. 
    \]
    First, the inner action factors through $\msf{m}_{x, n}$ by \Cref{prop:trivial_inner_action_even}. Then, since the diagonal action of $M(F)_x$ is trivial and the composition
    \[
        M(F)_{x, n} \to G_b(F)_{x_b, n, m} \to \msf{m}_{x, n}
    \]
    is surjective, the inner action of $\msf{m}_{x, n}$ is inverse to the action of $M(F)_{x, n} \subset G(F)_{x, n, m}$. Then, the claim follows since $M(F)_{x, n}$ acts on $\kappa_{x, n, \phi}^*$ via $(\psi \circ X)^{-1}$. 
\end{proof}

\subsection{\'{E}tale cohomology of reductions for $n = 2m - 1$} \label{ssec:cohomology_reduction_odd}

In this section, we treat the case $n = 2m - 1$. In this case, $m_b = (m-1)+$ and $\kappa_{b, n, \phi}$ is a nontrivial induction of $\kappa_{x_b, n, \phi}$. First, we give a characterization of $\kappa_{b, n, \phi}$ as an irreducible representation of $G_b(F)_{x_b, n, m_b}$.

\begin{prop} \label{prop:Heisenberg_odd}
    There is a unique irreducible representation of $G_b(F)_{x_b, n, m_b}$ such that $G_b(F)_{x_b, n, m}$ acts by a scalar via the character 
    \[
        G_b(F)_{x_b, n, m} \to G_b(F)_{x_b, n, m} / G_b(F)_{x_b, 2m, m} \cong \msf{m}_{x, n} \xrightarrow{\psi \circ X} \Qlax. 
    \]
    It is isomorphic to $\kappa_{b, n, \phi}$ and we have $\dim \kappa_{b, n, \phi} = \sqrt{\lvert \msf{m}^\perp \rvert}$. Moreover, 
    \[
        \kappa_{x_b, n, \phi} \cong \kappa_{b, n, \phi}^{G_b(F)_{x_b, n+, n/2+}}. 
    \]
\end{prop}
\begin{proof}
    Let $\rho$ be an irreducible representation of $G_b(F)_{x_b, n, m_b}$ satisfying the given condition. First, we show that $\rho$ contains a fixed vector under $G_b(F)_{x_b, n+, n/2+}$. 

    By induction on $\tfrac{n}{2} < r \leq m$, we show that $\rho$ contains a $G_b(F)_{x_b, n+, r}$-fixed vector. For $r = m$, the claim follows from the assumption. Suppose that $\rho$ contains a $G_b(F)_{x_b, n+, r+}$-fixed vector for some $\tfrac{n}{2} < r < m$. Since $G_b(F)_{x_b, n+, r+} \subset G_b(F)_{x_b, n+, r}$ is normal and $G_b(F)_{x_b, n+, r} / G_b(F)_{x_b, n+, r+} \cong \msf{m}^\perp_{x_b, r}$ is abelian via the Moy-Prasad isomorphism, $\rho^{G_b(F)_{x_b, n+, r+}}$ contains a nonzero vector $v$ on which $G_b(F)_{x_b, n+, r} / G_b(F)_{x_b, n+, r+}$ acts by a character $\chi$. 

    The Moy-Prasad isomorphism induces identifications 
    \[
        \msf{m}^\perp_{x_b, r} \cong G_b(F)_{x_b, n+, r} / G_b(F)_{x_b, n+, r+} ,\quad 
        \msf{m}^\perp_{x_b, n - r} \cong G_b(F)_{x_b, n, n - r} / G_b(F)_{x_b, n, (n - r)+}. 
    \]
    The commutator pairing between $G_b(F)_{x_b, n, n - r}$ and $G_b(F)_{x_b, n+, r}$ induces a bilinear form
    \[      
        [-, -]_r \colon \msf{m}^\perp_{x_b, n - r} \times \msf{m}^\perp_{x_b, r} \to G_b(F)_{x_b, n} / G_b(F)_{x_b, n+} \cong \msf{m}_{x, n} \xrightarrow{X} k
    \]
    (see \cite[Proposition 13.2.5]{KP23}). Since $\phi$ is $(G, M)$-generic, this pairing is non-degenerate. Then, $\psi \circ [-, -]_r$ induces $\msf{m}^\perp_{x_b, n - r} \cong \Hom(\msf{m}^\perp_{x_b, r}, \Qlax)$, so we can take $j \in \msf{m}^\perp_{x_b, n - r}$ that maps to $\chi$ via the isomorphism. Then, for every $h \in G_b(F)_{x_b, n+, r}$, we have 
    \[
        h \cdot jv = j\cdot [j^{-1}, h] \cdot hv = j \cdot \psi([j^{-1}, h]_r) \chi(h) v = jv. 
    \]
    Thus, $jv$ is fixed by $G_b(F)_{x_b, n+, r}$. By induction, $\rho$ contains a $G_b(F)_{x_b, n+, n/2+}$-fixed vector. 
    
    Now, $G_b(F)_{x_b, n, n/2} / G_b(F)_{x_b, n+, n/2+}$ is a Heisenberg group, so it admits a unique irreducible representation $\kappa_{x_b, n, \phi}$ of central character $\psi \circ X$ by \cite[Lemma 1.2]{Ger77}. Since $\msf{m}_{x, n}  \cong G_b(F)_{x_b, n, n/2+} / G_b(F)_{x_b, n+, n/2+}$ acts on $\rho^{G_b(F)_{x_b, n+, n/2+}}$ by a scalar $\psi \circ X$, $\rho$ contains $\kappa_{x_b, n, \phi}$ as a $G_b(F)_{x_b, n, n/2}$-representation. By Frobenius reciprocity, we get $\kappa_{b, n, \phi} \subset \rho$ once we know the irreducibility of $\kappa_{b, n, \phi}$. Its dimension can be computed as
    \[  
        \dim \kappa_{b, n, \phi} = \dim \kappa_{x_b, n, \phi} \cdot [G_b(F)_{x_b, n, m_b} \colon G_b(F)_{x_b, n, n/2}] = \sqrt{\lvert \msf{m}^\perp_{x_b, n/2} \rvert} \cdot \lvert \msf{m}^\perp_{x_b, [m_b, n/2)} \rvert = \sqrt{\lvert \msf{m}^\perp \rvert}. 
    \]

    Now, we will show that $\kappa_{b, n, \phi}$ is irreducible (without resorting to the existence of $(G_b, M)$-supergeneric representations of depth $n$). We deduce the claim from \cite[Lemma 6.3.1]{BW16}: it is enough to show that the normalizer of $\psi \circ X$ in $G_b(F)_{x_b, n, m_b}$ as a character of $G_b(F)_{x_b, n, n/2+}$ is $G_b(F)_{x_b, n, n/2}$. It is easy to see that $G_b(F)_{x_b, n, n/2}$ normalizes $\psi \circ X$.
    
    Let $j \in G_b(F)_{x_b, n, m_b}$ be an element such that $j \in G_b(F)_{x_b, n, r} \backslash G_b(F)_{x_b, n, r+}$ for some $m_b \leq r < \tfrac{n}{2}$. Then, we can take $h \in G_b(F)_{x_b, n+, n - r}$ so that $(\psi \circ X)([j, h]_{n-r}) \neq 1$, and we have
    \[
        (\psi \circ X)(j h j^{-1}) = (\psi \circ X)([j, h]_{n-r} h) \neq (\psi \circ X)(h).
    \]
    Thus, $j$ does not normalize $\psi \circ X$ and we get the claim. 

    By the same argument, we see that $h \in G_b(F)_{x_b, n+, n/2+}$ acts by a nontrivial scalar on $j \cdot \kappa_{x_b, n, \phi} \subset \kappa_{b, n, \phi}$. Then we get $\kappa_{x_b, n, \phi} = \kappa_{b, n, \phi}^{G_b(F)_{x_b, n+, n/2+}}$ since we have
    \[
        \kappa_{b, n, \phi} = \bigoplus_{j \in G_b(F)_{x_b, n, m_b} / G_b(F)_{x_b, n, n/2}} j \cdot \kappa_{x_b, n, \phi}.  
    \]
\end{proof}

Now, we first identify $H_c^d(\msf{W}(n)_\phi, \Qla)[\psi^{-1}]$ as a representation of $G_b(F)_{x_b, n, m_b}$. First, we study the inner action on $\msf{W}(n)_\phi$ in the same way as \Cref{prop:trivial_inner_action_even}. 

\begin{prop} \label{prop:trivial_inner_action_odd}
    The inner action of $G_b(F)_{x_b, 2m, m}$ on $\msf{W}(n)_\phi$ is trivial when $n = 2m - 1$. 
\end{prop}
\begin{proof}
    By the same argument as in \Cref{prop:trivial_inner_action_even}, it is enough to show that $\msf{W}(n)_\phi$ has a $G_b(F)_{x_b, 2m, m}$-fixed point. Consider the sequence 
    \[
        \cl{M}_{G, b, \mu, G(F)_{x, 2m, m}} \to \cl{M}_{G, b, \mu, K_\phi} \to \cl{M}_{G, b, \mu, G(F)_{x, n, m}}. 
    \]
    It induces a sequence $\Img(\can_{n + 1, m}) \to \cl{W}(n)_\phi \to \Img(\can_{n, m})$ since $\can_{n, m}$ and $\can_{n + 1, m}$ are compatible thanks to the matching of the image of $x_\CM$. By \Cref{cor:stability_Img_can_nm}, this sequence is equivariant under $G_b(F)_{x_b, 2m, m}$. By passing to the reduction, we get a map 
    \[
        \msf{U}(n + 1) \to \msf{W}(n)_\phi \to \msf{U}(n)
    \]
    equivariant under $G_b(F)_{x_b, 2m, m}$. By the same argument as in \Cref{prop:trivial_inner_action_even}, the image of $\msf{U}(n + 1) \to \msf{W}(n)_\phi$ is a single point that is fixed by $G_b(F)_{x_b, 2m, m}$. Thus, we get the claim. 
\end{proof}

\begin{prop} \label{prop:positive_depth_diagonal_action_odd}
    The diagonal action of $M(F)_{x, 0+}$ on $\msf{W}(n)_\phi$ is trivial when $n = 2m - 1$. 
\end{prop}
\begin{proof}
    By \Cref{prop:equivariance_reduction_odd}, the diagonal action of $M(F)_{x, 0+}$ is trivial on $\msf{U}(n)$. Moreover, it commutes with the action of $G(F)_{x, n, m}$ on $\msf{W}(n)_\phi$. Thus, by the same argument as in \Cref{prop:trivial_inner_action_even}, it is enough to show that $\msf{W}(n)_\phi$ has a $M(F)_{x, 0+}$-fixed point. Since $x_{\CM, K_\phi}$ is stable under $M(F)_{x}$ by the proof of \Cref{prop:stability_under_diagonal_action}, its reduction is a fixed point under $M(F)_{x, 0+}$. Thus, we get the claim. 
\end{proof}

From now on, we compute the map in \Cref{prop:cohomology_reduction}. When $n = 2m - 1$, $\kappa_{x, n, \phi}$ is a character 
\[
    G(F)_{x, n, m} \to G(F)_{x, n, m} / G(F)_{x, n+1, m} \cong \msf{m}_{x, n} \xrightarrow{\psi \circ X} \Qlax. 
\] 
Moreover, $\msf{W}(n)_\phi \to \msf{U}(n)$ is a torsor under $k$. As in \Cref{ssec:reduction_odd}, we treat the cases $m = 1$ and $m \geq 2$ separately since the defining equations of reductions are different as in the Lubin-Tate case \cite[Section 3.9]{BW16}.

\subsubsection{The case $m \geq 2$}

Here, we treat the simpler case $m \geq 2$. In this case, $\msf{W}(n)_\phi$ is (non-canonically) isomorphic to a Heisenberg Deligne-Lusztig variety, so we may apply the computation in \cite{Tak26_Yu}. 

\begin{prop} \label{prop:cohomology_reduction_odd_partial}
    For $n = 2m - 1$ with $m \geq 2$, the natural map 
    \[
        R\Gamma_c(\msf{W}(n)_\phi, \Qla)[\psi^{-1}] \to R\Gamma(\msf{W}(n)_\phi, \Qla)[\psi^{-1}]
    \]
    is an isomorphism and isomorphic to $\kappa_{b, n, \phi}[-d]$ as representations of $G_b(F)_{x_b, n, m_b} \rtimes M(F)_{x, 0+}$. 
\end{prop}
\begin{proof}
    Recall the construction in \Cref{defi:reduction_of_specaff_odd}
    \begin{center}
        \begin{tikzcd} 
            \msf{W}(n)_\phi \ar[r] \ar[d] & \bb{A}^1 \ar[d, "t \mapsto t^q - t"] \\
            \msf{U}(n) \ar[r, "f_n"] & \bb{A}^1. 
        \end{tikzcd}
    \end{center}
    with a polynomial $f_{n} = \sum_{\alpha \in \Phi_{\mu<0}} c_{\alpha,n}X([du_\alpha, du_{-\alpha}]_{\mfr{g}}) \cdot u_{\alpha,n} u_{\beta_\alpha,n}^{q^{n_\alpha}}$. Here, $c_{\alpha,n}X([du_\alpha, du_{-\alpha}]_{\mfr{g}}) \in \ov{k}$ is nonzero since $\phi$ is $(G, M)$-generic. Now, the map $\alpha \mapsto \beta_\alpha$ is a permutation in $\Phi_{\mu < 0}$ and we may apply \cite[Theorem 4.1]{Tak26_Yu} to each cycle of this permutation to get that 
    \[
        R\Gamma_c(\msf{W}(n)_\phi, \Qla)[\psi^{-1}] \to R\Gamma(\msf{W}(n)_\phi, \Qla)[\psi^{-1}]
    \]
    is an isomorphism concentrated in degree $d$ and the $d$-th cohomology is of dimension 
    \[
        \prod_{\alpha \in \Phi_{\mu < 0}} q^{n_\alpha} = \sqrt{\lvert \msf{m}^\perp \rvert}. 
    \]
    It remains to identify the action of $G_b(F)_{x_b, n, m_b} \rtimes M(F)_{x, 0+}$ on $H_c^d(\msf{W}(n)_\phi, \Qla)[\psi^{-1}]$. First, the action of $M(F)_{x, 0+}$ is trivial by \Cref{prop:positive_depth_diagonal_action_odd}. Moreover, the action of $G_b(F)_{x_b, 2m, m}$ is also trivial by \Cref{prop:trivial_inner_action_odd}. Then, $G_b(F)_{x_b, n, m}$ acts on $H_c^d(\msf{W}(n)_\phi, \Qla)[\psi^{-1}]$ via a character 
    \[
        G_b(F)_{x_b, n, m} \to G_b(F)_{x_b, n, m} / G_b(F)_{x_b, 2m, m} \cong \msf{m}_{x, n} \xrightarrow{\psi \circ X} \Qlax
    \]
    since the diagonal action of $M(F)_{x, n}$ is trivial and $M(F)_{x, n} \subset G(F)_{x, n, m}$ acts by $(\psi \circ X)^{-1}$. Thus, $H_c^d(\msf{W}(n)_\phi, \Qla)[\psi^{-1}] \cong \kappa_{b, n, \phi}$ as a representation of $G_b(F)_{x_b, n, m_b}$ by \Cref{prop:Heisenberg_odd} since both sides have the same dimension $\sqrt{\lvert \msf{m}^\perp \rvert}$. Since $M(F)_{x, 0+}$ acts trivially on $\kappa_{x_b, n, \phi}$ by \cite[Theorem 2.27]{Tak26_Yu}, we get the desired claim. 
\end{proof}

Since $\kappa_{b, n, \phi}$ is irreducible as a representation of $G_b(F)_{x_b, n, m_b}$, \Cref{prop:cohomology_reduction_odd_partial} implies that $H_c^d(\msf{W}(n)_\phi, \Qla)[\psi^{-1}]$ is a twist of $\kappa_{b, n, \phi}$ by a character of $\msf{M} = M(F)_{x, 0} / M(F)_{x, 0+}$. Thus, it is enough to identify $H_c^d(\msf{W}(n)_\phi, \Qla)[\psi^{-1}]$ with $\kappa_{b, n, \phi}$ as an $\msf{M}$-representation. 

\begin{lem} \label{lem:normalization_t_factor}
    We may choose $t_\alpha \in \ov{k}^\times$ for each $\alpha \in \Phi_{\mu < 0}$ so that $\ov{c}_{\alpha, n} = t_\alpha t_{\beta_\alpha}^{q^{n_\alpha}}$ for every $\alpha \in \Phi_{\mu < 0}$ and $t_\alpha = t_\beta$ if $\alpha\vert_{Z_M^\circ} = \beta\vert_{Z_M^\circ}$. Here, $\ov{c}_{\alpha, n} \in \ov{k}$ is the reduction of $c_{\alpha, n}$. 
\end{lem}
\begin{proof}
    Since $t_\infty \in Z_{\cl{M}}^\circ$ and $\lambda \in X^*(Z_M^\circ)_{\bb{Q}}$, the constants $s_{\alpha, m - 1}$, $s_{\alpha, m}$, $r_\alpha$ and $n_\alpha$ depend only on $\alpha\vert_{Z_M^\circ}$. In particular, $c_{\alpha, n}$ depends only on $\alpha\vert_{Z_M^\circ}$ by \Cref{defi:constant_c_alpha_odd}. Now, the map $\alpha\vert_{Z_M^\circ} \mapsto \beta_\alpha \vert_{Z_M^\circ}$ is decomposed into permutations of characters of $Z_M^\circ$. For each cycle $\alpha_1, \alpha_2, \ldots, \alpha_\ell$, it is enough to take $t_1,\ldots,t_\ell \in \ov{k}$ so that 
    \[
        t_1 t_2^{q^{n_{\alpha_1}}} = \ov{c}_{\alpha_1, n}, \; t_2 t_3^{q^{n_{\alpha_2}}} = \ov{c}_{\alpha_2, n},\; \ldots,\; t_\ell t_1^{q^{n_{\alpha_\ell}}} = \ov{c}_{\alpha_\ell, n}. 
    \]
    The existence of such elements is immediate since $\ov{c}_{\alpha, n} \neq 0$ for every $\alpha$. 
\end{proof}

We take $t_\alpha \in \ov{k}^\times$ for each $\alpha \in \Phi_{\mu < 0}$ as in \Cref{lem:normalization_t_factor} and set $v_{\alpha} = t_\alpha u_{\alpha, n}$. Then, 
\[
    \msf{U}(n) = \Spec(\ov{k}[v_\alpha \vert \alpha \in \Phi_{\mu < 0}]) \to \msf{u}_{\mu < 0} ,\quad (v_\alpha) \mapsto \sum_{\alpha \in \Phi_{\mu < 0}} v_\alpha du_\alpha
\]
is an $\msf{M}$-equivariant isomorphism by \Cref{prop:equivariance_reduction_odd} and the requirement $t_\alpha = t_\beta$ if $\alpha\vert_{Z_M^\circ} = \beta\vert_{Z_M^\circ}$. In these coordinates, we have 
\[
    f_n = \sum_{\alpha \in \Phi_{\mu < 0}} X([du_\alpha, du_{-\alpha}]_{\msf{g}}) v_\alpha v_{\beta_\alpha}^{q^{n_\alpha}}. 
\]
Let $\cl{L}_\psi$ be the Artin-Schreier local system on $\bb{A}^1$ associated to an additive character $\psi$ (see \cite[Section 4.1]{Tak26_Yu}). Then, we would like to identify the trace of the $\msf{M}$-representation 
\[
    H_c^d(\msf{W}(n), \Qla)[\psi^{-1}] \cong H_c^d(\msf{U}(n), f_n^*\cl{L}_{\psi}). 
\]
We first try to understand the $\msf{M}$-action on the right-hand side. In fact, we will interpret the right-hand side as in the form of \cite[Proposition 4.15]{Tak26_Yu}. Let $\msf{m}^\perp \subset \Lie(\msf{G})$ be the orthogonal complement of $\msf{m} = \Lie(\msf{M})$ with respect to the $Z_{\msf{M}}^\circ$-action and we equip $\msf{m}^\perp$ with an $\msf{M}$-equivariant symplectic structure 
\[
    \langle -, - \rangle_X = X \circ \pi_{\msf{m}} \circ [-, -]_{\msf{g}} \colon \msf{m}^\perp \times \msf{m}^\perp \to \msf{g} \xrightarrow{\pi_{\msf{m}}} \msf{m} \xrightarrow{X} k. 
\]
Let $\Phi^{\msf{M}} \subset \Phi$ be the set of roots outside $\msf{M}$ and let $\msf{\ov{m}}^\perp$ be the vectorial group scheme over $\ov{k}$ associated to $\msf{m}^\perp$.

\begin{lem} \label{lem:M_action_on_W(n)_odd}
    Let $L_{\msf{m}^\perp} \colon \msf{\ov{m}}^\perp \to \msf{\ov{m}}^\perp$ be the Lang torsor $x \mapsto \sigma(x) - x$ and let $\pi_{\mu < 0}\colon \msf{\ov{m}}^\perp \to \msf{u}_{\mu < 0}$ be the weight projection. Then, we have the following. 
    \begin{enumerate}
        \item The projection $\pi_{\mu < 0}$ restricts to an isomorphism $L_{\msf{m}^\perp}^{-1}(\msf{u}_{\mu < 0}) \xrightarrow{\sim} \msf{u}_{\mu < 0}$. 
        \item Let $i_{\mu < 0} \colon \msf{u}_{\mu < 0} \to L_{\msf{m}^\perp}^{-1}(\msf{u}_{\mu < 0})$ be the inverse of $\pi_{\mu < 0}$. Then, $f_n$ is identified with 
        \[
            \msf{u}_{\mu < 0} \to \bb{A}^1 ,\quad 
            u \mapsto \langle u, \sigma(i_{\mu < 0}(u)) \rangle_X. 
        \]
        In particular, $f_n$ is invariant under $\msf{M}$. 
        \item The Cartesian diagram
        \begin{center}
            \begin{tikzcd} 
                \msf{W}(n)_\phi \ar[r] \ar[d] & \bb{A}^1 \ar[d, "t \mapsto t^q - t"] \\
                \msf{U}(n) \ar[r, "f_n"] & \bb{A}^1. 
            \end{tikzcd}
        \end{center}
        is equivariant under $\msf{M}$. Here, $\msf{M}$ acts trivially on $\bb{A}^1$. 
    \end{enumerate}
\end{lem}
\begin{proof}
    By \cite[Lemma 3.2]{Tak25z}, each $\sigma$-orbit of $\Phi^{\msf{M}}$ intersects with $\Phi_{\mu < 0}$. Let $v_\beta$ be the function on $\msf{\ov{m}}^\perp$ representing the coefficient of $du_\beta$. Then, $L_{\msf{m}^\perp}^{-1}(\msf{u}_{\mu < 0}) \subset \msf{\ov{m}}^\perp$ is defined by 
    \[
        v_{\beta} = v_{\sigma^{-1}\beta}^q, \quad \beta \in \Phi^{\msf{M}} - \Phi_{\mu < 0}. 
    \]
    Thus, we get (1). For (2), we have 
    \[
        v_{-\alpha}(\sigma(i_{\mu < 0}(u))) = v_{\beta_\alpha}(u)^{q^{n_\alpha}}
    \]
    by \Cref{lem:vareps}. Then, it is easy to see $f_n(u) = \langle u, \sigma(i_{\mu < 0}(u)) \rangle_X$. Since $\pi_{\mu < 0}$ and $\langle -, - \rangle_X$ are equivariant under the adjoint action of $\msf{M}$, it follows that $f_n$ is invariant under $\msf{M}$. 

    Finally, we prove (3). Now, $\msf{W}(n)_\phi$ admits two $\msf{M}$-actions: one is the reduction of the action on the generic fiber $\cl{W}(n)_\phi$ (denoted by $i_1$) and the other is defined by the Cartesian diagram (denoted by $i_2$). By construction, both actions commute with the $k$-action on $\msf{W}(n)_\phi$. Now, 
    \[
        i_1(h) \circ i_2(h^{-1}) \colon \msf{W}(n)_\phi \to \msf{W}(n)_\phi
    \]
    is an automorphism over $\msf{U}(n)$ commuting with the $k$-action for $h \in \msf{M}$. As in the proof of \Cref{prop:equivariance_under_HW_even}, it is enough (for deducing $i_1 = i_2$) to show that $i_1(h) \circ i_2(h)^{-1}$ has a fixed point. By the proof of \Cref{prop:stability_under_diagonal_action}, the reduction $\msf{x}_\phi \in \msf{W}(n)_\phi$ of the CM point $x_{\CM, K_\phi}$ is fixed by $i_1(h)$. Now, the $K_\phi$-level structure of $x_{\CM, K_\phi}$ associated via \Cref{lem:levelstrCp} and \Cref{lem:levelVn_odd} is given by $x = \id$. Thus, the coordinate of $\msf{x}_\phi$ via $\msf{W}(n)_\phi \subset \msf{U}(n) \times \bb{A}^1$ is given by $u_\alpha = 0$ and $t = 0$, so $\msf{x}_\phi$ is fixed by $i_2(h)$. Thus, we get the claim. 
\end{proof}

Now, the computation in \cite[Section 5]{Tak26_Yu} can be applied to our situation due to the description in \Cref{lem:M_action_on_W(n)_odd} (2). 

\begin{defi}
    Let $\Phi_0^{\msf{M}} \subset X^*(Z_{\msf{M}}^\circ)$ be the set of $Z_{\msf{M}}^\circ$-weights in $\msf{\ov{m}}^\perp$ and let $\msf{\ov{m}}^{\perp, \lambda} \subset \msf{\ov{m}}^\perp$ be the weight space for each $\lambda \in \Phi_0^{\msf{M}}$. Let $\Phi_{\sigma}^{\msf{M}}$ be the set of $\sigma$-orbits in $\Phi_0^{\msf{M}}$ and for each $C \in \Phi_{\sigma}^{\msf{M}}$, let $\msf{m}^{\perp, C} \subset \msf{m}^\perp$ be the $k$-linear subspace such that $\ov{\msf{m}}^{\perp, C} = \msf{m}^{\perp, C} \otimes \ov{k}$ is given by
    \[
        \ov{\msf{m}}^{\perp, C} = \bigoplus_{\lambda \in C} \msf{\ov{m}}^{\perp, \lambda}. 
    \]
    Then, $\msf{m}^{\perp, C}$ admits a $k_C$-linear structure for a finite extension $k_C / k$ of degree $\lvert C \rvert$. Let $\Sigma = \langle \sigma \rangle \times \{ \pm 1 \}$ and let $\Phi_{\Sigma}^{\msf{M}}$ be the set of $\Sigma$-orbits in $\Phi_0^{\msf{M}}$. We take the decomposition 
    \[
        \Phi_\Sigma^{\msf{M}} = \Phi_{\Sigma, \asym}^{\msf{M}} \sqcup \Phi_{\Sigma, \sym}^{\msf{M}}
    \]
    so that $C \in \Phi_{\Sigma, \sym}^{\msf{M}}$ if and only if $C \in \Phi_\Sigma^{\msf{M}}$ consists of a single $\sigma$-orbit. 
\end{defi}

\begin{prop} \label{prop:trace_computation_cohomology}
    For each $\gamma \in \msf{M}$, let 
    \[
        d_\gamma = \sum_{C \in \Phi_{\Sigma, \sym}^{\msf{M}}} (\dim_{k_C} \msf{m}^{\perp, C} - \dim_{k_C} \msf{m}^{\perp, C, \gamma}). 
    \]
    Here, the superscript $\gamma$ denotes the fixed subspace under the action of $\gamma$. Then 
    \[
        \tr(\gamma \vert H_c^d(\msf{u}_{\mu < 0}, f_n^*\cl{L}_{\psi})) = (-1)^{d_\gamma} \sqrt{\lvert \msf{m}^{\perp, \gamma} \rvert}. 
    \]
\end{prop}
\begin{proof}
    We apply \cite[Proposition 5.3]{Tak26_Yu} to $r = 2m$: the notation loc. cit. is translated into the notation of \Cref{lem:M_action_on_W(n)_odd} as $\pi_{I_1} = \pi_{\mu < 0}$, $\wtd{\msf{u}}_{I, r/2} = L_{\msf{m}^\perp}^{-1}(\msf{u}_{\mu < 0})$ and 
    \[
        f(v) = \langle \sigma(v), \pi_{I_1}(v) \rangle_\phi = - \langle u, \sigma(i_{\mu < 0}(u)) \rangle_X = - f_n(u)
    \]
    for $v = i_{\mu < 0}(u)$. Thus, we have an $\msf{M}$-equivariant isomorphism 
    \[
        H_c^d(\msf{u}_{\mu < 0}, f_n^*\cl{L}_{\psi}) \cong H_c^d(\msf{u}_{\mu < 0}, f^*\cl{L}_{\psi^{-1}}) \cong H_c^d(\msf{W}_{\Phi_{\mu < 0}, 2m}, \Qla)[\psi]. 
    \]
    Then, the claim follows from \cite[Theorem 2.27, Theorem 5.14]{Tak26_Yu}. 
\end{proof}
\begin{rmk} \label{rmk:comparison_Heisenberg_Deligne_Lusztig}
    The proof also shows that $\msf{W}(n)_\phi^\perf$ is \textit{non-canonically} (i.e.\ after making the choices in \Cref{lem:normalization_t_factor}) isomorphic to $\msf{W}_{\Phi_{\mu < 0}, 2m}$. 
\end{rmk}

Now, it remains to compare this formula with the trace of $\kappa_{b, n, \phi}$. 

\begin{prop} \label{prop:cohomology_reduction_odd}
    When $m \geq 2$, \Cref{prop:cohomology_reduction} holds for $n = 2m - 1$. 
\end{prop}
\begin{proof}
    By \Cref{prop:cohomology_reduction_odd_partial} and the subsequent comments, we have an isomorphism
    \[
        H_c^d(\msf{u}_{\mu < 0}, f_n^*\cl{L}_{\psi}) \cong H_c^d(\msf{W}(n)_\phi, \Qla)[\psi^{-1}] \cong \kappa_{b, n, \phi} \otimes \chi 
    \]
    equivariant under $G_b(F)_{x_b, n, m_b} \rtimes M(F)_{x, 0}$ for some character $\chi \colon \msf{M} \to \Qlax$. It is enough to show that $\chi$ is trivial. Let $V = H_c^d(\msf{u}_{\mu < 0}, f_n^*\cl{L}_{\psi})^{G_b(F)_{x_b, n+, n/2+}}$. By \Cref{prop:Heisenberg_odd}, $V \cong \kappa_{x_b, n, \phi}$ as a $G_b(F)_{x_b, n, n/2}$-representation. Since $V$ and $\kappa_{x_b, n, \phi}$ are stable under $\msf{M}$, it is enough to show 
    \begin{equation} \label{eq:trace_equality_M_odd}
        \tr(\gamma\vert V) = \tr(\gamma \vert \kappa_{x_b, n, \phi}) \quad (\gamma \in \msf{M}). 
    \end{equation}

\begin{lem} \label{lem:trace_for_kappa_odd}
    For each $\gamma \in \msf{M}$, $\tr(\gamma \vert \kappa_{x_b, n, \phi}) = (-1)^{d_\gamma} \cdot \sqrt{\lvert \msf{m}^{\perp, \gamma}_{x_b, n/2} \rvert}$. 
\end{lem}
\begin{proof}
    By \cite[Theorem 2.27]{Tak26_Yu}, it is enough to show 
    \[
        d_\gamma = \sum_{C \in \Phi_{\Sigma, \sym}^{\msf{M}}} (\dim_{k_C} \msf{m}_{x_b, n/2}^{\perp, C} - \dim_{k_C} \msf{m}_{x_b, n/2}^{\perp, C, \gamma}). 
    \]
    Here, $\msf{m}^\perp_{x_b, n/2}$ denotes the graded piece of the Moy-Prasad filtration on the orthogonal complement $\mfr{m}_b^\perp \subset \Lie(\cl{G}_b)$ of $\mfr{m} = \Lie(\cl{M})$. In fact, there is an identification 
    \[
        \msf{m}_{x_b, n/2}^\perp \subset \msf{m}^\perp
    \]
    as a subspace consisting of weight spaces for $\alpha \in \Phi^{\msf{M}}$ with $\alpha(x_b) \in \{ \pm \tfrac{1}{2} \}$ (normalized so that $\alpha(x) = 0$). For each $C \in \Phi_\sigma^{\msf{M}}$, the fractional part $\{ \alpha(x_b) \}$ is constant over $\alpha \in C$, for $x_b$ is fixed by $\mu(-\pi)\sigma$. Since $\{ (-\alpha)(x_b) \} = 1 - \{ \alpha(x_b) \}$, we have $\msf{m}^{\perp, C} \subset \msf{m}_{x_b, n/2}^\perp$ for every $C \in \Phi_{\Sigma, \sym}^{\msf{M}}$. In particular, $\msf{m}^{\perp, C}_{x_b, n/2} = \msf{m}^{\perp, C}$ and we get the claim. 
\end{proof}
    
    Next, we study $\tr(\gamma \vert V)$ for $\gamma \in \msf{M}$. Since $H_c^d(\msf{u}_{\mu < 0}, f_n^*\cl{L}_{\psi}) \cong \kappa_{b, n, \phi}$, we have 
    \[
        H_c^d(\msf{u}_{\mu < 0}, f_n^*\cl{L}_{\psi}) = \bigoplus_{j \in G_b(F)_{x_b, n, m_b} / G_b(F)_{x_b, n, n/2}} j \cdot V. 
    \]
    Then, $\tr(\gamma\vert H_c^d(\msf{u}_{\mu < 0}, f_n^*\cl{L}_{\psi}))$ is written as a sum over the $\gamma$-fixed locus 
    \[
        J_\gamma = \{ j \in G_b(F)_{x_b, n, m_b} / G_b(F)_{x_b, n, n/2} \mid j^{-1} \Ad(\gamma)(j) \in G_b(F)_{x_b, n, n/2} \}
    \]
    so that $\tr(\gamma\vert H_c^d(\msf{u}_{\mu < 0}, f_n^*\cl{L}_{\psi})) = \sum_{j \in J_\gamma} \tr(j^{-1} \Ad(\gamma)(j) \cdot \gamma \mid V)$. Now, via the left exact sequence
    \[
        0 \to G_b(F)_{x_b, n, r+} \to G_b(F)_{x_b, n, r} \to G_b(F)_{x_b, n, r} / G_b(F)_{x_b, n, r+} \cong \msf{m}_{x_b, r}^\perp, 
    \]  
    we get $\lvert J_\gamma \rvert \leq \prod_{n/2 < r < m} \lvert \msf{m}_{x_b, r}^{\perp, \gamma} \rvert = \lvert \msf{m}^{\perp, \gamma}_{x_b, [m_b, n/2)} \rvert$. Moreover, by \cite[Theorem 2.27]{Tak26_Yu}, we get 
    \[
        \lvert \tr(j^{-1} \Ad(\gamma)(j) \cdot \gamma \mid V) \rvert \leq \sqrt{\lvert \msf{m}^{\perp, \gamma}_{x_b, n/2} \rvert}. 
    \]
    By \Cref{prop:trace_computation_cohomology}, we have 
    \[
        \sqrt{\lvert \msf{m}^{\perp, \gamma} \rvert} = \lvert \tr(\gamma\vert H_c^d(\msf{u}_{\mu < 0}, f_n^*\cl{L}_{\psi})) \rvert \leq \lvert J_\gamma \rvert \cdot \sqrt{\lvert \msf{m}^{\perp, \gamma}_{x_b, n/2} \rvert} \leq \sqrt{\lvert \msf{m}^{\perp, \gamma} \rvert}. 
    \]
    Since both sides are equal, all inequalities should be equalities. In particular, since the sign of $\tr(\gamma\vert H_c^d(\msf{u}_{\mu < 0}, f_n^*\cl{L}_{\psi}))$ is $(-1)^{d_\gamma}$, we have $\lvert J_\gamma \rvert = \lvert \msf{m}^\perp_{x_b, [m_b, n/2)} \rvert$ and 
    \[
        \tr(j^{-1} \Ad(\gamma)(j) \cdot \gamma \mid V) = (-1)^{d_\gamma} \cdot \sqrt{\lvert \msf{m}^{\perp, \gamma}_{x_b, n/2} \rvert}. 
    \]
    By setting $j = 1$, we get the desired equality \eqref{eq:trace_equality_M_odd} from \Cref{lem:trace_for_kappa_odd}. 
\end{proof}

\subsubsection{The case $m = 1$}

Here, we treat the case $m = 1$. Though the polynomial $f_1$ looks different from the polynomials $f_n$ for $m \geq 2$, we may compute the cohomology similarly using the following induction principles. To simplify the exposition, we say that a polynomial $f \in \ov{k}[x_1, \ldots, x_N]$ satisfies the property (M) if the natural map 
\[
    H_c^\bullet(\bb{A}^N, f^* \cl{L}_\psi) \to H^\bullet(\bb{A}^N, f^* \cl{L}_\psi)
\]
is an isomorphism for every nontrivial additive character $\psi \colon k \to \Qlax$. In the following reduction procedure, we ignore the difference of Tate twists. 

\begin{lem} \label{lem:reduction_N_>=3}
    Let $f_N \in \ov{k}[x_1, \ldots, x_N]$ be a polynomial with $N \geq 3$ variables whose terms containing $x_N$ are given by
    \[
        a_{N - 1} x_{N - 1} x_N^{q^{d_{N - 1}}} + a_N x_N x_1^{q^{d_N}} \quad (a_{N-1}, a_N \in \ov{k}^\times,\; d_{N-1}, d_N \in \bb{Z}_{\geq 1}). 
    \]
    Let $\pr \colon \bb{A}^N \to \bb{A}^{N - 1}$ be the projection to the first $(N - 1)$-coordinates. 
    \begin{enumerate}
        \item $\pr_{!} f_N^* \cl{L}_\psi$ is concentrated on the closed locus $x_{N - 1} = - a_{N - 1}^{-1} a_N^{q^{d_{N-1}}} x_1^{q^{d_{N - 1} + d_N}}$. 
        \item Let $f_{N - 2} \in \ov{k}[x_1, \ldots, x_{N - 2}]$ be the restriction of $f_N$ by setting
        \[
            x_{N - 1} = - a_{N - 1}^{-1} a_N^{q^{d_{N - 1}}} x_1^{q^{d_{N - 1} + d_N}}, \quad x_N = 0. 
        \]
        Then, $f_N$ satisfies (M) if and only if $f_{N - 2}$ satisfies (M). Moreover, 
        \[
            H_c^\bullet(\bb{A}^N, f_N^* \cl{L}_\psi) \cong H_c^{\bullet- 2}(\bb{A}^{N-2}, f_{N-2}^* \cl{L}_\psi). 
        \]
    \end{enumerate}
\end{lem}
\begin{proof}   
    This is an extraction of the inductive argument in \cite[Section 4.3.1]{Tak26_Yu}. 
\end{proof}

\begin{lem} \label{lem:reduction_N_=2}
    Let $f_N \in \ov{k}[x_1, \ldots, x_N]$ be a polynomial with $N \geq 2$ variables whose terms containing $x_N$ are given by
    \[
        a_{N - 1} x_{N - 1} x_N^{q^{d_{N - 1}}} + a_N x_N x_{N - 1}^{q^{d_N}} \quad (a_{N-1}, a_N \in \ov{k}^\times,\; d_{N-1}, d_N \in \bb{Z}_{\geq 1}). 
    \]
    Let $\pr \colon \bb{A}^N \to \bb{A}^{N - 1}$ be the projection to the first $(N - 1)$-coordinates. 
    \begin{enumerate}
        \item Let $A_{N - 1} \subset \ov{k}$ be the set of solutions $a \in \ov{k}$ such that 
        \[
            a_{N - 1}^{-1} a_N^{q^{d_{N-1}}} a^{q^{d_{N - 1} + d_N}} + a = 0. 
        \]
        Then $\pr_{!} f_N^* \cl{L}_\psi$ is concentrated on the closed locus $x_{N - 1} \in A_{N - 1}$. 
        \item For each $a \in A_{N - 1}$, let $f_{N - 2, a} \in \ov{k}[x_1, \ldots, x_{N - 2}]$ be the restriction of $f_N$ by setting $x_{N - 1} = a$ and $x_N = 0$. Then, $f_N$ satisfies (M) if and only if $f_{N - 2, a}$ satisfies (M) for every $a \in A_{N - 1}$. Moreover, 
        \[
            H_c^\bullet(\bb{A}^N, f_N^* \cl{L}_\psi) \cong \bigoplus_{a \in A_{N - 1}} H_c^{\bullet- 2}(\bb{A}^{N-2}, f_{N-2, a}^* \cl{L}_\psi). 
        \]
    \end{enumerate}
\end{lem}
\begin{proof}
    This is an extraction of the inductive argument in \cite[Section 4.3.2]{Tak26_Yu}. 
\end{proof}

Now, we will apply this reduction procedure to the given polynomial $f_1$. For this, we need to equip the index set $\Phi_{\mu < 0}$ with a suitable order. We will use the value $\alpha(x_b)$ for $\alpha \in \Phi$. Recall that the linear function $\alpha$ on $\cl{A}(\breve{G}, \breve{T})$ is normalized so that $\alpha(x) = 0$ and we have $0 < \alpha(x_b) < 1$ for $\alpha \in \Phi_{\msf{\ov{N}}}$. 

\begin{defi}
    For each subset $I \subset [0, 1]$, let 
    \[
        \Phi_{\mu < 0}^{I} = \{ \alpha \in \Phi_{\mu < 0} \mid \alpha(x_b) \in I \}
    \]
    and let 
    $
    \bb{A}_{I} = \Spec(\ov{k}[u_{\alpha, 1} \vert \alpha \in \Phi_{\mu < 0}^{I}]) \subset \msf{U}(1)
    $
    be the closed subvariety defined by $u_{\alpha, 1} = 0$ for $\alpha \notin \Phi_{\mu < 0}^{I}$. When $I = \{r\}$, the sub/superscript $I$ is simply denoted by $r$. 
\end{defi}

Now, the main property of $f_1$ enabling our induction procedure is as follows. 

\begin{prop} \label{prop:highest_term_f_r}
    For each $\tfrac{1}{2} < r < 1$, let $f_r$ be the restriction of $f_1$ to $\bb{A}_{[1-r, r]}$. For each $\alpha \in \Phi^{r}_{\mu < 0}$, the terms of $f_{r}$ containing $u_{\alpha, 1}$ are given by 
    \[
        a_1 u_{\alpha, 1} u_{\beta_\alpha, 1}^{q^{n_\alpha}} + a_2 u_{\gamma_\alpha, 1} u_{\alpha, 1}^{q^{n_{\gamma_\alpha}}}
    \]
    for some $a_1, a_2 \in \ov{k}^\times$. 
\end{prop}
\begin{proof}
    Recall the definition of $f_1$ as in \Cref{defi:constant_c_alpha_m=1} and we will use the same notation as there. Since $x_b$ is fixed by $\mu(-\pi)\sigma$, we have 
    \[
        (\sigma^i \alpha)(x_b) = \left\{ \begin{alignedat}{4}
            & \alpha(x_b) & \quad & (i < n_{\gamma_\alpha}) \\
            & \alpha(x_b) - 1 & \quad & (i = n_{\gamma_\alpha})
        \end{alignedat}
        \right.
    \]
    for each $\alpha \in \Phi_{\mu < 0}$. Then, it can be easily verified that if we take the decomposition 
    \[
        g_0^- = u \ov{n}, \quad u \in \msf{U}_{\mu > 0} = \msf{N} \cap \sigma(\ov{\msf{N}}) ,\quad \ov{n} \in \msf{\ov{N}} \cap \sigma(\msf{\ov{N}})
    \]
    over the closed locus $\bb{A}_{[1-r, r]} \subset \msf{U}(1)$, then the coordinate $u_{-\alpha}$ of $u$ for each $- \alpha \in \Phi_{\mu > 0}$ is 
    \[
        u_{-\alpha} = \left\{ \begin{alignedat}{4}
            & 0 & \quad & (\alpha(x_b) > r) \\
            & - u_{\beta_{\alpha}, 1}^{q^{n_{\alpha}}} & \quad & (\alpha(x_b) = r). 
        \end{alignedat}
        \right.
    \]  
    When $\alpha(x_b) < r$, $u_{-\alpha}$ is complicated. 
    Nevertheless, it is easy to see that the contribution of $u_{\alpha, 1}$ for $\alpha \in \Phi_{\mu < 0}^r$ is concentrated on $u_{-\beta}$ with $\beta \in \Phi_{\mu < 0}^{< 1 - r} \sqcup \{ \gamma_\alpha \}$ and the contribution to $u_{-\gamma_\alpha}$ is $- u_{\alpha, 1}^{q^{n_{\gamma_\alpha}}}$. On the other hand, the coefficient $v_\alpha$ of $\Ad(\ov{n})(\sum c_{\alpha, 1} u_{\alpha, 1} du_\alpha) \in \msf{u}_{\mu < 0}$ for each $du_\alpha$ ($\alpha \in \Phi_{\mu < 0}$) satisfies 
    \[
        v_\alpha = \left\{ \begin{alignedat}{4}
            & 0 & \quad & (\alpha(x_b) < 1 - r) \\
            & c_{\alpha, 1}u_{\alpha, 1} & \quad & (\alpha(x_b) = 1 - r). 
        \end{alignedat}
        \right.
    \]  
    When $\alpha(x_b) > 1 - r$, $v_\alpha$ is complicated. Nevertheless, the contribution of $u_{\alpha, 1}$ for $\alpha \in \Phi_{\mu < 0}^r$ is concentrated on $v_{\beta}$ with $\beta \in \Phi_{\mu < 0}^{> r} \sqcup \{ \alpha \}$ and the contribution to $v_{\alpha}$ is $c_{\alpha, 1}u_{\alpha, 1}$.

    As a result, the contribution of $u_{\alpha, 1}$ ($\alpha \in \Phi^r_{\mu < 0}$) to $(X \circ \pi_{\msf{m}})(\Ad(u)(\sum v_\alpha du_\alpha))$ is given by 
    \[
        (X\circ \pi_{\msf{m}})\bigl(\Ad(i_{-\alpha}(- u_{\beta_{\alpha}, 1}^{q^{n_{\alpha}}}))(c_{\alpha, 1} u_{\alpha, 1} du_\alpha) + \Ad(i_{- \gamma_\alpha}(- u_\alpha^{q^{n_{\gamma_\alpha}}}))(c_{\gamma_\alpha, 1} u_{\gamma_\alpha, 1} du_{\gamma_\alpha})\bigr). 
    \]
    By the genericity of $X$, it is easy to see that this has the form of the claim. 
\end{proof}

Now, we may apply \Cref{lem:reduction_N_>=3} and \Cref{lem:reduction_N_=2} to $f_r$ and $u_{\alpha, 1}$ ($\alpha \in \Phi^r_{\mu < 0}$). However, \Cref{lem:reduction_N_=2} does not directly reduce the case $f_r$ to $f_{r -}$ since $A_{N - 1}$ could contain non-zero elements. To reduce to the case $a = 0$, we use the invariance of $f_1$ under the inner action. 

Recall from \Cref{prop:inner_action_U(n)_reduction} that the inner action of $G_b(F)_{x_b, 1}$ on $\msf{U}(1)$ is trivial. We need the whole inner action. 

\begin{lem} \label{lem:reduction_inner_action}
    For each $j \in G_b(F)_{x_b, 0+}$, let $\ov{n} \in \msf{\ov{N}}(\ov{k})$ be the image of $j$ under 
    \[
        G_b(F)_{x_b, 0} \to \cl{G}(O_{\breve{F}}) \to \msf{G}(\ov{k}). 
    \]
    Let $\ov{n}_{\mu < 0} \in \msf{U}_{\mu < 0}$ be the projection of $\ov{n}$ under $\msf{\ov{N}} \cong \msf{U}_{\mu < 0} \times (\msf{\ov{N}} \cap \sigma(\msf{\ov{N}}))$. Via the natural identification $\msf{U}_{\mu < 0} \cong \msf{u}_{\mu < 0}$ and \eqref{eq:U(n)_isomorphism}, we have $j \cdot u = \Ad(\ov{n})(u) + \ov{n}_{\mu < 0}$ for $u \in \msf{u}_{\mu < 0}$. 
\end{lem}
\begin{proof}
    Let $x \in \cl{U}(1)(\bb{C}_p)$ be an arbitrary point and let $u_\alpha^\flat \in \pi^\flat \cdot O_{\bb{C}_p}^\flat$ be an element with $u_\alpha(x) = (u_\alpha^\flat)^\sharp$. We will compute the reduction of $j \cdot x \in \cl{U}(1)$. By \cite[Proposition 2.15]{Tak25z},  we can take $v_\alpha^\flat \in \pi^\flat \cdot O_{\bb{C}_p}^\flat$ for each $\alpha \in \Phi_{\mu < 0}$ so that 
    \begin{equation} \label{eq:inner_action_n=1}
        \iota j \mu([\pi^\flat] - \pi) \prod_{\alpha \in \Phi_{\mu < 0}} i_\alpha([u_\alpha^\flat]) \sigma(j)^{-1} \sigma(\iota)^{-1} = \mu([\pi^\flat] - \pi) \prod_{\alpha \in \Phi_{\mu < 0}} i_\alpha([v_\alpha^\flat])
    \end{equation}
    for some $\iota \in \cl{G}(W_{O_F}(O_{\bb{C}_p}^\flat))$ that is trivial modulo $[\pi^\flat]$. Let 
    \[
        u_{\alpha, 1}^\flat = \pi^{\flat, -1} u_\alpha^\flat, \quad v_{\alpha, 1}^\flat = \pi^{\flat, -1} v_\alpha^\flat \in O_{\bb{C}_p}^\flat. 
    \]
    Then $u_{\alpha, 1}(j\cdot x) = (v_{\alpha, 1}^\flat)^\sharp$, so $j \cdot \sum_{\alpha \in \Phi_{\mu < 0}} u_{\alpha, 1}^\flat du_\alpha \equiv \sum_{\alpha \in \Phi_{\mu < 0}} v_{\alpha, 1}^\flat du_\alpha$ modulo $\ov{k}$ in $\msf{u}_{\mu < 0}(\ov{k})$. 

    Let $\cl{G}(W_{O_F}(O_{\bb{C}_p}^\flat)[\tfrac{1}{\pi}])_i = \Ker(\cl{G}(W_{O_F}(O_{\bb{C}_p}^\flat)[\tfrac{1}{\pi}]) \to \cl{G}(W_{O_F}(O_{\bb{C}_p}^\flat)[\tfrac{1}{\pi}] / [\pi^{\flat, i}]))$ for $i = 1, 2$ and we will use the Moy-Prasad isomorphism 
    \begin{equation} \label{eq:MP_isom_n=1}
        \cl{G}(W_{O_F}(O_{\bb{C}_p}^\flat)[\tfrac{1}{\pi}])_1 / \cl{G}(W_{O_F}(O_{\bb{C}_p}^\flat)[\tfrac{1}{\pi}])_2 \cong [\pi^\flat] \cdot \mfr{g} \otimes W_{O_F}(O_{\bb{C}_p}^\flat)[\tfrac{1}{\pi}] / [\pi^\flat]
    \end{equation}
    to rewrite \eqref{eq:inner_action_n=1} modulo $[\pi^{\flat, 2}]$. Since $\sigma(\iota)$ is trivial modulo $[\pi^{\flat, q}]$ and $\Ad(b\sigma)(j) = j$, we get
    \[
        \bar{\iota} - \pi^{-1} \Ad(j)(d\mu + \sum_{\alpha \in \Phi_{\mu < 0}} [u_{\alpha, 1}^\flat] du_\alpha) = -\pi^{-1} d\mu - \pi^{-1} \sum_{\alpha \in \Phi_{\mu < 0}} [v_{\alpha, 1}^\flat] du_\alpha. 
    \] 
    Here, $d\mu \in \mfr{\breve{g}}$ is the differential of $\mu$ and $\bar{\iota}$ is the image of $\iota$ under \eqref{eq:MP_isom_n=1}. Since $\bar{\iota}$ lies in $\mfr{g} \otimes W_{O_F}(O_{\bb{C}_p}^\flat) / [\pi^\flat]$, we have 
    \[
        j \cdot \sum_{\alpha \in \Phi_{\mu < 0}} u_{\alpha, 1}^\flat du_\alpha \equiv \sum_{\alpha \in \Phi_{\mu < 0}} v_{\alpha, 1}^\flat du_\alpha = (\Ad(\ov{n}) - \id)(d\mu) + \Ad(\ov{n})(\sum_{\alpha \in \Phi_{\mu < 0}} [u_{\alpha, 1}^\flat] du_\alpha)
    \]
    in $\breve{\mfr{g}} \otimes O_{\bb{C}_p}^\flat / \pi^{\flat}$ since $j \equiv \ov{n}$ modulo $\pi$. As $(\ov{n}_{\mu < 0})^{-1} \ov{n}$ commutes with $\mu$, we have $\Ad(\ov{n})(d\mu) = \Ad(\ov{n}_{\mu < 0})(d \mu)$. Then, $(\Ad(\ov{n}) - \id)(d\mu) = [\ov{n}_{\mu < 0}, d\mu] = \ov{n}_{\mu < 0}$, so we get the claim. 
\end{proof}
\begin{cor} \label{cor:orbit_injectivity}
    For each $\tfrac{1}{2} < r < 1$, the inner action of $G_b(F)_{x_b, 1 - r}$ on $\msf{U}(1)$ stabilizes $\bb{A}_{\geq s}$ for every $s \leq 1 - r$. Moreover, the induced map
    \[
        \msf{m}^\perp_{x_b, 1 - r} \cong G_b(F)_{x_b, 1 - r} / G_b(F)_{x_b, (1 - r)+} \to \bb{A}_{1 - r}(\ov{k}) ,\quad j \mapsto j \cdot 0
    \]
    via $\bb{A}_{1 - r} = \bb{A}_{\geq 1 - r} / \bb{A}_{> 1 - r}$ is injective. 
\end{cor}
\begin{proof}
    The first claim is immediate from \Cref{lem:reduction_inner_action}. For the second claim, it is enough to show that the composition
    \[
        G_b(F)_{x_b, 1 - r} / G_b(F)_{x_b, (1 - r)+} \cong \msf{m}^\perp_{x_b, 1 - r} \to \msf{u}_{\mu < 0}(\ov{k})
    \]
    is injective. Here, the second map is a natural projection. Then, the injectivity follows from the fact that every $\sigma$-orbit of $\Phi^{\msf{M}}$ contains an element of $\Phi_{\mu < 0}$. 
\end{proof}

\begin{prop} \label{prop:induction_principle_n=1}
    Let $\tfrac{1}{2} < r < 1$ and $\pr \colon \bb{A}_{[1-r, r]} \to \bb{A}_{(1-r, r)} \times \bb{A}_{1-r}$ be the natural projection. 
    \begin{enumerate}
        \item The $\Qla$-local system $f_r^*\cl{L}_\psi$ is equivariant under $G_b(F)_{x_b, 1-r}$. Here, the inner action on $\bb{A}_{[1-r, r]}$ is induced by $\bb{A}_{[1-r, r]} = \bb{A}_{\geq 1 - r} / \bb{A}_{> r}$. 
        \item There is a subset $A_{1-r} \subset \bb{A}_{1-r}(\ov{k})$ of size $\lvert \msf{m}_{x_b, 1-r}^\perp \rvert$ such that $\pr_! f_r^* \cl{L}_\psi$ is supported on $\bb{A}_{(1-r, r)} \times A_{1-r}$ and $G_b(F)_{x_b, 1 - r} / G_b(F)_{x_b, (1-r)+}$ acts on $A_{1-r}$ transitively. 
        \item The polynomial $f_r$ satisfies (M) if and only if $f_{r-}$ satisfies (M). Moreover, 
        \[
            H_c^\bullet (\bb{A}_{[1-r, r]}, f_r^* \cl{L}_\psi) = \bigoplus_{j \in G_b(F)_{x_b, 1 - r} / G_b(F)_{x_b, (1-r)+}} j \cdot H_c^{\bullet - 2d_r} (\bb{A}_{(1-r, r)}, f_{r-}^* \cl{L}_\psi)
        \]
        equivariantly under $G_b(F)_{x_b, 1 - r}$ for $d_r = \dim \bb{A}_r$. 
    \end{enumerate}
\end{prop}
\begin{proof}
    First, (1) clearly holds for $r = 1$. We prove the claims by induction on $r$: we first deduce the claims (2) and (3) for $r$ from (1) and then deduce the claim (1) for $r-$. 
    
    Suppose that (1) holds for $r$. By \Cref{prop:highest_term_f_r}, for each cycle $C \subset \Phi_{\mu < 0}^r \sqcup \Phi_{\mu < 0}^{1 - r}$ in the permutation $\alpha \mapsto \beta_\alpha$, we may apply \Cref{lem:reduction_N_>=3} and \Cref{lem:reduction_N_=2} successively to $u_{\alpha, 1}$ with $\alpha \in C \cap \Phi^{r}_{\mu < 0}$. Then \Cref{lem:reduction_N_=2} (1) implies that $\pr_! f_r^* \cl{L}_\psi$ is supported on $\bb{A}_{(1-r, r)} \times A_{1-r}$ for some subset $A_{1 - r}$ containing $0$ whose size is at most 
    \[
        \prod_{\alpha \in \Phi_{\mu < 0}^{r} \sqcup \Phi_{\mu < 0}^{1 - r}} q^{n_\alpha} = \lvert \msf{m}_{x_b, 1 - r}^\perp \rvert. 
    \]
    The support of $\pr_! f_r^* \cl{L}_\psi$ is stable under the action of $G_b(F)_{x_b, 1 - r}$, so $\lvert A_{1-r} \rvert \leq  \lvert \msf{m}_{x_b, 1 - r}^\perp \rvert$ implies
    \[
        A_{1 - r} = G_b(F)_{x_b, 1-r} / G_b(F)_{x_b,(1-r)+} \cdot 0
    \]
    by \Cref{cor:orbit_injectivity}. Thus, we get the claim (2) for $r$. Moreover, \Cref{lem:reduction_N_>=3} and \Cref{lem:reduction_N_=2} imply that $f_r$ satisfies (M) if and only if 
    \[
        \pr_! f_r^* \cl{L}_\psi [2d_r] \vert_{\bb{A}_{(1-r, r)} \times a}
    \]
    satisfies (M) for every $a \in A_{1-r}$. Since $ \pr_! f_r^* \cl{L}_\psi [2d_r]$ is equivariant under $G_b(F)_{x_b, 1-r}$ and $\pr_! f_r^* \cl{L}_\psi\vert_{\bb{A}_{(1-r, r)}[2d_r] \times 0} \cong f_{r-}^* \cl{L}_\psi$, we get the claim (3). Then, we get the claim (1) for $r-$ since $\bb{A}_{(1-r, r)} \times \{0\}$ is stable under $G_b(F)_{x_b, (1-r)+}$. 
\end{proof}

Now, the computation of $H_c^\bullet(\msf{U}(1), f_1^* \cl{L}_\psi)$ is reduced to $H_c^\bullet(\bb{A}_{1/2}, f_{1/2}^* \cl{L}_\psi)$. Then, we can directly apply the computation of \cite{Tak26_Yu} by \Cref{prop:highest_term_f_r}. 

\begin{prop} \label{prop:cohomology_reduction_n=1_partial}
    The natural map 
    \[
        R\Gamma_c(\msf{W}(1)_\phi, \Qla)[\psi^{-1}] \to R\Gamma(\msf{W}(1)_\phi, \Qla)[\psi^{-1}]
    \]
    is an isomorphism and isomorphic to $\kappa_{b, 1, \phi}[-d]$ as representations of $G_b(F)_{x_b, 0+} \rtimes M(F)_{x, 0+}$. 
\end{prop}
\begin{proof}   
    For the first claim, it is enough to show that $f_{1/2}$ satisfies (M) by \Cref{prop:induction_principle_n=1}. Since $f_{1/2}$ is of the form 
    \[
        \sum_{\alpha \in \Phi_{\mu < 0}^{1/2}} a_{\alpha} u_{\alpha, 1} u_{\beta_\alpha, 1}^{q^{n_\alpha}}
    \]
    for nonzero constants $a_\alpha \in \ov{k}^\times$ (cf.\ the proof of \Cref{prop:highest_term_f_r}), the claim follows from \cite[Theorem 4.1]{Tak26_Yu}. Moreover, we get 
    \begin{equation} \label{eq:reduction_to_1/2}
        H_c^d(\msf{U}(1), f_1^* \cl{L}_\psi) \cong \cInd_{G_b(F)_{x_b, 1/2}}^{G_b(F)_{x_b, 0+}} H_c^w(\bb{A}_{1/2}, f_{1/2}^* \cl{L}_\psi)
    \end{equation}
    from \Cref{prop:induction_principle_n=1} (3) by setting $w = \dim \bb{A}_{1/2}$. In particular, the result loc. cit. implies 
    \[
        \dim H_c^d(\msf{U}(1), f_1^* \cl{L}_\psi) = \prod_{1/2 < r < 1} \lvert \msf{m}^\perp_{x_b, 1-r} \rvert \cdot \prod_{\alpha \in \Phi_{\mu < 0}^{1/2}} q^{n_\alpha} = \sqrt{\lvert \msf{m}^\perp \rvert}. 
    \]
    By \Cref{prop:positive_depth_diagonal_action_odd}, $M(F)_{x, 0+}$ acts trivially on $H_c^d(\msf{U}(1), f_1^* \cl{L}_\psi)$, so $G_b(F)_{x_b, 1, 1}$ acts by a scalar via $\psi \circ X$. Thus, the claim follows from the characterization in \Cref{prop:Heisenberg_odd}. 
\end{proof}

As in the case $m \geq 2$, it remains to compute the trace of $\gamma \in \msf{M}$ on $H_c^d(\msf{U}(1), f_1^* \cl{L}_\psi)$. Since \eqref{eq:reduction_to_1/2} is equivariant under $\msf{M}$, we may argue as in the case $m \geq 2$.

\begin{prop} \label{prop:cohomology_reduction_n=1}
    \Cref{prop:cohomology_reduction} holds for $n = 1$. 
\end{prop}
\begin{proof}
    By \Cref{prop:cohomology_reduction_n=1_partial}, we have an isomorphism
    \[
        H_c^d(\msf{U}(1), f_1^*\cl{L}_{\psi}) \cong H_c^d(\msf{W}(1)_\phi, \Qla)[\psi^{-1}] \cong \kappa_{b, 1, \phi} \otimes \chi 
    \]
    equivariant under $G_b(F)_{x_b, n, m_b} \rtimes M(F)_{x, 0}$ for some character $\chi \colon \msf{M} \to \Qlax$. It is enough to show that $\chi$ is trivial. By the proof of \Cref{prop:cohomology_reduction_n=1_partial} (cf.\ \eqref{eq:reduction_to_1/2}), $H_c^w(\bb{A}_{1/2}, f_{1/2}^* \cl{L}_\psi) \cong \kappa_{x_b, 1, \phi}$ as a $G_b(F)_{x_b, 1, 1/2}$-representation. Then, it is enough to show 
    \[
        \tr(\gamma\vert H_c^w(\bb{A}_{1/2}, f_{1/2}^* \cl{L}_\psi)) = \tr(\gamma \vert \kappa_{x_b, 1, \phi}) \quad (\gamma \in \msf{M}).
    \]
    Now, $f_{1 /2}$ is explicitly given by 
    \[
        \sum_{\alpha \in \Phi_{\mu < 0}^{1/2}} c_{\alpha, 1} X([du_\alpha, du_{-\alpha}]_{\msf{g}}) u_{\alpha, 1} u_{\beta_\alpha}^{q^{n_\alpha}}. 
    \]
    Since $c_{\alpha, 1}$ only depends on $\alpha\vert_{Z_{M}^\circ}$, we may replace the coordinate $u_{\alpha, 1}$ by $v_\alpha = t_\alpha u_{\alpha, 1}$ for some nonzero constant $t_\alpha$ depending only on $\alpha\vert_{Z_M^\circ}$, so that we have
    \[
        f_{1/2}(v_\alpha) = \sum_{\alpha \in \Phi_{\mu < 0}^{1/2}} X([du_\alpha, du_{-\alpha}]_{\msf{g}}) v_{\alpha, 1} v_{\beta_\alpha}^{q^{n_\alpha}}
    \]
    by the same argument as in \Cref{lem:normalization_t_factor}. Since the $\msf{M}$-action on $\bb{A}_{1/2} = \Spec(\ov{k}[v_\alpha \vert \alpha \in \Phi^{1/2}_{\mu < 0}])$ is identified with the adjoint action by \Cref{prop:equivariance_reduction_odd}, it follows as in \Cref{lem:M_action_on_W(n)_odd} that 
    \[
        H_c^w(\bb{A}_{1/2}, f_{1/2}^* \cl{L}_\psi) \cong H_c^w(\msf{W}^b_{\Phi_{\mu < 0}, 1}, \Qla)[\psi^{-1}]
    \]
    equivariantly under $\msf{M}$. Here, $\msf{W}^b_{\Phi_{\mu < 0}, 1}$ denotes the Heisenberg Deligne-Lusztig variety for $M \subset G_b$ at depth $1$. Then, the claim follows from \cite[Theorem 6]{Tak26_Yu}. 
\end{proof}

\section{Summary of properties of special affinoids} \label{sec:summary}

In this section, we collect the properties of special affinoids $\cl{W}(n)_\phi$ to show \Cref{thm:explicit_geometry}. Recall that $F$ is a finite extension of $\bb{Q}_p$ and $(M, \mu, x)$ is an unramified length-zero triple. 

\begin{thm} \label{thm:explicit_geometry_strongest}
    For $n \geq 0$, let $\cl{U}(n) \subset \cl{M}_{\cl{G}, b, \mu}$ be the open ball as in \Cref{defi:U(n)_radius} and let
    \[
        J_n = K_n, \quad 
        J_{b, n} = \left\{ \begin{alignedat}{4}
            & K_{b, n} & \; & (n = 2m) \\
            & K_{b, n - 1} & \; & (n = 2m - 1)
        \end{alignedat}
        \right.
    \]
    (see \Cref{defi:level_subgroups}). Then, we have the following. 
    \begin{enumerate}
        \item $\cl{U}(n)$ is stable under $J^0_{b, n}$ and $j \cdot \cl{U}(n) \cap \cl{U}(n) = \emptyset$ for $j \in G_b(F) - J_{b, n}^0$. Moreover, 
        \[
            \Delta(m)(\cl{U}(n)) = \cl{U}(n) \quad (m \in M(F)_x) ,\quad
            \sigma_E(\cl{U}(n)) = z_\mu \cdot \cl{U}(n)
        \]
        as open subsets of $\cl{M}_{\cl{G}, b, \mu}$. Here, $z_\mu$ acts via the $G_b(F)$-action (cf.\ \Cref{lem:zero_dimensional_LSV}). 
        \item The restriction of $\pi_\GM$ to $\bigsqcup_{j \in \cl{G}_b(O_F) / J^0_{b, n}} j \cdot \cl{U}(n)$ is an open immersion. 
        \item There is a unique $\breve{F}$-valued point $x_\CM \in \cl{U}(n) \cap \cl{M}_{\cl{M}, b, \mu}$ and a unique section
        \[
            \can_n \colon \cl{U}(n) \to \cl{M}_{G, b, \mu, J^0_n}
        \]
        sending $x_\CM$ into its natural lift via $\cl{M}_{\cl{M}, b, \mu} \subset \cl{M}_{G, b, \mu, J^0_n}$. Moreover, the open image $\Img(\can_n) \subset \cl{M}_{G, b, \mu, J^0_n}$ is stable under $J_{b, n}^0$. 
        \item For every $(G, G_b, M)$-supergeneric $\rho \in \Irr^\sm(M(F)_x)$ of depth $n$, the natural map
        \begin{equation*}
            R\Gamma_c(\cl{U}(n)_{\bb{C}_p}, \can_n^*\cl{L}_{R^{J_n}_M(\rho)})[R^{J_{b, n}}_M(\rho)] \to R\Gamma(\cl{U}(n)_{\bb{C}_p}, \can_n^*\cl{L}_{R^{J_n}_M(\rho)})[R^{J_{b, n}}_M(\rho)]
        \end{equation*}
        is an isomorphism and isomorphic to $R^{J_{b, n}}_M(\rho)[-d]$ in $\cl{D}(J_{b, n}, \Qla)$. 
    \end{enumerate}
    Moreover, let $\cl{U}^\infty(n) \subset \cl{M}_{G, b, \mu, \infty}$ be the inverse image of $\Img(\can_n)$. Then
    \[
        \Sht^b_{J_n, -\mu}(n) = \bigsqcup_{m \in M(F)_x / M(F)_{x, 0}} m \cdot \cl{U}^\infty(n) / \und{J_n \times J_{b, n}} = \Sht^1_{J_{b, n}, \mu}(n)
    \]
    are special open neighborhoods at levels $(J_n, J_{b, n})$ of depth $n$. 
\end{thm}
\begin{proof}   
    The case $n = 0$ is proved in \Cref{prop:verification_at_depth_zero}. For $n \geq 1$, each property is proved in the following propositions. 
    \begin{itemize}
        \item (1) is proved in \Cref{prop:U(n)_stability}, \Cref{prop:stability_U(n)_diagonal} and \Cref{prop:action_outside_J_bn}. 
        \item (2) follows from \Cref{thm:computation_period_map} since $\bigsqcup_{j \in \cl{G}_b(O_F) / J^0_{b, n}} j \cdot \cl{U}(n) \subset \cl{U}(0)$ by (1). 
        \item (3) follows from \Cref{cor:canonical_CM_structure} and \Cref{prop:stability_Img_can_n}. 
        \item (4) follows from \Cref{prop:nearby_cycle_positive_depth} and \Cref{prop:cohomology_reduction}. 
    \end{itemize}
    Moreover, the last claim follows from the construction in \Cref{prop:restatement_LSV}. 
\end{proof}

A key ingredient in the proof of (4) is the following property of positive-depth special affinoids $\cl{W}(n)_\phi$. Let $\phi \colon M(F) \to \Qlax$ be a $(G, G_b, M)$-supergeneric character of depth $n$ such that $\rho \otimes \phi^{-1}$ is trivial on $M(F)_{x, n}$ (see \Cref{prop:transfer_regular_supercuspidal}). 

\begin{thm} \label{thm:property_of_special_affinoids}
    For each $n \geq 1$, there is a unique refinement of $\can_n$ to a section
    \[
        \can_{n, m} \colon \cl{U}(n)_{\breve{F}_n} \to \cl{M}_{G, b, \mu, G(F)_{x, n, m}}
    \]
    sending $x_\CM$ to $x_{\CM, n}$. Let $\cl{W}(n)_\phi \subset \cl{M}_{G, b, \mu, K_\phi, \breve{F}_n}$ be the inverse image of $\Img(\can_{n, m})$. Then, $\cl{W}(n)_{\phi, \bb{C}_p}$ has good reduction and its reduction $\msf{W}(n)_\phi$ is isomorphic to a Heisenberg Deligne-Lusztig variety for $n \geq 2$. 
\end{thm}
\begin{proof}
    The construction of $\can_{n, m}$ is given in \Cref{thm:refined_trivialization}. Then, the latter claim is proved separately according to the parity of $n$. 
    \begin{itemize}
        \item When $n = 2m - 1$, the latter claim is proved in \Cref{prop:good_reduction_odd} and the comparison with Heisenberg Deligne-Lusztig varieties for $m \geq 2$ is given in \Cref{rmk:comparison_Heisenberg_Deligne_Lusztig}. 
        \item When $n = 2m$, the latter claim is proved in \Cref{prop:good_reduction_even} and the comparison with Heisenberg Deligne-Lusztig varieties is proved in \Cref{prop:equivariance_under_HW_even}. 
    \end{itemize}
    \vspace{-\topsep}
\end{proof}

As a direct consequence of the existence of special open neighborhoods, we obtain an explicit contribution of the Hecke operator at degree $0$. 

\begin{cor} \label{cor:Hecke_operator_unramified_CM}
    Let $\rho \in \Irr^\sm(M(F)_x)$ be $(G, G_b, M)$-supergeneric of depth $n$ and let $\pi = R^{G(F)}_M(\rho)$ and $\pi_b =  R^{G_b(F)}_M(\rho)$. For some smooth character $\xi_E \colon W_E \to \Qlax$, we have 
    \[
        \pi_b \boxtimes \xi_E \subset H^0(i_b^* T_{-\mu} i_{1!}(\pi)). 
    \]
    In particular, their Fargues-Scholze parameters are equal, i.e.\ $\varphi^\FS_\pi = \varphi^\FS_{\pi_b}$. 
\end{cor}
\begin{proof}
    It follows from \Cref{thm:explicit_geometry_strongest} and \Cref{prop:consequence_Hecke_operator}. 
\end{proof}

This holds without any restriction on $p$ and $q$. 

\begin{rmk}
    When $\pi$ and $\pi_b$ are constituents of Kaletha's regular supercuspidal $L$-packets \cite{Kal19}, the equality $\varphi^\FS_\pi = \varphi^\FS_{\pi_b}$ follows from \cite[Corollary 3.4.12]{Var24} and \cite[Corollary 1.3]{Han26} for sufficiently large $p$. 
\end{rmk}


Our construction of special open neighborhoods in this paper is carried out via the explicit geometry of local Shimura varieties. It would be interesting to find a direct stacky construction as in \Cref{thm:canonical_level_structures}. 


\renewcommand\bibfont{\footnotesize}
\printbibliography

\end{document}